\documentclass[10pt, letterpaper]{amsart}
\usepackage[margin=1in]{geometry}
\usepackage{preamble_article_BBK}

\title{On the Beilinson--Bloch--Kato conjecture for polarized motives}
\author{Hao Peng}

\begin{document}

\maketitle

\begin{abstract}
We study the Beilinson--Bloch--Kato conjecture for polarized motives. In the conjugate self-dual case, we show that if the central $L$-value does not vanish, then the associated Bloch--Kato Selmer group with coefficients in a suitable local field vanishes. In the self-dual analytic rank-zero case, we reduce the conjecture to a conjecture in the endoscopic Rankin--Selberg case related to the orthogonal Gross--Prasad periods.
\end{abstract}

\tableofcontents

\section{Introduction}

The Beilinson--Bloch--Kato conjecture for motives vastly generalizes the rank part of the Birch--Swinnerton-Dyer conjecture for elliptic curves. In this article, we study the Beilinson--Bloch--Kato conjecture for motives associated with self-dual (resp. conjugate self-dual) automorphic representations of $\GL_{2r}(\Ade_F)$, where $F\subset\bb C$ is a totally real number field (resp. a CM field).

Let $\Mot^\rat(F, E)$ denote the pseudo-Abelian category of Chow motives over $F$ with coefficients in a number field $E$ (see, for example, \cite{And04}). For the complex conjugation $\cc\in\Gal(\bb C/\bb R)$, a polarization of $M$ is an isomorphism $M^\cc\xr\sim M^\vee(1)$ in the category $\Mot^\rat(F, E)$. For any Chow motive $M\in \Mot^\rat(F, E)$ and any finite place $\lbd$ of $E$, there is a $\lbd$-adic realization $M_\lbd$, which is a representation of $\Gal(\ovl F/F)$ with $E_\lbd$-coefficients. We consider the Bloch--Kato Selmer group
\begin{equation*}
\bx H^1_f(F, M_\lbd)=\ker\bigg(\bx H^1(F, M_\lbd)\to\prod_{w\nmid \ell}\bx H^1(I_w, M_\lbd)\times\prod_{w\mid\ell}\bx H^1(F_w, M_\lbd\otimes_{\bb Q_\ell}\bb B_{\crys, \ell})\bigg),
\end{equation*}
where $\ell$ is the underlying rational prime of $\lbd$ and $\bb B_{\crys, \ell}$ is the $\ell$-adic crystalline period ring.

For example, if $A$ is an Abelian variety over $F$ of dimension $g$ and $M=h^{2g-1}(A)(g)$ is the Albanese motive of $A$ with coefficient field $\bb Q$, then $\bx H^1_f(F, M_\ell)$ is canonically isomorphic to
\begin{equation*}
\bb Q_\ell\otimes_{\bb Z_\ell}\plim_n\Sel_{\ell^n}(A/F)
\end{equation*}
for every rational prime $\ell$. Here $\Sel_k(A/F)$ is the mod-$k$ Selmer group of $A$ over $F$ for every positive integer $k$.

Suppose $M$ is a polarized motive in $\Mot^\rat(F, E)$ and $\lbd$ is a finite place of $E$. Conjecturally, the $L$-function $L(s, M_\lbd)$ attached to $M_\lbd$ has a meromorphic continuation to the entire complex plane, satisfying a functional equation
\begin{equation*}
L(s, M_\lbd)=\eps(M)c(M)^{-s}L(-s, M_\lbd)
\end{equation*}
where $\eps(M)\in\{\pm1\}$ is the root number and $c(M)$ is the conductor; see \cite{Del79}. Assuming this conjectural functional equation, we recall the following Beilinson--Bloch--Kato conjecture \cite{B-K90}. 

\begin{conj*}[Beilinson--Bloch--Kato]
\begin{equation*}
\ord_{s=0}L(s, M_\lbd)=\dim_{E_\lbd}\bx H_f^1(F, M_\lbd)-\dim_{E_\lbd}\bx H^0(F, M_\lbd).
\end{equation*}
\end{conj*}

We focus on the analytic rank-zero case, that is, when $L(0, M_\lbd)$ is nonzero. 

\subsection{The conjugate self-dual case}

Let $F\subset \bb C$ be a totally imaginary quadratic extension of a totally real number field $F_+\subset \bb R$. We first state a less technical main result.

\begin{mainthm}\label{ismsiieppwso}
Let $r$ be a positive integer and let $A$ be a modular elliptic curve over $F_+$. Suppose that $F_+\ne\bb Q$ whenever $r>1$, and that $A$ has no complex multiplication over $\ovl F$. If the central critical value
\begin{equation*}
L\paren{r, \Sym^{2r-1}A_F}
\end{equation*}
is nonzero, then for all but finitely many rational primes $\ell$, the Bloch--Kato Selmer group
\begin{equation*}
\bx H^1_f\paren{F, \Sym^{2r-1}\etH^1(A_{\ovl F}, \bb Q_\ell)(r)}
\end{equation*}
vanishes.
\end{mainthm}
\begin{rem}
The finite set of rational primes $\ell$ that are excluded in Theorem~\ref{ismsiieppwso} can be effectively bounded in terms of $F$, $A$, and $r$. The condition $F_+\ne\bb Q$ is imposed because Hypothesis~\ref{iifififmieiemss} on the cohomology of unitary Shimura varieties is not yet known for $N\ge 4$ if $F_+=\bb Q$. This condition is not used elsewhere.
\end{rem}
\begin{rem}
When $r=1$, Theorem~\ref{ismsiieppwso} recovers part of the Birch--Swinnerton-Dyer conjecture for $A_F$. If, in addition, $F_+=\bb Q$, then this is covered by Kolyvagin's work \cite{Kol90}, which introduced the \emph{Heegner point Euler system}; it uses the Gross--Zagier formula \cite{G-Z86} to pass to the analytic rank-one case. For $F_+\ne \bb Q$, the corresponding result was later established in \cites{Zha01, Lon06, Lon07, Nek12}.

When $r=1$ and $F_+=\bb Q$, there are other approaches to Theorem~\ref{ismsiieppwso}. In \cite{Kat04}, Kato used $p$-adic families of Beilinson elements in the $K$-theory of modular curves to construct Selmer classes via $p$-adic Hodge theory, known as \emph{Kato's Euler system}. In \cite{B-D05}, Bertolini and Darmon developed a different approach that constructs Selmer classes via level-raising congruences on Shimura curves, known as the \emph{bipartite Euler system} (see \cite{How06} for a systematic formulation). Under mild additional assumptions, these yield alternative proofs of Kolyvagin's result that do not invoke the Gross--Zagier formula.
\end{rem}

Theorem~\ref{ismsiieppwso} is a special case of a more general result concerning the Bloch--Kato Selmer groups of Galois representations attached to conjugate self-dual automorphic representations. We first introduce the automorphic representations we study.

\begin{defi}\label{eirinenriedusnisimaos}
An isobaric automorphic representation $\Pi$ of $\GL_N(\Ade_F)$ with $N\ge 2$ is called a (conjugate self-dual) \tbf{relevant automorphic representation} if 
\begin{enumerate}
\item
$\Pi$ is conjugate self-dual in the sense that its contragredient $\Pi^\vee$ is isomorphic to $\Pi\circ\cc$;
\item
for every Archimedean place $w$ of $F$, $\Pi_w$ is isomorphic to the irreducible principal series representation induced from the characters $(\arg^{1-N}, \arg^{3-N}, \ldots, \arg^{N-1})$, where $\arg: \bb C^\times\to \bb C^\times$ is the argument character defined by the formula $\arg(z)=z/\sqrt{z\ovl z}$;
\item
either one of the following holds:
\begin{enumerate}
\item
$\Pi$ is cuspidal.
\item
$N$ is odd and $\Pi$ is an isobaric sum of a character of $\GL_1(\Ade_F)$ with a cuspidal automorphic representation of $\GL_{N-1}(\Ade_F)$. 
\end{enumerate}
\end{enumerate}

A relevant representation $\Pi$ is called \tbf{almost cuspidal} if it is not cuspidal, in which case we write $\Pi=\Pi^\flat\boxplus\chi$, where $\chi$ is a conjugate self-dual character of $F^\times\bsh\Ade_F^\times$.
\end{defi}
\begin{rem}\label{ismsiehiemfies}
Note that our definition of relevant automorphic representations is slightly more general than that of \cite[Definition~1.1.3]{LTXZZ}: A representation $\Pi$ of $\GL_N(\Ade_F)$ is relevant in their sense if and only if it is cuspidal and relevant in our sense.

If $\Pi$ is a relevant automorphic representation, then it is regular algebraic in the sense of \cite[Defintion~3.12]{Clo90}. Moreover, it is well known that the Asai $L$-function $L(s, \Pi', \Asai^{(-1)^N})$ is regular at $s=1$ for each isobaric factor $\Pi'$ of $\Pi$ (see, for example, \cite[Theorem~9.1]{F-P23}).
\end{rem}

We now state our main result in terms of automorphic representations analogous to Theorem~\ref{ismsiieppwso}.

\begin{mainthm}\label{ismsieifmeifmss}
Let $r$ be a positive integer and $\Pi$ be a relevant automorphic representation of $\GL_{2r}(\Ade_F)$. Let $E$ be a strong coefficient field of $\Pi$ (see \textup{Definition~\ref{sisieifnieeimfsi}}). If the central critical value
\begin{equation*}
L(\frac{1}{2}, \Pi)
\end{equation*}
is nonzero, then for every admissible place $\lbd$ of $E$ with respect to $\Pi$, the Bloch--Kato Selmer group
\begin{equation*}
\bx H^1_f(F, \rho_{\Pi, \lbd}(r))
\end{equation*}
vanishes. Here $\rho_{\Pi, \lbd}$ is the Galois representation of $F$ with coefficients in $E_\lbd$ attached to $\Pi$ as described in \textup{Proposition~\ref{ieieeinfeieiites}} and \textup{Definition~\ref{sisieifnieeimfsi}}.
\end{mainthm}
\begin{rem}
In the setting of the unitary Friedberg--Jacquet periods, M.~Zanarella studied automorphic representations $\Pi$ in a framework close to ours, under the additional assumption that $\Pi$ is self-dual \cite{Zan24}. His argument relies on the conjecture of Leslie--Xiao--Zhang \cite{LXZ25}; see also \cite{LXZ25a} for recent progress on this conjecture. 
\end{rem}

The notion of admissible places appearing in Theorem~\ref{ismsieifmeifmss} is defined in Definition~\ref{issisoeeiriens}, which consists of a long list of assumptions. It is expected that all but finitely many finite places are admissible (with respect to $\Pi$). Indeed, we have the following family of abstract examples where all but finitely many finite places are admissible.

\begin{mainthm}\label{lsslienefifnieiwsws}
Let $r$ and $\Pi$ be as in \textup{Theorem~\ref{ismsieifmeifmss}}. Suppose that
\begin{enumerate}
\item
$F_+\ne \bb Q$ if $r>1$;
\item
There exists a finite place $w$ of $F$ \sut $\Pi_w$ is supercuspidal;
\item
There exists a good inert place $\mfk p$ of $F_+$ (see \textup{Definition~\ref{issieniefeifmies}}) \sut $\Pi_{\mfk p}$ is a Steinberg representation.
\end{enumerate}
Let $E$ be a strong coefficient field of $\Pi$ (see \textup{Definition~\ref{sisieifnieeimfsi}}). If the central critical value
\begin{equation*}
L(\frac{1}{2}, \Pi)
\end{equation*}
is nonzero, then for all but finitely many finite places $\lbd$ of $E$, the Bloch--Kato Selmer group
\begin{equation*}
•\bx H^1_f(F, \rho_{\Pi, \lbd}(r))
\end{equation*}
vanishes.
\end{mainthm}
\begin{rem}
In condition (b) of Theorem~\ref{lsslienefifnieiwsws}, if $F$ is Galois over $\bb Q$ or contains an imaginary quadratic field, then a good inert place of $F_+$ is just a finite place of $F_+$ that is inert in $F$.
\end{rem}

Using theta correspondence and a Burger--Sarnak type principle for Fourier--Jacobi periods on a pair of global unitary group $(\mbf U_{2r}, \mbf U_{2r})$, we reduce Theorem~\ref{ismsieifmeifmss} to the following theorem concerning central critical values of Rankin--Selberg $L$-functions. Let $n\ge 2$ be a positive integer. Denote by $n_0$ and $n_1$ the unique even and odd numbers in $\{n, n+1\}$, respectively.

\begin{mainthm}\label{ismsieiemiwmws}
Let $\Pi_0, \Pi_1$ be relevant automorphic representations of $\GL_{n_0}(\Ade_F)$ and $\GL_{n_1}(\Ade_F)$, respectively, such that $\Pi_0$ is cuspidal and $\Pi_1$ is almost cuspidal of the form $\Pi_1=\Pi^\flat_1\boxplus\uno$ where $\uno$ is the trivial representation of $\GL_1(\Ade_F)$. Assume $F_+\ne \bb Q$ if $n>2$, and assume there is a finite place $w$ of $F$ over a place of $F_+$ inert in $F$ \sut $\Pi^\flat_{1, w}$ is square-integrable. Let $E\subset\bb C$ be a strong coefficient field of both $\Pi_0$ and $\Pi_1$ (see \textup{Definition~\ref{sisieifnieeimfsi}}). If the central critical value
\begin{equation*}
L(\frac{1}{2}, \Pi_0)\cdot L(\frac{1}{2}, \Pi_0\times\Pi_1^\flat)
\end{equation*}
is nonzero, then for every admissible place $\lbd$ of $E$ with respect to $(\Pi_0, \Pi_1)$, the Bloch--Kato Selmer group
\begin{equation*}
\bx H^1_f(F, \rho_{\Pi_0, \lbd}(n_0/2))
\end{equation*}
vanishes.
\end{mainthm}
\begin{rem}
Theorem~\ref{ismsieiemiwmws} is analogous to one of the main results of \cite{LTXZZ} that concerns the analytic rank-zero case, where they assumed that $\Pi_1$ is relevant and cuspidal. Via the Gan--Gross--Prasad conjecture \cite{GGP12}, which is established in our case in \cite{BCZ22}, the theorem can be regarded as relating nonvanishing unitary Gan--Gross--Prasad periods on a pair of unitary groups $(\bx U_{2r}, \bx U_{2r+1})$ to the vanishing of Bloch--Kato Selmer groups.
\end{rem}
\begin{rem}
The notion of admissible places appearing in Theorem~\ref{ismsieiemiwmws} is defined in Definition~\ref{oiamsisieiwps}, which consists of a long list of assumptions. The admissibility condition here is weaker than the analogous admissibility condition in \cite[Definition~8.1.1]{LTXZZ}. It is expected that if the two automorphic representations $\Pi_0$ and $\Pi_1$ are not correlated in terms of Langlands functoriality, then all but finitely many finite places of $E$ are admissible with respect to $(\Pi_0, \Pi_1)$. For example, if we assume that
\begin{enumerate}
\item
$F$ is Galois over $\bb Q$ or contains an imaginary quadratic field,
\item
for each $\alpha\in\{0, 1\}$, there exists a finite place $w_\alpha$ of $F$ \sut $\Pi_{\alpha, w_\alpha}$ is supercuspidal, and
\item
there exists a finite place $\mfk p_+$ of $F_+$ underlying a unique place $\mfk p$ of $F$, 
\sut $\Pi_{0,\mfk p}$ is a Steinberg representation and $\Pi^\flat_{1, \mfk p}$ is unramified with Satake parameter not containing $1$,
\end{enumerate}
then all but finitely many finite places of $E$ are admissible with respect to $(\Pi_0, \Pi_1)$; see Lemma~\ref{isimsieienfifmeiss}.
\end{rem}

\subsection{The self-dual case}

We now state analogous conjectures in the self-dual case. Let $F\subset \bb R$ be a totally real number field.

\begin{mainconj}\label{osoppowiur}
Let $r$ be a positive integer and let $A$ be a modular elliptic curve over $F$ with no complex multiplication over $\ovl F$. If the central critical value
\begin{equation*}
L\paren{r, \Sym^{2r-1}A}
\end{equation*}
is nonzero, then for all rational primes $\ell$ greater than an effective constant depending only on $A$ and $r$, the Bloch--Kato Selmer group
\begin{equation*}
\bx H^1_f\paren{F, \Sym^{2r-1}\etH^1(A_{\ovl F}, \bb Q_\ell)(r)}
\end{equation*}
vanishes.
\end{mainconj}
\begin{rem}\enskip
\begin{enumerate}
\item
Theorem~\ref{ismsiieppwso} is implied by Conjecture~\ref{osoppowiur}. In fact, we can even drop the assumption $F_+\ne\bb Q$ in Theorem~\ref{ismsiieppwso} if Conjecture~\ref{osoppowiur} is true.
\item
When $r=1$, Conjecture~\ref{osoppowiur} is established by Kolyvagin \cite{Kol90} when $F=\bb Q$ using the Gross--Zagier formula \cite{G-Z86}, and later generalized to the case when $F\ne\bb Q$ in \cites{Zha01, Lon06, Lon07, Nek12}.
\item
When $r=2$ and $F=\bb Q$, Conjecture~\ref{osoppowiur} is known by work of H. Wang \cite{Wan22a} and N. Sweeting \cite{Swe25} using the bipartite Euler system method.
\end{enumerate}
\end{rem}

When $r\cdot [F: \bb Q]$ is even, it appears that Conjecture~\ref{osoppowiur} would follow from Theorem~\ref{ismsiieppwso} provided the following analytic statement holds:

\begin{description}
\item[$(\bx{NV}_r)$]For any elliptic curve $A$ over $F$ with no complex multiplication over $\ovl F$ such that $L\paren{r, \Sym^{2r-1}A}$ is nonzero, there exists a totally negative element $D\in F^\times$ effectively bounded by $F, A$, and $r$ satisfying that the central critical value
\begin{equation*}
L\paren{r, \Sym^{2r-1}A^D}
\end{equation*}
is nonzero, where $A^D$ is the twist of $A$ by the quadratic extension $F(\sqrt D)/F$.
\end{description}
If $r=1$, $(\bx{NV}_1)$ holds by the nonvanishing theorem of Friedberg--Hoffstein for quadratic twists with prescribed local behavior; cf. \cite[Theorem~B]{F-H95}. For $r\ge 2$, $(\bx{NV}_r)$ appears to lie beyond current techniques; even the case $r=2$ seems difficult (see, for example, \cites{R-Y23, BFKMMS, HJL23}).



Alternatively, using theta correspondence and a Burger--Sarnak type principle for Fourier--Jacobi periods on the symplectic-metaplectic pair $(\Sp_{2r}, \tld\Sp_{2r})$, we show that Conjecture~\ref{osoppowiur} can be reduced to another conjecture of Gan--Gross--Prasad type, relating nonvanishing of orthogonal Gross--Prasad periods to vanishing of Bloch--Kato Selmer groups.
\begin{mainconj}\label{ocneienipwowie}
Let $r$ be a positive integer and let $A$ be an elliptic curve over $F$ with no complex multiplication over $\ovl F$. Suppose that there exist
\begin{enumerate}
\item
a self-dual automorphic representation $\Pi$ of $\GL_{2r+1}(\mbf A_F)$ that is either cuspidal or an isobaric sum of a self-dual cuspidal automorphic representation of $\GL_{2r}(\Ade_F)$ with a nontrivial quadratic character of $F^\times\bsh \Ade_F^\times$;
\item
a pair $(\mbf V, \mbf V_\sharp)$ in which $\mbf V$ is a quadratic space of dimension $2r+1$ over $F$ that is positive definite at every Archimedean place of $F$ satisfying $-\disc(\mbf V)\notin F^{\times2}$, and $\mbf V_\sharp\defining \mbf V\oplus Fe$ where $e$ has norm $1$;
\item
cuspidal automorphic representations $\pi_0\subset\mcl A_0(\bx O(\mbf V))$ and $\pi_1\subset\mcl A_0(\bx O(\mbf V_\sharp))$ with trivial Archimedean components and with Arthur parameters $\Sym^{2r-1}A$ and $\Pi\boxplus\uno$, respectively (see \textup{Definition~\ref{formaml-parameineirmes}});\footnote{Here $\uno$ denotes the trivial automorphic character of $\GL_1(\Ade_F)$.} and
\item
cusp forms $f_0\in\pi_0$ and $f_1\in\pi_1$,
\end{enumerate}
\sut the orthogonal Gross--Prasad period
\begin{equation}\label{oeieiimfimes}
\mcl P_{\bx{GP}}(f_0, f_1)\defining \int_{\bx O(\mbf V)(F)\bsh \bx O(\mbf V)(\Ade_F)}f_0(h)f_1(\iota(h))\bx dh
\end{equation}
is nonzero. Here $\iota: \bx O(\mbf V)\inj\bx O(\mbf V_\sharp)$ is the embedding induced by the inclusion $\mbf V\subset\mbf V_\sharp$. Let $E\subset\bb C$ be a strong coefficient field of $\Pi$ (see Definition~\ref{psosieifmmiess}). Then there exists an effective constant $N(F, A, r)$ depending only on $F$, $A$, and $r$, \sut the Bloch--Kato Selmer group
\begin{equation*}
\bx H^1_f\paren{F, \Sym^{2r-1}\etH^1(A_{\ovl F}; \bb Q_\ell)(r)}
\end{equation*}
vanishes for all rational primes $\ell>N(F, A, r)$ underlying a preadmissible place $\lbd$ of $E$ with respect to $(A,\Pi)$.
\end{mainconj}
\begin{rem}
When $r=1$ and $F=\bb Q$, Conjecture~\ref{ocneienipwowie} is known by results of Y.~Liu \cite{Liu16} under suitable conditions, obtained using Hirzebruch--Zagier cycles and the bipartite Euler system method.
\end{rem}
\begin{rem}
The notion of preadmissible places appearing in Conjecture~\ref{ocneienipwowie} is a preliminary notion defined in Definition~\ref{aosppwoienuvneu}. It is expected that, if $\Pi$ is not correlated to $A$ in the sense of Langlands functoriality, then all but finitely many finite places of $E$ are admissible with respect to $(A, \Pi)$. For example, if there exist finite places $\mfk p, \mfk q$ of $F$ \sut
\begin{enumerate}
\item
$A$ has split multiplicative reduction at $\mfk p$, 
\item
$\Pi_{\mfk p}$ is unramified with Satake parameter of the form $\{-1, \alpha_1^{\pm1}, \ldots, \alpha_r^{\pm1}\}$ satisfying $\alpha_i\notin\{\pm1\}$ for every $1\le i\le r$, and
\item
$\Pi_{\mfk q}$ is either supercuspidal or an isobaric sum of a self-dual supercuspidal representation with a quadratic character,
\end{enumerate}
then there exists an effective constant $N(F, A,\Pi_{\mfk p}, \Pi_{\mfk q})$ depending on $F, A, \Pi_{\mfk p}$, and $\Pi_{\mfk q}$ \sut every finite place $\lbd$ of $E$ with underlying prime $\ell$ greater than $N(F, A, \Pi_{\mfk p}, \Pi_{\mfk q})$ is preadmissible \wrt $(A, \Pi)$; see Lemma~\ref{oeeoitiirunmes}.
\end{rem}

\begin{mainthm}\label{isioseoeimimfies}
If \textup{Conjecture~\ref{ocneienipwowie}} holds, then \textup{Conjecture~\ref{osoppowiur}} holds.
\end{mainthm}

\begin{rem}
In view of the Gross--Prasad conjecture for orthogonal groups \cite{G-P92, G-P94, I-I10}, Conjecture~\ref{ocneienipwowie} may be viewed as a natural analogue of Theorem~\ref{ismsieiemiwmws}. It will be studied in the author's forthcoming Ph.D. thesis \cite{Pen26} via orthogonal Shimura varieties and bipartite Euler system method, along the lines of the argument for Theorem~\ref{ismsieiemiwmws} (see \S\ref{oeieiitiidmiis}). In particular, for $F=\bb Q$ we expect to establish Conjecture~\ref{ocneienipwowie}, and hence also \textup{Conjecture~\ref{osoppowiur}}.
\end{rem}

\subsection{Strategy of proof}\label{oeieiitiidmiis}

The main innovation of this paper is an extensive use of local and global seesaw identities to deduce Theorem~\ref{ismsieifmeifmss} (resp. Conjecture~\ref{osoppowiur}) from Theorem~\ref{ismsieiemiwmws} (resp. Conjecture~\ref{ocneienipwowie}). The method of seesaw has proved to be a very useful tool in theta lifting of automorphic representations, yet our work seems to be the first to directly apply it to study arithmetic questions. 

For simplicity, we restrict to the self-dual case and assume $F=\bb Q$. Let $r$ be a positive integer and $A$ be an elliptic curve over $\bb Q$. By Newton--Thorne \cite{N-T21}, the odd symmetric power $\Sym^{2r-1}A$ is modular and associated with a self-dual cuspidal representation $\Pi_0$ of $\GL_{2r}(\Ade_{\bb Q})$. Rather than viewing $\Pi_0$ as the standard functorial transfer of a cuspidal automorphic representation on a special orthogonal group $\SO_{2r+1}$ as in previous work \cite{Liu16, LTXZZ, Zan24, Swe25}, we regard $\Pi_0$ as a generic elliptic $A$-parameter for the metaplectic group $\tld\Sp_{2r}$ in the Shimura--Waldspurger correspondence framework of \cite{G-I18}. In particular, by Arthur's multiplicity formula proved by Gan--Ichino \cite{G-I18}, there exists a genuine cuspidal automorphic representation $\tld\sigma_0$ of $\tld\Sp_{2r}(\Ade_{\bb Q})$ with $A$-parameter $\Pi_0$.\footnote{In fact, we twist $\Pi_0$ by a nontrivial quadratic character so that the quadratic space $\mbf V_{2r+1}$ (defined below) satisfies $-\disc(\mbf V_{2r+1})\notin\bb Q^{\times2}$.} Since the central critical value $L(\frac{1}{2}, \Pi_0)$ is nonzero (and $\Pi_0$ is tempered at every rational prime), the Rallis inner product formula \cite{Yam14}, together with local conservation relations, yields a positive definite quadratic space $\mbf V_{2r+1}$ of dimension $2r+1$ over $\bb Q$ \sut the global theta lift of $\tld\sigma_0$ to $\bx O(\mbf V_{2r+1})(\Ade_{\bb Q})$ is a nonzero cuspidal automorphic representation $\pi_0$ with trivial Archimedean components. 

We use the seesaw diagram
\begin{equation*}
\begin{tikzcd}[sep=large]
\Sp_{2r}\times\Sp_{2r}\ar[dr, dash] &\bx O(\mbf V_{2r+2})\\
\Sp_{2r}\ar[u, dash]\ar[ur, dash] &\bx O(\mbf V_{2r+1})\times\bx O(\mbf V_1),\ar[u, dash]
\end{tikzcd}
\end{equation*}
where $\mbf V_1=\bb Qe$ is a 1-dimensional quadratic space with $\norml{e}=1$, and $\mbf V_{2r+2}=\mbf V_{2r+1}\oplus \mbf V_1$. Fix a sufficiently large rational prime $\ell$ and a nontrivial additive character $\psi$ of $\bb Q\bsh\Ade_{\bb Q}$. Suppose we can find a cuspidal automorphic representation $\sigma_1$ of $\Sp_{2r}(\Ade_{\bb Q})$ such that the \tbf{Fourier--Jacobi period integral}
\begin{equation*}
\mcl{FJ}(\tilde\vp_0, \vp_1; \phi)\defining\int_{\Sp_{2r}(\bb Q)\bsh\Sp_{2r}(\Ade_{\bb Q})}\tilde\vp_0(g)\vp_1(g)\theta(g; \phi)\bx dg
\end{equation*}
is nonzero on the pair $(\sigma_1, \tld\sigma_0)$ for some Schwartz function $\phi$, where $\theta(g; \phi)$ is the theta function. Then it follows from the global seesaw identity that the theta lift of $\sigma_1$ to $\bx O(\mbf V_{2r+2})$ is a nonzero cuspidal automorphic representation $\pi_1$, and the orthogonal Gross--Prasad period integral \eqref{oeieiimfimes} is nonzero on the pair $(\pi_0, \pi_1)$. If we can further guarantee that
\begin{enumerate}
\item
$\pi_1$ has trivial Archimedean component,
\item
the Arthur parameter of $\pi_1$ is of the form $\Pi\boxplus\uno$ as in the statement of Conjecture~\ref{ocneienipwowie}; and
\item
$\ell$ underlies a preadmissible place $\lbd$ of $E$ with respect to $(A,\Pi)$,
\end{enumerate}
then Conjecture~\ref{osoppowiur} follows.

The shape of $\pi_1$ is determined by the shape of $\sigma_1$ via Prasad's conjecture \cite{A-G17}. Fix a large prime $p$. If the $L$-parameter of $\sigma_{1, p}$ contains a chosen $2r$-dimensional irreducible local Galois representation $\phi_p$ as a subrepresentation, we would know $\Pi$ is either cuspidal or an isobaric sum of a self-dual cuspidal automorphic representation and a quadratic Dirichlet character. Note that we cannot guarantee that $\Pi$ is cuspidal, because there exist no irreducible self-dual local Galois representations of odd dimension greater than one (when $p$ is odd); see \cite[Proposition~4]{Pra99}. Condition (1) and (3) would follow if we can control the Archimedean place of $\sigma_1$ and can choose $\phi_p$ with desired good properties.

To achieve these requirements, we prove a Burger--Sarnak type principle for Fourier--Jacobi periods on the pair $(\Sp_{2r}, \tld\Sp_{2r})$, in the spirit of \cites{B-S91, H-L98, Pra07, Zha14}. More precisely, suppose
\begin{enumerate}
\item
$\sigma_{1, p}$ is a supercuspidal representation of $\Sp_{2r}(\bb Q_p)$ that is induced from a compact open subgroup \sut the pair $(\tld\sigma_{0, p}, \sigma_{1, p})$ satisfies the Fourier--Jacobi case of the local Gan--Gross--Prasad restriction problem:
\begin{equation}\label{oieneifmeiiws}
\Hom_{\Sp_{2r}(\bb Q_p)}(\tld\sigma_{0, p}\otimes\omega_{\psi_p}\otimes\sigma_{1, p}, \bb C)\ne 0,
\end{equation}
where $\omega_{\psi_p}$ is the local Weil representation associated to $\psi_p$.
\item
The contragredient of $\sigma_{1, \infty}$ is a holomorphic discrete series of $\Sp_{2r}(\bb R)$ with scalar lowest $K$-type of weight $(r+1, \ldots, r+1)$.
\end{enumerate}
We show that there exists a cuspidal automorphic representation $\sigma_1$ of $\Sp_{2r}$ which globalizes $\sigma_{1, p}$ and $\sigma_{1, \infty}$ simultaneously, \sut the Fourier--Jacobi period integral on the pair $(\sigma_1, \sigma_0)$ is nonzero. The local restriction condition \eqref{oieneifmeiiws} then follows from the (now established) local Gan--Gross--Prasad conjecture, Prasad's conjecture, and a local seesaw identity; see \S\ref{ieiieiepeoies}.

The local Galois representation $\phi_p$ used to enforce the pre-admissibility condition at some place $\lbd$ above $\ell$ will be constructed in Appendix~\ref{oggieeotiieimfiiwis}. Let $\iota_\ell: \bb C\xr\sim\ovl{\bb Q_\ell}$ be a fixed isomorphism which induces a place $\lbd$ of $E$. We require $\phi_p$ to satisfy:
\begin{enumerate}
\item
$\phi_p$ is self-dual of orthogonal type;
\item
$\phi_p$ is residually absolutely irreducible;
\item
there exists an arithmetic Frobenius lift $\Frob_p\in\Gal(\ovl{\bb Q_p}/\bb Q_p)$ \sut the eigenvalues $\{\alpha_1, \ldots, \alpha_{2r}\}$ of $\iota_\ell\rho(\Frob_p)$ are $\ell$-adic units and their reductions in $\ovl{\bb F_\ell}$ avoid a prescribed finite subset of $\ovl{\bb F_\ell}$; moreover, $\alpha_i^2\ne p^2\alpha_j^2$ in $\ovl{\bb F_\ell}$ for any $1\le i\ne j\le 2r$.
\end{enumerate}
The explicit construction is more complicated than we expected. Indeed, in the conjugate self-dual variant, we need to split and distribute the analogous requirements between two distinct finite places.

We now turn to the conjugate self-dual setting and discuss the proof of Theorem~\ref{ismsieiemiwmws}, which is another main theorem. Following the bipartite Euler system arguments via level-raising congruences, pioneered by Bertolini and Darmon for Shimura curves \cite{B-D05}, we bound the Bloch--Kato Selmer group by constructing global Galois cohomology classes that are deeply ramified at prescribed primes. These classes originate from the cohomology of products of unitary Shimura varieties attached to (standard indefinite) unitary groups $\mbf U_{n_0}$ and $\mbf U_{n_1}$ via level-raising congruences, and are realized as the image of the diagonal cycle under the Hecke-localized Abel--Jacobi map. Their ramifications are detected by relating them to unitary Gan--Gross--Prasad periods on definite Shimura sets through the basic uniformization of the special fibers of the integral models---this is the so-called \emph{first explicit reciprocity law}.

Our argument follows \cite{LTXZZ} but requires modifications for the almost cuspidal setting. The results of \cite{LTXZZ} do not apply verbatim, since several of their standing hypotheses are tailored to the cuspidal case. For example, the computation of the Hecke--Galois module of the Shimura varieties is more delicate: when $\pi$ is a cuspidal representation of $\bx U_{n_1}$ with base change $\BC(\pi)\cong \Pi_1=\Pi_1^\flat\boxplus\uno$, then the $\pi^\infty$-isotypic part of the middle-degree (projective limit) cohomology of the Shimura variety
\begin{equation*}
\etH^{n_1-1}\paren{\bSh(\mbf U_{n_1})_{\ovl F}, \ovl{\bb Q_\ell}\paren{\frac{n_1-1}{2}}}[\iota_\ell\pi^\infty]
\end{equation*}
is a $\ovl{\bb Q_\ell}[\Gal(\ovl F/F)]$-module isomorphic to either the trivial character or the Galois representation $\rho^\cc_{\Pi_1^\flat, \lbd}(r)$, determined by Arthur's multiplicity formula. Here $\iota_\ell: \bb C\xr\sim\ovl{\bb Q_\ell}$ is a fixed isomorphism inducing a place $\lbd$ of $E$. More subtly, our construction of $\Pi_1$ via the Burger--Sarnak type principle fixes its local components of finitely many places, but does not a priori control ramifications at the remaining places. To compensate, we replace the notion of admissibility of \cite{LTXZZ} with a weaker variant adapted to the analytic rank-zero situation.

As in \cite{LTXZZ}, the geometric input has two parts: (i) the study of Tate cycles in the special fiber of the semistable integral model of Shimura varieties attached to $\mbf U_{n_1}$; and (ii) an arithmetic level-raising property for Shimura varieties attached to $\mbf U_{n_0}$. In our application, non-cuspidality is allowed only on the odd-unitary side, while the even-unitary representation remains cuspidal. Moreover, the Tate cycle argument works provided the Satake parameter at the given unramified place is ``generic'' enough.

Finally, under the hypotheses of Theorem~\ref{ismsieiemiwmws}, we are not able to prove the vanishing of the larger Bloch--Kato Selmer group
\begin{equation*}
\bx H^1_f(F, \rho_{\Pi_0, \lbd}\otimes\rho_{\Pi_1, \lbd}(n)),
\end{equation*}
although this is predicted by the Beilinson--Bloch--Kato conjecture. This limitation is intrinsic to our simplifying conditions when applying the bipartite Euler system method: Let $\bx R^\flat$ be a self-dual lattice in $\rho_{\Pi_0, \lbd}\otimes\rho_{\Pi_1^\flat, \lbd}(n)$. For any very good inert place $\mfk p$ of $F_+$ at which both arithmetic level-raising and the Tate cycle conditions apply, we cannot show deep ramification of the Hecke-localized Abel--Jacobi image of the diagonal cycle in the singular part of the local Galois cohomology
\begin{equation*}
\bx H^1_\sing(F_{\mfk p}, \bx R^\flat/\lbd^m),
\end{equation*}
for any $m\ge 1$. Indeed, under further conditions we impose, these cohomology spaces vanish; see~\S\ref{psleiieuiremfs}. We do not know how to circumvent this limitation.

Let us briefly summarize this article. In \S\ref{oigmeietemfueusmws}, we recall certain background materials related to automorphic representations and Galois representations. In \S\ref{oeiidimruuemfes}, we consider the conjugate self-dual Rankin--Selberg case. In \S\S\ref{oeoeieuturuesw}--\ref{ptiieirmeiehfieims}, we collect certain background results from \cite{LTXZZ} and extend them to the almost cuspidal situation. In \S\S\ref{Tiemsopeoireisss}--\ref{sosieemiifeiifemiws}, we 
compute the local part of the Abel--Jacobi image of the diagonal cycle. In \S\ref{osiieeuoeoutyrui}, we define the notion of admissible places in the almost cuspidal situation, and check in good situations that all but finitely many finite places are admissible. In \S\ref{psleiieuiremfs}, we prove Theorem~\ref{ismsieiemiwmws}. In \S\ref{ieiieiepeoies}, we collect the necessary background results related to theta correspondence that will be used in \S\ref{ososieifieiifmes}. Finally, in \S\ref{ososieifieiifmes}, we apply the Burger--Sarnak type principle and seesaw relation to prove the main theorems: Theorems~\ref{ismsiieppwso}, \ref{ismsieifmeifmss} and \ref{lsslienefifnieiwsws} are proved in \S\ref{psoseieiruefeimsis}, and Theorem~\ref{isioseoeimimfies} is proved in \S\ref{psosieuueiures}. In Appendix~\ref{oggieeotiieimfiiwis}, we construct certain (conjugate) self-dual local Galois representations with good properties, which will be used in the Burger--Sarnak type principle for Fourier--Jacobi periods.

\subsection{Notation and conventions}

In this subsection, we set up some common notations and conventions for the entire article, including the appendix.

\begin{note}[Generalities]\enskip
\begin{itemize}
\item
Let $\bb N =\{0, 1, 2, 3, . . . \}$ be the monoid of nonnegative integers and set $\bb Z_+=\bb N\setm\{0\}$. We write $\bb Z$, $\bb Q$, $\bb R$, and $\bb C$ for the integers, rational numbers, real numbers, and complex numbers, respectively.
\item
We take square roots only of positive real numbers and always choose the positive root.
\item
For any set $S$, we denote by $\uno_S$ the characteristic function of $S$, and by $\id_S: S\to S$ the identity map. We write $\id$ for $\id_S$ if $S$ is clear from context. Let $\#S$ be the cardinality of $S$.
\item
For any set $X$, let $\uno\in X$ denote the distinguished trivial element (this notation is only used when the notion of triviality is clear from context).
\item
The eigenvalues or generalized eigenvalues of a matrix over a field $k$ are counted with multiplicity, i.e., by the dimension of the corresponding eigenspace or generalized eigenspace.
\item
For each rational prime $p$, we fix an algebraic closure $\ovl{\bb Q_p}$ of $\bb Q_p$ with residue field $\ovl{\bb F_p}$. For every integer $r\in\bb Z_+$, we denote by $\bb Q_{p^r}$ the unique unramified extension of $\bb Q_p$ of degree $r$ inside $\ovl{\bb Q_p}$, and by $\bb F_{p^r}$ its residue field.
\item
We use standard notations from category theory. The category of sets is denoted by $\Set$. The category of schemes is denoted by $\Sch$.
\item
All rings are commutative and unital, and ring homomorphisms preserve units.
\item
If a base ring is not specified in the tensor operation $\otimes$, then it is $\bb Z$.
\item
For a ring $L$ and a set $S$, denote by $L[S]$ the $L$-module of $L$-valued functions on $S$ of finite support.
\item
For each square matrix $M$ over a ring, we write $M^\top$ for its transpose.
\item
Suppose $\tld\Gamma, G$ are groups, $\Gamma\subset \tld\Gamma$ is a subgroup, and $L$ is a ring.
\begin{itemize}
\item
We denote by $\Gamma^\ab$ the maximal abelian quotient of $\Gamma$;
\item
For a homomorphism $\rho: \Gamma\to \GL_r(L)$ for some $r\in\bb Z_+$, we denote by $\rho^\vee: \Gamma\to \GL_r(L)$ the contragredient homomorphism, which is defined by the formula $\rho^\vee(x)=(\rho(x)^\top)^{-1}$.
\item
For a group homomorphism $\rho: \Gamma\to G$ and an element $\gamma\in\tld\Gamma$ that normalizes $\Gamma$, let $\rho^\gamma: \Gamma\to G$ denote the homomorphism defined by $\rho^\gamma(x)=\rho(\gamma x\gamma^{-1})$.
\item
We say that two homomorphisms $\rho_1, \rho_2: \Gamma\to G$ are conjugate if there exists an element $g\in G$ \sut $\rho_1=g\circ\rho_2\circ g^{-1}$.
\end{itemize}
\item
For any positive integer $n\in\bb Z_+$, let $\mu_n$ denote the finite diagonalizable group scheme over $\bb Z$ of $n$-th roots of unity.
\item
Denote by $\cc\in\Gal(\bb C/\bb R)$ the complex conjugation.
\item
For each field $k$, we denote by $\chr k$ the characteristic of $k$.
\item
If $G$ is a real Lie group or a totally disconnected locally compact group and $\pi$ is an irreducible admissible representation of $G$, we denote by $\pi^\vee$ the contragredient of $\pi$. We do not use $\tilde\pi$ for the contragredient of $\pi$.
\end{itemize}
\end{note}

\begin{note}[Number fields]
A subfield of $\bb C$ is called a number field if it is a finite extension of $\bb Q$. Suppose $F$ is a number field.
\begin{itemize}
\item
We denote by $\mcl O_F$ the ring of integers of $F$. We will not distinguish between prime ideals of $\mcl O_F$ and the corresponding finite places of $F$; we denote by $\fPla_F$ the set of finite places of $F$, by $\Pla^\infty_F=\Hom(F, \bb C)$ the set of infinite places (also called Archimedean places) of $F$, and by $\Pla_F=\fPla_F\cup \Pla^\infty_F$ the set of all places of $F$.
\item
For each finite set $\Pla$ of finite places of $F$, we write
\begin{equation*}
\Ade_{F, \Pla}\defining\prod\nolimits_{v\in\Pla}F_v, \quad \Ade_F^\Pla\defining\prod_{v\in \Pla_F\setm\Pla} \nolimits' F_v, \quad \Ade_F\defining \Ade_F^\infty\times(F\otimes_{\bb Q}\bb R)\end{equation*}
If $F=\bb Q$, we omit $\bb Q$ from the notation.
\item
Let $\ovl F$ denote the Galois closure of $F$ in $\bb C$, and set $\Gal_F=\Gal(\ovl F/F)$.
\item
For each rational prime $\ell$, let $\ve_\ell: \Gal_F\to\bb Z_\ell^\times$ denote the $\ell$-adic cyclotomic character. If $v$ is a finite place of $F$, we continue to write $\ve_\ell$ for its restriction to $\Gal_{F_v}$.
\item
We fix the following conventions. For each finite place $v\in\fPla_F$:
\begin{itemize}
\item
write $\mcl O_v$ and $F_v$ for the completion of $\mcl O_F$ (resp. $F$) at $v$;
\item
let $\kappa_v$ denote the residue field, $\norml{v}\defining\#\kappa_v$, and write $\chr \kappa_v$ for the residue characteristic.
\item
we fix an algebraic closure $\ovl{F_v}$ of $F_v$ and an embedding $\iota_v: \ovl F\inj \ovl{F_v}$ extending $F\inj F_v$; via $\iota_v$ we regard $\Gal_{F_v}\defining \Gal(\ovl{F_v}/F_v)$ as a decomposition subgroup of $\Gal_F$;
\item
for any map $r: \Gal_F\to X$, we write $r_v\defining r|_{\Gal_{F_v}}$;
\item
let $I_v\subset \Gal_{F_v}$ denote the inertia subgroup;
\item
fix an algebraic closure $\ovl{\kappa_v}$ of $\kappa_v$, and identify $\Gal_{\kappa_v}\defining \Gal(\ovl{\kappa_v}/\kappa_v)$ with $\Gal_{F_v}/I_v$,
\item
fix $\phi_v\in \Gal_{F_v}$ lifting the arithmetic Frobenius in $\Gal_{\kappa_v}$, and
\item
let $W_{F_v}$ denote the Weil group, and denote by $\Art_v: F_v^\times\to W_{F_v}^\ab$ the local reciprocity map (also called the Artin map), normalized so that uniformizers are sent to geometric Frobenius classes.
\item
for every automorphism $\tau\in\Aut(F)$, denote by $v^\tau$ the place defined by $v^\tau(x)\defining v(\tau^{-1}x)$ for every $x\in F$.
\end{itemize}
\item
For each finite set $S$ of rational primes, set $\Pla_F(S)\defining \{v\in\fPla_F: \chr\kappa_v\in S\}$. If $S=\{p\}$ is a singleton, we write simply $\Pla_F(p)\defining \Pla_F(\{p\})=\{v\in\fPla_F: v|p\}$.
\item
Two subsets $\Pla_1, \Pla_2$ of finite places of $F$ are called \tbf{strongly disjoint} if $\{\chr\kappa_v: v\in\Pla_1\}$ is disjoint from $\{\chr\kappa_v: v\in\Pla_2\}$.
\end{itemize}
\end{note}

\begin{note}[Automorphic representations]
Suppose $F$ is a number field. Let $G$ be either the metaplectic double cover $\tld\Sp_{2n}$ of a symplectic group $\Sp_{2n}$ over $F$ or an algebraic group over $F$ whose central connected component is a connected reductive group.
\begin{itemize}
\item
If $G$ is a metaplectic group $\tld\Sp_{2n}$, then an automorphic form $f$ on $G(\Ade_F)$ is called \tbf{genuine} if the nontrivial element in $\ker\big(\tld\Sp_{2n}(\Ade_F)\to\Sp_{2n}(\Ade_F)\big)$ acts by $-1$ on $f$. For simplicity, if $G$ is not metaplectic, then every automorphic form on $G(\Ade_F)$ is called genuine. We denote by $\mcl A_0(G(\Ade_F))$ the space of genuine cusp forms on $G(\Ade_F)$.
\item
Suppose $\pi$ is an automorphic representation of $G(\Ade_F)$.
\begin{itemize}
\item
We write $\pi_v$ for its local component at $v$, for every place $v$ of $F$.
\item
We denote by $\pi^\vee$ the contragredient of $\pi$.
\item
For any automorphism $\tau\in\Aut(F)$, we denote by $\pi^\tau$ the automorphic representation of $G(\Ade_F)$ satisfying $\pi^\tau_v\cong\pi_{v^\tau}$ for every place $v$ of $F$.
\item
For any cuspidal genuine automorphic representation $\pi\subset \mcl A_0(G(\Ade_F))$, we write $\ovl\pi$ for its conjugation.
\end{itemize}
\end{itemize}
\end{note}

\subsection{Acknowledgments}

I wish to thank Rui Chen and Jialiang Zou for many valuable discussions on the theta correspondence. I am grateful to Yifeng Liu for sharing an earlier draft of \cite{LTXZZb}. I also thank Weixiao Lu, Hang Xue, and Murilo C. Zanarella for helpful conversations, and Daniel Disegni, Zhiyu Zhang for comments on an earlier draft. Finally, I am deeply indebted to my Ph.D. advisor, Wei Zhang, for his invaluable guidance and encouragement.

\section{Automorphic representations and Galois representations}\label{oigmeietemfueusmws}

In this section, we introduce the automorphic representations relevant to us and their associated Galois representations.

\subsection{The conjugate self-dual case}

In this subsection, we fix a positive integer $N\in\bb Z_+$, an imaginary quadratic extension $F$ of a totally real number field $F_+$, and a relevant representation $\Pi$ of $\GL_N(\Ade_F)$ (see~Definition~\ref{eirinenriedusnisimaos}).

If $\mbf V$ is a Hermitian space of dimension $N$ over $F$ and $\pi$ is a discrete automorphic representation of $\bx U(\mbf V)(\Ade_{F_+})$, let $\BC(\pi)$ denote the automorphic base change of $\pi$ as defined in~\cite[Definition~3.2.3]{LTXZZ} (see also Definition~\ref{formaml-parameineirmes}), which always exists by \cite[Theorem~2.1]{C-Z24}.

\begin{prop}\label{ieieeinfeieiites}\enskip
\begin{enumerate}
\item
For every finite place $w$ of $F$, $\Pi_w$ is tempered.
\item
Suppose $\Pi$ is cuspidal. For every rational prime $\ell$ and every isomorphism $\iota_\ell: \bb C\xr\sim\ovl{\bb Q_\ell}$, there exists a semisimple continuous homomorphism
\begin{equation*}
\rho_{\Pi, \iota_\ell}: \Gal_F\to \GL_N(\ovl{\bb Q_\ell}),
\end{equation*}
unique up to conjugation, satisfying that
\begin{equation*}
\WD_\ell\paren{\rho_{\Pi, \iota_\ell}|_{\Gal_{F_w}}}^{\bx F\dash\sems}\cong\iota_\ell\rec_N\paren{\Pi_w\otimes\largel{\det}^{\frac{1-N}{2}}},
\end{equation*}
for every finite place $w$ of $F$, where $\rec_N$ is the local Langlands correspondence for $\GL_N(F_w)$. Moreover, $\rho_{\Pi, \iota_\ell}^\cc$ and $\rho_{\Pi, \iota_\ell}^\vee(1-N)$ are conjugate.
\item
Suppose $N$ is odd and $\Pi=\Pi^\flat\boxplus\chi$ is almost cuspidal. For every rational prime $\ell$ and every isomorphism $\iota_\ell: \bb C\xr\sim\ovl{\bb Q_\ell}$, there exist semisimple continuous homomorphisms
\begin{equation*}
\rho_{\Pi^\flat, \iota_\ell}: \Gal_F\to \GL_{N-1}(\ovl{\bb Q_\ell}), \quad \rho_{\chi, \iota_\ell}: \Gal_F\to \GL_1(\ovl{\bb Q_\ell}),
\end{equation*}
unique up to conjugation, satisfying that
\begin{equation*}
\WD_\ell\paren{\rho_{\Pi^\flat, \iota_\ell}|_{\Gal_{F_w}}}^{\bx F\dash\sems}\cong\iota_\ell\rec_N\paren{\Pi^\flat_w\otimes\largel{\det}^{\frac{1-N}{2}}}
\end{equation*}
and
\begin{equation*}
\WD_\ell\paren{\rho_{\chi, \iota_\ell}|_{\Gal_{F_w}}}^{\bx F\dash\sems}=\iota_\ell\paren{\chi_w\otimes\largel{\det}^{\frac{1-N}{2}}}\circ\Art_w^{-1},
\end{equation*}
for every finite place $w$ of $F$, where $\rec_{N-1}$ is the local Langlands correspondence for $\GL_{N-1}(F_w)$. Moreover, $\rho_{\Pi^\flat, \iota_\ell}$ and $\rho_{\Pi^\flat, \iota_\ell}^\vee(1-N)$ are conjugate. Let $\rho_\Pi$ denote the direct sum Galois representation $\rho_{\Pi^\flat, \iota_\ell}\boxplus \rho_{\chi, \iota_\ell}$.
\end{enumerate}
\end{prop}
\begin{proof}
These follow from standard results, see for example \cite[Theorem~3.2.3]{C-H13}, \cite[Theorem~1.1]{Car12} and \cite[Theorem~1.1]{Car14}
\end{proof}
\begin{rem}
If $\chi$ is trivial, then $\rho_{\chi, \iota_\ell}$ equals $\ve_\ell^{(1-N)/2}$.
\end{rem}

\begin{lm}\label{iseinifiemesss}
Let $\ell$ be a rational prime with a fixed isomorphism $\iota_\ell: \bb C\xr\sim\ovl{\bb Q_\ell}$. If $N$ is odd, then $\rho_{\Pi, \iota_\ell}(\frac{N-1}{2})$ is pure of weight 0 at every finite place $w$ of $F$. If $N$ is even, then $\rho_{\Pi, \iota_\ell}(\frac{N}{2})$ is pure of weight $-1$ at every finite place $w$ of $F$.
\end{lm}
\begin{proof}
It suffices to show that $\rho_{\Pi, \iota_\ell}(\frac{N-1}{2})$ (resp. $\rho_{\Pi, \iota_\ell}(\frac{N}{2})$) is pure of some weight when $N$ is odd (resp. even). By \cite[Lemma~1.4(3)]{T-Y07} and Proposition~\ref{ieieeinfeieiites}, this follows from the fact that $\Pi_w$ is tempered for any finite place $w$ of $F$.
\end{proof}

\subsection{The self-dual case}

In this subsection, we fix a positive integer $r$ and a totally real number field $F$. Let $\Pla^\bad$ denote the (finite) set of finite places of $F$ whose underlying rational prime ramifies in $F$.

\begin{defi}\label{oreieigieifiewp}
An isobaric automorphic representation $\Pi$ of $\GL_{2r+1}(\Ade_F)$ is called a (self-dual) \tbf{relevant automorphic representation} if 
\begin{enumerate}
\item
$\Pi$ is self-dual in the sense that its contragredient $\Pi^\vee$ is isomorphic to $\Pi$;
\item
$\Pi$ has nontrivial central character $\chi_{(-1)^{r+1}\mfk d}$, where $\chi_{(-1)^{r+1}\mfk d}$ is the quadratic character of $\Ade_F$ attached to a quadratic extension $F(\sqrt{(-1)^{r+1}\mfk d})$ of $F$, where $\mfk d$ is a totally positive element in $F^\times$;
\item
$\Pi_\infty$ has infinitesimal character $(r-1, r-2, \ldots, 1-r)$; and
\item
$\Pi$ is either cuspidal or an isobaric sum of a cuspidal automorphic representation of $\GL_{2r}(\Ade_F)$ and a nontrivial quadratic character of $F^\times\bsh \Ade_F^\times$.
\end{enumerate}
\end{defi}

We fix a relevant representation $\Pi$ of $\GL_{2r+1}(\Ade_F)$, and denote by $\Pla^\Pi$ the smallest (finite) set of finite places of $F$ containing $\Pla^\bad$ \sut $\Pi_v$ is unramified for every finite place $v$ of $F$ not in $\Pla^\Pi$.

\begin{prop}\label{ddddddeieopws}\enskip
\begin{enumerate}
\item
For every finite place $v$ of $F$, $\Pi_v$ is tempered.
\item
For every rational prime $\ell$ and every isomorphism $\iota_\ell: \bb C\xr\sim\ovl{\bb Q_\ell}$, there exists a semisimple continuous homomorphism
\begin{equation*}
\rho_{\Pi, \iota_\ell}: \Gal_F\to \GL_{2r+1}(\ovl{\bb Q_\ell}),
\end{equation*}
unique up to conjugation, satisfying that
\begin{equation*}
\WD_\ell\paren{\rho_{\Pi, \iota_\ell}|_{\Gal_{F_w}}}^{\bx F\dash\sems}\cong\iota_\ell\rec_{2r+1}\paren{\Pi_w\otimes\largel{\det}^{-r}},
\end{equation*}
for every finite place $v$ of $F$, where $\rec_{2r+1}$ is the local Langlands correspondence for $\GL_{2r+1}(F_v)$. Moreover, $\rho_{\Pi, \iota_\ell}$ and $\rho_{\Pi, \iota_\ell}^\vee(-2r)$ are conjugate.
\end{enumerate}
\end{prop}
\begin{proof}
These follow from standard results, see for example \cite[Theorem~3.2.3]{C-H13}, \cite[Theorem~1.1]{Car12} and \cite[Theorem~1.1]{Car14}
\end{proof}

\begin{defi}\label{ocoeieifihiemfeos}
For each finite place $w$ of $F$ not lying above $\Pla^\Pi_+$, let $\bm\alpha(\Pi_w)$ denote the Satake parameter of $\Pi_w$, and let $\bb Q(\Pi_w)$ denote the subfield of $\bb C$ generated by the coefficients of the polynomial
\begin{equation*}
\prod_{\alpha\in\bm\alpha(\Pi_w)}\paren{T-\alpha}\in\bb C[T].
\end{equation*}

We define the \tbf{coefficient field} (also called the \tbf{Hecke field}) of $\Pi$ to be the compositum of the fields $\bb Q(\Pi_w)$ for all finite places $w$ of $F$ not lying above $\Pla^\Pi_+$, denoted by $\bb Q(\Pi)$. 
\end{defi}

\begin{defi}\label{psosieifmmiess}
We say a number field $E\subset \bb C$ is a \tbf{strong coefficient field} of $\Pi$ if $E$ contains $\bb Q(\Pi)$, and for every finite place $\lbd$ of $E$ with underlying prime $\ell$, there exists a continuous homomorphism
\begin{equation*}
\rho_{\Pi, \lbd}: \Gal_F\to \GL_{2r+1}(E_\lbd)
\end{equation*}
necessarily unique up to conjugation, such that for every isomorphism $\iota_\ell: \bb C\xr\sim \ovl{\bb Q_\ell}$ inducing the place $\lbd$,  $\rho_{\Pi, \lbd}\otimes_{E_\lbd}\ovl{\bb Q_\ell}$ and $\rho_{\Pi, \iota_\ell}$ (see Proposition~\ref{ddddddeieopws}) are conjugate.
\end{defi}
\begin{rem}\label{oeietureps}
By the argument of \cite[Proposition~3.2.5]{C-H13}, a strong coefficient field of $\Pi$ exists.
\end{rem}

\subsection{Galois theoretic arguments}

Let $F_+$ be a subfield of $\bb R$ and $F$ be a quadratic extension of $F_+$ contained in $\bb C$ that is not contained in $\bb R$. We fix an odd rational prime $\ell$ that is unramified in $F$, and consider a finite extension $E_\lbd/\bb Q_\ell$, with ring of integers $\mcl O_\lbd$ and the maximal ideal $\lbd$ of $\mcl O_\lbd$. We freely use the notation of \cite[\S2]{LTXZZ}. For example,

\begin{itemize}
\item
If $\Gamma$ is a topological group and $L$ is a $\bb Z_\ell$-ring that is finite over either $\bb Z_\ell$ or $\bb Q_\ell$, then an $L[\Gamma]$-module $M$ is called \tbf{weakly semisimple} if $M$ is an object of $\Mod(\Gamma, L)$, and the natural map $M^\Gamma\to M_\Gamma$ is an isomorphism.
\item
For each positive integer $N\in\bb Z_+$, we define the group scheme $\mrs G_N\defining (\GL_N\times\GL_1)\rtimes\{1, \cc\}$ with $\cc^2=1$ and $\cc (g, \mu)\cc=(\mu g^{-\top}, \mu)$ for $(g, \mu)\in \GL_N\times\GL_1$. Denote by $\nu: \mrs G_N\to \GL_1$ the homomorphism \sut $\nu|_{\GL_N\times\GL_1}$ is the projection to the $\GL_1$ factor and $\nu(\cc)=-1$.
\item
For an $\mcl O_\lbd$-module $M$ and an element $x\in M$, the \tbf{exponent} of $x$ is defined to be 
\begin{equation*}
\exp_\lbd(x, M)\defining\min\{d\in\bb N\cup\{\infty\}|\lbd^dx=0\}.
\end{equation*}
\item
For a finite place $w$ of $F$ over $\ell$ and an object $\bx R$ in $\Mod(F_w, \mcl O_\lbd)$ that is crystalline with Hodge--Tate weights in $[a, b]$ where $b$ and $-a$ are nonnegative integers and $b-a\le \ell-2$, let $\bx H^1_\ns(F_w, \bx R)$ denote the $\mcl O_\lbd$-submodule of $\bx H^1(F_w, \bx R)$ consisting of elements $s$ represented by an extension
\begin{equation*}
0\to \bx R_0\to \bx R_s\to \bb Z_\ell\to 0
\end{equation*}
in the category $\Mod(F_w, \bb Z_\ell)$ \sut $\bx R_s$ is crystalline. Here $\bx R_0$ is the underlying $\bb Z_\ell[\Gal_{F_w}]$-module of $\bx R$.
\item
For a finite place $w$ of $F$ not over $\ell$ and an object $\bx R$ in $\Mod(F_w, \mcl O_\lbd)$, we set $\bx H^1_\sing(F_w, \bx R)\defining \bx H^1(I_{F_w}, \bx R)^{\Gal_{\kappa_w}}$, and denote by $\bx H^1_\ns(F_w, \bx R)$ the kernel of the canonical map
\begin{equation*}
\partial_w: \bx H^1(F_w, \bx R)\to \bx H^1_\sing(F_w, \bx R).
\end{equation*}
$\bx H^1_\ns(F_w,\bx R)$ is canonically isomorphic to $\bx H^1(\kappa_w, \bx R^{I_{F_w}})$.
\end{itemize}

We recall the following definition of Bloch--Kato Selmer groups from \cite{B-K90}.

\begin{defi}
For an object $\bx R\in \Mod(F, E_\lbd)$, the \tbf{Bloch--Kato Selmer group} $\bx H_f^1(F, \bx R)$ attached to $\bx R$ is defined to be
\begin{equation*}
\bx H^1_f(F, \bx R)\defining \ker\bigg(\bx H^1(F, \bx R)\to\prod_{w\in\fPla_F\setm\Pla_F(p)}\bx H^1_\sing(F_w, \bx R)\times\prod_{w\in\Pla_F(p)}\bx H^1(F_w, \bx R\otimes_{\bb Q_\ell}\bb B_\crys)\bigg)
\end{equation*}
\end{defi}

\begin{defi}
For an object $\bx R\in \Mod(F, \mcl O_L)_{\bx{fr}}$, the (integral) \tbf{Bloch--Kato Selmer group} $\bx H^1_f(F, \bx R)$ attached to $\bx R$ is defined to be the inverse image of $\bx H^1_f(F, \bx R\otimes\bb Q)$ under the natural map
\begin{equation*}
\bx H^1(F, \bx R)\to\bx H^1(F, \bx R\otimes\bb Q).
\end{equation*}
Moreover, for each $m\in \bb Z_+\cup\{\infty\}$, the (mod-$\lbd^m$) \tbf{Bloch--Kato Selmer group} $\bx H^1_{f, \bx R}(F, \ovl{\bx R}^{(m)})$ is defined to be the image of $\bx H^1_f(F, \bx R)$ under the natural map $\bx H^1(F, \bx R)\to \bx H^1(F, \ovl{\bx R}^{(m)})$.
\end{defi}

To end this subsection, we study two ``general image'' conditions for integral Galois modules.

\begin{lm}\label{osiseiheifmiesw}
Let $F'/F_+$ be a totally real finite Galois extension contained in $\bb R$ and a polynomial $\mrs P(T)\in \bb Z[T]$. For each $\alpha\in\{0, 1\}$, we take an object $\bx R_\alpha\in \Mod(F, \mcl O_\lbd)_{\bx{fr}}$ with the associated homomorphism $\rho_\alpha: \Gal_F\to \GL(\bx R_\alpha)$, together with a $(1-\alpha)$-polarization $\Xi: \bx R_\alpha^\cc\xr\sim \bx R_\alpha^\vee(1-\alpha)$. We assume that $\rk\bx R_0=n_0=2r_0$ is even and $\rk\bx R_1=n_1=2r_1+1$ is odd. Set $\bx R=\bx R_0\otimes\bx R_1$ and $\Xi: \bx R\xr\sim\bx R^\vee(1)$. For every positive integer $m\in\bb Z_+$, consider the following statement
\begin{description}
\item[$(\bx{GI}^m_{\bx R_0, \bx R_1, F', \mrs P})$] The image of the restriction of the homomorphism
\begin{equation*}
\paren{\ovl\rho^{(m)}_{\bx R_0,+}, \ovl\rho^{(m)}_{\bx R_1,+}, \ovl\ve_\ell^{(m)}}: \Gal_F\to \mrs G_{n_0}(\mcl O_\lbd/\lbd^m)\times\mrs G_{n_1}(\mcl O_\lbd/\lbd^m)\times(\mcl O_\lbd/\lbd^m)^\times
\end{equation*}
(see \cite[Notation~2.6.1]{LTXZZ}) to $\Gal_{F'}$ contains the element $(\gamma_0, \gamma_1, \xi)$ satisfying
\begin{enumerate}[(a)]
\item
$\mrs P(\xi)$ is invertible in $\mcl O_\lbd/\lbd^m$;
\item
for each $\alpha\in\{0, 1\}$, $\gamma_\alpha$ belongs to $\GL_{n_\alpha}(\mcl O_\lbd/\lbd^m)\times(\mcl O_\lbd/\lbd^m)^\times\times\{\cc\}$ with order coprime to $\ell$;
\item
the kernels of $(h_{\gamma_0}-1)^{n_0}, (h_{\gamma_1}-1)^{n_1}$ and $(h_{\gamma_0}\otimes h_{\gamma_1}-1)$ (see \cite[Notation~2.6.2]{LTXZZ}) are all free over $\mcl O_\lbd/\lbd^m$ of rank $1$;
\item
for each $\alpha\in\{0, 1\}$, $h_{\gamma_\alpha}$ does not have an eigenvalue that is equal to $-\xi$ in $\kappa_\lbd$.
\end{enumerate}
\end{description}
Then $(\bx{GI}^1_{\bx R_0, \bx R_1, F', \mrs P})$ implies $(\bx{GI}^m_{\bx R_0, \bx R_1, F', \mrs P})$ for every $m\in\bb Z_+$.
\end{lm}
\begin{proof}
This is \cite[Lemma~2.7.1]{LTXZZ}.
\end{proof}

\begin{lm}\label{ososoieieiehfimiws}
Let $F'/F_+$ be a totally real finite Galois extension contained in $\bb R$ and a polynomial $\mrs P(T)\in \bb Z[T]$. We take object $\bx R\in \Mod(F, \mcl O_\lbd)_{\bx{fr}}$ with the associated homomorphism $\rho: \Gal_F\to \GL(\bx R)$, together with a $1$-polarization $\Xi: \bx R^\cc\xr\sim \bx R^\vee(1)$. We assume that $\rk\bx R=2r$. For every positive integer $m\in\bb Z_+$, consider the following statement

\begin{description}
\item[$(\bx{GI}^m_{\bx R, F', \mrs P})$] The image of the restriction of the homomorphism
\begin{equation*}
\paren{\ovl\rho^{(m)}_{\bx R,+}, \ovl\ve_\ell^{(m)}}: \Gal_F\to \mrs G_{2r}(\mcl O_\lbd/\lbd^m)\times(\mcl O_\lbd/\lbd^m)^\times
\end{equation*}
(see \cite[Notation~2.6.1]{LTXZZ}) to $\Gal_{F'}$ contains the element $(\gamma, \xi)$ satisfying
\begin{enumerate}[(a)]
\item
$\mrs P(\xi)$ is invertible in $\mcl O_\lbd/\lbd^m$;
\item
$\gamma$ belongs to $\GL_{2r}(\mcl O_\lbd/\lbd^m)\times(\mcl O_\lbd/\lbd^m)^\times\times\{\cc\}$ with order coprime to $\ell$;
\item
the kernels of $(h_\gamma-1)^{2r}$ (see \cite[Notation~2.6.2]{LTXZZ}) is all free over $\mcl O_\lbd/\lbd^m$ of rank $1$;
\item
$h_\gamma$ does not have an eigenvalue that is equal to $-\xi$ in $\kappa_\lbd$.
\end{enumerate}
\end{description}
Then $(\bx{GI}^1_{\bx R, F', \mrs P})$ implies $(\bx{GI}^m_{\bx R, F', \mrs P})$ for every $m\in\bb Z_+$.
\end{lm}
\begin{proof}
The argument of \cite[Lemma~2.7.1]{LTXZZ} goes through.
\end{proof}

\section{The conjugate self-dual Rankin--Selberg case}\label{oeiidimruuemfes}

In this section, we adapt the argument of \cite{LTXZZ} to prove Theorem~\ref{ismsieiemiwmws}. While our setup is not identical to that of \cite{LTXZZ}, we align our notation with theirs whenever possible and record any deviations as they arise. Fix a positive integer $N\ge2$ and set $r=\floor{\frac{N}{2}}$. We work in the following setting.

\begin{setup}\enskip
\begin{itemize}
\item
Let $F_+\subset\bb R$ be a totally real number field and let $F\subset\bb C$ be a quadratic CM extension of $F_+$.
\item
Denote by $\Pla^\infty$ (resp. $\Pla^\infty_+$) the set of Archimedean places of $F$ (resp. $F_+$), with $\tau_\infty$ (resp. $\udl\tau_\infty$) the default one induced by the inclusion $F\subset \bb C$ (resp. $F_+\subset\bb R$).
\item
Let $\Pla^\bad_+$ denote the (finite) set of finite places of $F_+$ whose underlying rational prime ramifies in $F$.
\item
For any place $v$ of $F_+$, we set $\mcl O_{F_v}\defining \mcl O_v\otimes_{\mcl O_{F_+}}\mcl O_F$ and $F_v\defining F\otimes_{F_+}F_v$.
\item
For every place $w$ of $F$ with underlying place $v$ of $F_+$, we identify $\Gal_{F_w}$ with $\Gal_{F_{+, v}}\cap\Gal_F$ (resp. $\cc(\Gal_{F_{+, v}}\cap\Gal_F)\cc$), if the embedding $\iota_v: \ovl F\to \ovl{F_{+, v}}$ induces (resp. does not induce) the place $w$.
\end{itemize}
\end{setup}

\subsection{Unitary Satake parameters and unitary Hecke algebras}\label{oeoeieuturuesw}

We recall the notation of the coefficient field for an automorphic representation of $\GL_N(\Ade_F)$. Let $\Pi$ be an irreducible relevant automorphic representation of $\GL_N(\Ade_F)$ that is cuspidal (resp. almost cuspidal) when $N$ is even (resp. $N$ is odd).

\begin{defi}\label{ifitneramieiiefolacies}
We denote by $\Pla^\Pi_+$ the smallest finite set of (finite) places of $F_+$ containing $\Pla^\bad_+$ so that $\Pi_w$ is unramified for every finite place $w$ of $F$ not lying above $\Pla^\Pi_+$.
\end{defi}

\begin{defi}\label{issienefienisws}
For each ring $L$, we define an \tbf{abstract Satake parameter} in $L$ of rank $N$ to be a multi-set $\bm\alpha$ consisting of $N$ elements in $L$. For two Satake parameters $\bm\alpha, \bm\alpha'$ in $L$ of dimension $n$ and $n'$, respectively, we can form their tensor product $\bm\alpha\otimes\bm\alpha'$ in the natural way, which is an abstract Satake parameter of dimension $nn'$.
\end{defi}

\begin{defi}\label{isioswooiemis}\enskip
\begin{itemize}
\item
For each finite place $w$ of $F$ not lying above $\Pla^\Pi_+$, let $\bm\alpha(\Pi_w)$ denote the Satake parameter of $\Pi_w$, which is an abstract Satake parameter in $\bb C$ of dimension $N$ (see Definition~\ref{issienefienisws}), and let $\bb Q(\Pi_w)$ denote the subfield of $\bb C$ generated by the coefficients of the polynomial
\begin{equation*}
\prod_{\alpha\in\bm\alpha(\Pi_w)}\bigg(T-\alpha\sqrt{\norml{w}}^{N-1}\bigg)\in\bb C[T].
\end{equation*}
\item
We define the \tbf{coefficient field} of $\Pi$ to be the compositum of the fields $\bb Q(\Pi_w)$ for all finite places $w$ of $F$ not lying above $\Pla^\Pi_+$, denoted by $\bb Q(\Pi)$, 
\item
For each finite place $v$ of $F_+$ not in $\Pla^\Pi_+$ and inert in $F$, the abstract Satake parameter $\bm\alpha(\Pi_v)$ at $v$ of rank $N$ is defined in \cite[Notation~3.14]{LTXZZ}, which is an abstract Satake parameter in $\bb C$ of dimension $N$.
\end{itemize}
\end{defi}

\begin{defi}\label{llieienieiheires}
Let $v$ be a finite place of $F_+$ inert in $F$, $L$ be a ring, and $P\in L[T]$ be a monic polynomial.
\begin{itemize}
\item
When $N$ is odd, we say $P$ is \tbf{Tate generic at $v$} if $P'(1)$ is invertible in $L$.
\item
When $N$ is odd, we say $P$ is \tbf{intertwining generic at $v$} if $P(-\norml{v})$ is invertible in $L$.
\item
When $N$ is even, we say $P$ is \tbf{level-raising special at $v$} if $P(\norml{v})=0$ and $P'(\norml{v})$ is invertible in $L$.
\item
When $N$ is even, we say $P$ is \tbf{intertwining generic at $v$} if $P(-1)$ is invertible in $L$.
\end{itemize}
\end{defi}

\begin{lm}
The coefficient field $\bb Q(\Pi)$ is a number field.
\end{lm}
\begin{proof}
We take a standard pair $(\mbf V, \pi)$ (in the sense of Definition~\ref{ssowosieiniemfes}) \sut $\BC(\pi)$ is isomorphic to $\Pi$ and $\pi_v$ is unramified for each finite place $v$ of $F_+$ not in $\Pla^\Pi_+$. Such a standard pair always exists by Arthur's multiplicity formula \cite[Theorem~1.7.1]{KMSW}. For each finite place $v$ of $F_+$, we denote by $\bb Q(\pi_v)$ the fixed field of the group
\begin{equation*}
\{\tau\in\Aut(\bb C): \pi_v\otimes_{\bb C, \tau}\bb C\cong \pi_v\}.
\end{equation*}
If $v$ is not contained in $\Pla^\Pi_+$, then $\bb Q(\BC(\pi_v))$ equals $\bb Q(\pi_v)$ by \cite[Lemma 2.25 and Lemma~4.5]{S-T14}. 
Moreover, the composite field of $\bb Q(\pi_v)$ for all finite places $v$ not in $\Pla^\Pi_+$ is a number field by \cite[Proposition~2.15]{S-T14}. Thus the assertion follows. 
\end{proof}

\begin{defi}\label{sisieifnieeimfsi}
We say a number field $E\subset \bb C$ is a \tbf{strong coefficient field} of $\Pi$ if $E$ contains $\bb Q(\Pi)$, and for every finite place $\lbd$ of $E$ with underlying prime $\ell$, there exists a continuous homomorphism
\begin{equation*}
\rho_{\Pi, \lbd}: \Gal_F\to \GL_N(E_\lbd)
\end{equation*}
necessarily unique up to conjugation, such that for every isomorphism $\iota_\ell: \bb C\xr\sim \ovl{\bb Q_\ell}$ inducing the place $\lbd$,  $\rho_{\Pi, \lbd}\otimes_{E_\lbd}\ovl{\bb Q_\ell}$ and $\rho_{\Pi, \iota_\ell}$ (see Proposition~\ref{ieieeinfeieiites}) are conjugate.
\end{defi}
\begin{rem}
By the argument of \cite[Proposition~3.2.5]{C-H13}, a strong coefficient field of $\Pi$ exists. Moreover, if $N$ is odd and $\Pi$ is almost cuspidal of the form $\Pi=\Pi^\flat\boxplus\uno$ where $\uno$ is the trivial character of $\GL_1(\Ade_F)$, then for every strong coefficient $E$ with a finite place $\lbd$, the homomorphism $\rho_{\Pi, \lbd}$ is of the form
\begin{equation*}
\rho_{\Pi, \lbd}=\rho_{\Pi^\flat, \lbd}\oplus\ve_\ell^{(1-N)/2}.
\end{equation*}
\end{rem}

%

\begin{defi}
For any $\mcl O_{F_+}$-ring $R$, a \tbf{Hermitian space} over $\mcl O_F\otimes_{\mcl O_{F_+}}R$ of dimension $N$ is a projective $\mcl O_F\otimes_{\mcl O_{F_+}}R$-module $V$ of rank $N$ together with a  perfect pairing
\begin{equation*}
(-, -)_V: V\times V\to\mcl O_F\otimes_{\mcl O_{F_+}}R
\end{equation*}
that is $\mcl O_F\otimes_{\mcl O_{F_+}}R$-linear in the first variable and $(\mcl O_F\otimes_{\mcl O_{F_+}}R, \cc\otimes\id)$-linear in the second variable, and satisfies $(x, y)_V=(y, x)_V^\cc$ for any $x, y\in V$. We write $\bx U(V)$ for the group of $R$-linear isometries of $V$, which is a reductive group scheme over $R$.

We denote by $V_\sharp\defining V\oplus Re$ the orthogonal direct sum Hermitian space where we set $\norml{e}=1$. If $f: V\to V'$ is an isometry of Hermitian spaces over $R$, we write $f_\sharp: V_\sharp\to V'_\sharp$ for the induced isometry of Hermitian spaces over $\mcl O_F\otimes_{\mcl O_{F_+}}R$.
\end{defi}

\begin{defi}\label{INIENIEMisimw}\enskip
\begin{enumerate}
\item
For a finite place $v$ of $F_+$ not in $\Pla_+$, let $\Lbd_{N, v}$ denote the unique up to isomorphism Hermitian space over $\mcl O_{F_v}$ of dimension $N$, and $\bx U_{N, v}$ its unitary group over $\mcl O_{F_{+,v}}$. We define spherical Hecke algebra
\begin{equation*}
\bb T_{N, v}\defining \bb Z[\bx U_{N, v}(\mcl O_v)\bsh \bx U_{N, v}(F_{+, v})/\bx U_{N, v}(\mcl O_v)].
\end{equation*}
\item
For a finite set $\Pla_+$ of finite places of $F_+$ containing $\Pla^\bad_+$, we define the abstract unitary Hecke algebra away from $\Pla_+$ to be the restricted tensor product ring
\begin{equation*}
\bb T_N^{\Pla_+}\defining \bigotimes_v{}'\bb T_{N, v}
\end{equation*}
over all finite places of $F_+$ not in $\Pla_+$, \wrt the unit elements.
\item
The Hecke character $\phi_\Pi: \bb T_N^{\Pla^\Pi_+}\to\bb C$ attached to $\Pi$ is defined in \cite[Construction~3.1.10]{LTXZZ}. By \cite[Lemma 2.2.3]{B-G14}, $\phi_\Pi$ takes value in $\bb Q(\Pi)$. Furthermore, $\phi_\Pi$ takes values in $\mcl O_{\bb Q(\Pi)}$. In fact, if we take a standard pair $(\mbf V, \pi)$ (in the sense of Definition~\ref{ssowosieiniemfes}) \sut $\BC(\pi)$ is isomorphic to $\Pi$ and $\pi_v$ is unramified for each finite place $v$ of $F_+$ not in $\Pla^\Pi_+$. Such a standard pair always exists by Arthur's multiplicity formula \cite[Theorem~1.7.1]{KMSW}. Then $\phi_\Pi$ is identical to the spherical Hecke character of $\pi$, which is easily seen to be valued in algebraic integers. 
\end{enumerate}
\end{defi}

\subsection{Unitary Shimura varieties}

Let $\mbf V$ be a Hermitian space over $F$ of dimension $N$.

\begin{defi}\label{ssowosieiniemfes}\enskip
\begin{enumerate}
\item
Recall from \cite[Definition~3.1.11]{LTXZZ} that, for any finite set $\Pla_+$ of finite places of $F_+$ we have
\begin{enumerate}
\item
a category $\mfk K(\mbf V)^{\Pla_+}$ whose objects are neat compact open subgroups of $\bx U(\mbf V)(\Ade_{F_+}^{\infty, \Pla})$ and whose morphisms are double cosets. There is also a subcategory $\mfk K(\mbf V)^{\Pla_+}$ of $\mfk K(\mbf V)^{\Pla_+}$ consisting of the same objects but allowing only identity double cosets; and
\item
a category $\mfk K(\mbf V)_\sp^{\Pla_+}$ consisting of pairs $(\mdc K_\flat, \mdc K_\sharp)$, where $\mdc K_\flat$ (resp. $\mdc K_\sharp$) is an object of $\mfk K(\mbf V)^{\Pla_+}$ (resp. $\mfk K(\mbf V_\sharp)^{\Pla_+}$) \sut $\mdc K_\flat$ is contained in $\mdc K_\sharp$. There are the obvious functors 
\begin{equation*}
(-)_\flat: \mfk K(\mbf V)_\sp^{\Pla_+}\to \mfk K(\mbf V)^{\Pla_+}, \quad (-)_\sharp: \mfk K(\mbf V)_\sp^{\Pla_+}\to\mfk K(\mbf V_\sharp)^{\Pla_+}.
\end{equation*}
When $\Pla_+$ is the empty set, we suppress it from all the notations above.
\end{enumerate}
\item
We say $\mbf V$ is standard definite if it has signature $(N, 0)$ at each real place of $F_+$. We say $\mbf V$ is standard indefinite if it has signature $(N-1, 1)$ at $\udl\tau_\infty$ and $(N, 0)$ at other real places of $F_+$.
\item
For a discrete automorphic representation $\pi$ of $\bx U(\mbf V)(\Ade_{F_+})$, we say $(\mbf V, \pi)$ is a standard pair if one of the following holds:
\begin{enumerate}
\item
$\mbf V$ is standard definite, and $\pi^\infty$ appears in
\begin{equation*}
\lim_{\mdc K\in\mfk K'(\mbf V)}\bb C[\bSh(\mbf V, \mdc K)].
\end{equation*}
\item
$\mbf V$ is standard indefinite, and $\pi^\infty$ appears in
\begin{equation*}
\lim_{\mdc K\in\mfk K'(\mbf V)}\iota_\ell^{-1}\etH^i(\bSh(\mbf V, \mdc K)_{\ovl F}, \ovl{\bb Q_\ell}),
\end{equation*}
for some rational prime $\ell$ with isomorphism $\iota_\ell:\bb C\xr\sim \ovl{\bb Q_\ell}$ and some $i\in\bb N$.
\end{enumerate}
\end{enumerate}
\end{defi}

\begin{prop}\label{issiieifimeifes}
Let $\pi$ be a discrete automorphic representation of $\bx U(\mbf V)(\Ade_{F_+})$ \sut $(\mbf V, \pi)$ is a standard pair. For every rational prime $\ell$ and every isomorphism $\iota_\ell: \bb C\xr\sim \ovl{\bb Q_\ell}$, there exists a semisimple continuous homomorphism
\begin{equation*}
\rho_{\BC(\pi), \iota_\ell}: \Gal_F\to \GL_N(\ovl{\bb Q_\ell}),
\end{equation*}
unique up to conjugation, satisfying that
\begin{equation}
\WD_\ell\paren{\rho_{\BC(\pi), \iota_\ell}|_{\Gal_{F_w}}}^{\bx F\dash\sems}\cong\iota_\ell\rec_N\paren{\BC(\pi)_w\otimes\largel{\det}^{\frac{1-N}{2}}},
\end{equation}
for every finite place $w$ of $F$, where $\rec_N$ is the local Langlands correspondence for $\GL_N(F_w)$. Moreover, $\rho_{\Pi, \iota_\ell}^\cc$ and $\rho_{\Pi, \iota_\ell}^\vee(1-N)$ are conjugate.
\end{prop}
\begin{proof}
This follows from Arthur's multiplicity formula \cite[Theorem~2.6]{C-Z24} and standard results, see for example \cite[Theorem~3.2.3]{C-H13}, \cite[Theorem~1.1]{Car12} and \cite[Theorem~1.1]{Car12}. 
\end{proof}

When $\mbf V$ is standard definite (resp. standard indefinite), there are functors $\bSh(\mbf V, -): \mfk K(\mbf V)\to \Set$ (resp. $\bSh(\mbf V, -): \mfk K(\mbf V)\to \Sch/F$) of Shimura sets (Shimura varieties) attached to $\Res_{F_+/\bb Q}\bx U(\mbf V)$, as defined in~\cite[\S3.2]{LTXZZ}.

\begin{hypothesis}\label{iifififmieiemss}
Suppose $\mbf V$ is a standard indefinite Hermitian space over $F$ of dimension $N$, and $\pi$ is a discrete automorphic representation of $\bx U(\mbf V)(\Ade_{F_+})$ \sut the functorial lift $\BC(\pi)$ is a relevant automorphic representation of $\GL_N(\Ade_F)$ (see Definition~\ref{formaml-parameineirmes}). For every isomorphism $\iota_\ell: \bb C\xr\sim \ovl{\bb Q_\ell}$, we consider the $\ovl{\bb Q_\ell}[\Gal_F]$-module
\begin{equation*}
W^{N-1}(\pi^\infty)\defining \Hom_{\ovl{\bb Q_\ell}[\bx U(\mbf V)(\Ade_{F_+}^\infty)]}\bigg(\iota_\ell\pi^\infty, \ilim_{\mfk K'(\mbf V)}\etH^{N-1}\paren{\bSh(\mbf V, \mdc K)_{\ovl F}, \ovl{\bb Q_\ell}}\bigg).
\end{equation*}
\begin{enumerate}
\item
If $\rho_{\BC(\pi), \iota_\ell}$ is irreducible, then $W^{N-1}(\pi^\infty)$ is isomorphic to $\rho_{\BC(\pi), \iota_\ell}^\cc$.
\item
If $N$ is odd, $\BC(\pi)=\Pi^\flat\boxplus\chi$ is almost cuspidal, and $\rho_{\Pi^\flat, \iota_\ell}$ is irreducible, then $W^{N-1}(\pi^\infty)$ is isomorphic to either $\rho_{\Pi^\flat, \iota_\ell}^\cc$ or $\rho_{\chi, \iota_\ell}^\cc$. Moreover, if there is a finite place $w$ of $F$ over a place of $F_+$ inert in $F$ \sut $\Pi^\flat_w$ is square-integrable, then there exists a unique irreducible admissible representation $\pi_1^\infty$ of $\bx U(\mbf V)(\Ade_F^\infty)$ \sut $\pi_1^\infty$ is isomorphic to $\pi^\infty$ away from $w$, and $W^{N-1}(\pi^\infty)\oplus W^{N-1}(\pi_1^\infty)$ is conjugate to $\rho_{\Pi, \iota_\ell}^\cc$ as $\Gal_F$-representations.
\end{enumerate}
\end{hypothesis}

\begin{prop}\label{oososiemiehifws}
\textup{Hypothesis~\ref{iifififmieiemss}} holds if $N\le 3$ or $F_+\ne \bb Q$.
\end{prop}
\begin{proof}
The case for $N=2$ is established by Liu \cite[Theorem~D.6]{Liu21}. The case for $N=3$ and $F_+=\bb Q$ follows from the main result of \cite{Rog92}. The case for $N\ge 3$ when $F_+\ne \bb Q$ will be established in a sequel to \cite{KSZ21}, assuming the full endoscopic classification for unitary groups. Note that the full endoscopic classification for such unitary groups is established by Chen-Zou \cite[Corollary~7.6]{C-Z24}.
\end{proof}

We recall the following definition of cohomological Hecke characters from \cite{LTXZZ}.

\begin{defi}\label{cohomomoegienifmos}
Let $N\in\bb Z_+$ be a positive integer, and $\Pla_+$ a finite set of finite places of $F_+$ containing $\Pla^\bad_+$. Consider a homomorphism $\phi: \bb T_N^{\Pla_+}\to \kappa$ with $\kappa$ a field. We say $\phi$ is \tbf{cohomologically generic} if
\begin{equation*}
\etH^i(\bSh(\mbf V, \mdc K)_{\ovl F}, \kappa)_{\bb T_N^{\Pla_+'}\cap\ker\phi}=0
\end{equation*}
holds for any tuple $(\Pla_+', i, \mbf V, \mdc K)$ in which
\begin{itemize}
\item
$\Pla_+'$ is a finite set of finite places of $F_+$ containing $\Pla_+$,
\item
$i$ is a nonnegative integer distinct from $N-1$,
\item
$\mbf V$ is a standard indefinite Hermitian space over $F$ of dimension $N$, and
\item
$\mdc K$ is an object of $\mfk K(\mbf V)$ of the form $\mdc K=\mdc K_{\Pla_+'}\times\prod_{v\in\fPla_{F_+}\setm\Pla_+'}\bx U(\Lbd)(\mcl O_v)$ for some self-dual $\prod_{v\in\fPla_{F_+}\setm\Pla_+'}\mcl O_{F_v}$-lattice $\Lbd$ in $\mbf V\otimes_{F_+}\Ade_{F_+}^{\fPla_{F_+}\setm\Pla_+'}$.
\end{itemize}
\end{defi}

\subsection{Generalized CM type and reflexive closure}

We denote by $\bb N[\Pla^\infty_F]$ the commutative monoid freely generated by the set $\Pla^\infty_F$, which admits an action of $\Aut(\bb C)$ via the set $\Pla^\infty_F$.

\begin{defi}
A \tbf{generalized CM type} of rank $N$ is an element
\begin{equation*}
\Psi=\sum_{\tau\in\Pla_\infty}r_\tau\tau\in\bb N[\Pla^\infty_F]
\end{equation*}
satisfying $r_\tau+r_{\tau^\cc}=N$ for every $\tau\in\Pla^\infty_F$. For such $\Psi$, we define its reflex field $F_\Psi\subset\bb C$ to be the fixed subfield of the stabilizer of $\Psi$ in $\Aut(\bb C)$. A \tbf{CM type} is simply a generalized CM type of rank 1.
\end{defi}

\begin{defi}
We define the \tbf{reflexive closure} of $F$, denoted by $F_{\bx{rflx}}$, to be the subfield of $\bb C$ generated by $F$ and the intersections of $F_\Phi$ for all CM types $\Phi$ of $F$. Set $F_{\bx{rflx}, +}\defining (F_{\bx{rflx}})^{\cc=1}$.
\end{defi}

\begin{defi}\label{issieniefeifmies}
We say a finite place $\mfk p$ of $F_+$ is \tbf{good inert} if it is inert in $F$ and splits completely in $F_{\bx{rflx}, +}$. By abuse of notation, we also denote by $\mfk p$ the induced finite place of $F$. We say a good inert place $\mfk p$ is \tbf{very good inert} if the following are satisfied:
\begin{enumerate}
\item
the underlying rational prime $p$ of $\mfk p$ is odd and unramified in $F$;
\item
$\mfk p$ is of degree one over $\bb Q$, that is, $F_{+, \mfk p}=\bb Q_p$.
\end{enumerate}
\end{defi}
\begin{rem}
A finite place $\mfk p$ of $F_+$ is very good inert in our sense if it is very special inert in the sense of \cite[Definition~3.3.4]{LTXZZ}.
\end{rem}

\subsection{Preparation for Tate classes and arithmetic level-raising}\label{ptiieirmeiehfieims}

We now work in the following setting.

\begin{setup}\label{issientieitiemfes}\enskip
\begin{itemize}
\item
Let $\Pi$ be a relevant representation of $\GL_N(\Ade_F)$ that is cuspidal (resp. almost cuspidal of the form $\Pi=\Pi^\flat\boxplus\uno$) if $N$ is even (resp. odd). Here $\uno$ is the trivial character of $\GL_1(\Ade_F)$.
\item
Let $E\subset \bb C$ be a strong coefficient field of $\Pi$ (see Definition~\ref{sisieifnieeimfsi}).
\item
Let $\Pla^{\bx{min}}_+$ be a finite set of finite places of $F_+$ that contains $\Pla^\Pi_+$ (see Definition~\ref{ifitneramieiiefolacies}).
\item
Let $\lbd$ be a finite place of $E$ whose underlying prime $\ell$ satisfies $\Pla^{\bx{min}}_+\cap\Pla_{F_+}(\ell)=\vn$. We fix an isomorphism $\iota_\ell: \bb C\xr\sim\ovl{\bb Q_\ell}$ that induces the place $\lbd$.
\item
Let $\Pla^\lr_+$ be a finite set of finite places of $F_+$ that are inert in $F$, which is strongly disjoint from $\Pla^{\bx{min}}_+$ and satisfies $\ell\nmid\norml{v}(\norml{v}^2-1)$ for any $v\in\Pla^\lr_+$. 
\item
Let $\Pla_+$ be a finite set of finite places of $F_+$ containing $\Pla^{\bx{min}}_+$ and $\Pla^\lr_+$.
\item
Let $\phi_\Pi: \bb T_N^{\Pla_+}\to \mcl O_E$ be the Hecke character attached to $\Pi$ (see Definition~\ref{INIENIEMisimw}).

\item
Let $\rho_{\Pi, \lbd}: \Gal_F\to \GL_N(E_\lbd)$ be the continuous homomorphism attached to $\Pi$ (see Definition~\ref{sisieifnieeimfsi}). In particular, $\rho_{\Pi, \lbd}^\cc$ and $\rho_{\Pi, \lbd}^\vee(1-N)$ are conjugate.
\item
Let $\mrs V_N^\circ=(\mbf V_N^\circ, \Lbd_N^\circ, \mdc K_N^\circ)$ be a triple, where\footnote{Compared with \cite[\S6.1]{LTXZZ}, we omit the assumption that $(\mdc K_N^\circ)_v$ is transferable when $N$ is even, which is possible by \cite[Remark~8.2]{LTXZZb}}
\begin{enumerate}
\item
$\mbf V_N^\circ$ is a standard definite Hermitian space over $F$ of dimension $N$ (see Definition~\ref{ssowosieiniemfes}) such that $(\mbf V_N^\circ)_v$ is not split for $v\in\Pla^\lr_+$ when $N$ is even;
\item
$\Lbd_N^\circ$ is a self-dual $\prod_{v\in\fPla_{F_+}\setm\Pla^{\bx{min}}_+}\mcl O_{F_v}$-lattice in $\mbf V^\circ_N\otimes_{F_+}\Ade_{F_+}^{\infty, \Pla^{\bx{\min}}}$;
\item
$\mdc K_N^\circ$ is an object in $\mfk K(\mbf V_N^\circ)$ of the form
\begin{equation*}
\mdc K_N^\circ=\prod_{v\in\Pla_+}(\mdc K_N^\circ)_v\times\prod_{v\in\fPla_{F_+}\setm\Pla_+}\bx U(\Lbd_N^\circ)(\mcl O_v),
\end{equation*}
satisfying that when $N$ is even, $(\mdc K_N^\circ)_v$ is a hyperspecial maximal subgroup of $\bx U(\mbf V_N^\circ)(F_v)$ for $v\in\Pla_+\setm(\Pla^\lr_+\cup\Pla^{\bx{min}}_+)$, and is a special maximal subgroup of $\bx U(\mbf V_N^\circ)(F_v)$ for $v\in\Pla^\lr_+$
\end{enumerate}
such that
\begin{equation*}
\frac{\mcl O_\lbd[\bSh(\mbf V_N^\circ, \mdc K_N^\circ)]}{\bb T_N^{\Pla_+}\cap\ker\phi_{\Pi_N}}
\end{equation*}
is nontrivial when $N$ is even.
\item
Let $m\in\bb Z_+$ be a positive integer.
\item
Let $\mfk p$ be a very good inert place of $F_+$ with the underlying rational prime $p$ (see Definition~\ref{issieniefeifmies}), satisfying\footnote{Compared with \cite[\S6.1]{LTXZZ}, we omit the assumption (PI6), because it will be redundant for applications in view of \cite[Lemma 4.2.4(2)]{LTX24}.}
\begin{enumerate}
\item[(P1)]
$\mfk p$ is strongly disjoint from $\Pla_+$;
\item[(P2)]
$\ell$ does not divide $p(p^2-1)$;
\item[(P3)]
There exists a CM type $\Phi$ containing $\tau_\infty$ with $\bb Q_{p^2}^\Phi=\bb Q_{p^2}$ (we refer to \cite[\S3.3]{LTXZZ} for the definitions).
\item[(P4)]
If $N$ is even, then $P_{\bm\alpha(\Pi_{\mfk p})}\modu{\lbd^m}$ is level-raising special at $\mfk p$; if $N$ is odd, then $P_{\bm\alpha(\Pi_{\mfk p})}\modu\lbd$ is Tate generic at $\mfk p$ (see Definition~\ref{llieienieiheires});
\item[(P5)]
$P_{\bm\alpha(\Pi_{\mfk p})}\modu\lbd$ is intertwining generic at $\mfk p$.
\end{enumerate}
In particular, we can and will apply the construction and notations in \cite[\S5.1]{LTXZZ} to the datum $(\mbf V_N^\circ, \{\Lbd_{N, \mfk q}^\circ\}|_{\mfk q|p})$; cf. the beginning of \cite[\S5.2]{LTXZZ}. Denote by
\begin{equation*}
\mfk m\defining \bb T_N^{\Pla_+\cup\Pla_{F_+}(p)}\cap\ker\paren{\bb T_N^{\Pla_+}\xr{\phi_\Pi}\mcl O_E\to\mcl O_E/\lbd}
\end{equation*}
and
\begin{equation*}
\mfk n\defining \bb T_N^{\Pla_+\cup\Pla_{F_+}(p)}\cap\ker\paren{\bb T_N^{\Pla_+}\xr{\phi_\Pi}\mcl O_E\to\mcl O_E/\lbd^m}
\end{equation*}
the two ideals of $\bb T_N^{\Pla_+\cup\Pla_{F_+}(p)}$.
\item
Let $\mrs T=(\Phi, \mbf W_0, \mdc K_0^p, \iota_p, \varpi)$ be a quintuple of data as in \cite[\S5.1]{LTXZZ} with $\bb Q_p^\Phi=\bb Q_{p^2}$, which is possible because $\mfk p$ is very good inert.
\item
Let $\Lbd_{N, \mfk p}^\bullet$ be a lattice in $\mbf V_N^\circ\otimes_FF_{\mfk p}$ satisfying
\begin{itemize}
\item
$\Lbd_{N, \mfk p}^\circ\subset\Lbd_{N, \mfk p}^\bullet\subset p^{-1}\Lbd^\circ_{N, \mfk p}$, and
\item
$p\Lbd_{N, \mfk p}^\bullet\subset\Lbd_{N, \mfk p}^{\bullet, \vee}$ and $\Lbd_{N,\mfk p}^{\bullet, \vee}/p\Lbd_{N,\mfk p}^\bullet$ has length $\frac{1-(-1)^N}{2}$.
\end{itemize}
\item
Let $\mdc K_{N, \mfk p}^\bullet$ denote the stabilizer of $\Lbd_{N, \mfk p}^\bullet$ in $\bx U(\mbf V_N'\otimes_FF_{\mfk p})$, and set $\mdc K_{N, p}^\bullet\defining \mdc K_{N, \mfk p}^\bullet\times\prod_{\mfk q\in\Pla_{F_+}(p)\setm\{\mfk p\}}\mdc K_{\mfk q}^\circ$.
\item
Let $\mrs U_N=\paren{\mbf V_N', \{\Lbd_{N, \mfk q}'\}_{\mfk q|p}, \mdc K_{N,p}', \mtt j_N}$ be an indefinite uniformization datum for $\mbf V_N^\circ$, which means
\begin{itemize}
\item
$\mbf V'_N$ is a standard indefinite Hermitian space over $F$ of dimension $N$;
\item
for every place $\mfk q$ of $F_+$ lying above $p$ other than $\mfk p$, $\Lbd_{N, \mfk q}'$ is a self-dual $\mcl O_{F_{\mfk q}}$-lattice in $\mbf V'\otimes_FF_{\mfk q}$;
\item
$\Lbd_{N, \mfk p}'$ is an $\mcl O_{F_{\mfk p}}$-lattice in $V'\otimes_FF_{\mfk p}$ satisfying $\Lbd'_{N, \mfk p}\subset(\Lbd'_{N, \mfk p})^\vee$ and $(\Lbd'_{N, \mfk p})^\vee/\Lbd'_{N, \mfk p}$ has length 1;
\item
$\mdc K_{N, p}'=\prod_{\mfk q\in \Pla_{F_+}(p)}\mdc K_{N, \mfk q}'$, where $\mdc K_{N, \mfk q}'$ is the stabilizer of $\Lbd_{N, \mfk q}'$ in $\bx U(\mbf V_N'\otimes_FF_{\mfk q})$ for each $\mfk q\in \Pla_{F_+}(p)$; and
\item
$\mtt j_N: \mbf V_N^\circ\otimes_{\bb Q}\Ade^{\infty, p}\to \mbf V_N'\otimes_{\bb Q}\Ade^{\infty, p}$ is an isometry.
\end{itemize}
\item
Set $\mdc K_N^{p, \circ}\defining (\mdc K_N^\circ)^p$, and $\mdc K_N^\bullet\defining \mdc K_N^{p, \circ}\times\mdc K_{N, p}^\bullet$.
\item
Set $\bx X^?_N\defining \bx X_{\mfk p}^?(\mbf V_N^\circ, \mdc K_N^{p, \circ})$ for meaningful pairs $(\bx X, ?)\in \{\mbf M, \bx M, \bx B, \bx S\}\times \{\ ,\eta,\circ,\bullet, \dagger\}$, and let $(\bx E_s^{p, q}, \bx d_s^{p, q})$ denote the weight spectral sequence abutting to the cohomology $\bx H^\bullet_{\mfk T}(\ovl{\bx M}_N, \bx R\Psi\mcl O_\lbd(r))$ from \cite[\S5.9]{LTXZZ}.
\end{itemize}
\end{setup}

\begin{assumption}\label{iisieifnieimeifs}
The composite homomorphism $\bb T_N^{\Pla_+}\xr{\phi_\Pi}\mcl O_E\to \kappa_\lbd$ is cohomologically generic (see Definition~\ref{cohomomoegienifmos}).
\end{assumption}

\begin{assumption}\label{isismiwisiws}
The Galois representation $\rho_{\Pi, \lbd}$ (resp. $\rho_{\Pi^\flat, \lbd}$) is residually absolutely irreducible when $N$ is even (resp. $N$ is odd).
\end{assumption}

Under Assumption~\ref{isismiwisiws}, we get a residual representation $\ovl\rho_{\Pi, \lbd}$, which is unique up to conjugation and $(1-N)$-polarizable in the sense of \cite[Definition~2.5.3]{LTXZZ}. Then we obtain a continuous homomorphism
\begin{equation}\label{neeineiemirheos}
\ovl\rho_{\Pi, \lbd, +}: \Gal_{F_+}\to \mrs G_N(\kappa_\lbd)
\end{equation}
from \cite[Construction~2.5.4]{LTXZZ}.

\begin{defi}
We say a standard pair $(\mbf V, \pi)$ (see Definition~\ref{ssowosieiniemfes}) with $\dim_F\mbf V=N$ is $\Pi$-congruent (outside $\Pla_+\cup\Pla_{F_+}(p)$) if for any finite place $v$ of $F_+$ not in $\Pla_+\cup\Pla_{F_+}(\{p, \ell\})$, $\pi_v$ is unramified, and the two homomorphisms $\iota_\ell\phi_{\alpha(\BC(\pi)_v)}$ and $\iota_\ell\phi_{\alpha(\Pi_v)}$ from $\bb T_{N, v}$ to $\ovl{\bb Q_\ell}$, taking values in $\ovl{\bb Z_\ell}$, coincide in $\ovl{\bb F_\ell}$.
\end{defi}

\begin{lm}\label{isnsieoeimeifes}
Assume \textup{Assumption~\ref{isismiwisiws}}. Then the natural maps
\begin{equation*}
\cetH^i\paren{\bSh\paren{\mbf V_N', \mtt j_N\mdc K_N^{p\circ}\mdc K_{p, N}'}_{\ovl F}, \mcl O_\lbd}_{\mfk m}\to \etH^i\paren{\bSh\paren{\mbf V_N', \mtt j_N\mdc K_N^{p\circ}\mdc K_{p, N}'}_{\ovl F}, \mcl O_\lbd}_{\mfk m}
\end{equation*}
\begin{equation*}
\bx H^i_{\mfk T, c}\paren{\ovl{\bx M}_N^\bullet, \mcl O_\lbd}_{\mfk m}\to \bx H^i_{\mfk T}\paren{\ovl{\bx M}_N^\bullet, \mcl O_\lbd}_{\mfk m}
\end{equation*}
are both isomorphisms for every $i\in\bb N$.
\end{lm}
\begin{proof}
We follow the proof of \cite[Lemma~6.1.11]{LTXZZ}. We abbreviate $\bSh\defining \bSh\paren{\mbf V_N', \mtt j_N\mdc K_N^{p\circ}\mdc K_{p, N}'}$. By \cite[Lemma~5.2.7]{LTXZZ} and the description of the weight spectral sequences $(\bx E_s^{p, q}, \bx d_s^{p, q})$ in \cite[Lemma~5.9.2]{LTXZZ} (for $N$ odd) and \cite[Lemma~5.9.3]{LTXZZ} for $N$ even, it suffices to show that the first map is an isomorphism for every $i\in\bb N$. This is trivial if $F_+\ne \bb Q$, because in that case $\bSh$ is proper.

If $F_+=\bb Q$, then the Witt index of $\mbf V_N'$ is $1$. In that case, the Shimura variety $\bSh$ has a unique toroidal compactification \cite{AMRT}, which we denote by $\tld\bSh$. Since the choice of the relevant combinatorial data is unique, $\tld\bSh$ is smooth over $F$. As $\mtt j_N\mdc K_N^{p\circ}\mdc K_{p, N}'$ is neat, the boundary $Z\defining \tld\bSh\setm\bSh$ is geometrically isomorphic to a disjoint union of abelian varieties of dimension $N-2$. In particular, $\etH^i(Z_{\ovl F}, \mcl O_F)$ is a finite free $\mcl O_\lbd$-module. Let $\pi^{\prime\infty}$ be an irreducible admissible representation of $\bx U(\mbf V_N')(\Ade_{F_+}^\infty)$ that appears in $\etH^i(Z_{\ovl F}, \mcl O_\lbd)\otimes_{\mcl O_\lbd, \iota_\ell^{-1}}\bb C$. Then $\pi^{\prime\infty}$ extends to an automorphic representation $\pi'$ of $\bx U(\mbf V_N')(\Ade_{F_+})$ that is a subquotient of the parabolic induction of a cuspidal automorphic representation $\pi_{\mbf L}$ of $\mbf L(\Ade_{F_+})$ where $\mbf L$ is the unique proper Levi subgroup of $\bx U(\mbf V_N')$ up to conjugation.

We write $\mbf L=\bx U(\mbf V_{N-2})\times\Res_{F/F_+}\GL_1$, where $\mbf V_{N-2}$ is a standard definite Hermitian spaces of dimension $N-2$ contained in $\mbf V_N$ (if $N=2$, $\bx U(\mbf V_{N-2})$ denotes the trivial group). Then we can write $\pi_{\mbf L}=\pi_{N-2}\boxtimes\chi$ where $\pi_{N-2}$ is a cuspidal automorphic representation of $\bx U(\mbf V_{N-2})$ and $\chi$ is an automorphic character of $\GL_1(\Ade_F)$. In particular, $\BC(\pi')$ is of the form $\BC(\pi')=\BC(\pi_{N-2})\boxplus\chi\boxplus\chi^{-1}$. Then it is impossible that the (semi-simplified) residual representation of $\rho_{\BC(\pi')}$ is conjugate to that of $\rho_\Pi$, as the latter has an irreducible component of at least dimension $\max(2, N-1)$ (Note that $2$ is even). Thus $\etH^i(Z_{\ovl F}, \mcl O_\lbd)_{\mfk m}$ vanishes, because for any automorphic representation $\pi$ \sut $\pi^\infty$ appearing in $\etH^i(Z_{\ovl F}, \mcl O_\lbd)_{\mfk m}\otimes_{\mcl O_\lbd, \iota_\ell^{-1}}\bb C$, $\rho_{\BC(\pi)}$ should have (semi-simplified) residual representation conjugate to that of $\rho_\Pi$. This implies that
\begin{equation*}
\cetH^i(\bSh, \mcl O_\lbd)_{\mfk m}\xr\sim\etH^i(\bSh, \mcl O_\lbd)_{\mfk m}
\end{equation*}
is an isomorphism for every $i\in\bb N$.
\end{proof}

\begin{lm}\label{issiienneifheims}
Let $(\mbf V, \pi)$ be a $\Pi$-congruent standard pair. If \textup{Assumption~\ref{isismiwisiws}} holds, then $\BC(\pi)$ is a relevant automorphic representation of $\GL_N(\Ade_F)$ (see \textup{Definition~\ref{eirinenriedusnisimaos}}).
\end{lm}
\begin{proof}
Let $\rho_{\BC(\pi), \iota_\ell}: \Gal_F\to \GL_N(\ovl{\bb Q_\ell})$ denote the Galois representation attached to $\pi$ (see Proposition~\ref{issiieifimeifes}). Since $(\mbf V, \pi)$ is $\Pi$-congruent, by the Chebotarev density theorem, $\rho_{\BC(\pi), \iota_\ell}$ admits a lattice whose semisimplified residual representation $\ovl\rho_{\BC(\pi), \iota_\ell}$ is isomorphic to $\ovl\rho_{\Pi, \lbd}\otimes_{\kappa_\lbd}\ovl{\bb F_\ell}$, which is irreducible (resp. the sum of an irreducible Galois representation with a character) if $N$ is even (resp. odd). If $N$ is even, then $\rho_{\BC(\pi), \iota_\ell}$ is irreducible, so $\BC(\pi)$ must be cuspidal. If $N$ is odd, then $\rho_{\BC(\pi), \iota_\ell}$ is either irreducible or a sum of a character and an irreducible Galois representation. In the former case, $\BC(\pi)$ is cuspidal and conjugate self-dual. Assume now that $N$ is odd and $\BC(\pi)$ is not cuspidal. Then $\BC(\pi)$ must be the isobaric sum of a conjugate self-dual cuspidal automorphic representation of $\GL_{N-1}(\Ade_F)$ and a conjugate self-dual character $\chi$ of $\GL_1(\Ade_F)$.

We now show that $\BC(\pi)$ is relevant. By the above argument, it suffices to show that $\BC(\pi)_w$ is isomorphic to $\Pi_w$ for every infinite place $w$ of $F$. Let
\begin{equation*}
\bx H^i_{(2)}(\bSh)\defining \plim_{\mdc K\in\mfk K'(\mbf V)}\bx H^i_{(2)}(\bSh(\mbf V, \mdc K), \bb C)
\end{equation*}
be the $L^2$-cohomology as defined in \cite[\S6]{Fal83}, It follows from (an analogue of) Lemma~\ref{isnsieoeimeifes} that there are isomorphisms
\begin{equation*}
\lim_{\mdc K\in\mfk K'(\mbf V)}\iota_\ell^{-1}\cetH^i(\bSh(\mbf V, \mdc K)_{\ovl F}, \ovl{\bb Q_\ell})_{\mfk m}\cong \iota_\ell\bx H^i_{(2)}(\bSh)_{\mfk m}\cong \lim_{\mdc K\in\mfk K'(\mbf V)}\iota_\ell^{-1}\etH^i(\bSh(\mbf V, \mdc K)_{\ovl F}, \ovl{\bb Q_\ell})_{\mfk m}.
\end{equation*}
In particular, $\pi^\infty$ appears in $\iota_\ell\bx H^i_{(2)}(\bSh)$. By Borel--Casselman's decomposition of $\bx H_{(2)}^i(\bSh)$ \cite{B-C83}, we see that $\pi_\infty$ is cohomological for the trivial representation of $\Res_{F_+/\bb Q}\bx U(\mbf V)$. In particular, $\BC(\pi)_w$ is isomorphic to $\Pi_w$ for every infinite place $w$ of $F$.
\end{proof}

\begin{lm}\label{changeoeinfifenfiems}
Let $N$ be odd and assume \textup{Assumption~\ref{iisieifnieimeifs}} and \textup{Hypothesis~\ref{iifififmieiemss}} for $N$. Then for any object $\mdc K_N^{\prime\mfk p}\in \mfk K(\mbf V_N')^{\mfk p}$ and hyperspecial maximal subgroup $\mdc K_{\mfk p}^{\prime\hs}$ of $\bx U(\mbf V_N')(F_{+, \mfk p})$, there are isomorphisms
\begin{equation*}
\etH^i(\bSh(\mbf V_N', \mdc K_N^{\prime\mfk p}\mdc K_{N, \mfk p}')_{\ovl F}, \mcl O_\lbd)_{\mfk m}\cong \etH^i(\bSh(\mbf V_N', \mdc K_N^{\prime\mfk p}\mdc K_{N, \mfk p}^{\prime\hs})_{\ovl F}, \mcl O_\lbd)_{\mfk m}
\end{equation*}
\end{lm}
\begin{proof}
As both $\mdc K_{N, \mfk p}'$ and $\mdc K_{N, \mfk p}^{\prime\hs}$ are special maximal subgroups of $\bx U(\mbf V_N')(F_{+, \mfk p})$, the proof of \cite[Lemma~8.1.7]{LTXZZ} goes through, noticing that for every cuspidal automorphic representation $\pi'$ of $\bx U(\mbf V')(\Ade_{F_+})$ appearing in either
\begin{equation*}
\etH^i(\bSh(\mbf V_N', \mdc K_N^{\prime\mfk p}\mdc K_{N, \mfk p}')_{\ovl F}, \mcl O_\lbd)_{\mfk m}\otimes_{\mcl O_\lbd}\ovl{\bb Q_\ell}
\end{equation*}
or 
\begin{equation*}
\etH^i(\bSh(\mbf V_N', \mdc K_N^{\prime\mfk p}\mdc K_{N, \mfk p}^{\prime\hs})_{\ovl F}, \mcl O_\lbd)_{\mfk m}\otimes_{\mcl O_\lbd}\ovl{\bb Q_\ell},
\end{equation*}
the semisimplified residual representations of $\rho_{\BC(\pi'), \iota_\ell}$ and $\rho_{\Pi, \iota_\ell}$ are conjugate as $\ovl{\bb F_\ell}[\Gal_F]$-modules by the Chebotarev density theorem.
\end{proof}

\subsection{Tate classes in the odd rank case}\label{Tiemsopeoireisss}

In this subsection, we assume that $N$ is odd, and work in the setting of Setup~\ref{issientieitiemfes}.

\begin{lm}\label{issieiemfieehireriems}
$\bx H^i_{\mfk T}(\ovl{\bx M}_N^\dagger, \mcl O_\lbd)_{\mfk m}$ vanishes for every odd integer $i$.
\end{lm}
\begin{proof}
We follow the proof of \cite[Lemma~6.2.1]{LTXZZ}. If $i\ne 2r-1$, this follows from \cite[Lemma~5.6.2(1)]{LTXZZ}. We now assume $i=2r-1$.

Suppose that $\pi^{\infty, p}$ is an irreducible admissible representation of $\bx U(\mbf V_N^\circ)(\Ade_{F_+}^{\infty, p})$ that appears in the cohomology $\bx H_{\mfk T}^{2r-1}(\ovl{\bx M}^\dagger_N, \mcl O_\lbd)_{\mfk m}\otimes_{\mcl O_\lbd, \iota_\ell^{-1}}\bb C$. By \cite[Proposition~5.6.4]{LTXZZ}, we may complete $\pi^{\infty, p}$ to an automorphic representation $\pi$ as in that proposition, such that $(\mbf V_N^\circ, \pi)$ is a $\Pi$-congruent standard pair, and that $\BC(\pi_{\mfk p})$ is a constituent of an unramified principal series of $\GL_N(F_{\mfk p})$, whose Satake parameter contains $-p$ and $-p^{-1}$ (which is then different from $\bm\alpha(\Pi_{\mfk p})$ in $\ovl{\bb F_\ell}$ by (P5)). On the other hand, the semisimplified residual representations of $\rho_{\BC(\pi), \iota_\ell}$ and $\rho_{\Pi, \iota_\ell}$ are isomorphic. In particular, they have the same generalized Frobenius eigenvalues in $\ovl{\bb F_\ell}$ at the unique place of $F$ over $\mfk p$. However, this is not possible by Arthur's multiplicity formula (see \cite[Theorem~1.7.1]{KMSW}), Proposition~\ref{ieieeinfeieiites}(3) and Proposition~\ref{issiieifimeifes}. Therefore, we must have $\bx H^{2r-1}_{\mfk T}(\ovl{\bx M}_N^\dagger, \mcl O_\lbd)_{\mfk m}=0$.
\end{proof}

\begin{prop}\label{isnsnieiiheinfies}
Assume \textup{Assumption~\ref{iisieifnieimeifs}} and \textup{Hypothesis~\ref{iifififmieiemss}} for $N$.
\begin{enumerate}
\item
$\bx E_{2, \mfk m}^{p, q}$ vanishes unless $(p, q)=(0, 2r)$, and $\bx E_{2, \mfk m}^{0, 2r}$ is canonically isomorphic to $\bx H_{\mfk T}^{2r}(\ovl{\bx M}_N, \bx R\Psi\mcl O_\lbd(r))_{\mfk m}$, which is a free $\mcl O_\lbd$-module.
\item
The set of generalized Frobenius eigenvalues of the $\kappa_\lbd[\Gal_{\bb F_{p^2}}]$-module $\bx E_{2, \mfk m}^{0, 2r}\otimes_{\mcl O_\lbd}\kappa_\lbd$ is contained in the set of roots of $P_{\bm\alpha(\Pi_{\mfk p})}\modu\lbd$ in $\ovl{\bb F_\ell}$.
\item
The $\mcl O_\lbd[\Gal_{\bb F_{p^2}}]$-module $\bx E_{2, \mfk m}^{0, 2r}$ is weakly semisimple.
\item
The localization of the map $\nabla^1$ at $\mfk m$ induces an isomorphism
\begin{equation*}
\nabla^1_{\mfk m}: \paren{\bx E_{2, \mfk m}^{0, 2r}}_{\Gal_{\bb F_{p^2}}}\xr\sim \mcl O_\lbd[\bSh(\mbf V_N^\circ, \mdc K_N^\circ)]_{\mfk m}.
\end{equation*}
\end{enumerate}
\end{prop}
\begin{proof}
For (1), by Lemma~\ref{issieiemfieehireriems}, the same proof of \cite[Lemma~6.2.2(3)]{LTXZZ} goes through.

Next we prove parts (2)-(4). Firstly it follows from the proof of \cite[Theorem~6.2.3]{LTXZZ} and \cite[Lemma~4.2.4]{LTX24} that $\nabla_{\mfk m}^1$ is surjective. By \cite[Lemma~5.2.7]{LTXZZ} and part (1), there is an isomorphism
\begin{equation}\label{issiemifmiefiws}
\bx E_{2, \mfk m}^{0, 2r}\cong \etH^{2r}\paren{\bSh\paren{\mbf V', \mtt j_N\mdc K_N^{\circ, p}\mdc K_{N, p}'}_{\ovl F}, \mcl O_\lbd(r)}_{\mfk m}
\end{equation}
of $\mcl O_\lbd[\Gal_{\bb Q_{p^2}}]$-modules. By Lemma~\ref{issiienneifheims}, \ref{isnsieoeimeifes}, Hypothesis~\ref{iifififmieiemss} and Arthur's multiplicity formula \cite[Theorem~1.7.1]{KMSW}, there is an isomorphism
\begin{equation}\label{ismieeifhes}
\etH^{2r}\paren{\bSh\paren{\mbf V', \mtt j_N\mdc K_N^{\circ, p}\mdc K'_{N, p}}_{\ovl F}, \mcl O_\lbd(r)}_{\mfk m}\otimes_{\mcl O_\lbd}\ovl{\bb Q_\ell}\cong \bigoplus_{\pi^{\prime\infty}}\paren{W^{N-1}(\pi^{\prime\infty})}^{\oplus d(\pi^{\prime\infty})}
\end{equation}
of $\ovl{\bb Q_\ell}[\Gal_F]$-modules, where $d(\pi^{\prime\infty})=\dim\paren{\pi^{\prime\infty}}^{\mtt j_N\mdc K_N^{\circ, p}\mdc K_{N, p}'}$; and the sum is taken over all admissible irreducible representations $\pi^{\prime\infty}$ of $\bx U(\mbf V')(\Ade_{F_+}^\infty)$ that is the finite part of some automorphic representation $\pi'$ of $\bx U(\mbf V')(\Ade_{F_+})$ satisfying $(\mbf V', \pi')$ is a standard pair. Here we choose such a $\pi'$ for each $\pi^{\prime\infty}$ appearing in the direct sum. For the proof of parts (2-4), we may replace $E_\lbd$ by a finite extension inside $\ovl{\bb Q_\ell}$ \sut $W^{N-1}(\pi^{\prime\infty})$ is defined over $E_\lbd$ for each $\pi^{\prime\infty}$ appearing in the previous direct sum. For each such $\pi^{\prime\infty}$, $W^{N-1}(\pi^{\prime\infty})$ is conjugate to an irreducible subrepresentation of $\rho_{\Pi, \iota_\ell}^\cc$ by Hypothesis~\ref{iifififmieiemss}. Thus part (2) follows from Equations~\eqref{issiemifmiefiws},~\eqref{ismieeifhes} and part (1).

For (3), we choose a $\Gal_F$-stable $\mcl O_\lbd$-lattice $\bx R^{N-1}(\pi^{\prime\infty})$ of $W^{N-1}(\pi^{\prime\infty})$ for each $\pi^{\prime\infty}$ appearing in the previous direct sum. We claim that $\bx R^{N-1}(\pi^{\prime\infty})$ is weakly semisimple, which implies part (3) by \cite[Lemma~2.1.4(1)]{LTXZZ}. By (P4), we know $\ovl\rho_{\Pi, \lbd}^\cc(r)$ is weakly semisimple, and 
\begin{equation*}
\dim_{\kappa_\lbd}\ovl\rho_{\Pi, \lbd}(r)^{\Gal_{\bb F_{p^2}}}=0, \quad\dim_{\kappa_\lbd}\ovl\rho_{\Pi^\flat, \lbd}(r)^{\Gal_{\bb F_{p^2}}}=1.
\end{equation*}
If $\dim_{E_\lbd}W^{N-1}(\pi^{\prime\infty})$ is odd, then
\begin{equation*}
\dim_{E_\lbd}W^{N-1}(\pi^{\prime\infty})^{\Gal_{\bb F_{p^2}}}\ge 1.
\end{equation*}
Thus $\bx R^{N-1}(\pi^{\prime\infty})$ is weakly semisimple by \cite[Lemma~2.1.5]{LTXZZ}. On the other hand, if $\dim_{E_\lbd}W^{N-1}(\pi^{\prime\infty})$ is even, then $\BC(\pi')$ is almost cuspidal, and $\bx R^{N-1}(\pi^{\prime\infty})\otimes_{\mcl O_\lbd}\kappa_\lbd$ is conjugate to $\rho_{\Pi^\flat, \lbd}$ as $\kappa_\lbd[\Gal_F]$-modules. Thus $\bx R^{N-1}(\pi^{\prime\infty})$ is also weakly semisimple by \cite[Lemma~2.1.5]{LTXZZ}.

For (4): By the above discussion, it suffices to show
\begin{equation*}
\sum_{\pi^{\prime\infty}}d(\pi^{\prime\infty})\le\dim_{E_\lbd}\mcl O_\lbd[\bSh(\mbf V_N^\circ, \mdc K_N^\circ)]_{\mfk m}\otimes_{\mcl O_\lbd}E_\lbd,
\end{equation*}
where $\pi^{\prime\infty}$ is taken over all those appearing in the previous direct sum satisfying $\dim_{E_\lbd}W^{N-1}(\pi^{\prime\infty})$ is odd. This assertion follows from Lemma~\ref{changeoeinfifenfiems} and Lemma~\ref{lsomiiaueimipws} below.
\end{proof}
\begin{lm}\label{lsomiiaueimipws}
Let $\pi'$ be an automorphic representation of $\bx U(\mbf V')(\Ade_{F_+})$ \sut $\BC(\pi')$ is relevant. If $N$ is odd and $\BC(\pi')=\Pi^\flat\boxplus\chi$ is almost cuspidal, we further assume the following conditions.
\begin{itemize}
\item
The local component $\pi'_{\mfk p}$ is unramified with Satake parameter containing 1 exactly once.
\item
Set $\mfk I=\{N-1, N-3, \ldots, 3-N, 1-N\}$, in particular $\chi_{\udl\tau_\infty}(z)=\arg(z)^{a_\chi}$ for some $a_\chi\in \mfk I$. Let $\kappa_\chi: \mu_2^{\mfk I}\to \bb C^\times$ denote the character that takes value $-1$ on the generator corresponding to $a_\chi\in \mfk I$ and takes value $1$ on all other generators. Then $\pi'_{\udl\tau_\infty}$ is isomorphic to the discrete series $\pi^{\kappa_\chi}$ (with Harish-Chandra parameter $(r, r-1, \ldots, 1-r, -r)$) as defined in \cite[Notation~3.14]{L-L21}.
\end{itemize}
Consider the admissible irreducible representation $\pi\defining \pi_{\udl\tau_\infty}\otimes\pi_{\mfk p}\otimes(\pi')^{\udl\tau_\infty, \mfk p}$ of $\bx U(\mbf V^\circ)(\Ade_{F_+})$ where
\begin{itemize}
\item
$\pi_{\udl\tau_\infty}$ is the trivial representation of $\bx U(\mbf V^\circ\otimes_FF_{\udl\tau_\infty})$;
\item
$\pi_{\mfk p}$ is an unramified representation of $\bx U(\mbf V^\circ\otimes_FF_{\mfk p})$ satisfying $\BC(\pi_{\mfk p})=\BC(\pi'_{\mfk p})$.
\end{itemize}
Then the automorphic multiplicity of $\pi$ is $1$.
\end{lm}
\begin{proof}
This follows from Arthur's multiplicity formula for tempered global L-packets; cf.\cite[Theoerm~1.7.1]{KMSW}.
\end{proof}

\subsection{Arithmetic level-raising in the even rank case}\label{ososiemeoireuuems}

In this subsection, we assume that $N$ is even and work in the setting of Setup~\ref{issientieitiemfes}.

We recall the following definition of rigid residual Galois representations from \cite[\S 3.6]{LTXZZa}.

\begin{defi}\label{rigidindiremiesLGoos}
Let $\ovl r: \Gal_{F_+}\to \mrs G_N(\kappa_\lbd)$ be a continuous homomorphism satisfying
\begin{equation*}
\ovl r^{-1}(\GL_N(\kappa_\lbd)\times\GL_1(\kappa_\lbd))=\Gal_F, \quad \nu\circ\ovl r=\eta_{F/F_+}^N\ovl\ve_\ell^{1-N}.
\end{equation*}
We say $\ovl r$ is rigid for $(\Pla^{\bx{min}}_+, \Pla^\lr_+)$ if the following are satisfied:
\begin{enumerate}
\item
For $v\in\Pla^{\bx{min}}_+$, any lifting $r: \Gal_{F_+}\to \mrs G_N(\kappa_\lbd)$ with $\nu\circ r=\eta_{F/F_+}^N\ve_\ell^{1-N}$ is minimally ramified as defined in \cite[Definition~3.4.8]{LTXZZa}.
\item
For $v\in\Pla^\lr_+$, the set of generalized eigenvalues of $\ovl r_v^\natural(\phi_w)$ contains the pair $\{\norml{v}^{-N}, \norml{v}^{-N+2}\}$ exactly once, where $w$ is the unique place of $F$ over $v$.
\item
For $v\in \Pla_{F_+}(\ell)$, $\ovl r_v^\natural$ is regular Fontaine--Laffaille crystalline as defined in \cite[Definition~3.2.4]{LTXZZa}.
\item
For a finite place of $F_+$ not in $\Pla^{\bx{min}}_+\cup\Pla^\lr_+\cup\Pla_{F_+}(\ell)$, $\ovl r_v$ is unramified.
\end{enumerate}
\end{defi}

We state the following variant of the R=T theorem in \cite{LTXZZa} suitable for our case. We apply the discussion of \cite[\S3]{LTXZZa} to the pair $(\ovl r, \chi)=(\ovl\rho_{\Pi, \lbd, +}, \ve_\ell^{1-N})$. Suppose $\ovl r$ is rigid for $(\Pla^{\bx{min}}_+, \Pla^\lr_+)$. For each $?\in\{\bx{mix, unr, ram}\}$, we consider the global deformation problem (see \cite[Definition~3.6]{LTXZZa})
\begin{equation*}
\mrs S^?=\paren{\ovl r, \ve_\ell^{1-N}, \Pla^{\bx{min}}_+\cup\Pla^\lr_+\cup\{\mfk p\}\cup\Pla_{F_+}(\ell), \{\mrs D_v\}_{v\in \Pla^{\bx{min}}_+\cup\Pla^\lr_+\cup\{\mfk p\}\cup\Pla_{F_+}(\ell)}}
\end{equation*}
where
\begin{itemize}
\item
for $v\in \Pla^{\bx{\min}}_+$, $\mrs D_v$ is the local deformation problem classifying all liftings of $\ovl r_v$;
\item
for $v\in \Pla^\lr_+$, $\mrs D_v$ is the local deformation problem $\mrs D^{\bx{ram}}$ of $\ovl r_v$ from \cite[Definition~3.34]{LTXZZa};
\item
for $v=\mfk p$, $\mrs D_v$ is the local deformation problem $\mrs D^?$ of $\ovl r_v$ from \cite[Definition~3.34]{LTXZZa};
\item
for $v\in \Pla_{F_+}(\ell)$, $\mrs D_v$ is the local deformation problem $\mrs D^\FL$ of $\ovl r_v$ from \cite[Definition~3.12]{LTXZZa}.
\end{itemize}
Then the global universal deformation ring $\msf R^\univ_{\mrs S^?}$ is defined in \cite[Proposition~3.7]{LTXZZa}. Set $\msf R^?\defining \msf R^\univ_{\mrs S^?}$ for short. Then there are canonical surjective homomorphisms $\msf R^{\bx{mix}}\to \msf R^\unr$ and $\msf R^{\bx{mix}}\to \msf R^{\bx{ram}}$ of $\mcl O_\lbd$-rings. We have the following corollary of the $\msf R=\msf T$ theorem from \cite{LTXZZa}.

\begin{thm}\label{isnsieniehfbeienis}
Assume \textup{Assumptions~\ref{iisieifnieimeifs}} and \textup{Hypothesis~\ref{iifififmieiemss}} for $N$. We further assume that $\ell\ge 2(N+1)$, $\ovl\rho_{\Pi, \lbd, +}$ (\textup{Equation~\eqref{neeineiemirheos}}) is rigid for $(\Pla^{\min}_+, \Pla_+^\lr)$, and $\ovl\rho_{\Pi, \lbd}|_{\Gal_{F(\mu_\ell)}}$ is absolutely irreducible.
\begin{itemize}
\item
Let $\msf T^\unr$ denote the image of $\bb T_N^{\Pla_+\cup\{\mfk p\}}$ in $\End_{\mcl O_\lbd}\paren{\mcl O_\lbd[\bSh(\mbf V_N^\circ, \mdc K_N^\circ)]}$. Then there is a canonical isomorphism $\msf R^\unr\cong\msf T^\unr$ of nonzero local rings \sut $\mcl O_\lbd[\bSh(\mbf V_N^\circ, \mdc K_N^\circ)]$ is a nonzero finite free module over $\msf R^\unr$.
\item
Let $\msf T^{\bx{ram}}$ denote the image of $\bb T_N^{\Pla_+\cup\{\mfk p\}}$ in $\End_{\mcl O_\lbd}\paren{\bx H^{2r-1}_{\mfk T}(\ovl{\bx M}_N, \bx R\Psi\mcl O_\lbd)}$. Then there is a canonical isomorphism $\msf R^{\bx{ram}}\cong\msf T^{\bx{ram}}$ of nonzero local rings \sut $\mcl O_\lbd[\bSh(\mbf V_N^\circ, \mdc K_N^\circ)]$ is a nonzero finite free module over $\msf R^{\bx{ram}}$.
\end{itemize}
\end{thm}
\begin{proof}
For (1): This follows from \cite[Theorem~3.38]{LTXZZa}, except that when $v\in \fPla_{F_+}\setm\Pla_+^{\bx{min}}$, the level group $(\mdc K_N^\circ)_v$ is a hyperspecial but may not be the stabilizer of a self-dual lattice in $\bx U(\mbf V_N^\circ)(F_{+, v})$. However, the proof of \cite[Theorem~3.38]{LTXZZa} goes through.

For (2): By \cite[Proposition~3.6.1]{LTXZZ}, we know $\msf T^{\bx{ram}}$ is nonzero. Thus by \cite[Lemma~5.2.7]{LTXZZ} and the same reason as in (1), the assertion follows from \cite[Theorem~3.38]{LTXZZa} (with $(\Pla^{\bx{\min}}_+, \Pla^\lr_+)$ replaced by $(\Pla^{\bx{\min}}_+, \Pla^\lr_+\cup\{\mfk p\})$).
\end{proof}

\begin{prop}\label{isnsieiINEifehiw}
Assume \textup{Assumptions~\ref{iisieifnieimeifs}} and \textup{Hypothesis~\ref{iifififmieiemss}} for $N$. Assume further that $\ell\ge 2(N+1)$, $\ovl\rho_{\Pi, \lbd, +}$ (\textup{Equation~\eqref{neeineiemirheos}}) is rigid for $(\Pla^{\min}_+, \Pla_+^\lr)$, and $\ovl\rho_{\Pi, \lbd}|_{\Gal_{F(\mu_\ell)}}$ is absolutely irreducible.
\begin{enumerate}
\item
$\bx H^i_{\mfk T}\paren{\ovl{\bx M}_N^\bullet, \mcl O_\lbd}_{\mfk m}$ is a free $\mcl O_\lbd$-module for every $i\in\bb Z_+$.
\item
$\bx E_{2, \mfk m}^{p, q}$ is a free $\mcl O_\lbd$-module, and vanishes unless $p+q=2r-1$ and $|p|\le 1$.
\item
The set of generalized Frobenius eigenvalues of the $\kappa_\lbd[\Gal_{\bb F_{p^2}}]$-module $\bx H^{2r-1}_{\mfk T}(\ovl{\bx M}_N^\bullet, \mcl O_\lbd(r))_{\mfk m}\otimes_{\mcl O_\lbd}\kappa_\lbd$ is contained in the set of roots of $P_{\bm\alpha(\Pi_{0, \mfk p})}(p^{-1}T)\modu\lbd$ in $\ovl{\bb F_\ell}$, and does not contain $1$ or $p^2$.
\item
The quotient modulo $\mfk n$ of the map $\nabla_{\mfk m}^0$ induces an isomorphism
\begin{equation*}
\nabla_{/\mfk n}^0: \bx F_{-1}\bx H^1\paren{I_{\bb Q_{p^2}}, \bx H^{2r-1}_{\mfk T}\paren{\ovl M_N, \bx R\Psi\mcl O_\lbd(r)}/\mfk n}\xr\sim \mcl O_\lbd[\bSh(\mbf V^\circ, \mdc K_N^\circ)]/\mfk n.
\end{equation*}
Here $\bx F_{-1}$ is the degree $-1$ term of the monodromy filtration.
\item
There is a natural isomorphism
\begin{equation*}
\bx F_{-1}\bx H^1\paren{I_{\bb Q_{p^2}}, \etH^{2r-1}\paren{\ovl{\bx M}_N, \bx R\Psi\mcl O_\lbd(r)}/\mfk n}\cong \bx H^1_\sing\paren{\bb Q_{p^2}, \etH^{2r-1}\paren{\ovl{\bx M}_N, \bx R\Psi\mcl O_\lbd(r)}/\mfk n}.
\end{equation*}
\item
There exists a positive integer $\mu\in\bb Z_+$ and an isomorphism
\begin{equation*}
\etH^{2r-1}\paren{\bSh\paren{\mbf V_N', \mtt j_N\mdc K_N^{\infty, p}\mdc K_{N, p}'}_{\ovl F}, \mcl O_\lbd(r)}/\mfk n\cong \paren{\ovl{\bx R}^{(m)\cc}}^{\oplus\mu}
\end{equation*}
of $\mcl O_\lbd[\Gal_F]$-modules, where $\bx R$ is a $\Gal_F$-stable $\mcl O_\lbd$-lattice in $\rho_{\Pi, \lbd}(r)$, unique up to homothety.
\end{enumerate}
\end{prop}
\begin{proof}
For (1)-(3): Using Theorem~\ref{isnsieniehfbeienis}, the proof of \cite[Theorem~6.3.4(1)-(3)]{LTXZZ} goes through.

For (6)-(7): By the proof of \cite[Theorem~6.3.4(4)]{LTXZZ}, these follow from \cite[Proposition~6.4.1]{LTXZZ}, \cite[Lemma~4.2.4(2)]{LTX24} and Theorem~\ref{isnsieniehfbeienis}.
\end{proof}

\subsection{First explicit reciprocity law}\label{sosieemiifeiifemiws}

We now work in the following setup.

\begin{setup}\label{isnsiemfmiemiw}\enskip
\begin{itemize}
\item
Let $n\ge 2$ be an integer. Among $\{n, n+1\}$, $n_0=2r_0$ (resp. $n_1=2r_1+1$) be the unique even (resp. odd) number in the set $\{n, n+1\}$. In particular, $r_0+r_1=n$.
\item
Let $\Pi_0$ be a cuspidal relevant representation of $\GL_{n_0}(\Ade_F)$, and let $\Pi_1$ be an almost cuspidal relevant representation of $\GL_{n_1}(\Ade_F)$ of the form $\Pi_1=\Pi_1^\flat\boxplus\uno$, where $\uno$ is the trivial character of $\GL_1(\Ade_F)$ (see Definition~\ref{eirinenriedusnisimaos}).
\item
Let $E\subset \bb C$ be a strong coefficient field of $\Pi$ (see Definition~\ref{sisieifnieeimfsi}).
\item
For each $\alpha\in\{0,1\}$ and each finite place $\lbd$ of $E$, let $\rho_{\Pi_\alpha, \lbd}:\Gal_F\to \GL_{n_\alpha}(E_\lbd)$ be the continuous homomorphism attached to $\Pi_\alpha$ (see Definition~\ref{sisieifnieeimfsi}). In particular, $\rho_{\Pi_\alpha, \lbd}^\cc$ and $\rho_{\Pi_\alpha, \lbd}^\vee(1-n_\alpha)$ are conjugate.
\item
For each $\alpha\in\{0,1\}$, let $\phi_{\Pi_\alpha}: \bb T_{n_\alpha}^{\Pla^{\bx{\Pi_0}}_+\cup\Pla^{\bx{\Pi_1}}_+}\to\mcl O_E$ be the restriction of the Hecke character defined in Definition~\ref{INIENIEMisimw}.
\end{itemize}
\end{setup}

We further assume that we are in the following setting.

\begin{setup}\label{ismsienifmeitjes}
Let $(\lbd, \Pla_+^{\lr, \bx{\Rmnum1}}, \Pla_+^{\bx{\Rmnum1}}, \mrs V^\circ, m, \mfk p, \mrs T, \mrs V^\bullet, \mrs U)$ be a nonuple, where
\begin{itemize}
\item
$\lbd$ is a finite place of $E$ whose underlying prime $\ell$ satisfies $\Pla^{\Pi_0}_+\cap\Pla_{F_+}(\ell)=\vn$ and $\ell\ge 2(n_0+1)$.
\item
$\Pla^{\lr, \bx{\Rmnum1}}_+$ is a finite set of finite inert places of $F_+$ strongly disjoint from $\Pla^{\Pi_0}_+\cup\Pla^{\Pi_1}_+$ (see Definition~\ref{ifitneramieiiefolacies}) satisfying $\ell\nmid\norml{v}(\norml{v}^2-1)$ for any $v\in\Pla^{\lr, \bx{\Rmnum1}}_+$.
\item
$\Pla^{\bx{\Rmnum1}}_+$ is a finite set of finite places of $F_+$ containing $\Pla_+^{\lr, \bx{\Rmnum1}}$ and $\Pla^{\Pi_0}_+\cup\Pla^{\Pi_1}_+$.

\item
$\mrs V^\circ=(\mbf V_n^\circ, \mbf V_{n+1}^\circ; \Lbd_n^\circ, \Lbd_{n+1}^\circ; \mdc K_n^\circ, \mdc K_\sp^\circ, \mdc K_{n+1}^\circ)$ is a septuple, where\footnote{Compared with \cite[\S7.2]{LTXZZ}, we omit the assumption that $(\mdc K_N^\circ)_v$ is transferable when $N$ is even, which is possible by \cite[Remark~8.2]{LTXZZb}}
\begin{enumerate}
\item
$\mbf V_n^\circ$ is a standard definite Hermitian space over $F$ of dimension $N$ (see Definition~\ref{ssowosieiniemfes}), and $\mbf V_{n+1}^\circ=(\mbf V_n^\circ)_\sharp$, such that $(\mbf V_{n_0}^\circ)_v$ is not split for $v\in \Pla^{\lr, \bx{\Rmnum1}}_+$.
\item
$\Lbd_n^\circ$ is a self-dual $\prod_{v\in\Pla_F^\infty\setm\Pla_+^{\bx{\Rmnum1}}}\mcl O_{F_v}$-lattice in $\mbf V^\circ_n\otimes_{F_+}\Ade_{F_+}^{\infty, \Pla_+^{\bx{\Rmnum1}}}$;
\item
$\mdc K_n^\circ$ is an object in $\mfk K(\mbf V_n^\circ)$ and $(\mdc K_\sp^\circ, \mdc K_{n+1}^\circ)$ is an object in $\mfk K(\mbf V_n^\circ)_\sp$ of the forms
\begin{equation*}
\mdc K_n^\circ=\prod_{v\in\Pla^{\bx{\Rmnum1}}_+}(\mdc K_n^\circ)_v\times\prod_{v\in\fPla_+\setm\Pla^{\bx{\Rmnum1}}_+}\bx U(\Lbd_n^\circ)(\mcl O_v),
\end{equation*}
\begin{equation*}
\mdc K_\sp^\circ=\prod_{v\in\Pla^{\bx{\Rmnum1}}_+}(\mdc K_\sp^\circ)_v\times\prod_{v\in\fPla_+\setm\Pla^{\bx{\Rmnum1}}_+}\bx U(\Lbd_n^\circ)(\mcl O_v),
\end{equation*}
\begin{equation*}
\mdc K_{n+1}^\circ=\prod_{v\in\Pla^{\bx{\Rmnum1}}_+}(\mdc K_{n+1}^\circ)_v\times\prod_{v\in\fPla_+\setm\Pla^{\bx{\Rmnum1}}_+}\bx U(\Lbd_{n+1}^\circ)(\mcl O_v),
\end{equation*}
satisfying
\begin{itemize}
\item
$(\mdc K_\sp^\circ)_v\subset (\mdc K_n^\circ)_v$ for $v\in\Pla^{\bx{\Rmnum1}}_+$, and
\item
$(\mdc K_{n_0}^\circ)_v$ is a hyperspecial maximal subgroup of $\bx U(\mbf V_{n_0}^\circ)(F_v)$ for $v\in\Pla^{\bx{\Rmnum1}}_+\setm(\Pla^{\lr, \bx{\Rmnum1}}_+\cup\Pla^{\Pi_0}_+)$, and is a special maximal subgroup of $\bx U(\mbf V_{n_0}^\circ)(F_v)$ for $v\in\Pla^{\lr, \bx{\Rmnum1}}_+$
\end{itemize}
\end{enumerate}
\sut
\begin{equation*}
\frac{\mcl O_\lbd[\bSh(\mbf V_{n_0}^\circ, \mdc K_{n_0}^\circ)]}{\bb T_{n_0}^{\Pla_+^{\bx{\Rmnum1}}}\cap\ker\phi_{\Pi_{n_0}}}
\end{equation*}
is nontrivial.
\item
$m\in\bb Z_+$ is a positive integer,
\item
$\mfk p$ is a very good inert place of $F_+$ with the underlying rational prime $p$ (see Definition~\ref{issieniefeifmies}), satisfying\footnote{Compared with \cite[\S7.2]{LTXZZ}, we incorporate (PI7) into (PI4), and omit assumption (PI6) as it will be redundant for applications in view of \cite[Lemma 4.2.4(2)]{LTX24}.}
\begin{enumerate}
\item[(PI1)]
$\mfk p$ is strongly disjoint from $\Pla^{\bx{\Rmnum1}}_+$;
\item[(PI2)]
$\ell$ does not divide $p(p^2-1)$;
\item[(PI3)]
There exists a CM type $\Phi$ containing $\tau_\infty$ as in \cite[\S5.1]{LTXZZ} with $\bb Q_{p^2}^\Phi=\bb Q_{p^2}$ (we refer to \cite[\S3.3]{LTXZZ} for the definitions).
\item[(PI4)]
$P_{\bm\alpha(\Pi_{0, \mfk p})}\modu{\lbd^m}$ is level-raising special at $\mfk p$, $P_{\bm\alpha(\Pi_{1, \mfk p})}\modu\lbd$ is Tate generic at $\mfk p$, and $P_{\bm\alpha(\Pi_{0, \mfk p})\otimes\bm\alpha(\Pi_{1, \mfk p})}\modu{\lbd^m}$ is level-raising special at $\mfk p$ (see Definition~\ref{llieienieiheires});
\item[(PI5)]
$P_{\bm\alpha(\Pi_{\alpha, \mfk p})}\modu\lbd$ is intertwining generic at $\mfk p$ for each $\alpha\in\{0, 1\}$.
\end{enumerate}
In particular, we can and will apply the construction and notations in \cite[\S5.10]{LTXZZ} to the datum $(\mbf V_n^\circ, \{\Lbd_{n, \mfk q}^\circ\}|_{\mfk q|p})$. For each $\alpha\in\{0, 1\}$, denote by
\begin{equation*}
\mfk m_\alpha\defining \bb T_{n_\alpha}^{\Pla^{\bx{\Rmnum1}}_+\cup\Pla_{F_+}(p)}\cap\ker\paren{\bb T_{n_\alpha}^{\Pla^{\bx{\Pi_0}}_+\cup\Pla^{\bx{\Pi_1}}_+}\xr{\phi_\Pi}\mcl O_E\to\mcl O_E/\lbd}
\end{equation*}
and
\begin{equation*}
\mfk n_\alpha\defining \bb T_{n_\alpha}^{\Pla^{\bx{\Rmnum1}}_+\cup\Pla_{F_+}(p)}\cap\ker\paren{\bb T_{n_\alpha}^{\Pla^{\bx{\Pi_0}}_+\cup\Pla^{\bx{\Pi_1}}_+}\xr{\phi_\Pi}\mcl O_E\to\mcl O_E/\lbd^m}
\end{equation*}
the two ideals of $\bb T_{n_\alpha}^{\Pla^{\bx{\Rmnum1}}_+\cup\Pla_{F_+}(p)}$.
\item
$\mrs T=(\Phi, \mbf W_0, \mdc K_0^p, \iota_p, \varpi)$ is a quintuple of data as in \cite[\S5.1]{LTXZZ} with $\bb Q_p^\Phi=\bb Q_{p^2}$.
\item
$\mrs V^\bullet=(\Lbd_{n, \mfk p}^\bullet, \Lbd_{n+1, \mfk p}^\bullet; \mdc K_{n, p}^\bullet, \mdc K_{n+1, p}^\bullet, \mdc K_{\sp, p}^\bullet; \mdc K_{n, p}^\dagger, \mdc K_{\sp, p}^\dagger, \mdc K_{n+1, p}^\dagger)$ is an octuple of data as in \cite[Notation~5.10.13]{LTXZZ}. For each $\alpha\in\{0, 1\}$, we set $\mdc K_{n_\alpha}^{\circ, p}\defining (\mdc K_{n_\alpha}^\circ)^p$, and $\mdc K_{n_\alpha}^\bullet\defining \mdc K_{n_\alpha}^{\circ, p}\times \mdc K_{n_\alpha, p}^\bullet$.
\item
$\mrs U=\paren{\mbf V_n', \mtt j_n, \{\Lbd_{n, \mfk q}'\}_{\mfk q|p}; \mbf V_{n+1}', \mtt j_{n+1}, \{\Lbd_{n+1, \mfk q}'\}_{\mfk q|p}}$ is a sextuple in which $\paren{\mbf V_n', \mtt j_n, \{\Lbd_{n, \mfk q}'\}_{\mfk q|p}}$ is an indefinite uniformization datum for $\mbf V_n^\circ$ as in Setup~\ref{issientieitiemfes}, $\mbf V_{n+1}'\defining (\mbf V_n')_\sharp$, $\mtt j_{n+1}\defining (\mtt j_n)_\sharp$, and $\Lbd_{n+1, \mfk q}=(\Lbd_{n, \mfk q})_\sharp$ for each $\mfk q|p$. Then $\paren{\mbf V_{n+1}', \mtt j_{n+1}, \{\Lbd_{n+1, \mfk q}'\}_{\mfk q|p}}$ is an indefinite uniformization datum for $\mbf V_{n+1}^\circ$. For each $\alpha\in\{0, 1\}$, let $\mdc K_{n_\alpha, \mfk q}'$ denote the stabilizer of $\Lbd_{n_\alpha, \mfk q}'$, and set $\mdc K_{n_\alpha, p}'\defining\prod_{\mfk q|p}\mdc K_{n_\alpha, \mfk q}'$.
\end{itemize}
\end{setup}

For each $\alpha\in\{0, 1\}$, we set $\bx X^?_{n_\alpha}\defining \bx X_{\mfk p}^?(\mbf V_{n_\alpha}^\circ, \mdc K_{n_\alpha}^{p, \circ})$ for meaningful pairs $(\bx X, ?)\in \{\mbf M, \bx M, \bx B, \bx S\}\times \{\ ,\eta,\circ,\bullet, \dagger\}$, and let $({}^\alpha\bx E_s^{p, q}, {}^\alpha\bx d_s^{p, q})$ denote the weight spectral sequence abutting to the cohomology $\bx H_{\mfk T}^\bullet(\ovl{\bx M}_{n_\alpha}, \bx R\Psi\mcl O_\lbd(r_\alpha))$ from \cite[\S5.9]{LTXZZ}.

\begin{assumption}\label{ieieniemeifs}
$\rho_{\Pi_0, \lbd}$ and $\rho_{\Pi^\flat_1, \lbd}$ are residually absolutely irreducible.
\end{assumption}

Under Assumption~\ref{ieieniemeifs}, for each $\alpha\in\{0, 1\}$, we get a residual representation $\ovl\rho_{\Pi_\alpha, \lbd}$, which is unique up to conjugation and $(1-n_\alpha)$-polarizable in the sense of \cite[Definition~2.5.3]{LTXZZ}. Then we obtain a continuous homomorphism
\begin{equation}\label{ititttemows}
\ovl\rho_{\Pi_\alpha, \lbd, +}: \Gal_{F_+}\to \mrs G_{n_\alpha}(\kappa_\lbd)
\end{equation}
from \cite[Construction~2.5.4]{LTXZZ}.

\begin{assumption}\label{issmienfeiifems}
Assumption~\ref{ieieniemeifs} holds, $\ovl\rho_{\Pi_0, \lbd, +}$ is rigid for $(\Pla^{\Pi_0}_+, \Pla^{\lr, \bx{\Rmnum1}}_+)$ (see Definition~\ref{rigidindiremiesLGoos}), and $\ovl\rho_{\Pi_0, \lbd}|_{\Gal_{F(\mu_\ell)}}$ is absolutely irreducible.
\end{assumption}

\begin{assumption}\label{isnsiiehifeireiheiss}
For each $\alpha\in\{0, 1\}$, the composite homomorphisms $\bb T_{n_\alpha}^{\Pla^{\bx{min}_+}}\xr{\phi_{\Pi_\alpha}}\mcl O_E\to \kappa_\lbd$ is cohomologically generic (see Definition~\ref{cohomomoegienifmos}).
\end{assumption}

In the following we will freely use the notation from \cite[\S7.2]{LTXZZ}.

We apply the construction and notation of \cite[\S5.11]{LTXZZ}, evaluating on the object $(\mdc K_n^{\circ, p}, \mdc K_{n+1}^{\circ, p})\in \mfk K(\mbf V_n^\circ)^p\times \mfk K(\mbf V_{n+1}^\circ)^p$. In particular, we obtain the blow-up morphism $\sigma: \mbf Q\to \mbf P$ from \cite[Notation~5.11.1]{LTXZZ}, and the localized weight spectral sequence $\paren{\bb E^{p, q}_{s, (\mfk m_0, \mfk m_1)}, \bx d^{p, q}_{s, (\mfk m_0, \mfk m_1)}}$ abutting to the cohomology $\bx H_{\mfk T}^\bullet(\ovl{\bx Q}, \bx R\Psi\mcl O_\lbd(n))_{(\mfk m_0,\mfk m_1)}$ from \cite[(5.27)]{LTXZZ}.

\begin{lm}\label{Kemmeiefeifemis}
Assume \textup{Assumptions~\ref{issmienfeiifems}, \ref{isnsiiehifeireiheiss}} and \textup{Hypothesis~\ref{iifififmieiemss}} for each $N\in\{n, n+1\}$. Then
\begin{enumerate}
\item
For any $(?_0, ?_1)\in\{\circ, \bullet, \dagger\}^2$ and any $i\in\bb Z$, there is a canonical isomorphism
\begin{equation*}
\bx H^{i}_{\mfk T}\paren{\ovl{\bx P}^{?_0, ?_1}, \mcl O_\lbd(i)}_{(\mfk m_0, \mfk m_1)}\cong \bplus_{i_0+i_1=i}\bx H^{i_0}_{\mfk T}\paren{\ovl{\bx M}_{n_0}^{?_0}, \mcl O_\lbd}_{\mfk m_0}\otimes_{\mcl O_\lbd}\bx H^{i_0}_{\mfk T}\paren{\ovl{\bx M}_{n_1}^{?_1}, \mcl O_\lbd}_{\mfk m_1}
\end{equation*}
in $\Mod(\Gal_{\bb F_{p^2}}, \mcl O_\lbd)_{\bx{fr}}$.
\item
$\bb E^{p, q}_{2, (\mfk m_0, \mfk m_1)}$ vanishes unless $(p, q)\in\{(-1, 2n), (0, 2n-1), (1, 2n-2)\}$, and canonical isomorphisms
\begin{equation*}
\begin{cases}
&\bb E_{2, (\mfk m_0, \mfk m_1)}^{-1, 2n}\cong {}^0\bx E_{2, \mfk m_0}^{-1, 2r_0}\otimes_{\mcl O_\lbd}{}^1\bx E^{0, 2r_1}_{2, \mfk m_1},\\
&\bb E_{2, (\mfk m_0, \mfk m_1)}^{0, 2n-1}\cong {}^0\bx E_{2, \mfk m_0}^{0, 2r_0-1}\otimes_{\mcl O_\lbd}{}^1\bx E^{0, 2r_1}_{2, \mfk m_1},\\
&\bb E_{2, (\mfk m_0, \mfk m_1)}^{1, 2n-2}\cong {}^0\bx E_{2, \mfk m_0}^{1, 2r_0-2}\otimes_{\mcl O_\lbd}{}^1\bx E^{0, 2r_1}_{2, \mfk m_1},
\end{cases}
\end{equation*}
in $\Mod(\Gal_{\bb F_{p^2}}, \mcl O_\lbd)_\lr$. In particular, $\bx H_{\mfk T}^i\paren{\ovl{\bx Q}, \bx R\Psi\mcl O_\lbd(n)}_{(\mfk m_0,\mfk m_1)}$ vanishes unless $i=2n-1$.
\item
If $\bb E^{i, 2n-1-i}_{2, (\mfk m_0, \mfk m_1)}(-1)$ has a nontrivial subquotient on which $\Gal_{\bb F_{p^2}}$ acts trivially, then $i=1$.
\item
For any $(?_0, ?_1)\in\{\circ, \bullet, \dagger\}^2$ and any $i\in\bb Z$, $\bx H^{2i}_{\mfk T}\big(\ovl{\bx Q}^{?_0, ?_1}, \mcl O_\lbd(i)\big)_{(\mfk m_0, \mfk m_1)}$ is weakly semisimple.
\item
The canonical map $\bx H^i_{\mfk T, c}(\ovl{\bx Q}^{(c)}, \mcl O_\lbd)_{(\mfk m_0, \mfk m_1)}\to \bx H^i_{\mfk T}(\ovl{\bx Q}^{(c)}, \mcl O_\lbd)_{(\mfk m_0, \mfk m_1)}$ is an isomorphism for any integers $c$ and $i$.
\end{enumerate}
\end{lm}
\begin{proof}
For (1), By \cite[Lemma~5.6.2]{LTXZZ}, Lemma~\ref{isnsnieiiheinfies}(1) and Lemma~\ref{isnsieiINEifehiw}(1), we know that $\bx H^{i_\alpha}\paren{\ovl{\bx M}_{n_\alpha}^{?_\alpha}, \mcl O_\lbd}_{\mfk m_\alpha}$ is a free $\mcl O_\lbd$-module for every $(\alpha, i_\alpha,  ?_\alpha)\in \{0, 1\}\times\bb N\times\{\circ, \bullet, \dagger\}$. Thus (1) follows from Lemma~\ref{isnsieoeimeifes} and the \Kunneth formula.

For (2), Using Lemma~\ref{issieiemfieehireriems}, Propositions~\ref{isnsnieiiheinfies}, \ref{isnsieiINEifehiw}(2) and Lemma~\ref{isnsieoeimeifes}, the proof of \cite[Lemma~7.2.5(2)]{LTXZZ} goes through.

For (3), by inspecting the proof of \cite[Lemma~7.2.5(3)]{LTXZZ}, the assertion follows from Proposition~\ref{isnsnieiiheinfies}(2) and Proposition~\ref{isnsieiINEifehiw}(3).

For (4): Using Proposition~\ref{isnsnieiiheinfies}, the proof of \cite[Lemma~7.2.5(4)]{LTXZZ} goes through.

Part (5) follows from part (1), Lemma~\ref{isnsieoeimeifes} and \cite[Lemma~5.11.3(3-5)]{LTXZZ}.
\end{proof}

By Lemma~\ref{Kemmeiefeifemis}(2), we obtain a coboundary map
\begin{equation*}
\AJ_{\mbf Q}: Z_{\mfk T}^n(\mbf Q^\eta)\to \bx H^1\paren{\bb Q_{p^2}, \bx H^{2n-1}_{\mfk T}\paren{\ovl{\bx Q}, \bx R\Psi\mcl O_\lbd(n)}_{(\mfk m_0, \mfk m_1)}}.
\end{equation*}

By our choice of $\mdc K_n^\circ$ and $(\mdc K_\sp^\circ, \mdc K_{n+1}^\circ)$, we obtain a finite morphism
\begin{equation*}
\mbf M_{\mfk p}(\mbf V_n^\circ, \mdc K_\sp^\circ)\to \mbf P.
\end{equation*}
Denote by $\mbf P_\sp$ the corresponding cycle, and by $\mbf Q_\sp$ the strict
transform of $\mbf P_\sp$ under $\sigma$, and $\bx Q_\sp$ the special fiber of $\mbf Q_\sp$.

We recall the construction of potential map from \cite[\S5.11]{LTXZZ}. For each $r\in\bb Z$, set
\begin{equation*}
B^r(\bx Q, \mcl O_\lbd)\defining \ker\paren{\delta_0^*: \bx H^{2r}_{\mfk T}\paren{\ovl{\bx Q}^{(0)}, \mcl O_\lbd(r)}\to \bx H^{2r}_{\mfk T}\paren{\ovl{\bx Q}^{(1)}, \mcl O_\lbd(r)}},
\end{equation*}
and
\begin{align*}
B_{n-r}(\bx Q, \mcl O_\lbd)\defining \coker\bigg(\delta_{1, !}: &\bx H^{2(n+r-2)}_{\mfk T}\paren{\ovl{\bx Q}^{(1)}, \mcl O_\lbd(n+r-2)}\\
&\to \bx H^{2(n+r-1)}_{\mfk T}\paren{\ovl{\bx Q}^{(0)}, \mcl O_\lbd(n+r-1)}\bigg),
\end{align*}
where $\delta_0^*$ is a linear combination of pullback maps and $\delta_{1, !}$ is a linear combination of pushforward maps; see \cite[p.~262]{LTXZZ}. Denote by $B^n(\bx Q, \mcl O_\lbd)^0$ and $B_n(\bx Q, \mcl O_\lbd)_0$ the kernel and cokernel of the tautological map
\begin{equation*}
B^n(\bx Q, \mcl O_\lbd)\to B_{n-1}(\bx Q, \mcl O_\lbd),
\end{equation*}
respectively. By \cite[Lemma~2.4]{Liu19}, the composite map
\begin{equation*}
\bx H_{\mfk T}^{2(n-1)}\paren{\ovl{\bx Q}^{(0)}, \mcl O_\lbd(n-1)}\xr{\delta_0^*}\bx H_{\mfk T}^{2(n-1)}\paren{\ovl{\bx Q}^{(1)}, \mcl O_\lbd(n-1)}\xr{\delta_{1, !}}\bx H_{\mfk T}^{2n}\paren{\ovl{\bx Q}^{(0)}, \mcl O_\lbd(n)}
\end{equation*}
factors through a unique map $B_n(\bx Q, \mcl O_\lbd)_0\to B_n(\bx Q, \mcl O_\lbd)^0$. Set
\begin{equation*}
C_n(\bx Q, \mcl O_\lbd)\defining B_n(\bx Q, \mcl O_\lbd)_0^{\Gal_{\bx F_{p^2}}}, \quad C^n(\bx Q, \mcl O_\lbd)\defining B^n(\bx Q, \mcl O_\lbd)^0_{\Gal_{\bx F_{p^2}}}.
\end{equation*}
Then we obtain a \tbf{potential map}
\begin{equation*}
\Delta^n: C_n(\bx Q, \mcl O_\lbd)\to C^n(\bx Q, \mcl O_\lbd).
\end{equation*}
In particular, the cycle $\bx Q_\sp$ gives rise to a class $\cl(\bx Q_\sp)\in C^n(\bx Q, \mcl O_\lbd)$.

\begin{prop}\label{issienifeiehriefmew}
Assume \textup{Assumptions~\ref{issmienfeiifems}, \ref{isnsiiehifeireiheiss}} and \textup{Hypothesis~\ref{iifififmieiemss}} for each $N\in\{n, n+1\}$. There is a canonical isomorphism
\begin{equation*}
\bx H^1_\sing\paren{\bb Q_{p^2}, \bx H^{2n-1}_{\mfk T}\paren{\ovl{\bx Q}, \bx R\Psi\mcl O_\lbd(n)}_{(\mfk m_0, \mfk m_1)}}\cong\coker\Delta^n_{(\mfk m_0, \mfk m_1)},
\end{equation*}
under which $\partial \AJ_{\mbf Q}(\mbf Q_\sp^\eta)$ is identified with the image of $\cl(\bx Q_\sp)$ in $\coker\Delta^n_{(\mfk m_0, \mfk m_1)}$.
\end{prop}
\begin{proof}
Using Lemma~\ref{Kemmeiefeifemis}, the proof of \cite[Proposition~7.2.7]{LTXZZ} goes through.
\end{proof}

For each $\alpha\in\{0, 1\}$, we set $\bSh_{n_\alpha}'\defining \bSh(\mbf V'_{n_\alpha}, \mtt j_{n_\alpha}\mdc K_{n_\alpha}^{\circ, p}\mdc K_{n_\alpha, p}')$. By \cite[Construction~5.11.7 and Remark~5.11.8]{LTXZZ}, we obtain a map
\begin{equation*}
\nabla: C^n(\bx Q, \mcl O_\lbd)\to \mcl O_\lbd[\bSh(\mbf V_{n_0}^\circ, \mdc K_{n_0}^\circ)]\otimes_{\mcl O_\lbd}\mcl O_\lbd[\bSh(\mbf V_{n_1}^\circ, \mdc K_{n_1}^\circ)].
\end{equation*}

Under Assumption~\ref{ieieniemeifs} and Assumption~\ref{isnsiiehifeireiheiss},
\begin{equation*}
\etH^{2n}\paren{(\bSh_{n_0}'\times_{\Spec F}\bSh_{n_1}')_{\ovl F}, \mcl O_\lbd}_{(\mfk m_0, \mfk m_1)}
\end{equation*}
vanishes. This follows from \cite[Lemma~5.2.7]{LTXZZ}, Lemma~\ref{isnsieoeimeifes}, and the \Kunneth formula. In particular, we obtain an Abel--Jacobi map
\begin{equation*}
\AJ: \bx Z^n\paren{\bSh_{n_0}'\times_{\Spec F}\bSh_{n_1}'}\to \bx H^1\paren{F, \etH^{2n-1}\paren{\paren{\bSh_{n_0}'\times_{\Spec F}\bSh_{n_1}'}_{\ovl F}, \mcl O_\lbd(n)}_{(\mfk m_0, \mfk m_1)}}
\end{equation*}
and its natural projection
\begin{equation*}
\ovl\AJ: \bx Z^n\paren{\bSh_{n_0}'\times_{\Spec F}\bSh_{n_1}'}\to \bx H^1\paren{F, \etH^{2n-1}\paren{\paren{\bSh_{n_0}'\times_{\Spec F}\bSh_{n_1}'}_{\ovl F}, \mcl O_\lbd(n)}/(\mfk n_0, \mfk n_1)}.
\end{equation*}
Let $\bSh'_\sp$ denote the cycle given by the finite morphism $\bSh(\mbf V'_n, \mtt j_n\mdc K_\sp^{\circ, p}\mdc K_{n, p}')\to \bSh_{n_0}'\times_{\Spec F}\bSh_{n_1}'$.

\begin{prop}\label{lsisineioonfiemis}
Assume \textup{Assumptions~\ref{issmienfeiifems}, \ref{isnsiiehifeireiheiss}} and \textup{Hypothesis~\ref{iifififmieiemss}} for each $N\in\{n, n+1\}$.
\begin{enumerate}
\item
The map $\nabla$ descends modulo $(\mfk n_0, \mfk n_1)$ to an isomorphism
\begin{equation*}
\nabla_{/(\mfk n_0, \mfk n_1)}: \coker\Delta^n/(\mfk n_0,\mfk n_1)\xr\sim\mcl O_\lbd[\bSh(\mbf V_{n_0}^\circ, \mdc K_{n_0}^\circ)]\otimes_{\mcl O_\lbd}\mcl O_\lbd[\bSh(\mbf V_{n_1}^\circ, \mdc K_{n_1}^\circ)]/(\mfk n_0, \mfk n_1).
\end{equation*}
\item
The Hecke operator $(p+1)\mtt I^\circ_{n_0, \mfk p}\otimes\mtt T^\circ_{n_1, \mfk p}$ acts invertible on
\begin{equation*}
\mcl O_\lbd[\bSh(\mbf V_{n_0}^\circ, \mdc K_{n_0}^\circ)]\otimes_{\mcl O_\lbd}\mcl O_\lbd[\bSh(\mbf V_{n_1}^\circ, \mdc K_{n_1}^\circ)]/(\mfk n_0, \mfk n_1);
\end{equation*}
denote its inverse by $\mtt T^\circ$. Moreover,
\begin{equation*}
\nabla_{/(\mfk n_0, \mfk n_1)}(\partial_{\mfk p}\AJ_{\mbf Q})(\mbf Q_\sp^\eta)=\mtt T^\circ\uno_{\bSh(\mbf V^\circ_n, \mdc K_\sp^\circ)},
\end{equation*}
where $\uno_{\bSh(\mbf V^\circ_n, \mdc K_\sp^\circ)}$ is the pushforward of the characteristic function along the map $\bSh(\mbf V^\circ_n, \mdc K_\sp^\circ)\to \bSh(\mbf V^\circ_n, \mdc K_n^\circ)\times\bSh(\mbf V^\circ_{n+1}, \mdc K_{n+1}^\circ)$.
\item
\begin{align*}
&\exp_\lbd\paren{\partial_{\mfk p}\loc_{\mfk p}\ovl\AJ(\bSh'_\sp), \bx H^1_\sing\paren{F_{\mfk p}, \etH^{2n-1}\paren{\paren{\bSh_{n_0}'\times_{\Spec F}\bSh_{n_1}'}_{\ovl F}, \mcl O_\lbd(n)}/(\mfk n_0, \mfk n_1)}}\\
&=\exp_\lbd\paren{\uno_{\bSh(\mbf V^\circ_n, \mdc K_\sp^\circ)}, \mcl O_\lbd\Bkt{\bSh(\mbf V^\circ_{n_0}, \mdc K_{n_0}^\circ)\times \bSh(\mbf V^\circ_{n_1}, \mdc K_{n_1}^\circ)}/(\mfk n_0, \mfk n_1)}.
\end{align*}
\end{enumerate}
\end{prop}
\begin{proof}
For (1): We follow the proof of \cite[Theorem~7.2.8(2)]{LTXZZ}. Firstly, by Proposition~\ref{Kemmeiefeifemis}(1), Proposition~\ref{isnsnieiiheinfies}(4) and Proposition~\ref{isnsieiINEifehiw}(3), the map $\nabla_{/(\mfk n_0, \mfk n_1)}$ is surjective. Thus it remains to show that the domain and the target of $\nabla_{/(\mfk n_0, \mfk n_1)}$ are isomorphic as $\mcl O_\lbd$-modules. By the proof of \cite[Theorem~7.2.8(2)]{LTXZZ}, this follows from Proposition~\ref{issienifeiehriefmew}, Lemma~\ref{Kemmeiefeifemis}(2, 3), Proposition~\ref{isnsnieiiheinfies}(4), and Proposition~\ref{isnsieiINEifehiw}(4, 5).

For (2): $p+1$ is invertible in $\mcl O_\lbd$ by (PI2); $\mtt I_{n_0,\mfk p}^\circ\otimes\mtt T_{n_1, \mfk p}^\circ$ is invertible by (PI4, PI5), \cite[Propositions~B.3.5(1), B.4.3(2)]{LTXZZ} and \cite[Lemma~4.2.4(1)]{LTX24}; 

For (3): This follows from part (2) by the proof of \cite[Corollary~7.2.9]{LTXZZ}.
\end{proof}

\subsection{Admissible places}\label{osiieeuoeoutyrui}

We now work in the setting of Setup~\ref{isnsiemfmiemiw}.

\begin{defi}\label{oiamsisieiwps}
We say that a finite place $\lbd\in \fPla_E$, with underlying prime $\ell$, is an admissible place (with respect to $(\Pi_0, \Pi_1))$ if the following hold:\footnote{Compared to \cite[Definition~8.1.1]{LTXZZ}, we omitted assumption (L3) because we will not consider the Bloch--Kato Selmer group of the Galois representation $\rho_{\Pi_0, \lbd}\otimes\rho_{\Pi_1, \lbd}$.}
\begin{enumerate}
\item[(L1)]
$\ell\ge2(n_0+1)$;
\item[(L2)]
$\Pla^{\Pi_0}_+$ does not contain places lying above $\ell$;
\item[(L3)]
The residual representations $\ovl\rho_{\Pi_0, \lbd}$ and $\ovl\rho_{\Pi_1^\flat, \lbd}$ are both absolutely irreducible. Fix $\Gal_F$-stable $\mcl O_\lbd$-lattices $\bx R_0\subset\rho_{\Pi_0, \lbd}(r_0)$ and $\bx R_1^\flat\subset\rho_{\Pi_1^\flat, \lbd}(r_1)$ (which are unique up to homothety), together with isomorphisms $\Xi_0: \bx R_0\xr\sim \bx R_0^\vee(1)$ and $\Xi^\flat_1: \bx R_1^\flat\xr\sim (\bx R_1^\flat)^\vee$. Set $\bx R_1\defining \bx R_1^\flat\oplus\mcl O_\lbd$ and $\Xi_1\defining\Xi_1^\flat\oplus\id: \bx R_1\xr\sim\bx R_1^\vee$.
\item[(L4-1)]
One of the following holds:
\begin{enumerate}
\item
The image of $\Gal_F$ in $\GL(\ovl{\bx R_0})$ contains a nontrivial scalar element;
\item
$\ovl{\bx R_0}$ is a semisimple $\kappa_\lbd[\Gal_F]$-module and $\Hom_{\kappa_\lbd[\Gal_F]}(\End(\ovl{\bx R_0}), \ovl{\bx R_0})=0$;
\end{enumerate} 
\item[(L4-2)]
$(\bx{GI}^1_{F', \mrs P, \bx R_0, \bx R_1})$ from Lemma~\ref{osiseiheifmiesw} holds for $F'=F_{\bx{rflx}, +}$ and $\mrs P(T)=T^2-1$;
\item[(L5)]
The homomorphism $\ovl\rho_{\Pi_0, \lbd, +}$ is rigid for $(\Pla^{\Pi_0}_+, \vn)$ (see Definition~\ref{rigidindiremiesLGoos}), and $\ovl\rho_{\Pi_0, \lbd}|_{\Gal_{F(\mu_\ell)}}$ is absolutely irreducible; and
\item[(L6)]
The composite homomorphism $\bb T_{n_\alpha}^{\Pla^{\Pi_0}_+\cup\Pla^{\Pi_1}_+}\xr{\phi_{\Pi_\alpha}}\mcl O_E\to \kappa_\lbd$ is cohomologically generic (Definition~\ref{cohomomoegienifmos}) for every $\alpha\in\{0, 1\}$.
\end{enumerate}
\end{defi}

To end this subsection, we give several examples where it is known that all but finitely many finite places $\lbd$ of $E$ are admissible.

\begin{lm}\label{ssoeienfefmfeiskw}
Suppose that
\begin{enumerate}
\item
there exists an elliptic curve $A_0$ over $F_+$ \sut for every finite place $\lbd$ of $E$,
\begin{equation*}
\rho_{\Pi_0, \lbd}\cong\Sym^{n_0-1}\etH^1(A_{\ovl F}, E_\lbd)|_{\Gal_F};
\end{equation*}
\item
there exists a good inert place $\mfk p$ of $F_+$ (see \textup{Definition~\ref{issieniefeifmies}}) \sut $A_0$ has split multiplicative reduction at $\mfk p$, and $\Pi^\flat_{1, \mfk p}$ is a supercuspidal $B$-avoiding good representation (see \textup{Definition~\ref{osoeiiinfiemisw}}) for
\begin{equation*}
B=\{-\norml{\mfk p}, \norml{\mfk p}^{1\pm1}, \norml{\mfk p}^{1\pm3}, \ldots, \norml{\mfk p}^{1\pm(2r-1)}\}
\end{equation*}
\end{enumerate}
\wrt any isomorphism $\iota_\ell: \bb C\xr\sim\ovl{\bb Q_\ell}$ where $\ell$ is not a rational prime underlying $\mfk p$. Then all but finitely many finite places $\lbd$ of $E$ are admissible (with respect to $(\Pi_0, \Pi_1)$).
\end{lm}
\begin{proof}
We show that every condition in Definition~\ref{oiamsisieiwps} excludes only finitely many finite places of $E$. By \cite[Th\'eor\`eme~6]{Ser72}, for sufficiently large prime $\ell$, the homomorphism
\begin{equation*}
\ovl\rho_{A, \ell}|_{\Gal_F}: \Gal_F\to\GL\paren{\etH^1(A_{\ovl F},\bb F_\ell)}
\end{equation*}
is surjective. So we may assume that $\ell$ is large such that this is the case.

For (L1) and (L2), this is trivial.

For (L3), $\ovl\rho_{\Pi_0, \lbd}$ is clearly absolutely irreducible, and the condition that $\ovl\rho_{\Pi_1^\flat, \lbd}$ is absolutely irreducible only excludes finitely many finite places $\lbd$ of $E$ by~\cite[Theorem~4.5.(1)]{LTXZZa} and condition (2).

For (L4-1), condition (a) always holds.

For (L4-2), because $A_0$ has split multiplicative reduction at $\mfk p$, $\Pi_{0, \mfk p}$ is the Steinberg representation by \cite[\S15]{Roh94}. Thus (L4-2) excludes only finitely many finite places $\lbd$ of $E$, by the same reasoning as in the proof of \cite[Lemma 8.1.4]{LTXZZ}.

For (L5), by~\cite[Corollary~4.2]{LTXZZa}, the condition that $\ovl\rho_{\Pi_0, \lbd, +}$ is rigid for $(\Pla^{\min}_+, \vn)$ excludes only finitely many finite places $\lbd$ of $E$. The second condition is clearly satisfied.

For (L6), for each $\alpha\in\{0, 1\}$, we choose a finite place $w_\alpha$ of $F$ \sut $\Pi_{\alpha, w_\alpha}$ is unramified with Satake parameter $\{a_{\alpha, 1}, \ldots, a_{\alpha, n_\alpha}\}$. By Proposition~\ref{ieieeinfeieiites}, $|a_{\alpha, i}|=1$ for every $1\le i\le n_\alpha$. Thus, for every sufficiently large rational prime $\ell$, $a_{\alpha, i}/a_{\alpha_j}\ne \norml{w}$ for $1\le i\ne j\le n_\alpha$ even in $\ovl{\bb F_\ell}$. Suppose $\lbd$ is a finite place of $E$ lying above $\ell$. We fix an isomorphism $\iota_\ell: \bb C\xr\sim\ovl{\bb Q_\ell}$ which induces $\lbd$. Applying the Chebotarev density theorem to the representation $\ovl\rho_{\Pi, \lbd}\oplus\ovl\ve_\ell$ of $\Gal_F$, we see that there are infinitely many finite places $w'_\alpha$ of $F$ that are of degree $1$ over $\bb Q$ satisfying that 
\begin{itemize}
\item
$\Pi_{\alpha, w'_\alpha}$ is unramified with Satake parameter $\{a'_{\alpha, 1}, \ldots, a'_{\alpha, n_\alpha}\}$ in which $\iota_\ell(a'_{\alpha, i})$ is an $\ell$-adic unit for every $1\le i\le n_\alpha$, and
\item
$\iota_\ell(a'_{\alpha, i}/a'_{\alpha, j})\ne\norml{w'_\alpha}\in \ovl{\bb F_\ell}$ for $1\le i\ne j\le n_\alpha$. 
\end{itemize}
Then it follows from \cite[Theorem~1.5]{Y-Z25} that (L6) holds for $\lbd$.
\end{proof}

\begin{lm}\label{isimsieienfifmeiss}
Suppose that
\begin{enumerate}
\item
there exists a very good inert place $\mfk p$ of $F_+$ (see \textup{Definition~\ref{issieniefeifmies}}) \sut $\Pi_{0, \mfk p}$ is Steinberg, and $\Pi^\flat_{1, \mfk p}$ is unramified with Satake parameter not containing $1$; and
\item
for each $\alpha\in\{0, 1\}$, there exist a finite place $w_\alpha$ of $F$ \sut $\Pi_{\alpha, w_\alpha}$ is supercuspidal;
\end{enumerate}
Then all but finitely many finite places $\lbd$ of $E$ are admissible (with respect to $(\Pi_0, \Pi_1)$).
\end{lm}
\begin{proof}
We show that every condition in Definition~\ref{oiamsisieiwps} excludes only finitely many finite places of $E$.

For (L1) and (L2), this is trivial.

For (L3), this follows from \cite[Theorem 4.5.(1)]{LTXZZa} by (2).

For (L5), this follows from \cite[Theorem 4.8]{LTXZZa} by (2).

For (L6), this follows from the same reasoning as in the proof of Lemma~\ref{ssoeienfefmfeiskw}.

For (L4-1), this follows by the same reasoning as in the proof of \cite[Lemma 8.1.4]{LTXZZ}.

For (L4-2), this follows by the same reasoning as in the proof of \cite[Lemma 8.1.4]{LTXZZ}.
\end{proof}

\subsection{Proof of Theorem D}\label{psleiieuiremfs}

The following lemma is crucial for the proof of Theorem~\ref{ismsieiemiwmws}, which is essentially the solution of the Gan--Gross--Prasad conjecture for unitary groups \cites{J-R11, Zha14, BPLZZ, BCZ22}.

\begin{lm}\label{Gan-Gorsinsieinfeinids}
We work in the setting of \textup{Setup~\ref{isnsiemfmiemiw}}. If $L(\frac{1}{2}, \Pi_0\times\Pi_1)\ne 0$, then there exists
\begin{itemize}
\item a standard definite Hermitian space $\mbf V^\circ_n$ of dimension $n$ over $F$, together with a self-dual $\prod_{v\in\fPla_+\setm(\Pla^{\Pi_0}_+\cup\Pla^{\Pi_1}_+)}\mcl O_{F_v}$-lattice $\Lbd_n^\circ$ in $\mbf V_n^\circ\otimes_{F_+}\Ade_{F_+}^{\Pla_{+, \infty}\cup\Pla^{\Pi_0}_+\cup\Pla^{\Pi_1}_+}$, and we set $\mbf V_{n+1}^\circ=(\mbf V_n^\circ)_\sharp$ and $\Lbd_{n+1}^\circ=(\Lbd_n^\circ)_\sharp$.
\item objects $\mdc K_n\in \mfk K(\mbf V_n^\circ)$ and $(\mdc K_\sp^\circ, \mdc K_{n+1}^\circ)\in \mfk K(\mbf V_n^\circ)_\sp$ of the forms
\begin{equation*}
\mdc K_n^\circ=\prod_{v\in\Pla^{\Pi_0}_+\cup\Pla^{\Pi_1}_+}(\mdc K_n^\circ)_v\times\prod_{v\in\fPla_+\setm(\Pla^{\Pi_0}_+\cup\Pla^{\Pi_1}_+)}\bx U(\Lbd_n^\circ)(\mcl O_v),
\end{equation*}
\begin{equation*}
\mdc K_\sp^\circ=\prod_{v\in\Pla^{\Pi_0}_+\cup\Pla^{\Pi_1}_+}(\mdc K_\sp^\circ)_v\times\prod_{v\in\fPla_+\setm(\Pla^{\Pi_0}_+\cup\Pla^{\Pi_1}_+)}\bx U(\Lbd_n^\circ)(\mcl O_v),
\end{equation*}
\begin{equation*}
\mdc K_{n+1}^\circ=\prod_{v\in\Pla^{\Pi_0}_+\cup\Pla^{\Pi_1}_+}(\mdc K_{n+1}^\circ)_v\times\prod_{v\in\fPla_+\setm(\Pla^{\Pi_0}_+\cup\Pla^{\Pi_1}_+)}\bx U(\Lbd_{n+1}^\circ)(\mcl O_v),
\end{equation*}
satisfying
\begin{itemize}
\item
$\mdc K_{\sp, v}^\circ\subset \mdc K_{n, v}^\circ$ for $v\in\Pla^{\Pi_0}_+\cup\Pla^{\Pi_1}_+$, and
\item
$\mdc K_{n_0, v}^\circ$ is hyperspecial maximal subgroup of $\bx U(\mbf V_{n_\alpha}^\circ)(F_v)$ for $v\in\Pla^{\Pi_0}_+\setm\Pla^{\Pi_1}_+$,
\end{itemize}
\end{itemize}
\sut
\begin{equation*}
\sum_{s\in\bSh(\mbf V_n^\circ, \mdc K_\sp^\circ)}f(s)\ne 0
\end{equation*}
for some $f\in \mcl O_E\Bkt{\bSh(\mbf V_{n_0}^\circ, \mdc K_{n_0}^\circ)}\Bkt{\ker\phi_{\Pi_0}}\otimes_{\mcl O_E}\mcl O_E\Bkt{\bSh(\mbf V_{n_1}^\circ, \mdc K_{n_1}^\circ)}\Bkt{\ker\phi_{\Pi_1}}$. Here we regard $f$ as a function on $\bSh(\mbf V_n^\circ, \mdc K_\sp^\circ)$ via the map $\bSh(\mbf V^\circ_n, \mdc K_\sp^\circ)\to \bSh(\mbf V^\circ_n, \mdc K_n^\circ)\times\bSh(\mbf V^\circ_{n+1}, \mdc K_{n+1}^\circ)$.
\end{lm}
\begin{proof}
In view of Remark~\ref{ismsiehiemfies}, this follows from the direction $(1)\implies (2)$ of \cite[Theorem~1.1.5.1]{BCZ22}. Note that since our $\Pi_n$ and $\Pi_{n+1}$ are relevant representations of $\GL_n(\Ade_F)$ and $\GL_{n+1}(\Ade_F)$, respectively, the Hermitian space in (2) of \cite[Theorem~1.1.5.1]{BCZ22} is standard definite.
\end{proof}

\begin{thm}\label{iieiieierhfieiswp}
We work in the setting of \textup{Setup~\ref{isnsiemfmiemiw}}. Assume there is a finite place $w$ of $F$ lying above a place of $F_+$ inert in $F$ \sut $(\Pi^\flat_1)_w$ is square-integrable, and assume \textup{Hypothesis~\ref{iifififmieiemss}} for each $N\in\{n, n+1\}$. If the central critical value
\begin{equation*}
L(\frac{1}{2}, \Pi_0)\cdot L(\frac{1}{2}, \Pi_0\times\Pi_1^\flat)
\end{equation*}
does not vanish, then for all admissible finite places $\lbd$ of $E$ (with respect to $(\Pi_0, \Pi_1)$), the Bloch--Kato Selmer group $\bx H_f^1\paren{F, \rho_{\Pi_0, \lbd}(r_0)}$ vanishes.
\end{thm}
\begin{proof}
The proof is a variant of that of \cite[Theorem~8.2.2]{LTXZZ}. By Lemma~\ref{Gan-Gorsinsieinfeinids}, we may fix the choices of $\mbf V_n^\circ, \mbf V_{n+1}^\circ, \Lbd_n^\circ, \Lbd_{n+1}^\circ; \mdc K_n^\circ, \mdc K_\sp^\circ, \mdc K_{n+1}^\circ$ in that lemma \sut
\begin{equation*}
\sum_{s\in\bSh(\mbf V_n^\circ, \mdc K_\sp^\circ)}f(s)\ne 0
\end{equation*}
for some $f\in \mcl O_E\Bkt{\bSh(\mbf V_{n_0}^\circ, \mdc K_{n_0}^\circ)}\Bkt{\ker\phi_{\Pi_0}}\otimes_{\mcl O_E}\mcl O_E\Bkt{\bSh(\mbf V_{n_1}^\circ, \mdc K_{n_1}^\circ)}\Bkt{\ker\phi_{\Pi_1}}$.

Let $\lbd$ be an admissible finite place of $E$ with the underlying rational prime $\ell$. We choose a $\Gal_F$-stable $\mcl O_\lbd$-lattice $\bx R_0$ in $\rho_{\Pi_0, \lbd}(r_\alpha)$, unique up to homothety, with a fixed isomorphism $\Xi_0: \bx R_0\xr\sim \bx R_0^\vee(1)$; and a $\Gal_F$-stable $\mcl O_\lbd$-lattice $\bx R_1^\flat$ in $\rho_{\Pi_1^\flat, \lbd}(r_1)$, unique up to homothety, with a fixed isomorphism $\Xi^\flat_1: \bx R_1^\flat\xr\sim (\bx R_1^\flat)^\vee$. Set $\bx R_1\defining \bx R_1^\flat\oplus\mcl O_\lbd$, with a fixed isomorphism $\Xi: \bx R_1\xr\sim\bx R_1^\vee$. We write $\bx R\defining \bx R_0\otimes \bx R_1$ and $\Xi\defining \Xi_0\otimes \Xi_1: \bx R\xr\sim \bx R^\vee(1)$. Define two nonnegative integers $m_{\bx{per}}$ and $m_{\bx{lat}}$ as follows.
\begin{enumerate}
\item
Let $m_{\bx{per}}$ denote the largest nonnegative integer such that
\begin{equation*}
\sum_{s\in\bSh(\mbf V_n^\circ, \mdc K_\sp^\circ)}f(s)\in\lbd^{m_{\bx{per}}}\mcl O_E
\end{equation*}
for every $f\in \mcl O_E\Bkt{\bSh(\mbf V_{n_0}^\circ, \mdc K_{n_0}^\circ)}\Bkt{\ker\phi_{\Pi_0}}\otimes_{\mcl O_E}\mcl O_E\Bkt{\bSh(\mbf V_{n_1}^\circ, \mdc K_{n_1}^\circ)}\Bkt{\ker\phi_{\Pi_1}}$.
\item
We choose a standard indefinite Hermitian space $\mbf V_{n_1}$ over $F$ of rank $n_1$, together with a fixed isomorphism $\bx U((\mbf V_{n_1}^\circ)^\infty)\cong \bx U(\mbf V_{n_1}^\infty)$ of reductive groups over $\Ade_{F_+}^\infty$. In particular, we obtain the Shimura variety $\bSh(\mbf V_{n_1}, \mdc K_{n_1}^\circ)$. By Hypothesis~\ref{iifififmieiemss}, there is an isomorphism
\begin{equation*}
\etH^{2r_1}\paren{\bSh(\mbf V_{n_1}, \mdc K_{n_1}^\circ)_{\ovl F}, E_\lbd(r_1)}/\ker\phi_{\Pi_1}\cong (\bx R_1^\cc\otimes_{\mcl O_\lbd}E_\lbd)^{\oplus\mu_1}
\end{equation*}
of $E_\lbd[\Gal_F]$-modules for some positive integer $\mu_1\in\bb Z_+$. We fix a map
\begin{equation*}
\etH^{2r_1}\paren{\bSh(\mbf V_{n_1}, \mdc K_{n_1}^\circ)_{\ovl F}, \mcl O_\lbd(r_1)}/\ker\phi_{\Pi_1}\to (\bx R_1^\cc)^{\oplus\mu_1}
\end{equation*}
of $\mcl O_\lbd[\Gal_F]$-modules whose kernel and cokernel are both $\mcl O_\lbd$-torsion. Then we denote by $m_{\bx{lat}}$ the smallest nonnegative integer \sut both the kernel and the cokernel are annihilated by $\lbd^{m_{\bx{lat}}}$.
\end{enumerate}
We start to prove the theorem by contradiction, hence assume
\begin{equation*}
\dim_{E_\lbd}\bx H_f^1\paren{F, \rho_{\Pi_0, \lbd}(r_0)}\ge 1.
\end{equation*}
Tate a sufficiently large positive integer $m$ which will be determined later. By Lemma~\ref{iseinifiemesss}, we may apply \cite[Proposition~2.4.6]{LTXZZ} by taking $\Pla$ to be the set of places of $F$ lying above $\Pla^{\Pi_0}_+$. Then we obtain a submodule $S$ of $\bx H^1_{f, \bx R}(F, \ovl{\bx R_0}^{(m)})$ that is free of rank $1$ over $\mcl O_\lbd/\lbd^{m-m_\Pla}$ \sut $\loc_w|_S=0$ for every finite place $w$ of $F$ lying above $\Pla^{\Pi_0}_+$. We now apply the discussion of \cite[\S2.3]{LTXZZ} to the submodule $S\subset \bx H^1(F, \ovl{\bx R_0}^{(m)})$. By (L4-1) and \cite[Lemma~2.3.4]{LTXZZ}, we obtain an injective map
\begin{equation*}
\theta_S: \Gal(F_S/F_{\ovl\rho^{(m)}})\to \Hom_{\mcl O_\lbd}(S, \ovl{\bx R_0}^{(m)})
\end{equation*}
whose image generates an $\mcl O_\lbd$-submodule containing $\lbd^{\mfk r_{\ovl{\bx R_0}^{(m)}}}\Hom_{\mcl O_\lbd}(S, \ovl{\bx R_0}^{(m)})$, which further contains $\lbd^{\mfk r_{\bx R_0}}\Hom_{\mcl O_\lbd}(S, \ovl{\bx R_0}^{(m)})$ by \cite[Lemma~2.3.3]{LTXZZ} and (L3) (Here $\mfk r_{\ovl{\bx R_0}^{(m)}}$ and $\mfk r_{\bx R_0}$ are reducibility depths defined in \cite[Definition~2.3.2, Proposition~2.3.3]{LTXZZ}). By (L4-2) and Lemma~\ref{osiseiheifmiesw}, we may choose an element $(\gamma_0, \gamma_1, \xi)$ in the image of $(\ovl\rho_{\Pi_0, \lbd}^{(m)}, \ovl\rho_{\Pi_1, \lbd}^{(m)}, \ovl\ve_\ell^{(m)})|_{\Gal_{F_{\bx{rflx}, +}}}$ satisfying conditions (a-d) in Lemma~\ref{osiseiheifmiesw}. In particular, the natural inclusion
\begin{equation}\label{osoepofmiesss}
\big(\ovl{\bx R_0}^{(m)}\big)^{h_{\gamma_0}}\to \big(\ovl{\bx R}^{(m)}\big)^{h_{\gamma_0}\otimes h_{\gamma_1}}
\end{equation}
is an isomorphism of free $\mcl O_\lbd/\lbd^m$-modules of rank 1. By \cite[Proposition~2.6.6]{LTXZZ} (with $m_0=m_\Pla$ and $r_S=1$), we may fix an $(S, \gamma)$-abundant element $\Psi\in G_{S, \gamma}$ (see \cite[Definition~2.6.5]{LTXZZ}).

By the Chebotarev density theorem, we can choose a $\gamma$-associated place (see \cite[Definition~2.6.3]{LTXZZ}) $w_+^{(m)}$ of $F_+^{(m)}$ satisfying $\Psi_{w^{(m)}}=\Psi$ and whose underlying prime $\mfk p$ of $F_+$ (with its underlying rational prime $p$ and an isomorphism $\iota_p: \bb C\xr\sim \ovl{\bb Q_p}$ under which $\udl\tau_\infty$ and $\mfk p$ correspond) is a very good inert place satisfying (PI1)-(PI5) and
\begin{itemize}
\item[(PI6)]
the natural map
\begin{equation*}
\frac{\etH^{2r_1}\paren{\bSh(\mbf V_{n_1}, \mdc K_{n_1}^\circ)_{\ovl F}, \mcl O_\lbd(r_1)}}{\bb T^{\Pla^{\Pi_0}_+\cup\Pla^{\Pi_1}_+\cup\Pla_{F_+}(p)}_{n_1}\cap\ker\phi_{\Pi_1}}\to\frac{\etH^{2r_1}\paren{\bSh(\mbf V_{n_1}, \mdc K_{n_1}^\circ)_{\ovl F}, \mcl O_\lbd(r_1)}}{\ker\phi_{\Pi_1}}
\end{equation*}
is an isomorphism.
\end{itemize}

We can choose a quintuple $\mrs T=(\Phi, \mbf W_0, \mdc K_0^p, \iota_p, \varpi)$ as in \cite[\S5.1]{LTXZZ} with $\bb Q_p^\Phi=\bb Q_{p^2}$, an octuple
\begin{equation*}
\mrs V^\bullet=(\Lbd_{n, \mfk p}^\bullet, \Lbd_{n+1, \mfk p}^\bullet; \mdc K_{n, p}^\bullet, \mdc K_{n+1, p}^\bullet, \mdc K_{\sp, p}^\bullet; \mdc K_{n, p}^\dagger, \mdc K_{\sp, p}^\dagger, \mdc K_{n+1, p}^\dagger)
\end{equation*}
as in \cite[Notation~5.10.13]{LTXZZ}, and a sextuple $\mrs U$ as in Setup~\ref{ismsienifmeitjes}. We are now working in the setting of Setup~\ref{ismsienifmeitjes} with
\begin{equation*}
\lbd, \quad\Pla^{\lr, \bx{\Rmnum1}}_+=\vn, \quad\Pla^{\bx{\Rmnum1}}_+=\Pla^{\Pi_0}_+\cup\Pla^{\Pi_1}_+, \quad\mrs V^\circ=(\mbf V_n^\circ, \mbf V_{n+1}^\circ; \Lbd_n^\circ, \Lbd_{n+1}^\circ; \mdc K_n^\circ, \mdc K_\sp^\circ, \mdc K_{n+1}^\circ),\quad m, \quad \mfk p,\quad\mrs T, \quad\mrs V^\bullet, \quad\mrs U
\end{equation*}
specified.

By the definition of $m_{\bx{per}}$,
\begin{equation}\label{ospwiemieficcws}
\exp_\lbd\paren{\uno_{\bSh(\mbf V^\circ_n, \mdc K_\sp^\circ)}, \mcl O_\lbd\Bkt{\bSh(\mbf V^\circ_{n_0}, \mdc K_{n_0}^\circ)\times \bSh(\mbf V^\circ_{n_1}, \mdc K_{n_1}^\circ)}/(\mfk n_0, \mfk n_1)}\ge m-m_{\bx{per}},
\end{equation}
where $\uno_{\bSh(\mbf V^\circ_n, \mdc K_\sp^\circ)}$ is the pushforward of the characteristic function along the map $\bSh(\mbf V^\circ_n, \mdc K_\sp^\circ)\to \bSh(\mbf V^\circ_n, \mdc K_n^\circ)\times\bSh(\mbf V^\circ_{n+1}, \mdc K_{n+1}^\circ)$.

We claim that there exists an element $c_1\in \bx H^1(F, \ovl{\bx R_0}^{(m), \cc})$ \sut
\begin{equation}\label{opwpwoppp}
\exp_\lbd\paren{\partial_{\mfk p}\loc_{\mfk p}(c_1), \bx H^1_\sing(F_{\mfk p}, \ovl{\bx R_0}^{(m), \cc})}\ge m-m_{\bx{per}}-m_{\bx{lat}},
\end{equation}
and for every finite place $w$ of $F$ not lying above $\Pla^{\Pi_0}_+\cup\{\mfk p\}$, 
\begin{equation}\label{ismieieniikuwqoap}
\loc_w(c_1)\in\bx H^1_\ns(F_w, \ovl{\bx R_0}^{(m), \cc}).
\end{equation}

We first prove the theorem assuming the existence of such $c_1$. Fix a generator $s_1$ of the submodule $S\subset \bx H^1_{f, \bx R_0}(F, \ovl{\bx R}^{(m)})$. We also identify $\ovl{\bx R_0}^{(m), \cc}$ with $(\ovl{\bx R_0}^{(m)})^*(1)$ via the polarization $\Xi_0$. We now compute the local Tate pairing $\bra{s_1, c_1}_w$ (see \cite[Equation~(2.2)]{LTXZZ}) for every finite place $w$ of $F$.
\begin{itemize}
\item
Suppose $w$ is lying above $\Pla^{\Pi_0}_+$. Then $\loc_w(s_1)$ vanishes by our choice of $S$. Thus $\bra{s_1, c_1}_w=0$.
\item
Suppose $w$ is lying above $\Pla_{F_+}(\ell)$. Then by (L2), $(\bx R_0)_{\bb Q}$ is crystalline with Hodge--Tate weights in $[-r_0, r_0+1]$. Thus $\loc_w(c_1)$ is in $\bx H^1_\ns(F_w, \ovl{\bx R}^{(m)})$ by \cite[Lemma~2.4.3(2)]{LTXZZ} and (L1). By \eqref{ismieieniikuwqoap}, \cite[Lemma~2.2.7]{LTXZZ} and (L1), $\lbd^{m_\bx{dif}}\bra{s_1, c_1}_w$ vaniehes, where $\mfk d_\lbd=\lbd^{m_{\bx{dif}}}\subset\mcl O_\lbd$ is the different ideal of $E_\lbd$ over $\bb Q_\ell$.
\item
Suppose $w$ is not lying above $\Pla^{\Pi_0}_+\cup\Pla_{F_+}(\ell)\cup\{\mfk p\}$. Then by (L2), $\bx R_0$ is unramified. Thus $\loc_w(c_1)$ is in $\bx H^1_\ns(F_w, \ovl{\bx R}^{(m)})$ by \cite[Lemma~2.4.3(1)]{LTXZZ}. By \eqref{ismieieniikuwqoap} and \cite[Lemma~2.2.3]{LTXZZ}, $\bra{s_1, c_1}_w$ vanishes.
\item
Suppose $w$ is the unique place lying above $\mfk p$. Then
\begin{equation*}
\exp_\lbd\paren{\loc_w(s_1), \bx H^1_\ns(F_w, \ovl{\bx R_0}^{(m)})}\ge m-m_\Pla-\mfk r_{\bx R_0}
\end{equation*}
by \cite[Proposition~2.6.7]{LTXZZ}; and
\begin{equation*}
\exp_\lbd\paren{\bra{s_1, c_1}_w, \mcl O_\lbd/\lbd^m}\ge m-m_{\bx{per}}-m_{\bx{lat}}-m_\Pla-\mfk r_{\bx R_0}
\end{equation*}
by \eqref{opwpwoppp} and \cite{Rub00}*{Proposition \Rmnum1.4.3.(\rmnum2)}.
\end{itemize}
Therefore, as long as we take $m$ such that $m>m_{\bx{per}}+m_{\bx{lat}}+m_\Pla+\mfk r_{\bx R_0}+m_{\bx{dif}}$, we will have a contradiction to the relation
\begin{equation*}
\sum_{w\in\Pla_F^\infty}\bra{s_1, c_1}_w=0.
\end{equation*}

The theorem is proved assuming the claim.

We now consider the claim on the existence of $c_1$. 
It follows from (L5), (L6) and Proposition~\ref{isnsieiINEifehiw}(6) that there exists an isomorphism
\begin{equation*}
\ovl\Upsilon_0: \etH^{2r_0-1}\paren{\bSh\paren{\mbf V_{n_0}', \mtt j_{n_0}\mdc K_{n_0}^{\infty, p}\mdc K_{n_0, p}'}_{\ovl F}, \mcl O_\lbd(r_0)}/\mfk n_0\xr\sim \paren{\ovl{\bx R_0}^{(m), \cc}}^{\oplus\mu_0}
\end{equation*}
of $\mcl O_\lbd[\Gal_F]$-modules, for some positive integer $\mu_0\in\bb Z_+$. It follows from Lemma~\ref{changeoeinfifenfiems} that there exists an isomorphism
\begin{equation*}
\etH^{2r_1}(\bSh(\mbf V_{n_1}, \mdc K_{n_1}^\circ)_{\ovl F}, \mcl O_\lbd)_{\mfk m_1}\cong \etH^{2r_1}(\bSh(\mbf V_{n_1}', \mtt j_{n_1}(\mdc K_{n_1}^{p\circ})\mdc K_{n_1, p}')_{\ovl F}, \mcl O_\lbd)_{\mfk m_1}
\end{equation*}
of $\mcl O_\lbd[\Gal_F]$-modules. Thus, by (PI8) and the definition of $m_{\bx{lat}}$, we may fix a map
\begin{equation*}
\Upsilon_1: \frac{\etH^{2r_1}(\bSh(\mbf V_{n_1}', \mtt j_{n_1}(\mdc K_{n_1}^{p\circ})\mdc K_{n_1, p}')_{\ovl F}, \mcl O_\lbd(r_1))}{\bb T^{\Pla^{\Pi_0}_+\cup\Pla^{\Pi_1}_+\cup\Pla_{F_+}(p)}_{n_1}\cap\ker\phi_{\Pi_1}}\to\paren{\bx R_1^\cc}^{\oplus\mu_1}
\end{equation*}
of $\mcl O_\lbd[\Gal_F]$-modules whose kernel and cokernel are both annihilated by $\lbd^{m_{\bx{lat}}}$.

To continue, we adopt the notational abbreviation prior to Proposition~\ref{lsisineioonfiemis}. By Lemma~\ref{isnsieoeimeifes} and the \Kunneth formula, we obtain a map
\begin{equation*}
\ovl\Upsilon\defining \ovl\Upsilon_0\otimes\Upsilon_1: \etH^{2n-1}\paren{\paren{\bSh_{n_0}'\times_{\Spec F}\bSh_{n_1}'}_{\ovl F}, \mcl O_\lbd(n)}/(\mfk n_0, \mfk n_1)\to\paren{\ovl{\bx R}^{(m), \cc}}^{\oplus\mu_0\mu_1}
\end{equation*}
of $\mcl O_\lbd[\Gal_F]$-modules whose kernel and cokernel are both annihilated by $\lbd^{m_{\bx{lat}}}$. Consider the class
\begin{equation*}
\ovl\AJ(\bSh_\sp')\in\bx H^1\paren{F, \etH^{2n-1}\paren{\paren{\bSh_{n_0}'\times_{\Spec F}\bSh_{n_1}'}_{\ovl F}, \mcl O_\lbd(n)}/(\mfk n_0, \mfk n_1)}.
\end{equation*}
Here $\bSh'_\sp$ denotes the cycle associated to the finite morphism $\bSh(\mbf V'_n, \mtt j_n\mdc K_\sp^{\circ, p}\mdc K_{n, p}')\to \bSh_{n_0}'\times_{\Spec F}\bSh_{n_1}'$. It follows frm Proposition~\ref{lsisineioonfiemis}(3) and \eqref{ospwiemieficcws} that
\begin{equation}\label{osooeimsseiws}
\exp_\lbd\paren{\partial_{\mfk p}\loc_{\mfk p}\ovl\AJ(\bSh'_\sp), \bx H^1_\sing\paren{F_{\mfk p}, \etH^{2n-1}\paren{\paren{\bSh_{n_0}'\times_{\Spec F}\bSh_{n_1}'}_{\ovl F}, \mcl O_\lbd(n)}/(\mfk n_0, \mfk n_1)}}
\ge m-m_{\bx{per}}.
\end{equation}
For each $1\le i\le \mu_0$ and each $1\le j\le \mu_1$, let
\begin{align*}
\ovl\Upsilon_{i, j}: \etH^{2n-1}\paren{\paren{\bSh_{n_0}'\times_{\Spec F}\bSh_{n_1}'}_{\ovl F}, \mcl O_\lbd(n)}/(\mfk n_0, \mfk n_1)&\xr{\ovl\Upsilon}\paren{\ovl{\bx R}^{(m), \cc}}^{\oplus\mu_0\mu_1}\\
&=\paren{\ovl{\bx R_0}^{(m), \cc}}^{\oplus\mu_0\mu_1}\oplus \paren{\ovl{\bx R_0\otimes\bx R_1^\flat}^{(m), \cc}}^{\oplus\mu_0\mu_1}\\
&\xr{\pr_{i, j}}\ovl{\bx R_0}^{(m), \cc}
\end{align*}
denote the composition of $\ovl\Upsilon$ with the projection to the $(i, j)$-th $\ovl{\bx R_0}^{(m), \cc}$-factor, and set
\begin{equation*}
c_{i, j}\defining \bx H^1(F, \ovl\Upsilon_{i, j})(\AJ(\bSh_\sp'))\in \bx H^1(F, \ovl{\bx R_0}^{(m), \cc}).
\end{equation*}
Then it follows from \eqref{osooeimsseiws} and \eqref{osoepofmiesss} that
\begin{equation*}
\max_{1\le i\le\mu_0}\max_{1\le j\le \mu_1}\exp_\lbd\paren{\partial_{\mfk p}\loc_{\mfk p}(c_{i, j}), \bx H^1_\sing\paren{F_{\mfk p}, \ovl{\bx R_0}^{(m), \cc}}}\ge m-m_{\bx{per}}-m_{\bx{lat}}.
\end{equation*}
Thus we obtain \eqref{opwpwoppp} by taking $c_1=c_{i, j}$ for some $i, j$. On the other hand, by (L6),
\begin{equation*}
\mrs H_\alpha\defining \etH^{n_\alpha-1}\paren{\paren{\bSh_{n_\alpha}'}_{\ovl F}, \mcl O_\lbd(n)}_{\mfk m_\alpha}
\end{equation*}
is a finite free $\mcl O_\lbd$-module for each $\alpha\in\{0, 1\}$. By Lemma~\ref{isnsieoeimeifes} and the \Kunneth formula, the following composition map
\begin{equation*}
\mrs H_0\otimes_{\mcl O_\lbd}\mrs H_1\xr{1\otimes\Upsilon_1}\mrs H_0\otimes_{\mcl O_\lbd}(\bx R_1^{\flat, \cc}\oplus\mcl O_\lbd)^{\oplus\mu_1}\xr{1\otimes\pr_j}\mrs H_0\xr{\ovl\Upsilon_0}\paren{\ovl{\bx R_0}^{(m), \cc}}^{\oplus\mu_0}\xr{\pr_i}\ovl{\bx R_0}^{(m), \cc}
\end{equation*}
is equal to $\ovl\Upsilon_{i, j}$, where $\pr_i, \pr_j$ are obvious projection maps for every $1\le i\le \mu_0$ and every $1\le j\le \mu_1$. Thus
\begin{equation*}
c_{i, j}=\bx H^1_\sing(F, \ovl\Upsilon_{i, j})(\AJ(\bSh_\sp')).
\end{equation*}
Let $w$ be a finite place of $F$. By Lemma~\ref{issiienneifheims}, \ref{iseinifiemesss} and Hypothesis~\ref{iifififmieiemss}, $\mrs H_0\otimes_{\mcl O_\lbd}\mrs H_1$ is pure of weight $-1$ at $w$. Thus
\begin{equation*}
\bx H^1(F_w, \mrs H_0\otimes_{\mcl O_\lbd}\mrs H_1)
\end{equation*}
vanishes if $w$ is not lying above $\ell$, and 
\begin{equation*}
\bx H^1_f(F_w, \mrs H_0\otimes_{\mcl O_\lbd}\mrs H_1)=\bx H^1_{\bx{st}}(F_w, \mrs H_0\otimes_{\mcl O_\lbd}\mrs H_1)
\end{equation*}
if $w$ is lying above $\ell$. Then it follows from \cite[Theorem~5.9]{N-N16} and the proof of \cite[Theorem~3.1(ii)]{Nek00} that $\AJ(\bSh_\sp')$ is contained in $\bx H^1_f(F, \mrs H_0\otimes_{\mcl O_\lbd}\mrs H_1)$. Hence
\begin{equation*}
\bx H^1(F, (1\otimes\pr_j)\circ(1\otimes\Upsilon_1))(\AJ(\bSh_\sp'))\in \bx H^1_f(F, \mrs H_0),
\end{equation*}
by definition of Bloch--Kato Selmer groups. Therefore, for every finite place $w$ of $F$ not lying above $\Pla_+^{\Pi_0}\cup\{\mfk p\}$, 
\begin{equation*}
\loc_w(c_{i, j})=\bx H^1(F_w, \pr_i\circ\ovl\Upsilon_0)\paren{\loc_w\paren{\bx H^1(F, (1\otimes\pr_j)\circ(1\otimes\Upsilon_1))(\AJ(\bSh_\sp'))}}
\end{equation*}
is contained in $\bx H^1_\ns(F_w, \ovl{\bx R_0}^{(m), \cc})$ by \cite[Lemma~2.4.3]{LTXZZ} and the fact that $\bSh_{n_0}'$ has good reduction at $w$.

The claim is proved.
\end{proof}

\section{Theta correspondence}\label{ieiieiepeoies}

In this appendix, we review some results on automorphic representations and theta correspondence that will be useful to us.

Let $K$ be a local or global field of characteristic zero, and let $K_1$ be an extension field of $K$ with degree at most two. Let $\cc$ denote the element in $\Gal(K_1/K)$ ith fixed field $K$. We fix a nontrivial additive character $\psi$ of $K$ (resp. of $K\bsh \Ade_K$) if $K$ is local (resp. global). For an element $d\in K^\times$, let $\chi_d$ denote the quadratic character of $K$ (resp. of $K^\times\bsh \Ade_K^\times$) corresponding to the quadratic extension $K(\sqrt{d})/K$ via local (resp. global) class field theory when $K$ is local (resp. global).
\subsection{The groups}

Suppose $\eps\in\{\pm1\}$ is a sign and $\msf W$ is a finite dimensional vector space over $K_1$ of dimension $n$ equipped with a nondegenerate $\eps$-Hermitian $\cc$-sesquilinear form
\begin{equation*}
\bra{-, -}_{\msf W}:\msf W\times\msf W\to K_1.
\end{equation*}
We denote by $G(\msf W)$ the group of elements of $\GL(\msf W)$ preserving the form $\bra{-, -}_{\msf W}$:
\begin{equation*}
G(\msf W)(R)=\{g\in\GL(R): \bra{gv, gw}_{\msf W}=\bra{v, w}_{\msf W}\}.
\end{equation*}
If $K_1\ne K$ or $\eps=1$, let the discriminant $\disc(\msf W)$ and Hasse--Witt invariant $\eps(\msf W)$ of $\msf W$ be normalized as in \cite[\S2.1]{Pen25a}. In particular, if $K_1=K$ and $\eps=1$, and $\msf W$ has an orthogonal basis $\{v_1, \ldots, v_n\}$ with $\bra{v_i, v_i}=a_i\in K^\times$ for $1\le i\le n$, then
\begin{equation*}
\disc(\msf W)=(-1)^{n(n-1)/2}\prod_{i=1}^na_i.
\end{equation*}
For notational simplicity, we define $\disc(\msf W)=1$ and $\eps(\msf W)=1$ if $K_1=K$ and $\eps=-1$. Then the neutral component of $G(\msf W)$ is a reductive group over $K$. There are several cases to consider:
\begin{enumerate}
\item
If $K_1=K$ and $\eps=1$, then $G(\msf W)=\bx O(\msf W)$ is an orthogonal group. If $\dim\msf W$ is odd, then $G(\msf W)$ is split (resp. non-quasi-split) if $\eps(\msf W)=1$ (resp. $\eps(\msf W)=-1$). If $\dim\msf W$ is even, then $G$ is split if $\disc(\msf W)=1, \eps(\msf W)=1$, $G$ is non-quasi-split if $\disc(\msf W)=1, \eps(\msf W)=-1$, and $G$ is quasi-split but non-split if $\disc(\msf W)\ne 1$;
\item
If $K_1=K$ and $\eps=-1$, then $G(\msf W)=\Sp(\msf W)$ is a symplectic group;
\item
If $K_1\ne K$, then $G(\msf W)=\bx U(\msf W)$ is a unitary group. $G(\msf W)$ is quasi-split except when $\dim\msf W$ is even and $\eps(\msf W)=-1$, in which case it is non-quasi-split.
\end{enumerate}

If $K_1=K$ and $\eps=1$, the determinant map on $\GL(\msf W)$ restricts to a nontrivial quadratic character $\det$ of $G(\msf W)=\bx O(\msf W)$.

If $K_1=K$ and $\eps=-1$, we will consider metaplectic group $\Mp(\msf W)$, which is the unique nonsplit $\bb C^1$-covering of $G(\msf W)=\Sp(\msf W)$:
\begin{equation*}
1\to \bb C^1\to \Mp(\msf W)\to \Sp(\msf W)\to 1.
\end{equation*}
Here $\bb C^1$ is the group of norm-$1$ elements in $\bb C^\times$. We can write $\Mp(\msf W)=\Sp(\msf W)\rtimes\bb C^1$, with multiplication law given by 
\begin{equation*}
(g_1, z_1)\cdot (g_2, z_2)=(g_1g_2, z_1z_2\cdot c(g_1, g_2))
\end{equation*}
for $g_1, g_2\in \Sp(\msf W)$ and $z_1, z_2\in\bb C^1$, where $c$ is the 2-cocycle of $\Sp(\msf W)$ in $\{\pm1\}$ given in \cite{Rao93}. $\Mp(\msf W)$ has a natural subgroup
\begin{equation*}
\tld\Sp(\msf W)\defining \Sp(\msf W)\rtimes\{\pm1\}\subset\Mp(\msf W),
\end{equation*}
which is a nonsplit double cover of $\Sp(\msf W)$. Let $\omega_{\msf W, \psi}$ denote the Weil representation of $\Mp(\msf W)$ \wrt $\psi$, defined via the Heisenberg group attached to the symplectic space $(\msf W, 2\bra{-, -}_{\msf W})$. We continue to write $\omega_{\msf W, \psi}$ for its restriction to $\tld\Sp(\msf W)$. When $K$ is global, we simply write $\omega_{\msf W}$ for the Weil representation.

These classical groups arise naturally in Howe's theory of reductive dual pairs in the symplectic group. We recall some basic facts about these reductive dual pairs and the splitting of the metaplectic cover over them.

Let $W$ be a vector space over $K_1$ equipped with nondegenerate $\eps$-Hermitian $\cc$-sesquilinear form
\begin{equation*}
\bra{-, -}_W: W\times W\to K_1,
\end{equation*}
and let $V$ be a vector space over $K_1$ equipped with nondegenerate $(-\eps)$-Hermitian $\cc$-sesquilinear form
\begin{equation*}
\bra{-, -}_V: V\times V\to K_1.
\end{equation*}
We distinguish the following cases:
\begin{itemize}
\item
(Case U) $K_1\ne K$, $W$ is Hermitian and $V$ is skew-Hermitian, or $W$ is skew-Hermitian and $V$ is Hermitian;
\item
(Case SO1) $K_1=K$, $W$ is symplectic and $V$ is orthogonal with $\dim V$ odd;
\item
(Case O1S) $K_1=K$, $W$ is orthogonal with $\dim V$ odd and $V$ is symplectic;
\item
(Case O2S) $K_1=K$, $W$ is orthogonal with $\dim V$ even and $V$ is symplectic;
\item
(Case SO2) $K_1=K$, $W$ is symplectic and $V$ is orthogonal with $\dim V$ even.
\end{itemize}
We collectively refer to Cases O1S and O2S as Case OS, and refer to Cases SO1 and SO2 as Case SO.

Let $G$ and $H$ be algebraic groups over $K$ defined by
\begin{equation*}
G=
\begin{cases}
G(W) & \text{in Cases U, OS, SO2}\\
\tld\Sp(W) &\text{in Case SO1},
\end{cases}
\quad H=
\begin{cases}
G(V) & \text{in Cases U, O2S, SO}\\
\tld\Sp(V) &\text{in Case O1S}
\end{cases}.
\end{equation*}

Let $\bb W=W\otimes_{K_1}V$, regarded as a vector space over $K$ and equipped with a symplectic form
\begin{equation*}
\tr_{K_1/K}\paren{\bra{-, -}_W\otimes_{K_1}\bra{-, -}_V}.
\end{equation*}
Then $(G(W), H(V))$ is a reductive dual pair in the symplectic group $\Sp(\bb W)$, and there is a natural map \begin{equation*}
\iota: G\times H\to G(W)\times H(V)\to \Sp(\bb W).
\end{equation*}

\subsection{Local Gan--Gross--Prasad conjecture}

In this subsection, we assume that $K$ is a \nA local field. We will focus on the group $G$. Fix a nontrivial additive character $\psi$ of $K$, and set $m\defining \dim(W)$.

If $G$ is isomorphic to a metaplectic $\tld\Sp_{2n}(K)$, then we say an irreducible admissible representation $\pi$ of $G(K)$ is \tbf{genuine} if the nontrivial element in $\ker\paren{\tld\Sp_{2n}(K)\to \Sp_{2n}}(K)$ acts by $-1$. For simplicity, if $G$ is not metaplectic, then every irreducible admissible representation of $G(K)$ is called genuine.

Let $\Pi(G)$ denote the set of all irreducible admissible genuine representations of $G(K)$. Denote by $\Phi(G)$ the set of equivalence classes of representations $\phi$ of $W_{K_1}\times \SL_2$ of dimension
\begin{equation*}
\begin{cases}
m-1 &\text{in Case }\bx U\\
m &\text{in Case }\bx{SO}1\\
m+1 &\text{in Case }\bx{SO}2\\
m-1 &\text{in Case }\bx{O1S}\\
m &\text{in Case }\bx{O2S}\\
\end{cases}
\end{equation*}
which are
\begin{equation*}
\begin{cases}
\text{conjugate self-dual of sign }(-1)^{m-1} &\text{in Case }\bx U\\
\text{self-dual of sign 1 such that }\det(\phi)=\chi_W&\text{in Case }\bx{O2S}\\
\text{self-dual of sign 1 such that }\det(\phi)=\uno&\text{in Case }\bx{SO1}\\
\text{self-dual of sign }-\eps\text{ such that }\det(\phi)=\uno&\text{otherwise}\\
\end{cases}
\end{equation*}
Elements of $\Phi(G)$ are called \tbf{$L$-parameters} for $G$. We denote by $\Phi_\temp(G)$ the subset of equivalence classes of tempered $L$-parameters, that is, the set of $\phi\in \Phi(G)$ \sut $\phi(W_E)$ is precompact.


Recall that there is a canonical local Langlands reciprocity map (depending on $\psi$ in the metaplectic case)
\begin{equation*}
\rec_W: \Pi(G)\to \Phi(G);
\end{equation*}
see \cite{G-S12, Art13, KMSW, A-G17, C-Z21, Ish24}. For any $\pi\in\Pi(G)$, $\pi$ is tempered if and only if $\rec(\pi)$ is tempered. For $\phi\in\Phi(G)$, we denote by $\Pi_\phi$ the inverse image of $\phi$, called the \tbf{$L$-packet} of $\phi$ on $G$.





We now state the tempered Bessel case of the local Gan--Gross--Prasad conjecture.

\begin{thm}\label{ssoeieuriwos}\enskip
\begin{enumerate}
\item
Suppose we are in Case $\bx U$. Set $V_\sharp\defining V\oplus L_{(-1)^{\dim V}}$ where $L_{(-1)^{\dim V}}$ is the Hermitian space of dimension $1$ and discriminant $(-1)^{\dim V}$. 
For any $\phi\in \Phi_\temp(\bx U(V))$ and $\phi_\sharp\in\Phi_\temp(\bx U(V_\sharp))$, there exists 
\begin{itemize}
\item
a unique pair $(V^\bullet, V^\bullet_\sharp)$ in which $V^\bullet$ is a Hermitian space over $K_1$ with $\dim V^\bullet=\dim V$ and $V^\bullet_\sharp\defining V^\bullet\oplus L_{(-1)^{\dim V}}$; and
\item
a pair of irreducible admissible representations $(\pi, \pi_\sharp)\in \Pi_\phi(\bx U(V^\bullet))\times \Pi_{\phi_\sharp}(\bx U(V^\bullet_\sharp))$,
\end{itemize}
satisfying
\begin{equation*}
\Hom_{\bx U(V^\bullet)}(\pi\otimes\pi_\sharp, \bb C)\ne 0.
\end{equation*}
\item
Suppose we are in Case $\SO$. Set $V_\sharp\defining V\oplus L_{(-1)^{\dim V+1}}$ where $L_{(-1)^{\dim V+1}}$ is the quadratic space of dimension $1$ and discriminant $(-1)^{\dim V+1}$.
Let $V_0$ (resp. $V_1$) denote the unique even (resp. odd) dimensional element in the set $\{V, V_\sharp\}$. For $\phi_0\in \Phi_\temp(\bx O(V_0))$ and $\phi_1\in\Phi_\temp(\bx O(V_1))$, there exist
\begin{itemize}
\item
a unique pair $(V_0^\bullet, V_1^\bullet)$ in which $V_0^\bullet$ is a quadratic space with $\dim(V_0^\bullet)=\dim(V_0)$ and $\disc(V_0^\bullet)=\disc(V_0)$ and $V_1^\bullet=V_0^\bullet\oplus L_{(-1)^{\dim V+1}}$; and
\item
a pair of irreducible admissible representations $(\pi_0, \pi_1)\in \Pi_{\phi_0}(\bx O(V_0^\bullet))\times \Pi_{\phi_1}(\bx O(V^\bullet_1))$,
\end{itemize}
satisfying
\begin{equation*}
\Hom_{\bx O(V_0^\bullet)}(\pi_0\otimes\pi_1, \bb C)\ne 0.
\end{equation*}
\end{enumerate}
\end{thm}
\begin{proof}
Case U is established by Beuzart-Plessis \cites{Beu14, Beu15, Beu16}. Case SO is established in \cite[Theorem~5.6]{A-G17}, extending the Gross--Prasad conjecture in the special orthogonal case established by Waldspurger \cite{Wal10a, Wal12, Wal12a, Wal12b}. Note that the assumptions on local Langlands correspondence for orthogonal groups in \cite[Theorem~5.6]{A-G17} are established in \cite{Art13, Ish24} for odd special orthogonal groups and in \cite[Theorem~4.4]{C-Z21} for even orthogonal groups.
\end{proof}

\subsection{Local theta lifts and Prasad's conjectures}

In this subsection, we assume that $K$ is a local field. We fix a nontrivial additive character $\psi$ of $K$ and a pair of characters $\chi=(\chi_W, \chi_V)$ of $K_1^\times$ \sut 
\begin{enumerate}
\item
In Case U, $\chi_W|_{K^\times}=\chi_{K_1}^{\dim W}$ and $\chi_V|_{K^\times}=\chi_{K_1}^{\dim V}$;
\item
In Case SO, $\chi_W$ is trivial and $\chi_V=\chi_{\disc(V)}$.
\item
In Case OS, $\chi_W=\chi_{\disc(W)}$ and $\chi_V$ is trivial.
\end{enumerate}
Note that $\chi_W^\cc=\chi_W^{-1}$ and $\chi_V^\cc=\chi_V^{-1}$.

Using $\psi$ and $\chi$, the natural map
\begin{equation*}
\iota_{W, V}: G\times H\to \Sp(\bb W)
\end{equation*}
can be lifted to a homomorphism 
\begin{equation*}
\tilde\iota_{W, V, \chi, \psi}\to \Mp(\bb W);
\end{equation*}
see \cite{Kud94} and \cite{HKS96}*{\S 1}.

Let $\omega_{\bb W, \psi}$ denote the Weil representation of $\Mp(\bb W)$ \wrt $\psi$. Using this splitting $\tilde\iota_{W, V, \chi, \psi}$, we obtain a representation
\begin{equation*}
\omega_{W, V, \psi, \chi}\defining \omega_{\bb W, \psi}\circ \tilde\iota_{W, V, \psi, \chi}
\end{equation*}
of $G\times H$, called the Weil representation of $G\times H$ (\wrt the auxiliary data above).
 
For any irreducible admissible genuine representation $\pi$ of $G(K)$, the maximal $\pi$-isotypic quotient of $\omega_{W, V, \psi, \chi}$ is of the form
\begin{equation*}
\pi\boxtimes\Theta_{W, V, \psi, \chi}(\pi),
\end{equation*}
where $\Theta_{W, V, \psi, \chi}(\pi)$ is either zero or a finite length smooth representation of $H(K)$ \cite{Kud86}. Let $\theta_{W, V, \psi, \chi}(\pi)$ denote the maximal semisimple quotient of $\Theta_{W, V, \psi, \chi}(\pi)$. We have the following standard properties.

\begin{prop}\label{oeiieieikifimeis}\enskip
\begin{enumerate}
\item
$\theta_{W, V, \psi, \chi}(\pi)$ is either zero or irreducible.
\item
If $K$ is \nA and $\pi$ is supercuspidal, then $\Theta_{W, V, \psi, \chi}(\pi)$ is either zero or irreducible.
\item
If $K=\bb R$ and at least one of $G$ and $H$ is compact, then $\Theta_{W, V, \psi, \chi}(\pi)$ is either zero or irreducible.
\end{enumerate}
\end{prop}
\begin{proof}
The first two follow from the Howe duality conjecture \cites{Kud86,MVW87, How89, Wal90,G-T16}. The third one is due to Howe \cite{How89a}. 
\end{proof}

The following theorem is known as the local conservation relation (also called the local theta dichotomy); see \cites{HKS96, K-R05, Min12, G-S12, S-Z15}.

\begin{thm}\label{oeoeieifnies}
Suppose $K$ is \nA and we are in Case $\bx U$ or $\SO$. If $V'$ is another nondegenerate $(-\eps)$-Hermitian vector space over $K_1$ with 
\begin{equation*}
\dim V+\dim V'=2(\dim W+2-[K_1: K]), \quad \eps(V')\ne \eps(V)
\end{equation*}
and moreover $\disc(V')=\disc(V)$ in Case $\SO$, then for any irreducible admissible genuine representation $\pi$ of $G(K)$, exactly one of the two theta lifts $\theta_{W,V, \psi, \chi}(\pi)$ and $\theta_{W, V', \psi, \chi}(\pi)$ is nonzero.
\end{thm}

To conclude this subsection, we recall Prasad's conjectures relating local theta correspondence and local Langlands correspondence.

\begin{thm}\label{sieieifnimidms}
Suppose $K$ is \nA and $\pi$ is an irreducible admissible representation of $G$.
\begin{enumerate}
\item
Suppose we are in Case $\bx U$ and $\dim V=\dim W$. Then there is a unique $(-\eps)$-Hermitian space $V^\bullet$ over $K_1$ with $\dim(V^\bullet)=\dim(V)$ \sut $\theta_{W, V^\bullet, \psi, \chi}(\pi)$ is nonzero. Moreover,
\begin{equation*}
\rec_{V^\bullet}(\theta_{W, V^\bullet, \psi, \chi}(\pi))=\rec_W(\pi)\otimes\chi_V^{-1}\chi_W.
\end{equation*}
\item
Suppose we are in Case $\bx U$ and $\dim V=\dim W+1$. If $\theta_{W, V^\bullet, \psi, \chi}(\pi)$ is nonzero, then
\begin{equation*}
\rec_{V^\bullet}(\theta_{W, V^\bullet, \psi, \chi}(\pi))=(\rec_W(\pi)\otimes\chi_V^{-1}\chi_W)\oplus\chi_W.
\end{equation*}
\item
Suppose we are in Case $\bx U$ and $\dim V=\dim W-1$. If $\rec_W(\pi)$ contains $\chi_V$ as a subrepresentation, then there is a unique $(-\eps)$-Hermitian space $V^\bullet$ over $K_1$ with $\dim(V^\bullet)=\dim(V)$ \sut $\theta_{W, V^\bullet, \psi, \chi}(\pi)$ is nonzero. Moreover,
\begin{equation*}
\rec_W(\pi)=(\rec_{V^\bullet}(\theta_{W, V, \psi, \chi}(\pi))\otimes\chi_W^{-1}\chi_V)\oplus\chi_V
\end{equation*}
\item
Suppose we are in Case $\bx{SO}1$ and $\dim V=\dim W+1$.
Then there exists a unique quadratic space $V^\bullet$ over $K$ with $\dim(V^\bullet)=\dim(V)$ and $\disc(V^\bullet)=\disc(V)$ \sut $\theta_{W, V^\bullet, \psi, \chi}(\pi)$ is nonzero. Moreover,
\begin{equation*}
\rec_{V^\bullet}(\theta_{W, V^\bullet, \psi, \chi}(\pi))=\rec_W(\pi)\otimes\chi_V.
\end{equation*}
\item
Suppose we are in Case $\bx O1\bx S$ and $\dim V=\dim W-1$.
Then there exists a unique element $\eps\in\{\pm1\}$ \sut $\theta_{W, V, \psi, \chi}(\pi\otimes\det\nolimits^{(1-\eps)/2})$ is nonzero. Moreover,
\begin{equation*}
\rec_V(\theta_{W, V, \psi, \chi}(\pi\otimes\det\nolimits^{(1-\eps)/2}))=\rec_W(\pi)\otimes\chi_W.
\end{equation*}
\item
Suppose we are in Case $\bx{SO}2$ and $\dim V=\dim W+2$. 
If $\theta_{W, V, \psi, \chi}(\pi)$ is nonzero, then
\begin{equation*}
\rec_V(\theta_{W, V, \psi, \chi}(\pi))=(\rec_W(\pi)\otimes\chi_V)\oplus\uno.
\end{equation*}
Here $\uno$ is the trivial representation of $W_K$.
\item
Suppose we are in Case $\bx O2\bx S$ and $\dim V=\dim W-2$. 
If $\rec_W(\pi)$ contains the trivial representation $\uno$ as a subrepresentation, then there exists a unique element $\eps\in\{\pm1\}$ \sut $\theta_{W, V, \psi, \chi}(\pi\otimes\det\nolimits^{(1-\eps)/2})$ is nonzero. Moreover,
\begin{equation*}
\rec_W(\pi)=(\rec_V(\theta_{W, V, \psi, \chi}(\pi\otimes\det\nolimits^{(1-\eps)/2}))\otimes\chi_W)\oplus\uno.
\end{equation*}
\end{enumerate}
\end{thm}
\begin{proof}
(4)-(6) are established by Gan--Ichino \cite{G-I16}. (4-5) are established by Gan-Savin \cite{G-S12} (cf. \cite[Theorem~B.8]{A-G17}). (6)-(7) are established by Atobe--Gan \cite[Theorem~4.4]{A-G17}.
\end{proof}

\subsection{Global theta lifts}\label{issoeoifmmieimws}

In this subsection, we assume that $K$ is a global field, and set $F\defining K, F_1\defining K_1$. We fix a conjugate self-dual automorphic character $\mu$ of $\Ade_{F_1}$ that satisfying $\mu_u(z)=z/\sqrt{z\ovl z}$ for every infinite place $u$ of $F_1$ and $z\in \bb C^\times$.

If $G$ (resp. $H$) is isomorphic to a metaplectic group $\tld\Sp_{2n}$, then the covering $\tld\Sp_{2n}(F_v)\to \Sp_{2n}(F_v)$ splits over the hyperspecial maximal compact subgroup $\mdc K_v$ for all but finitely many finite places $v$ of $F$. So we may regard $\mdc K_v$ as a compact open subgroup of $\tld\Sp_{2n}(F_v)$. In this case, the restricted tensor product 
\begin{equation*}
\prod_v{}'G(F_v)
\end{equation*}
\wrt the family $\{\mdc K_v\}_v$ contains $\oplus_v\mu_2$ as a central subgroup. Denote by $G(\Ade_F)$ (resp. $H(\Ade_F)$) the quotient of the above restricted tensor product by the central subgroup
\begin{equation*}
\{(z_v)\in\bplus\nolimits_v\mu_2: \prod_vz_v=1\}.
\end{equation*}
If $G$ (resp. $H$) is not isomorphic to a metaplectic group, we simply denote by $G(\Ade_F)$ (resp. $H(\Ade_F)$) the adelic points of $G$ (resp. $H$).

Similarly, for all but finitely many finite places $v$ of $F$, the metaplectic covering $\Mp(\bb W_v)\to \Sp(\bb W_v)$ splits over the hyperspecial maximal compact subgroup $\mdc K_v$. So we may regard $\mdc K_v$ as an open compact subgroup of $\Mp(\bb W_v)$. Then we define $\Mp(\bb W)(\Ade_F)$ as the quotient of the restricted tensor product
\begin{equation*}
\prod_v{}'\Mp(\bb W_v)
\end{equation*}
by the central subgroup
\begin{equation*}
\{(z_v)\in\prod_v\bb C^1: z_v=1\text{ for all but finitely many }v, \prod_vz_v=1\}.
\end{equation*}
The covering $\Mp(\bb W)(\Ade_F)\to \Sp(\bb W)(\Ade_F)$ canonically splits over the subgroup $\Sp(\bb W)(F)$. So we can regard $\Sp(\bb W)(F)$ as a subgroup of $\Mp(\bb W)(\Ade_F)$.

We fix a convenient set of parameters for the theta correspondence: a nontrivial additive character $\psi$ of $F\bsh \Ade_F$ and a pair of automorphic characters $\chi=(\chi_W, \chi_V)$ of $K_1\bsh\Ade_{K_1}^\times$ \sut 
\begin{enumerate}
\item
In Case U, $\chi_W=\mu^{(1+(-1)^{\dim W})/2}$ and $\chi_V=\chi^{(1+(-1)^{\dim V})/2}$;
\item
In Case SO, $\chi_W$ is trivial and $\chi_V=\chi_{\disc(V)}$;
\item
In Case OS, $\chi_W=\chi_{\disc(W)}$ and $\chi_V$ is trivial.
\end{enumerate}
Note that $\chi_W^\cc=\chi_W^{-1}$ and $\chi_V^\cc=\chi_V^{-1}$. The pair $(\chi, \psi)$ fixes a lifting
\begin{equation*}
\tilde\iota_{W,V}\defining\btimes_v{}'\tilde\iota_{W_v, V_v, \psi_v, \chi_v}: G(\Ade_F)\times H(\Ade_F)\to \Mp(\bb W)(\Ade_F)
\end{equation*}
of
\begin{equation*}
\iota_{W, V}\defining\btimes_v{}'\iota_{W_v, V_v}: G(\Ade_F)\times H(\Ade_F)\to \Sp(\bb W)(\Ade_F).
\end{equation*}
The global Weil representation $\omega_{\bb W}\defining \otimes'_v\omega_{\bb W_v, \psi_v}$ of 
\begin{equation*}
\prod_v{}'\Mp(\bb W_v)
\end{equation*}
factors through a representation $\omega_{\bb W}$ of $\Mp(\bb W)(\Ade_F)$. Using the lifting $\tilde\iota_{W, V}$, we obtain a representation
\begin{equation*}
\omega_{W, V}\defining \omega_{\bb W}\circ\tilde\iota_{W, V}
\end{equation*}
of $G(\Ade_F)\times H(\Ade_F)$. If $W$ is skew-Hermitian, we pair it with the $1$-dimensional Hermitian space $V'=F_1e$ with $\norml{e}=1$, and let $\omega_W$ denote the restriction of $\omega_{W, V'}$ to $G(\Ade_F)$, called the \tbf{Weil representation} of $G(\Ade_F)$. For each place $v$ of $F$, we denote by $\omega_{W_v, \psi_v}$ the local component of $\omega_W$, which is a representation of $G(F_v)$.

Let $\bb L$ be a Lagrangian subspace of $\bb W$. Then the Weil representation $\omega_{W, V}$ is realized on the space of Schwartz functions $\mcl S(\bb L(\Ade_F))$. For each Schwartz function $\phi\in \mcl S(\bb L(\Ade_F))$, define a theta function on $G(\Ade_F)\times H(\Ade_F)$ by
\begin{equation*}
\theta_{W, V}(g, h; \phi)\defining\sum_{x\in \bb L(F)}\omega_{W, V}(g, h)\phi(x), \quad (g, h)\in G(\Ade_F)\times H(\Ade_F).
\end{equation*}
Let $\pi\subset\mcl A_0(G(\Ade_F))$ be a genuine cuspidal automorphic representation of $G(\Ade_F)$. Then the \tbf{theta lift} $\theta_{W, V}(\pi)$ is defined to be the span of functions on $H(\Ade_F)$ of the form
\begin{equation*}
\theta_{W, V}(\vp; \phi): h\mapsto \int_{G(F)\bsh G(\Ade_F)}\ovl{\vp(g)}\theta_{W, V}(g, h; \phi)\bx dg,
\end{equation*}
for $\vp\in \pi$ and $\phi\in\mcl S(\bb L(\Ade_F))$. Here the measure $\bx dg$ denotes the Tamagawa measure on $G(F)\bsh G(\Ade_F)$ if $G$ is not metaplectic, and an arbitrary fixed Haar measure otherwise. Since theta functions are of moderate growth and $\vp$ is rapidly decreasing, these integrals converge and define automorphic forms on $H(\Ade_F)$.

We recall the following compatibility property between global and local theta lifts.

\begin{prop}\label{oeieeiiriemfies}
Let $\pi\subset\mcl A_0(G(\Ade_F))$ be a genuine cuspidal automorphic representation of $G(\Ade_F)$. If $\sigma\defining\theta_{W, V}(\pi)$ is contained in the space of square-integrable automorphic forms on $H(\Ade_F)$, then it is irreducible and isomorphic to the restricted tensor product $\otimes'_v\theta_{W_v, V_v, \psi_v, \chi_v}(\pi_v)$. If moreover $\sigma$ is cuspidal, then $\pi=\theta_{V, W}(\sigma)$.
\end{prop}
\begin{proof}
The first assertion follows from \cite[Corollary~7.1.3]{K-R94}. The second follows from \cite[Proposition 1.2]{GRS93}. 
\end{proof}

To end this section, we discuss relation between global theta correspondence and functorial lifts. We first recall the notion of Arthur parameters attached to discrete automorphic representations of orthogonal groups.

\begin{defi}\label{formaml-parameineirmes}
Let $\pi$ be an automorphic representation of $H(\Ade_F)$ contained in the space of square-integrable functions on $H(\Ade_F)$. There is a standard $L$-homomorphism 
\begin{equation*}
\xi:\LL H\to\LL\paren{\Res_{F_1/F}(\GL_N)_{F_1}}, \quad N=\dim V+
\begin{cases}
0 &\text{In Cases }\bx U\text{ or }\bx O1\bx S\text{ or }\bx{SO}2\\
1 &\text{In Case }\bx O2\bx S\\
-1 &\text{In Case }\bx{SO}1\\
\end{cases}
\end{equation*}
as defined in \cite[\S2.1]{Mok15} in Case U (the standard base change embedding) and in \cite[\S1.2]{Art13} in Case SO or OS. For each finite place $v$ of $F$ \sut $H_v$ is unramified, $\xi$ induces a map $\xi_*$ from the set of isomorphism classes of irreducible unramified representations of $\SO(\mbf V)(F_v)$ to that of $\GL_N(F_1\otimes_FF_v)$. A \tbf{functorial lift} of $\pi$ is defined to be an automorphic representation $\Pi$ of $\GL_N(\Ade_{F_1})$ that is a finite isobaric sum of discrete automorphic representations \sut $\Pi_v$ is isomorphic to $\FL(\pi_v)$ for all but finitely many finite places $v$ of $F$ \sut $\pi_v$ is unramified. A functorial lift $\FL(\pi)$ exists in Cases O2S, U, O1S, SO1, and SO2; see \cite{Art13}, \cite[Theorem~1.7.1]{KMSW}, \cite[Theorem~1.1]{G-I18}, \cite[Theorem~3.16]{Ish24}, \cite[Theorem~2.1]{C-Z24}, respectively. By strong multiplicity one for $\GL_N(\Ade_{F_1})$ \cite{Sha79}, this functorial lift is unique up to isomorphism, denoted by $\FL(\pi)$, and we will also call it the \tbf{Arthur parameter} of $\pi$. To align with the literature, in Case U we also refer to $\FL(\pi)$ as the \tbf{base change} of $\pi$ and write $\BC(\pi)$.
\end{defi}

\begin{prop}\label{peiemitimeifes}
Suppose we are in Cases $\bx U$ or $\SO$. Let $\pi\subset\mcl A_0(G(\Ade_F))$ be a cuspidal automorphic representation of $G(\Ade_F)$ such that $\theta_{W, V}(\pi)$ is an (irreducible) cuspidal automorphic representation of $H(\Ade_F)$. Then
\begin{equation*}
\FL(\theta_{W, V}(\pi))=
\begin{cases}
\BC(\pi) & \dim V=\dim W\text{ in Case }\bx U\\
(\BC(\pi)\otimes\mu^{-1})\boxplus \mu^{(1+(-1)^{\dim W})/2} & \dim V=\dim W+1\text{ in Case }\bx U\\
\FL(\pi) &\dim V=\dim W+1\text{ in Case }\bx{SO}1\\
(\FL(\pi)\otimes\chi_V)\boxplus\uno &\dim V=\dim W+2\text{ in Case }\bx{SO}2
\end{cases}.
\end{equation*}
\end{prop}
\begin{proof}
In Case U, this is \cite[Proposition~8.14]{Xue14}. In Case SO, by strong multiplicity one theorem \cite{J-S81}, it suffices to compare their localizations at finite places $v$ of $F$ where $H$ is split. Thus the assertion follows from Proposition~\ref{sieieifnimidms}.
\end{proof}

We recall the following criterion of nonvanishing of global theta lifts.

\begin{thm}\label{novnainsieiheirimfies}
Suppose we are in Case $\bx U$ or $\SO$. Suppose $\dim W=\dim V+1-[F_1: F]$ and $\pi\subset\mcl A_0(G(\Ade_F))$ is a genuine cuspidal automorphic representation of $G(\Ade_F)$. Assume that $\FL(\pi)_v$ is tempered for every finite place $v$ of $F$. If $\theta_{\mbf W, \mbf V}(\pi)$ is contained in $\mcl A_0(\bx U(\mbf V)(\Ade_F))$, then it is nonzero if and only if
\begin{itemize}
\item
for all places $v$ of $F$, the local theta lift $\theta_{W_v, V_v, \psi_v,\chi_v}(\pi_v)$ is nonzero, and
\item
$L(\FL(\pi)\otimes\chi_V; \frac{1}{2})$ is nonzero.
\end{itemize}
\end{thm}
\begin{proof}
This follows from \cite[Theorem~10.1]{Yam14}. In fact, it is not clear whether the standard $L$-function $L(s, \pi)$ for $\pi$ constructed by the doubling method in \cite{Yam14} and the standard $L$-function of $\FL(\pi)$ coincide. Nevertheless, it follows from Yamana's computation at unramified places \cite[Proposition~7.1]{Yam14} that their partial $L$-functions are equal. It follows from the temperedness assumption and \cite[Lemma~7.2]{Yam14} that
\begin{equation*}
\ord_{s=\frac{1}{2}}L(s, \pi\otimes\chi_V)=\ord_{s=\frac{1}{2}}L(s, \FL(\pi)\otimes\chi_V).
\end{equation*}
Now \cite[Theorem~10.1]{Yam14} applies.
\end{proof}

\section{Seesaw and proof of main theorems}\label{ososieifieiifmes}

In this section, we use seesaw identities (both local and global) to prove the main theorems. Let $r$ be a positive integer.

\subsection{The conjugate self-dual case}\label{psoseieiruefeimsis}

Let $F$ be a totally imaginary quadratic extension of a totally real number field $F_+$. Let $\mbf V_{2r}$ be a Hermitian space of dimension $2r$ over $F$, and $\mbf V_1$ be a Hermitian space of dimension $1$ over $F$ equipped with an element $e\in \mbf V_1$ satisfying $\norml{e}=1$. Let $\mbf W_{2r}$ be a skew-Hermitian space of dimension $2r$ over $F$. Set $\mbf V_{2r+1}\defining \mbf V_{2r}\oplus\mbf V_1$. Let $\iota: \bx U(\mbf V_{2r})\subset \mbf U(\mbf V_{2r+1})$ be the natural inclusion. We fix a nontrivial additive character $\psi$ of $F_+\bsh \Ade_{F_+}$, and use notations defined in \S\ref{ieiieiepeoies}.

Consider the inclusion
\begin{equation*}
\bx U(\mbf V_{2r})\times \bx U(\mbf V_1)\subset \bx U(\mbf V_{2r+1})
\end{equation*}
and the diagonal embedding
\begin{equation*}
\bx U(\mbf W_{2r})\subset \bx U(\mbf W_{2r})\times \bx U(\mbf W_{2r}).
\end{equation*}
$(\bx U(\mbf W_{2r}), \bx U(\mbf V_{2r+1}))$ and
\begin{equation*}
(\bx U(\mbf W_{2r})\times \bx U(\mbf W_{2r}), \bx U(\mbf V_{2r})\times \bx U(\mbf V_1))
\end{equation*}
are reductive dual pairs. In other words, there is a seesaw diagram:
\begin{equation}\label{sosiemfiemfiw}
\begin{tikzcd}[sep=large]
\bx U(\mbf W_{2r})\times\bx U(\mbf W_{2r})\ar[dr, dash] &\bx U(\mbf V_{2r+1})\\
\bx U(\mbf W_{2r})\ar[u, dash]\ar[ur, dash] &\bx U(\mbf V_{2r})\times\bx U(\mbf V_1)\ar[u, dash]
\end{tikzcd}.
\end{equation}

We fix a conjugate self-dual automorphic character $\mu$ of $\Ade_F$ satisfying $\mu_u(z)=z/\sqrt{z\ovl z}$ for $z\in \bb C^\times$ at every infinite place $u$ of $F$. Then we use the pair $(\psi, \chi)$ to define the (both local and global) theta correspondences between the pairs
\begin{equation*}
(\bx U(\mbf W_{2r}),\bx U(\mbf V_{2r})),\quad (\bx U(\mbf W_{2r}), \bx U(\mbf V_1)),\quad (\bx U(\mbf W_{2r}), \bx U(\mbf V_{2r+1}))
\end{equation*}
as defined in \S\ref{issoeoifmmieimws}. We record the following local seesaw identity attached to the seesaw diagram \eqref{sosiemfiemfiw}.

\begin{lm}\label{ososopeeiiemreis}
Let $\mfk p$ be a finite place of $F_+$ that is inert in $F$. For irreducible admissible representations $\pi_0$ of $\bx U(\mbf V_{2r})(F_{+, v})$ and $\sigma_1$ of $\bx U(\mbf W_{2r})(F_{+, v})$, there is a canonical isomorphism
\begin{align*}
\Hom_{\bx U(\mbf W_{2r})(F_{+, v})}(\Theta_{\mbf V_{2r}, \mbf W_{2r}}(\pi_0)&\otimes\omega_{\mbf W_{2r}}, \pi)\\
&\cong\Hom_{\bx U(\mbf V_{2r})(F_{+, v})}(\Theta_{\mbf W_{2r}, \mbf V_{2r+1}}(\sigma_1), \pi_0).
\end{align*}
\end{lm}
\begin{proof}
This is standard.
\end{proof}

We introduce the unitary Gan--Gross--Prasad periods and the Fourier--Jacobi periods.

\begin{defi}
Let $\pi_0\subset \mcl A_0(\bx U(\mbf V_{2r})(\Ade_{F_+}))$ and $\pi_1\subset \mcl A_0(\bx U(\mbf V_{2r+1})(\Ade_{F_+}))$ be cuspidal automorphic representations and $f_0\in\pi_0$ and $f_1\in\pi_1$ be cusp forms. We define the unitary \tbf{Gan--Gross--Prasad period}
\begin{equation*}
\mcl P_{\bx{GGP}}(f_0, f_1)\defining\int_{\bx U(\mbf V_{2r})(F_+)\bsh \bx U(\mbf V_{2r})(\Ade_{F_+})}f_0(h)f_1(\iota(h))\bx dh.
\end{equation*}
Here the measure $\bx dh$ is the Tamagawa measure on $\bx U(\mbf V_{2r})(\Ade_{F_+})$. This integral is absolutely convergent since $f_0$ and $f_1$ are rapidly decreasing.
\end{defi}

We set
\begin{equation*}
\bb W_{2r, 1}\defining\mbf W_{2r}\otimes_{F_+}\mbf V_1, \quad \bb W_{2r, 2r}\defining\mbf W_{2r}\otimes_{F_+}\mbf V_{2r}, \quad \bb W_{2r, 2r+1}\defining\mbf W_{2r+1}\otimes_{F_+}\mbf V_{2r+1}.
\end{equation*}
Then they are all symplectic spaces over $F_+$ as defined in \S\ref{issoeoifmmieimws}. Fix Lagrangian subspaces
\begin{equation*}
\bb L_{2r, 1}\subset \bb W_{2r, 1}, \quad \bb L_{2r, 2r}\subset \bb W_{2r, 2r},
\end{equation*}
then $\bb L_{2r, 2r}\defining \bb L_{2r, 2r}\oplus\bb L_{2r, 1}$ is a Lagrangian subspace of $\bb W_{2r, 2r+1}$. For each $n\in\{1, 2r, 2r+1\}$, let $\omega_{\mbf W_{2r}, \mbf V_n}$ denote the Weil representation, which can realized on the space of Schwartz functions $\mcl S(\bb L_{2r, n})$. Then
\begin{equation*}
\omega_{\mbf W_{2r}, \mbf V_{2r+1}}=\omega_{\mbf W_{2r}, \mbf V_{2r}}\hat\otimes\omega_{\mbf W_{2r}}.
\end{equation*}
In particular, if $\phi_{2r, 2r+1}=\phi_{2r,2r}\otimes\phi_{2r, 1}\in \mcl S(\bb L_{2r, 2r}(\Ade_{F_+}))\otimes\mcl S(\bb L_{2r, 1}(\Ade_{F_+}))$, then
\begin{equation*}
\theta_{\mbf W_{2r}, \mbf V_{2r+1}}(g, \iota(h); \phi_{2r, 2r+1})=\theta_{\mbf W_{2r}, \mbf V_{2r}}(g, h; \phi_{2r, 2r})\theta_{\mbf W_{2r}, \mbf V_1}(g, \phi_{2r, 1})
\end{equation*}
for every $(g, h)\in \bx U(\mbf W_{2r})(\Ade_{F_+})\times\bx U(\mbf V_{2r})(\Ade_{F_+})$.

\begin{defi}
Let $\sigma_0, \sigma_1\subset \mcl A_0(\bx U(\mbf W_{2r})(\Ade_{F_+}))$ be two cuspidal automorphic representations. Let $\vp_0\in\sigma_0, \vp_1\in\sigma_1$ be automorphic forms and $\phi\in\mcl S(\bb L_{2r, 1}(\Ade_{F_+}))$ be a Schwartz function. We define the \tbf{Fourier--Jacobi period}
\begin{equation*}
\mcl{FJ}(\vp_0, \vp_1; \phi)\defining\int_{\bx U(\mbf W_{2r})(F_+)\bsh\bx U(\mbf W_{2r})(\Ade_{F_+})}\vp_0(g)\vp_1(g)\theta_{\mbf W_{2r}, \mbf V_1}(g; \phi)\bx dg.
\end{equation*}
Here the measure $\bx dg$ is the Tamagawa measure on $\bx U(\mbf W_{2r})(\Ade_{F_+})$. This integral is absolutely convergent since $\vp_0$ and $\vp_1$ are rapidly decreasing and theta functions are of moderate growth.
\end{defi}

We will use the following global seesaw identity. 

\begin{lm}\label{osismeimefiemmitis}
Let $\sigma_1\subset \mcl A_0(\bx U(\mbf W_{2r})(\Ade_{F_+}))$ and $\pi_0\subset\mcl A_0(\bx U(\mbf V_{2r})(\Ade_{F_+}))$ be cuspidal automorphic representations such that
\begin{equation*}
\sigma_0=\theta_{\mbf V_{2r}, \mbf W_{2r}}(\ovl\pi_0)
\end{equation*}
is a cuspidal automorphic representation of $\bx U(\mbf W_{2r})(\Ade_F)$. Let $\vp_1\in \sigma_1$ and $f_0\in\pi_0$ be cusp forms and $\phi_{2r, 1}\in\mcl S(\bb L_{2r, 1}(\Ade_{F_+})), \phi_{2r, 2r}\in\mcl S(\bb L_{2r, 2r}(\Ade_{F_+}))$ be Schwartz functions. Then
\begin{align*}
\mcl{FJ}&\paren{\theta_{\mbf V_{2r}, \mbf W_{2r}}(\ovl f_0; \phi_{2r, 2r}), \vp_1; \phi_{2r, 1}}\\
&=\mcl P_{\bx{GGP}}\paren{f_0, \theta_{\mbf W_{2r},\mbf V_{2r+1}}(\ovl\vp_1; \phi_{2r, 2r}\otimes\phi_{2r, 1})}.
\end{align*}
\end{lm}
\begin{proof}
To save space, we write $[\bx U(\mbf V_{2r})]$ and $[\bx U(\mbf W_{2r})]$ for
\begin{equation*}
\bx U(\mbf V_{2r})(F_+)\bsh \bx U(\mbf V_{2r})(\Ade_{F_+})\quad\bx{ and }\quad \bx U(\mbf W_{2r})(F_+)\bsh \bx U(\mbf W_{2r})(\Ade_{F_+}),
\end{equation*}
respectively. Then
\begin{align*}
\mcl{FJ}&\paren{\theta_{\mbf V_{2r}, \mbf W_{2r}}(\ovl f_0; \phi_{2r, 2r}), \vp_1; \phi_{2r, 1}}\\
&=\int_{[\bx U(\mbf W_{2r})]}\vp_1(g)\theta_{\mbf W_{2r}, \mbf V_1}(g; \phi_{2r, 1})\int_{[\bx U(\mbf V_{2r})]}f_0(h)\theta_{\mbf W_{2r}, \mbf V_{2r}}(g, h; \phi_{2r, 2r})\bx dh\bx dg\\
&=\int_{[\bx U(\mbf V_{2r})]}f_0(h)\int_{[\bx U(\mbf W_{2r})]}\vp_1(g)\theta_{\mbf W_{2r}, \mbf V_1}(g; \phi_{2r, 1})\theta_{\mbf W_{2r}, \mbf V_{2r}}(g, h; \phi_{2r, 2r})\bx dg\bx dh\\
&=\int_{[\bx U(\mbf V_{2r})]}f_0(h)\int_{[\bx U(\mbf W_{2r})]}\vp_1(g)\theta_{\mbf W_{2r}, \mbf V_{2r+1}}(g, \iota(h); \phi_{2r, 2r}\otimes\phi_{2r, 1})\bx dg\bx dh\\
&=\mcl P_{\bx{GGP}}\paren{f_0, \theta_{\mbf W_{2r},\mbf V_{2r+1}}(\ovl\vp_1; \phi_{2r, 2r}\otimes\phi_{2r, 1})}
\end{align*}
\end{proof}

We now explain how to deduce Theorem~\ref{ismsieifmeifmss} from Theorem~\ref{ismsieiemiwmws}. The key ingredient is the following Burger--Sarnak type principle for Fourier--Jacobi periods on the pair of unitary groups $(\bx U(\mbf W_{2r}), \bx U(\mbf W_{2r}))$, in the spirit of \cites{B-S91, H-L98, Pra07, Zha14}. We first fix notation. For every infinite place $u$ of $F$, $\bx U(\mbf W_{2r})(F_{+, u})$ has a maximal compact subgroup $\mdc K_u\cong\bx U(r)\times\bx U(r)$. 
We fix such an isomorphism and denote by $\det_1^{m_1}\det^{m_2}$ the character of $\mdc K_u$ defined by
\begin{equation*}
(k_1, k_2)\mapsto\det(k_1)^{m_1}\det(k_2)^{m_2}.
\end{equation*}

\begin{prop}\label{ososieifnienfims}
Assume that $\mbf W_{2r}$ has signature $(r, r)$ at every infinite place. Suppose that
\begin{enumerate}
\item
$\Pla$ is a finite set of places of $F_+$ containing at least one finite place;
\item
$\sigma_0$ is an automorphic representation of $\bx U(\mbf W_{2r})(\Ade_{F_+})$; and
\item
$\otimes_{v\in\Pla}\tau_v$ is an irreducible admissible representation of $\prod_{v\in \Pla}\bx U(\mbf W_{2r})(F_{+, v})$ satisfying
\begin{enumerate}
\item
for every $v\in \Pla$, the space $\Hom_{\bx U(\mbf W_{2r})(F_{+, v})}(\sigma_{0,v}\otimes\omega_{\mbf W_{2r, v}, \psi_v}\otimes\tau_v, \bb C)$ is nonzero;
\item
for every finite place $v\in \Pla$, $\tau_v$ is compactly induced from an irreducible admissible representation $\nu_v$ of $Z_v\mdc K_v$, where $\mdc K_v$ is a compact open subgroup of $\bx U(\mbf W_{2r})(F_{+, v})$ and $Z_v$ is the center of $\bx U(\mbf W_{2r})(F_{+, v})$; and
\item
for every infinite place $u\in\Pla$, $\tau_u^\vee$ is a holomorphic discrete series that is a generalized Verma module in the sense of \cite{Gar05}. Moreover, if the lowest $\mdc K_u$-type of $\tau_u^\vee$ is the character $\det_1^{m_1}\det_2^{-m_2}$ for some positive integers $m_1, m_2$, then $\sigma_{0, u}$ has lowest $\mdc K_u$-type $\det_1^{m_1-1}\det_2^{-m_2}$ with multiplicity one.
\end{enumerate}
\end{enumerate}
Then there exists a cuspidal automorphic representation $\sigma_1$ of $\bx U(\mbf W_{2r})(\Ade_{F_+})$ satisfying
\begin{enumerate}
\item
for every place $v\in\Pla$, $\sigma_{1, v}$ is isomorphic to $\tau_v$; and
\item
there exist automorphic forms $\vp_0\in\sigma_0, \vp_1\in \sigma_1$ and a Schwartz function $\phi\in \mcl S(\bb L_{2r, 1}(\Ade_F))$ such that
\begin{equation*}
\mcl{FL}(\vp_0, \vp_1; \phi)\ne 0.
\end{equation*}
\end{enumerate}
\end{prop}
\begin{proof}
The proof is a variant of that of \cite[Proposition~2.14]{Zha14}. We write $\mbf G=\bx U(\mbf W_{2r})$. It follows from the hypothesis that $\tau_u$ is induced from its lowest $K$-type $\nu_u$. We consider the restriction of $\sigma_{0, v}\otimes\omega_{\psi_v, \chi_v}$ to $\mdc K_v$ for each $v\in\Pla$. By the assumption and Frobenius reciprocity, $\nu^\vee|_{\mdc K_v}$ is a quotient representation of $\sigma_{0, v}\otimes\omega_{\psi_v, \chi_v}|_{\mdc K_v}$. Because $\mdc K_v$ is compact, there exist an automorphic function $\vp_0\in\sigma_0$ on $\mbf G(\Ade_{F_+})$ and a Schwartz function $\phi\in \mcl S(\bb L_{2r, 1}(\Ade_{F_+}))$ \sut the $\prod_{v\in\Pla}\mdc K_v$-translates of $f\defining \vp_0\cdot\theta_{\mbf W_{2r}, \mbf V_1}(-; \phi)$ span a $\bb C$-vector space that is isomorphic to $\otimes_{v\in\Pla}\nu^\vee|_{\mdc K_v}$ as representations of $\prod_{v\in\Pla}\mdc K_v$. We can further assume that $f(1)\ne 0$. Indeed, $\mbf G(\Ade_{F_+}^\Pla)$ acts on the set of all such functions. If they all vanish at the identity element, then they would be identically zero by the weak approximation theorem according to which $\mbf G(F_+)$ is dense in $\mbf G(\Ade_{F_+, \Pla})$.

The group $\prod_{v\in\Pla}Z_v$ acts on $f$ by the character $\prod_{v\in\Pla}\omega_{\nu_v}^{-1}$, where $\omega_{\nu_v}$ is the central character of $\nu_v$ for every $v\in\Pla$. Thus the $\prod_{v\in\Pla}Z_v\mdc K_v$-translates of $f$ generates a $\bb C$-vector space that is isomorphic to $\prod_{v\in\Pla}\nu_v^{-1}$ as a representation of $\prod_{v\in\Pla}Z_v\mdc K_v$. As a result, if $v\in\Pla$ is finite, then the $\mbf G(F_{+, v})$-translates of $f$ generates a $\bb C$-vector space that is isomorphic to $\Ind_{Z_v\mdc K_v}^{\mbf G(F_{+, v})}\nu_v^{-1}$ as representations of $\mbf G(F_{+, v})$. On the other hand, if $u\in\Pla$ is infinite, then it follows from the relation between the lowest $\mdc K_u$-type of $\tau_u^\vee$ and $\sigma_{0, u}$ that $\bx U(\mfk g_u)\rtimes\mdc K_u$-translates of $f$ generate a $\bb C$-vector space that is isomorphic to $\tau_u^\vee$ as a $(\mfk g_u, \mdc K_u)$-module (Here $\bx U(\mfk g_u)$ is the universal enveloping algebra over $\bb C$ of the Lie algebra of $\mbf G(F_{+, u})$).

Since cusp forms are rapidly decreasing, $f$ is contained in $L^2(\mbf G(F_+)\bsh\mbf G(\Ade_{F_+}))$. Because the space of automorphic forms are $L^2$-dense, one can find an automorphic form $\vp_1$ on $\mbf G(\Ade_{F_+})$ \sut
\begin{equation*}
\int_{\mbf G(F_+)\bsh\mbf G(\Ade_{F_+})}f(g)\vp_1(g)\bx dg
\end{equation*}
is absolutely convergent and nonzero. Using Hecke projectors and properties of $f$, we can further assume that Hecke translates of $\vp_1$ generate a cuspidal automorphic representation $\sigma_1$ of $\mbf G(\Ade_{F_+})$ satisfying $\sigma_{1, v}\cong \tau_v$ for every $v\in\Pla$. Thus $\mcl{FL}(\vp_0, \vp_1; \phi)$ is nonzero.

The theorem is proved.
\end{proof}

We define the notion of admissible places for the coefficient field appearing in Theorem~\ref{ismsieifmeifmss}.

\begin{defi}\label{issisoeeiriens}
Let $\Pi$ be a relevant automorphic representation of $\GL_{2r}(\Ade_F)$ and $E$ be a strong coefficient field of $\Pi$ (see Definition~\ref{sisieifnieeimfsi}). We say that a finite place $\lbd$ of $\fPla_E$, with underlying prime $\ell$, is an admissible place (with respect to $\Pi$) if the following hold:
\begin{enumerate}
\item[($\Lbd$1)]
$\ell\ge 4r+2$.
\item[($\Lbd$2)]
$\Pla^\Pi_+$ does not contain places over $\ell$.
\item[($\Lbd$3)]
The residual representation $\ovl\rho_{\Pi_0, \lbd}$ is absolutely irreducible. Fix a $\Gal_F$-stable $\mcl O_\lbd$-lattice $\bx R\subset \rho_{\Pi, \lbd}(r)$ (which is unique up to homothety), together with an isomorphism $\Xi: \bx R\xr\sim \bx R^\vee(1)$.
\item[($\Lbd$4-1)]
Either one of the following two assumptions holds:
\begin{enumerate}
\item
The image of $\Gal_F$ in $\GL(\ovl{\bx R})$ contains a nontrivial scalar element.
\item
$\ovl{\bx R}$ is a semisimple $\kappa_\lbd[\Gal_F]$-module and $\Hom_{\kappa_\lbd[\Gal_F]}(\End(\ovl{\bx R}), \ovl{\bx R})=0$.
\end{enumerate} 
\item[($\Lbd$4-2)]
$(\bx{GI}^1_{F', \mrs P, \bx R})$ from Lemma~\ref{ososoieieiehfimiws} holds for $F'=F_{\bx{rflx}, +}$ and $\mrs P(T)=T^2-1$.
\item[($\Lbd$5)]
The homomorphism $\ovl\rho_{\Pi, \lbd, +}$ is rigid for $(\Pla^\Pi_+, \vn)$ (see Definition~\ref{rigidindiremiesLGoos}), and $\ovl\rho_{\Pi, \lbd}|_{\Gal_{F(\mu_\ell)}}$ is absolutely irreducible.
\item[($\Lbd$6)]
The composite homomorphism $\bb T_{2r}^{\Pla^\Pi_+}\xr{\phi_\Pi}\mcl O_E\to \kappa_\lbd$ is cohomologically generic (see Definition~\ref{INIENIEMisimw} and Definition~\ref{cohomomoegienifmos}).
\end{enumerate}
\end{defi}

\begin{lm}\label{sosieimfifmiesws}
Let $\Pi$ be a relevant automorphic representation of $\GL_{2r}(\Ade_F)$ and $E$ be a strong coefficient field of $\Pi$ (see Definition~\ref{sisieifnieeimfsi}). Suppose $F_+\ne \bb Q$ and one of the following two assumptions holds:
\begin{enumerate}
\item
$E=\bb Q$ and there exists a modular elliptic curve $A$ over $F_+$ with no complex multiplication over $\ovl F$ satisfying $\rho_{\Pi, \ell}\cong\Sym^{2r-1}\etH^1(A_{\ovl F}, \bb Q_\ell)|_{\Gal_F}$ for every rational prime $\ell$.
\item
There exists a finite place $w$ of $F$ \sut $\Pi_w$ is supercuspidal; and a good place $\mfk p$ of $F$ (see \textup{Definition~\ref{issieniefeifmies}}) \sut $\Pi_{\mfk p}$ is a Steinberg representation.
\end{enumerate}
Then all but finitely many finite places of $E$ are admissible (with respect to $\Pi$).
\end{lm}
\begin{proof}
We first consider case (1): By \cite[Th\'eor\`eme~6]{Ser72} and \cite{Lom15}, for sufficiently large rational prime $\ell$, the homomorphism
\begin{equation*}
\ovl\rho_{A, \ell}|_{F_{\bx{rflx}}}: \Gal_{F_{\bx{rflx}}}\to\GL\paren{\etH^1\paren{A_{\ovl F},\bb F_\ell}}
\end{equation*}
is surjective; let $\ell$ be such a rational prime. We fix an isomorphism $\etH^1\paren{A_{\ovl F},\bb F_\ell}\cong \bb F_\ell^{\oplus 2}$ \sut $\rho_{A, \ell}(\cc)$ is given by the matrix
\begin{equation*}
\begin{bmatrix} &1 \\1 & \end{bmatrix}\in\GL_2(\bb F_\ell).
\end{equation*}
We need to check that every condition in Definition~\ref{issisoeeiriens} excludes only finitely many rational primes $\ell$. 

For ($\Lbd$1-3) and ($\Lbd$4-1), this is clear.

For ($\Lbd$4-2), we suppose $\ell>2^{4r-2}$, so
\begin{equation*}
\Brace{2^{\pm1}, 2^{\pm3}, \ldots, 2^{\pm(2r-1)}}
\end{equation*}
consists of distinct elements in $\bb F_\ell$, and does not contain $-2\in\bb F_\ell$. We take an element $g\in\Gal_{F_{\bx{rflz}}}$ whose image under $\ovl\rho_{A, \ell}$ is
\begin{equation*}
\begin{bmatrix}2 & \\ & 1\end{bmatrix}\in \GL_2(\bb F_\ell).
\end{equation*}
Thus $(\bx{GI}^1_{F', \mrs P, \bx R})$ from Lemma~\ref{ososoieieiehfimiws} holds for $F'=F_{\bx{rflx}, +}$ and $\mrs P(T)=T^2-1$ holds by taking the image of $g\cc$ under $(\ovl\rho_{\Pi, \ell}, \ovl\ve_\ell)$.

For ($\Lbd$5), by~\cite[Corollary~4.2]{LTXZZa}, the condition that $\ovl\rho_{\Pi_0, \lbd, +}$ is rigid for $(\Pla^{\min}_+, \vn)$ excludes only finitely many finite places $\lbd$ of $E$. The second condition is clearly satisfied.

For ($\Lbd6)$, this follows from the same reasoning as in the proof of Lemma~\ref{ssoeienfefmfeiskw}.

We now consider case (2): We need to check that every condition in Definition~\ref{issisoeeiriens} excludes only finitely many rational primes $\ell$.

For ($\Lbd$1) and ($\Lbd$2), this is clear.

For ($\Lbd$3), this follows from \cite[Theorem~4.5.(1)]{LTXZZa}.

For ($\Lbd$4-1), this follows by the same reasoning as in the proof of \cite[Lemma~8.1.4]{LTXZZa}.

For ($\Lbd$4-2), note that, for all but finitely many finite place $\lbd$ of $E$ strongly disjoint from $\mfk p$,
\begin{equation*}
\{\norml{\mfk p}^{\pm1}\modu\lbd, \norml{\mfk p}^{\pm3}\modu\lbd, \ldots, \norml{\mfk p}^{2r-1}\modu\lbd\}
\end{equation*}
consists of distinct elements and does not contain $-1$. For every such $\lbd$ that also satisfies ($\Lbd$3), the condition $(\bx{GI}^1_{\bx R, F', \mrs P})$ from Lemma~\ref{ososoieieiehfimiws} holds for $F'=F_{\bx{rflx}, +}$ and $\mrs P(T)=T^2-1$, by taking the element $(\ovl\rho_{\bx R, \lbd}, \ovl\ve_\ell)(\phi_{\mfk p})$.

For ($\Lbd$5-2), this follows from \cite[Theorem~4.8]{LTXZZa}.

For ($\Lbd6)$, this follows from the same reasoning as in the proof of Lemma~\ref{ssoeienfefmfeiskw}.

Note that the primes that are excluded can be effectively bounded.
\end{proof}

We now prove Theorem~\ref{ismsieifmeifmss} using the Burger--Sarnak type principle (see Proposition~\ref{ososieifnienfims}) and seesaw identities.

\begin{thm}\label{osineiieihirmisws}
Let $\Pi_0$ be a relevant automorphic representation of $\GL_{2r}(\Ade_F)$ and $E$ be a strong coefficient  field of $\Pi_0$ (see \textup{Definition~\ref{sisieifnieeimfsi}}). If $L(\frac{1}{2}, \Pi_0)\ne 0$, then for every admissible place $\lbd$ of $E$ with respect to $\Pi_0$, the Bloch--Kato Selmer group $\bx H^1_f(F, \rho_{\Pi_0, \lbd}(n))$ vanishes.
\end{thm}
\begin{proof}
Let $\lbd$ be an admissible place with underlying rational prime $\ell$. We fix an isomorphism $\iota_\ell: \bb C\xr\sim\ovl{\bb Q_\ell}$ that induces the place $\lbd$. By ($\Lbd$4-2) and the Chebotarev density theorem, we can find a good inert place $\mfk p$ of $F_+$ (see \textup{Definition~\ref{issieniefeifmies}}) satisfying
\begin{itemize}
\item
the underlying prime of $\mfk p$ is larger than $\max(\ell, 2r+1)$; and
\item
$\ovl\rho_{\Pi, \lbd}(\phi_{\mfk p})$ has generalized eigenvalues $\{\norml{\mfk p}\cdot\alpha_{0, 1}^{\pm1}, \ldots, \norml{\mfk p}\cdot\alpha_{0, r}^{\pm1}\}\subset \ovl{\kappa_\lbd}^\times$ with $\alpha_{0, 1}=\norml{\mfk p}$ and $\alpha_{0, i}\notin\{\norml{\mfk p}^{\pm1}\modu\lbd\}$ for every $2\le i\le r$.
\end{itemize}
We take a $\Pi_{0, \mfk p}$-avoiding good representation $\Pi_{1, \mfk p}^{\flat, \prime}$ of $\GL_{2r}(F_{\mfk p})$ with respect to $\iota_\ell$ (see Definition~\ref{osoeiiinfiemisw}) satisfying
\begin{itemize}
\item
there exists a lift $F\in W_{F_{\mfk p}}$ of the arithmetic Frobenius element \sut the eigenvalues $\{\alpha_1, \ldots, a_{2r}\}$ of $\iota_\ell\rec_{2r}(\Pi_{1, \mfk p}^{\flat, \prime})(F^2)$ are $\ell$-adic units; and
\item
\begin{equation*}
\norml{\mfk p}^2\notin\{\alpha_i\alpha_j^{-1}|1\le i\ne j\le 2r\}\cup\{\alpha_i|1\le i\le 2r\}\subset\ovl{\bb F_\ell}.
\end{equation*}
holds.
\end{itemize}
Such a representation exists by Lemma~\ref{soosseteeiifeimiefisws}. 

We fix another prime $\mfk q$ of $F_+$ inert in $F$ \sut $2r\ell$ divides $\norml{\mfk q}^2-1$. By Lemma~\ref{tteoeoeires}, we can take a conjugate-orthogonal supercuspidal representation $\Pi_{1, \mfk q}^{\flat, \prime}$ of $\GL_{2r}(F_{\mfk q})$ whose associated Galois representation
\begin{equation*}
\iota_\ell\rec_{2r}(\Pi_{1, \mfk q}^{\flat, \prime}): W_{F_{\mfk q}}\to \GL_{2r}(\ovl{\bb Q_\ell})
\end{equation*}
is residually absolutely irreducible.

In this paragraph, let $v$ denote a place in $\{\mfk p, \mfk q\}$. By the local Gan--Gross--Prasad conjecture (see Theorem~\ref{ssoeieuriwos}(1)), there exists a Hermitian space $V'_v$ of dimension $2r$ over $F_v$ and irreducible admissible representations $\pi_{0, v}'$ and $\pi_{1, v}'$ of $\bx U(V'_v)$ and $\bx U(V'_{v,\sharp})$, respectively, satisfying
\begin{enumerate}
\item
$\BC(\pi_{0, v}')=\Pi_{0, v}$ and $\BC(\pi_{1, v}')=\Pi_{1, v}^{\flat, \prime}\boxplus\uno$, where $\uno$ is the trivial representation of $\GL_1(F_v)$; and
\item
$\Hom_{\bx U(V_v')}\paren{\pi'_{1, v}|_{\bx U(V_v')}\otimes\pi_{0, v}', \bb C}\ne 0$.
\end{enumerate}
In particular, $\pi_{1, v}'$ is supercuspidal by \cite[Corollaire~3.5]{M-R18}. By Prasad's conjecture (see Theorem~\ref{sieieifnimidms}(3)), there exist a unique skew-Hermitian space $W'_v$ of dimension $2r$ over $F_v$ \sut the contragredient theta lift
\begin{equation*}
\sigma_{1, v}'\defining(\theta_{V'_{\sharp,v}, W'_v}(\pi_{1, v}'))^\vee
\end{equation*}
is nonzero. Moreover, it follows from Prasad's conjectures (see Theorem~\ref{sieieifnimidms}) that $\BC(\sigma_{1, v}^{\prime, \vee})=\Pi_{1, v}^{\flat, \prime}\otimes\mu_v$. In particular, it follows from \cite[Theorem~8.1]{Fin21} and \cite[Corollaire~3.5]{M-R18} that $\sigma_{1, v}'$ is compactly induced from an irreducible representation of some compact open subgroup of $\bx U(W_v')$. Thus it follows from the local seesaw identity (see Lemma~\ref{ososopeeiiemreis}) and Proposition~\ref{oeiieieikifimeis} that the theta lift
\begin{equation*}
\sigma_{0, v}'\defining \theta_{V_{\sharp, v}', W_v'}(\pi_{0, v}^{\prime, \vee})
\end{equation*}
is also nonzero, and
\begin{equation*}
\Hom_{\bx U(W'_v)}(\sigma_{0, v}'\otimes\omega_{W_v', \psi_v}\otimes\sigma_{1,v}', \bb C)
\end{equation*}
is nonzero. Moreover, it follows from Prasad's conjectures (see Theorem~\ref{sieieifnimidms}(1)) that $\BC(\sigma_{0, v}')=\Pi_{0, v}^\vee$.

We now consider an infinite place $u$. let $W_u'$ be a skew-Hermitian space of dimension $2r$ and signature $(r, r)$ over $F_{+, u}$, and let $V_u'$ be a Hermitian space of dimension $2r$ and signature $(2r, 0)$ over $F_{+, u}$. Let
\begin{equation*}
\sigma_{1, u}'\defining(\theta_{V'_{u,\sharp}, W_u'}(\uno))^\vee
\end{equation*}
be the contragredient of the theta lift of the trivial representation of $\bx U(V'_{u, \sharp})$ to $\bx U(W_u')$, and let
\begin{equation*}
\sigma_{0, u}'\defining \theta_{V_u', W_u'}(\uno)
\end{equation*}
be the theta lift of the trivial representation of $\bx U(V_u')$ to $\bx U(W_u')$. Then it follows from classical calculation (see, for example, \cite[\S2.3]{Har07} and \cite{Li90}) that
\begin{itemize}
\item
$\sigma_{1, u}^{\prime, \vee}$ is a holomorphic discrete series representation with Harish-Chandra parameter
\begin{equation*}
\tau_1^\vee=\paren{\frac{2r+1}{2}, \ldots, \frac{3}{2}, -\frac{1}{2}, \ldots, -\frac{2r-1}{2}}
\end{equation*}
and the lowest $\mdc K_u$-type being the character $(k_1, k_2)\mapsto\det(k_1)^{r+1}\det(k_2)^{-r}$; and
\item
$\sigma_{0, u}'$ is a holomorphic discrete series representation with Harish-Chandra parameter
\begin{equation*}
\tau_0=\paren{\frac{2r-1}{2}, \ldots, -\frac{2r-1}{2}}
\end{equation*}
and the lowest $\mdc K_u$-type being the character $(k_1, k_2)\mapsto\det(k_1)^r\det(k_2)^{-r}$
\end{itemize}
for every infinite place $u$ of $F_+$. In particular, $\sigma_{1, u}^{\prime, \vee}$ is a generalized Verma module for every infinite place $u$ of $F_+$ (cf. \cite{Gar05}). Moreover, by the local seesaw identity (see Lemma~\ref{ososopeeiiemreis}) and Proposition~\ref{oeiieieikifimeis}, the space \begin{equation*}
\Hom_{\bx U(W_u)}(\sigma_{0, u}'\otimes\omega_{W_u'}\otimes\sigma_{1, u}', \bb C)
\end{equation*}
is nonzero for every infinite place $u$ of $F_+$. Moreover, $\BC(\sigma_{0, v}')=\Pi_{0, u}$.

By Arthur's multiplicity formula (see \cite[Theorem~1.7.1]{KMSW}), there exists a skew-Hermitian space $\mbf W_{2r}$ of dimension $2r$ over $F$ with signature $(r,r)$ at every infinite place satisfying $\mbf W_{2r, v}\cong \mbf W_v'$ for every $v\in\{\mfk p, \mfk q\}$, and a cuspidal automorphic representation $\sigma_0$ of $\bx U(\mbf W_{2r})$ satisfying $\sigma_{0, v}\cong \sigma_{0, v}'$ for every $v\in\{\mfk p, \mfk q\}\cup\Pla^\infty_{F_+}$ and $\BC(\sigma_0)$ is isomorphic to $\Pi_0^\vee$.

Because $L(\frac{1}{2}, \Pi_0)$ is nonzero, it follows from the local conservation relation (see Theorem~\ref{oeoeieifnies}), Theorem~\ref{ieieeinfeieiites}(1) and Theorem~\ref{novnainsieiheirimfies} that there exists a Hermitian space $\mbf V_{2r}$ of dimension $2r$ over $F$ with signature $(2r, 0)$ at every infinite place, \sut the conjugate global theta lift
\begin{equation*}
\pi_0\defining\ovl{\theta_{\mbf W_{2r}, \mbf V_{2r}}(\sigma_0)}
\end{equation*}
is an (irreducible) cuspidal automorphic representation of $\mbf V_{2r}(\Ade_F)$ with trivial Archimedean components. Then it follows from Proposition~\ref{oeieeiiriemfies} and the local conservation relation (see Theorem~\ref{oeoeieifnies}) that $\mbf V_{2r, v}\cong V_v'$ and $\pi_{0, v}\cong \pi_{0, v}'$, for every $v\in\{\mfk p, \mfk q\}$. Moreover, it follows from Lemma~\ref{peiemitimeifes} and Proposition~\ref{oeieeiiriemfies} that $\BC(\pi_0)$ is isomorphic to $\Pi_0$, and $\sigma_0=\theta_{\mbf V_{2r}, \mbf W_{2r}}(\ovl\pi_0)$. Set $\mbf V_{2r+1}\defining (\mbf V_{2r})_\sharp$.

It follows from the Burger--Sarnak type principle (see Proposition~\ref{ososieifnienfims}) that there exists a cuspidal automorphic representation $\sigma_1$ of $\bx U(\mbf W_{2r})(\Ade_{F_+})$ \sut $\sigma_{1, v}$ is isomorphic to $\sigma_{1, v}'$ for every $v\in\{\mfk p, \mfk q\}\cup\Pla^\infty_{F_+}$, together with automorphic forms $\vp_0\in\sigma_0, \vp_1\in\sigma_1$ and a Schwartz function $\phi\in\mcl S(\bb L_{2r, 1}(\Ade_{F_+}))$ \sut 
\begin{equation*}
\mcl{FJ}(\vp_0, \vp_1; \phi)\ne 0.
\end{equation*}
Set $\Pi_1^\flat\defining \BC(\ovl\sigma_1)$. Then $\Pi_{1, v}^\flat$ is isomorphic to $\Pi_{1, v}^{\flat, \prime}\otimes\mu_v$ for every $v\in\{\mfk p, \mfk q\}$ and $\Pi_{1, u}^\flat$ is isomorphic to $\BC(\sigma_{1, u}^{\prime, \vee})$ for every $u\in\Pla^\infty_{F_+}$. Set $\Pi_1\defining(\Pi_1^\flat\otimes\mu^{-1})\boxplus\uno$, where $\uno$ is the trivial representation of $\GL_1(\Ade_F)$. Then $\Pi_1$ is an almost cuspidal relevant representation of $\GL_{2r+1}(\Ade_F)$ (see Definition~\ref{eirinenriedusnisimaos}).

It follows from the global seesaw identity Lemma~\ref{osismeimefiemmitis} that 
\begin{equation*}
\pi_1\defining \theta_{\mbf W_{2r}, \mbf V_{2r+1}}(\ovl{\sigma_1})
\end{equation*}
is nonzero. Because $\mbf V_{2r+1}$ is anisotropic, we know $\pi_1$ is an (irreducible) cuspidal automorphic representation of $\bx U(\mbf V_{2r+1})(\Ade_{F_+})$. In particular, it follows from Lemma~\ref{peiemitimeifes} that $\pi_1$ has trivial Archimedean component, and $\pi_{1, v}$ is isomorphic to $\pi_{1, v}'$ for every $v\in\{\mfk p, \mfk q\}$. Moreover, it follows from Proposition~\ref{peiemitimeifes} that
\begin{equation*}
\BC(\pi_1)\cong(\BC(\ovl\sigma_1)\otimes\mu^{-1})\boxplus\uno=\Pi_1,
\end{equation*}
Thus it follows from the global seesaw identity again that there exist automorphic forms $f_0\in\pi_0$ and $f_1\in\pi_1$ \sut
\begin{equation*}
\mcl P_{\bx{GGP}}(f_0, f_1)\ne 0.
\end{equation*}

Let $E'$ be a strong coefficient field of $\Pi_1$ containing $E$. The isomorphism $\iota_\ell: \bb C\xr\sim\ovl{\bb Q_\ell}$ induces a place $\lbd'$ of $E'$ with underlying place $\lbd$ of $E$. We check that $\lbd'$ is an admissible place of $E$ with respect to $(\Pi_0,\Pi_1)$ (see Definition~\ref{oiamsisieiwps}).
\begin{itemize}
\item
(L1), (L2), (L4-1) and (L5) are satisfied by ($\Lbd$1), ($\Lbd$2), ($\Lbd$4-1) and ($\Lbd$5), respectively.
\item
For (L3), $\ovl\rho_{\Pi_0, \lbd'}$ is absolutely irreducible by $(\Lbd3)$. The restriction of $\rho_{\Pi^\flat_1, \lbd'}\otimes_{E'_{\lbd'}}\ovl{\bb Q_\ell}$ to $\Gal_{F_{\mfk q}}$ is residually absolutely irreducible by Proposition~\ref{ieieeinfeieiites} and the definition of $\Pi_{1, \mfk q}^{\flat,\prime}$. Thus $\rho_{\Pi^\flat, \lbd'}$ is residually absolutely irreducible.
\item
For (L4-2), it is easy to check that $(\bx{GI}^1_{F_{\bx{rflx}, +}, \mrs P})$ with $\mrs P(T)=T^2-1$ is satisfied by taking the element $(\ovl\rho_{\Pi_0, \lbd', +}, \ovl\rho_{\Pi_1, \lbd', +}, \ovl\ve_\ell)(\phi_{\mfk p})$.
\item
For (L6), if $\alpha=0$, then this follows from ($\Lbd6$). If $\alpha=1$, then this follows from the definition of $\Pi_{1, \mfk p}^{\flat, \prime}$ and the Chebotarev density theorem applied to the representation $\ovl\rho_{\Pi_1, \lbd'}\oplus\ovl\ve_\ell$ of $\Gal_F$, we see that there are infinitely many finite places $w$ of $F$ that are of degree 1 over $\bb Q$ satisfying that
\begin{enumerate}
\item
$\Pi_{1, w}$ is unramified with Satake parameter $\{\alpha_{1, 1}, \ldots, \alpha_{1, 2r+1}\}$ in which $\iota_\ell(\alpha_{1, i})$ is an $\ell$-adic unit for every $1\le i\le 2r+1$, and
\item
$\iota_\ell(\alpha_{1, i}/\alpha_{1, j})\ne\norml{w}\in \ovl{\kappa_{\lbd'}}$ for $1\le i\ne j\le 2r+1$. 
\end{enumerate}
Then it follows from \cite[Theorem~1.5]{Y-Z25} that (L6) holds for $\lbd'$.
\end{itemize}

As $F_+\ne \bb Q$, Hypothesis~\ref{iifififmieiemss} is known for every positive integer $N\ge 2$ by Proposition~\ref{oososiemiehifws}. We now apply (the proof of) Theorem~\ref{iieiieierhfieiswp} to get
\begin{equation*}
\bx H^1_f\paren{F, \rho_{\Pi_0, \lbd}(n)}\otimes_{E_\lbd}E'_{\lbd'}=\bx H^1_f\paren{F, \rho_{\Pi_0, \lbd'}(r)}=0.
\end{equation*}
Thus $\bx H^1_f\paren{F, \rho_{\Pi, \lbd}(n)}$ vanishes.
\end{proof}

We now deduce Theorem~\ref{ismsiieppwso} and Theorem~\ref{lsslienefifnieiwsws} from Theorem~\ref{ismsieifmeifmss}.

\begin{cor}
Let $A$ be a modular elliptic curve over $F_+$. Suppose that $F^+\ne\bb Q$ and $A$ has no complex multiplication over $\ovl F$. If the central critical value $L\paren{\Sym^{2r-1}A_F; r}$ does not vanish, then the Bloch--Kato Selmer group
\begin{equation*}
\bx H^1_f\paren{F, \Sym^{2r-1}\etH^1(A_{\ovl F}, \bb Q_\ell)(r)}
\end{equation*}
vanishes for all but finitely many rational primes $\ell$.
\end{cor}
\begin{proof}
By \cite[Theorem~A]{N-T22} and \cite{A-C89}, $\Sym^{2r-1}A_F$ is modular. Let $\Pi_0$ denote the automorphic representation of $\GL_{2r}(\Ade_F)$ attached to $\Sym^{2r-1}A_F$, which is a cuspidal relevant representation. Thus $\Pi_0$ has strong coefficient field $\bb Q$, and $\rho_{\Pi_0, \ell}$ is conjugate to $\Sym^{2r-1}\etH^1(A_{\ovl F}, \bb Q_\ell)$ as $\bb Q_\ell[\Gal_F]$-modules for every rational prime $\ell$. Moreover,
\begin{equation*}
L(\frac{1}{2}, \Pi_0)=L\paren{r, \Sym^{2r-1}A_F}.
\end{equation*}
As $F_+\ne \bb Q$, Hypothesis~\ref{iifififmieiemss} is known for every positive integer $N\ge 2$ by Proposition~\ref{oososiemiehifws}. Thus the assertion is an immediate consequence of Theorem~\ref{ismsieifmeifmss} and Lemma~\ref{sosieimfifmiesws}.
\end{proof}

\begin{cor}
Let $\Pi$ be a relevant automorphic representation of $\GL_{2r}(\Ade_F)$. Suppose that
\begin{enumerate}
\item
$F_+\ne \bb Q$;
\item
There exists a finite place $w$ of $F$ \sut $\Pi_w$ is supercuspidal;
\item
There exists a good inert place $\mfk p$ of $F$ (see \textup{Definition~\ref{issieniefeifmies}}) \sut $\Pi_{\mfk p}$ is a Steinberg representation.
\end{enumerate}
Let $E$ be a strong coefficient field of $\Pi$ (see \textup{Definition~\ref{sisieifnieeimfsi}}). If the central critical value $L(\frac{1}{2}, \Pi)$ does not vanish, then for almost every finite place $\lbd$ of $E$, the Bloch--Kato Selmer group $\bx H^1_f(F, \rho_{\Pi, \lbd}(n))$ vanishes.
\end{cor}
\begin{proof}
As $F_+\ne \bb Q$, Hypothesis~\ref{iifififmieiemss} is known for every positive integer $N\ge 2$ by Proposition~\ref{oososiemiehifws}. Thus the assertion is an immediate consequence of Theorem~\ref{ismsieifmeifmss} and Lemma~\ref{sosieimfifmiesws}.
\end{proof}

\subsection{The self-dual case}\label{psosieuueiures}

Let $F$ be a totally real number field. Let $\mbf V_{2r+1}$ be a quadratic space of dimension $2r$ over $F$ and let $\mbf V_1$ a quadratic space of dimension $1$ over $F$ of discriminant $1$. Let $\mbf W_{2r}$ be a symplectic space of dimension $2r$ over $F$. Set $\mbf V_{2r+2}\defining \mbf V_{2r+1}\oplus \mbf V_1$. Let $\iota: \bx O(\mbf V_{2r+1})\subset \bx O(\mbf V_{2r+2})$ be the natural inclusion. We fix a nontrivial additive character $\psi$ of $F\bsh \Ade_F$, and use notation defined in \S\ref{ieiieiepeoies}.

Consider the inclusion
\begin{equation*}
\bx O(\mbf V_{2r+1})\times \bx O(\mbf V_1)\subset \bx O(\mbf V_{2r+2})
\end{equation*}
and the diagonal embedding
\begin{equation*}
\Sp(\mbf W_{2r})\subset \Sp(\mbf W_{2r})\times \Sp(\mbf W_{2r}).
\end{equation*}
$(\Sp(\mbf W_{2r}), \bx O(\mbf V_{2r+2}))$ and
\begin{equation*}
(\Sp(\mbf W_{2r})\times \Sp(\mbf W_{2r}), \bx O(\mbf V_{2r+1})\times \bx O(\mbf V_1))
\end{equation*}
are reductive dual pairs. In other words there is a seesaw diagram:

\begin{equation*}
\begin{tikzcd}[sep=large]
\Sp(\mbf W_{2r})\times\Sp(\mbf W_{2r})\ar[dr, dash] &\bx O(\mbf V_{2r+2})\\
\Sp(\mbf W_{2r})\ar[u, dash]\ar[ur, dash] &\bx O(\mbf V_{2r+1})\times\bx O(\mbf V_1)\ar[u, dash]
\end{tikzcd}.
\end{equation*}

We record the following local seesaw identity attached to the seesaw diagram \eqref{sosiemfiemfiw}.

\begin{lm}\label{ososopeeiiettttmreis}
For irreducible admissible representations $\pi_0$ of $\bx O(\mbf V_{2r+1})(F_{+, v})$ and $\sigma_1$ of $\Sp(\mbf W_{2r})(F_{+, v})$, there is a canonical isomorphism
\begin{align*}
\Hom_{\Sp(\mbf W_{2r})(F_{+, v})}(\Theta_{\mbf V_{2r+1}, \mbf W_{2r}}(\pi_0)&\otimes\omega_{\mbf W_{2r}}, \pi)\\
&\cong\Hom_{\bx O(\mbf V_{2r+1})(F_{+, v})}\paren{\Theta_{\mbf W_{2r}, \mbf V_{2r+2}}(\sigma_1), \pi_0}.
\end{align*}
\end{lm}
\begin{proof}
This is standard.
\end{proof}

We introduce the orthogonal Gross--Prasad periods and the Fourier--Jacobi periods.

\begin{defi}
Let $\pi_0\subset \mcl A_0(\bx O(\mbf V_{2r+1})(\Ade_F))$ and $\pi_1\subset \mcl A_0(\bx O(\mbf V_{2r+2})(\Ade_F))$ be cuspidal automorphic representations, and $f_0\in\pi_0$ and $f_1\in\pi_1$ be cusp forms. We define the orthogonal \tbf{Gross--Prasad period}
\begin{equation*}
\mcl P_{\bx{GP}}(f_0, f_1)\defining\int_{\bx O(\mbf V_{2r+1})(F)\bsh \bx O(\mbf V_{2r+1})(\Ade_F)}f_0(h)f_1(\iota(h))\bx dh.
\end{equation*}
Here the measure $\bx dh$ is the Tamagawa measure on $\bx O(\mbf V_{2r+1})(\Ade_F)$. This integral is absolutely convergent since $f_0$ and $f_1$ are rapidly decreasing.
\end{defi}

We set
\begin{equation*}
\bb W_{2r, 1}\defining\mbf W_{2r}\otimes_F\mbf V_1, \quad \bb W_{2r, 2r+1}\defining\mbf W_{2r}\otimes_F\mbf V_{2r+1}, \quad \bb W_{2r, 2r+2}\defining\mbf W_{2r+1}\otimes_F\mbf V_{2r+2}.
\end{equation*}
Then they are all symplectic spaces over $F_+$ as defined in \S\ref{issoeoifmmieimws}. Fix Lagrangian subspaces
\begin{equation*}
\bb L_{2r, 1}\subset \bb W_{2r, 1}, \quad \bb L_{2r, 2r+1}\subset \bb W_{2r, 2r+1},
\end{equation*}
then $\bb L_{2r, 2r+1}\defining \bb L_{2r, 2r+1}\oplus\bb L_{2r, 1}$ is a Lagrangian subspace of $\bb W_{2r, 2r+2}$. For each $n\in\{1, 2r+1, 2r+2\}$, we denote by $\omega_{\mbf W_{2r}, \mbf V_n}$ the Weil representation, which can be realized on the space of Schwartz functions $\mcl S(\bb L_{2r, n})$. Then
\begin{equation*}
\omega_{\mbf W_{2r}, \mbf V_{2r+2}}=\omega_{\mbf W_{2r}, \mbf V_{2r+1}}\hat\otimes\omega_{\mbf W_{2r}}.
\end{equation*}
In particular, if $\phi_{2r, 2r+2}=\phi_{2r,2r+1}\otimes\phi_{2r, 1}\in \mcl S(\bb L_{2r, 2r+1}(\Ade_F))\otimes\mcl S(\bb L_{2r, 1}(\Ade_F))$, then
\begin{equation*}
\theta_{\mbf W_{2r}, \mbf V_{2r+2}}(g, \iota(h); \phi_{2r, 2r+2})=\theta_{\mbf W_{2r}, \mbf V_{2r+1}}(g, h; \phi_{2r, 2r+1})\theta_{\mbf W_{2r}, \mbf V_1}(g; \phi_{2r, 1})
\end{equation*}
for every $(g, h)\in \Sp(\mbf W_{2r})(\Ade_F)\times\bx O(\mbf V_{2r+1})(\Ade_F)$.

\begin{defi}
Let $\tilde\sigma_0\subset \mcl A_0(\tld\Sp(\mbf W_{2r})(\Ade_F))$ and $\sigma_1\subset\mcl A_0(\Sp(\mbf W_{2r})(\Ade_F))$ be genuine cuspidal automorphic representations. Let $\tilde\vp_0\in\tilde\sigma_0, \vp_1\in\sigma_1$ be cusp forms and $\phi\in\mcl S(\bb L_{2r, 1}(\Ade_F))$ be a Schwartz function. We define the \tbf{Fourier--Jacobi period}
\begin{equation*}
\mcl{FJ}(\tilde\vp_0, \vp_1; \phi)\defining\int_{\Sp(\mbf W_{2r})(F)\bsh\Sp(\mbf W_{2r})(\Ade_F)}\tilde\vp_0(\tilde g)\vp_1(g)\theta_{\mbf W_{2r}, \mbf V_1}(\tilde g; \phi)\bx dg.
\end{equation*}
Here $\tld g$ is an arbitrary lift of $g$ to $\tld\Sp(\mbf W_{2r})$, and the measure $\bx dg$ is the Tamagawa measure on $\Sp(\mbf W_{2r})(\Ade_F)$. This integral is absolutely convergent since $\tilde\vp_0$ and $\vp_1$ are rapidly decreasing and theta functions are of moderate growth.
\end{defi}

We will use the following global seesaw identity.

\begin{lm}
Let $\sigma_1\subset \mcl A_0(\Sp(\mbf W_{2r})(\Ade_{F_+}))$ and $\pi_0\subset\mcl A_0(\bx O(\mbf V_{2r+1})(\Ade_{F_+}))$ be cuspidal automorphic representations, \sut
\begin{equation*}
\tilde\sigma_0=\theta_{\mbf V_{2r+1}, \mbf W_{2r}}(\ovl\pi_0)
\end{equation*}
is a genuine cuspidal automorphic representation of $\tld\Sp(\mbf W_{2r})(\Ade_F)$. Let $\vp_1\in \sigma_1$ and $f_0\in\pi_0$ be cusp forms and $\phi_{2r, 1}\in\mcl S(\bb L_{2r, 1}(\Ade_{F_+})), \phi_{2r, 2r+1}\in\mcl S(\bb L_{2r, 2r+1}(\Ade_{F_+}))$ be Schwartz functions. Then
\begin{align*}
\mcl{FJ}&\paren{\theta_{\mbf V_{2r+1}, \mbf W_{2r}}(\ovl f_0, \phi_{2r, 2r+1}), \vp_1; \phi_{2r, 1}}\\
&=\mcl P_{\bx{GP}}\paren{f_0, \theta_{\mbf W_{2r},\mbf V_{2r+2}}(\ovl\vp_1; \phi_{2r, 2r+1}\otimes\phi_{2r, 1})}.
\end{align*}
\end{lm}
\begin{proof}
The proof is the same as that of Lemma~\ref{osismeimefiemmitis}, thus omitted.
\end{proof}

We now explain how to deduce Conjecture~\ref{osoppowiur} from Conjecture~\ref{ocneienipwowie}. The key ingredient is the following Burger--Sarnak type principle for Fourier--Jacobi periods on the pair $(\Sp(\mbf W_{2r}), \tld\Sp(\mbf W_{2r}))$, in the spirit of \cites{B-S91, H-L98, Pra07, Zha14}. We first fix notation. For every infinite place $u$ of $F$, $\Sp(\mbf W_{2r})(F_u)$ has a maximal compact subgroup $\mdc K_u\cong\bx U(r)$. 
Denote the preimage of $\mdc K_u$ in $\tld\Sp(\mbf W_{2r})(F_u)$ by $\tld{\mdc K}_u$. Under the identification $\mdc K_u\cong\bx U(r)$, one can realize
\begin{equation*}
\tld{\mdc K_u}=\{(g,z)|g\in \bx U(r), z\in\bb C^\times,\ \det(g)=z^2\}.
\end{equation*}
There is a genuine ``square-root of the determinant'' character
\begin{equation*}
\sqrt{\det}: \tld{\mdc K}_u\mapsto \bb C^\times,  (k, z)\mapsto z,
\end{equation*}
which satisfies $(\sqrt{\det})^2=\det$ (via the projection $\tld{\mdc K}_u\to \mdc K_u$).

\begin{prop}
Suppose that
\begin{enumerate}
\item
$\Pla$ is a finite set of places of $F$ containing at least one finite place;
\item
$\tld\sigma_0$ is a genuine automorphic representation of $\tld\Sp(\mbf W_{2r})(\Ade_F)$; and
\item
$\otimes_{v\in\Pla}\tau_v$ is an irreducible admissible representation of $\prod_{v\in\Pla}\Sp_{2r}(F_v)$ \sut
\begin{enumerate}
\item
for every $v\in\Pla$, the space $\Hom_{\Sp(\mbf W_{2r})(F_v)}(\tld\sigma_{0, v}\otimes\omega_{\mbf W_{2r, v}, \psi_v}\otimes\tau_v, \bb C)$ is nonzero;
\item
for every finite place $v\in\Pla$, $\tau_v$ is a supercuspidal representation that is compactly induced from a representation of a compact open subgroup $\mdc K_v$ of $\Sp_{2r}(F_v)$; and
\item
for every infinite place $u\in\Pla$, $\tau_u^\vee$ is a holomorphic discrete series that is a generalized Verma module in the sense of \cite{Gar05}. Moreover, if the lowest $\mdc K_u$-type of $\tau_u^\vee$ is the character $\det^m$ for some positive integer $m$, then $\tilde\sigma_{0, u}$ has lowest $\tld{\mdc K}_u$-type $(\sqrt{\det})^{2m-1}$ with multiplicity one.
\end{enumerate}
\end{enumerate}
Then there exists a cuspidal automorphic representation $\sigma_1$ of $\Sp(\mbf W_{2r})(\Ade_{F_+})$ satisfying
\begin{enumerate}
\item
for every $v\in\Pla$, $\sigma_{1,v}$ is isomorphic to $\tau_v$; and
\item
there exist genuine automorphic forms $\tilde\vp_0\in\tilde\sigma_0, \vp_1\in \sigma_1$ and a Schwartz function $\phi\in\mcl S(\bb L_{2r, 1}(\Ade_F))$ such that
\begin{equation*}
\mcl{FL}(\tilde\vp_0, \vp_1; \phi)\ne 0.
\end{equation*}
\end{enumerate}
\end{prop}
\begin{proof}
The proof is the same as that of Proposition~\ref{ososieifnienfims}, thus omitted. 
\end{proof}

We introduce the notion of preadmissible places for the coefficient fields appearing in Conjecture~\ref{ocneienipwowie}. This is a preliminary notation that can be refined.

\begin{defi}\label{aosppwoienuvneu}
Let $A$ be an elliptic curve over $F$ with $\End(A_{\ovl F})=\bb Z$, and $\Pi$ a relevant automorphic representation of $\GL_{2r+1}(\Ade_F)$. Let $E\subset\bb C$ be a strong coefficient field of $\Pi$ (Definition~\ref{psosieifmmiess}). We say that a finite place $\lbd\in\fPla_E$ with underlying prime $\ell$, is preadmissible (with respect to $(A, \Pi)$) if
\begin{enumerate}
\item[(pL1)]
The semi-simplified residual representation $\ovl\rho_{\Pi, \lbd}$ is either absolutely irreducible or a sum of a self-dual absolutely irreducible representation with a self-dual character.
\item[(pL2)]
There exists a finite place $\mfk p$ of $F$ and a finite extension $E'$ of $E$ in $\bb C$ with a finite place $\lbd'$ over $\lbd$ satisfying
\begin{enumerate}
\item
$\{\pm1, \pm\norml{\mfk p}^{\pm 1}\modu{\ell},\pm \norml{\mfk p}^{\pm 2}\modu{\ell}, \ldots,\pm \norml{\mfk p}^{\pm 4r}\modu{\ell}\}$ consists of distinct elements; 
\item
$E$ has good reduction at $\mfk p$ with $a_{\mfk p}(E)-(\norml{\mfk p}+1)$ is divisible by $\ell$;
\item
$\Pi_{\mfk p}$ is unramified with Satake parameter $\{-1, \alpha_1^{\pm1}, \ldots, \alpha_r^{\pm1}\}\subset \mcl O_{E'}$ \sut $\{\alpha_i|1\le i\le r\}$ is disjoint from $\{\pm1, \pm\norml{\mfk p}^{\pm1}, \ldots, \pm\norml{\mfk p}^{\pm4r}\}$ in $\kappa_{\lbd'}$.
\end{enumerate}
\item[(pL3)]
There exists a finite place $v$ of $F$ \sut $\Pi_v$ is unramified with Satake parameters $\{1, \alpha_{v, 1}^{\pm1}, \ldots, \alpha_{v, r}^{\pm1}\}$ satisfying
\begin{equation*}
\norml{v}\notin\{\alpha_{v, i}\alpha_{v, j}^{-1}|1\le i\ne j\le r\}\cup\{\alpha_{v, i}^{\pm1}|1\le i\le r\}\cup\{1\}\subset\ovl{\bb F_\ell}.
\end{equation*}
\end{enumerate}
\end{defi}

We give an example where it is known that all but finitely many finite places $\lbd$ of $E$ are admissible.

\begin{lm}\label{oeeoitiirunmes}
Let $A$ and $\Pi$ be as in Definition~\ref{aosppwoienuvneu}. If we assume that there exist finite places $\mfk p, \mfk q$ of $F$ \sut
\begin{enumerate}
\item
$A$ has split multiplicative reduction at $\mfk p$, and
\item
$\Pi_{\mfk p}$ is unramified with Satake parameter of the form $\{-1, \alpha_1^{\pm1}, \ldots, \alpha_r^{\pm1}\}$ satisfying $\alpha_i\ne\pm1$ for every $1\le i\le r$; and
\item
$\Pi_{\mfk q}$ is either supercuspidal or an isobaric sum of a self-dual supercuspidal representation with a self-dual character.
\end{enumerate}
then there exists an effective constant $N(F, A,\Pi_{\mfk p}, \Pi_{\mfk q})$ depending on $F, A, \Pi_{\mfk p}$, and $\Pi_q$ \sut every finite place $\lbd$ of $E$ with underlying prime $\ell$ greater than $N(F, A, \Pi_{\mfk p}, \Pi_{\mfk q})$ is admissible \wrt $(A, \Pi)$.
\end{lm}
\begin{proof}
We show that every condition in Definition~\ref{aosppwoienuvneu} excludes only finitely many finite places of $E$. 

For (pL1), the condition $\ovl\rho_{\Pi_1^\flat, \lbd}$ is absolutely irreducible only excludes finitely many finite places $\lbd$ of $E$ by~\cite[Theorem~4.5.(1)]{LTXZZa} and (3).

For (pL2), it follows from temperedness (see Proposition~\ref{ddddddeieopws}) that $|\alpha_i|=1$ for every $1\le i\le r$. Moreover, $\alpha_i$ is an algebraic number for every $1\le i\le r$ by Remark~\ref{oeietureps}. Thus it is clear that when $\ell$ is large, (pL2)(a, c) is satisfied. By the Chebotarev density theorem, (L2) is satisfied for all such $\ell$.

For (pL3), it follows from Proposition~\ref{ddddddeieopws}(1) that $\norml{\alpha_i}=1$ for every $1\le i\le 2r+1$. Thus, for every sufficiently large rational prime $\ell$,
\begin{equation*}
\norml{\mfk p}\notin\{\pm\alpha_i\alpha_j^{-1}|1\le i\ne j\le r\}\cup\{\pm\alpha_i^{\pm1}|1\le i\le r\}\cup\{\pm1\}\subset\ovl{\bb F_\ell}.
\end{equation*}
Suppose $\lbd$ is a finite place of $E$ over $\ell$, we fix an isomorphism $\iota_\ell: \bb C\xr\sim\ovl{\bb Q_\ell}$ which induces $\lbd$. Applying the Chebotarev density theorem to the representation $\ovl\rho_{\Pi, \lbd}\oplus\ovl\ve_\ell$ of $\Gal_F$, we see that (L3) holds for $\lbd$.
\end{proof}

\begin{thm}
Let $r$ be a positive integer, and let $A$ be a modular elliptic curve over $F$ with no complex multiplication over $\ovl F$. Assume \textup{Conjecture~\ref{ocneienipwowie}} holds for $r$ and $A$.  If the central critical value $L\paren{\Sym^{2r-1}A; r}$ does not vanish, then there is an effective constant $N(F, A, r)$ depending only on $F, A$, and $r$ \sut the Bloch--Kato Selmer group
\begin{equation*}
\bx H^1_f\paren{F, \Sym^{2r-1}\etH^1(A; \bb Q_\ell)(r)}.
\end{equation*}
vanishes for all rational primes $\ell$ greater than $N(F, A, r)$.
\end{thm}
\begin{proof}
By \cite[Theorem~A]{N-T22}, $\Sym^{2r-1}A$ is modular. Let $\Pi_0$ denote the automorphic representation of $\GL_{2r}(\Ade_F)$ attached to $\Sym^{2r-1}A$, which is a cuspidal automorphic representation. Thus $\Pi_0$ has strong coefficient field $\bb Q$, and $\rho_{\Pi_0, \ell}$ is conjugate to $\Sym^{2r-1}\etH^1(A_{\ovl F}, \bb Q_\ell)$ as $\bb Q_\ell[\Gal_F]$-modules for every rational prime $\ell$. Moreover,
\begin{equation*}
L(\frac{1}{2}, \Pi_0)=L\paren{r, \Sym^{2r-1}A}.
\end{equation*}

By \cite[Th\'eor\`eme~6]{Ser72} and \cite{Lom15}, there is an effective constant $N_1(F, A)$ depending only on $A$ \sut the homomorphism
\begin{equation*}
\ovl\rho_{A, \ell}: \Gal_F\to\GL\paren{\etH^1(A_{\ovl F},\bb F_\ell)}
\end{equation*}
is surjective for every rational prime $\ell$ greater than $N_1(F, A)$.

Suppose there is an effective constant $N_2(F, A, r)$ \sut Conjecture~\ref{ocneienipwowie} holds for any preadmissible finite place $\lbd$ of the strong coefficient field with underlying rational prime $\ell>N_2(F, A, r)$. We set
\begin{equation*}
N(F, A, r)\defining\max(N_1(F, A), N_2(F, A, r), 2^{16r}).
\end{equation*}
Let $\ell>N(F, A, r)$ be a rational prime with a fixed isomorphism $\iota_\ell: \bb C\xr\sim\ovl{\bb Q_\ell}$. Then we know the set
\begin{equation*}
B_2\defining \Brace{\pm1, \pm2^{\pm1}, \pm 2^{\pm2}, \ldots, \pm2^{\pm 4r}}
\end{equation*}
consists of distinct elements in $\bb F_\ell$. By the Chebotarev density theorem, we can find a finite place $\mfk p$ of $F$ satisfying
\begin{itemize}
\item
The rational prime $p$ underlying $\mfk p$ is larger than $\max(\ell, 2r)$;
\item
$A$ has good reduction at $\mfk p$; and
\item
$\ovl\rho_{A, \ell}(\phi_{\mfk p})$ has eigenvalues $\{2, 1\}$.
\end{itemize}
We fix a totally positive element $\mfk d\in F^\times$ satisfying $(-1)^{r+1}\mfk d\ne 1\in F^\times/(F^\times)^2$.

We take a supercuspidal $B_2$-avoiding good representation $\Pi_{1, \mfk p}^{\flat\flat, \prime}$ of $\GL_{2r}(F_{\mfk p})$ with respect to $\iota_\ell$ (see Definition~\ref{ossewwttoeiiinfiemisw}) satisfying
\begin{itemize}
\item
$\iota_\ell\rec_{2r}(\Pi_{1, \mfk p}^{\flat\flat, \prime})$ is residually absolutely irreducible; and
\item
there exists a lift $F\in W_{F_{\mfk p}}$ of the arithmetic Frobenius element \sut the eigenvalues $\{\alpha_1, \ldots, \alpha_{2r}\}$ of $\iota_\ell\rec_{2r}(\Pi_{1, \mfk p}^{\flat\flat, \prime})(F)$ are $\ell$-adic units; and
\item
\begin{equation*}
\norml{\mfk p}\notin\{\pm\alpha_i\alpha_j^{-1}|1\le i\ne j\le 2r\}\cup\{\pm\alpha_i|1\le i\le 2r\}\subset\ovl{\bb F_\ell}.
\end{equation*}
holds.
\end{itemize}
Such a representation exists by Lemma~\ref{sooeiifeimiefisws}. Set $\Pi_{1, \mfk p}^{\flat, \prime}\defining \Pi_{1, \mfk p}^{\flat\flat, \prime}\boxplus\chi_{(-1)^{r+1}\mfk d}$.

By the local Gan--Gross--Prasad conjecture (see Theorem~\ref{ssoeieuriwos}(2)), there exists a quadratic space $V'_{\mfk p}$ of dimension $2r+1$ over $F_{\mfk p}$ with $\disc(V'_{\mfk p})=(-1)^r\mfk d\in F_{\mfk p}^\times/(F_{\mfk p}^\times)^2$; and irreducible admissible representations $\pi_{0, \mfk p}'$ and $\pi_{1, \mfk p}'$ of $\bx O(V'_{\mfk p})$ and $\bx O(V'_{\mfk p,\sharp})$, respectively, satisfying
\begin{enumerate}
\item
$\FL(\pi_{0, \mfk p}')=\Pi_{0, \mfk p}$ and $\FL(\pi_{1, \mfk p}')=\Pi_{1, v}^{\flat, \prime}\boxplus\uno$, where $\uno$ is the trivial representation of $\GL_1(F_{\mfk p})$; and
\item
$\Hom_{\bx O(V_{\mfk p}')}\paren{\pi'_{1,\mfk p}|_{\bx O(V_{\mfk p}')}\otimes\pi_{0, \mfk p}', \bb C}\ne 0$.
\end{enumerate}
In particular, $\pi_{1, \mfk p}'$ is supercuspidal by \cite[Corollaire~3.5]{M-R18}. According to Prasad's conjecture (see Theorem~\ref{sieieifnimidms}(7)), upon changing $(\pi_{0, \mfk p}', \pi_{1, \mfk p}')$ to $(\pi_{0, \mfk p}'\otimes\det, \pi_{1, \mfk p}'\otimes\det)$ if necessary, we can assume that the contragredient theta lift
\begin{equation*}
\sigma_{1, \mfk p}'\defining(\theta_{V'_{\sharp, \mfk p}, W_{\mfk p}}(\pi_{1, \mfk p}'))^\vee
\end{equation*}
is nonzero, where $W_{\mfk p}$ is a symplectic space of dimension $2r$ over $F_{\mfk p}$. Moreover, it follows from Prasad's conjectures (see Theorem~\ref{sieieifnimidms}) that $\FL(\sigma_{1, \mfk p}^{\prime, \vee})=\Pi_{1, \mfk p}^{\flat, \prime}\otimes\chi_{(-1)^{r+1}\mfk d}$. In particular, it follows from \cite[Theorem~8.1]{Fin21} and \cite[Corollaire~3.5]{M-R18} that $\sigma_{1, \mfk p}'$ is supercuspidal and compactly induced from an irreducible representation of some compact open subgroup of $\Sp(W_{\mfk p})$. By the local seesaw identity (see Lemma~\ref{ososopeeiiettttmreis}) and Proposition~\ref{oeiieieikifimeis}, the theta lift
\begin{equation*}
\tilde\sigma_{0, \mfk p}'\defining \theta_{V_{\sharp, \mfk p}', W_{\mfk p}'}(\pi_{0, \mfk p}^{\prime, \vee})
\end{equation*}
is also nonzero, and
\begin{equation*}
\Hom_{\Sp(W_{\mfk p})}(\tilde\sigma_{0, \mfk p}'\otimes\omega_{W_{\mfk p}', \psi_{\mfk p}}\otimes\sigma_{1,\mfk p}', \bb C)
\end{equation*}
is nonzero. Moreover, it follows from Prasad's conjecture (see Theorem~\ref{sieieifnimidms}(5)) that
\begin{equation*}
\FL(\tilde\sigma_{0, \mfk p}')=\Pi_{0, \mfk p}^\vee\otimes\chi_{(-1)^r\mfk d}.
\end{equation*}

We now consider an infinite place $u$. Let $W_u$ be a symplectic space of dimension $2r$ over $F_u$, and let $V_u'$ be a quadratic space of dimension $2r+1$ over $F_u$ with signature $(2r+1, 0)$. Let
\begin{equation*}
\sigma_{1, u}'\defining(\theta_{V'_{u,\sharp}, W_u}(\uno))^\vee
\end{equation*}
be the contragredient of the theta lift of the trivial representation of $\bx O(V'_{u, \sharp})$ to $\Sp(W_u')$, and let
\begin{equation*}
\tilde\sigma_{0, u}'\defining \theta_{V_u', W_u}(\uno)
\end{equation*}
be the theta lift of the trivial representation of $\bx O(V_u')$ to $\tld\Sp(W_u)$. Then it follows from classical calculation (see, for example, \cite[Proposition~2.1]{K-R90} and \cite[Theorem~3.3]{A-B98}) 
that
\begin{itemize}
\item
$\sigma_{1, u}^{\prime, \vee}$ is a holomorphic discrete series representation with Harish-Chandra parameter
\begin{equation*}
\tau_1^\vee=\paren{r, r-1, \ldots, 1}
\end{equation*}
and the lowest $\mdc K_u$-type being the character $\det^{r+1}$; and
\item
$\tilde\sigma_{0, u}'$ is a holomorphic discrete series representation with Harish-Chandra parameter
\begin{equation*}
\tilde\tau_0=\paren{\frac{2r-1}{2}, \frac{2r-3}{2}, \ldots, \frac{1}{2}}
\end{equation*}
and the lowest $\tld{\mdc K}_u$-type being the character $\sqrt{\det}^{2r+1}$
\end{itemize}
for every infinite place $u$ of $F$. In particular, $\sigma_{1, u}^{\prime, \vee}$ is a generalized Verma module for every infinite place $u$ of $F$ (cf.~\cite{Gar05}). Moreover, by the local seesaw identity (see Lemma~\ref{ososopeeiiettttmreis}) and Proposition~\ref{oeiieieikifimeis}, the space \begin{equation*}
\Hom_{\Sp(W_u)}(\tilde\sigma_{0, u}'\otimes\omega_{W_u}\otimes\sigma_{1, u}', \bb C)
\end{equation*}
is nonzero for every infinite place $u$ of $F$. Moreover, $\FL(\sigma_{0, v}')=\Pi_{0, u}\otimes\mfk\chi_{(-1)^r}$.

By Arthur's multiplicity formula \cite[Theorem~1.4]{G-I18}, there exists a genuine cuspidal automorphic representation $\tilde\sigma_0$ of $\tld\Sp(\mbf W_{2r})$ satisfying $\tilde\sigma_{0, u}\cong \tilde\sigma_{0, u}'$ for every $u\in\Pla^\infty_{F_+}\cup\{\mfk p\}$ and $\FL(\tilde\sigma_0)\cong \Pi_0^\vee\otimes\chi_{(-1)^r\mfk d}$.

Because $L(\frac{1}{2}, \Pi_0)$ is nonzero, it follows from the local conservation relation (see Theorem~\ref{oeoeieifnies}), Theorem~\ref{ddddddeieopws}(1) and Theorem~\ref{novnainsieiheirimfies} that there exists a unique quadratic space $\mbf V_{2r+1}$ of dimension $2r+1$ over $F$ satisfying
\begin{itemize}
\item
$\mbf V_{2r+1}$ has signature $(2r+1, 0)$ at every infinite place of $F$;
\item
$\disc(\mbf V_{2r+1})=(-1)^r\mfk d\in F^\times/(F^\times)^2$; and
\item
the conjugate global theta lift
\begin{equation*}
\pi_0\defining\ovl{\theta_{\mbf W_{2r}, \mbf V_{2r+1}}(\tilde\sigma_0)}
\end{equation*}
is an (irreducible) cuspidal automorphic representation of $\mbf V_{2r+1}(\Ade_F)$ with trivial Archimedean components.
\end{itemize}
Then it follows from Proposition~\ref{oeieeiiriemfies} and the local conservation relation (see Theorem~\ref{oeoeieifnies}) that $\mbf V_{2r+1, \mfk p}$ is isomorphic to $V_{\mfk p}'$ and $\pi_{0, \mfk p}$ is isomorphic to $\pi_{0, \mfk p}'$. Moreover, it follows from Lemma~\ref{peiemitimeifes} and Proposition~\ref{oeieeiiriemfies} that $\FL(\pi_0)$ is isomorphic to $\Pi_0$, and
\begin{equation*}
\tilde\sigma_0=\theta_{\mbf V_{2r+1}, \mbf W_{2r}}(\ovl\pi_0).
\end{equation*}
Set $\mbf V_{2r+2}\defining (\mbf V_{2r+1})_\sharp$.

It follows from the Burger--Sarnak type principle (see Proposition~\ref{ososieifnienfims}) that there exists a cuspidal automorphic representation $\sigma_1$ of $\Sp(\mbf W_{2r})(\Ade_{F_+})$ \sut $\sigma_{1, v}$ is isomorphic to $\sigma_{1, v}'$ for every $v\in\{\mfk p, \mfk q\}\cup\Pla^\infty_{F_+}$; together with automorphic forms $\tilde\vp_0\in\tilde\sigma_0, \vp_1\in\sigma_1$ and a Schwartz function $\phi\in\mcl S(\bb L_{2r, 1}(\Ade_{F_+}))$ \sut 
\begin{equation*}
\mcl{FJ}(\tilde\vp_0, \vp_1; \phi)\ne 0.
\end{equation*}
Set $\Pi_1^\flat\defining \FL(\ovl\sigma_1)$, which satisfies $\Pi_{1, \mfk p}^\flat\cong\Pi_{1, \mfk p}^{\flat, \prime}\otimes\chi_{(-1)^{r+1}\mfk d}$ and $\Pi_{1, u}^\flat\cong\FL(\sigma_{1, u}^{\prime, \vee})\otimes\chi_{(-1)^{r+1}}$ for every $u\in\Pla^\infty_F$. Then $\Pi\defining \Pi_1^\flat\otimes\chi_{(-1)^{r+1}\mfk d}$ is a relevant automorphic representation of $\GL_{2r+1}(\Ade_F)$ (see Definition~\ref{oreieigieifiewp}). Set $\Pi_1\defining\Pi\boxplus\uno$, where $\uno$ is the trivial representation of $\GL_1(\Ade_F)$.

It follows from the global seesaw identity (see Lemma~\ref{osismeimefiemmitis}) that 
\begin{equation*}
\pi_1\defining \theta_{\mbf W_{2r}, \mbf V_{2r+2}}(\ovl{\sigma_1})
\end{equation*}
is nonzero. Because $\mbf V_{2r+2}$ is anisotropic, we know $\pi_1$ is an (irreducible) cuspidal automorphic representation of $\bx O(\mbf V_{2r+2})(\Ade_F)$. In particular, it follows from Lemma~\ref{peiemitimeifes} that $\pi_1$ has trivial Archimedean component, and $\pi_{1, \mfk p}$ is isomorphic to $\pi_{1, \mfk p}'$. Moreover, it follows from Proposition~\ref{peiemitimeifes} that
\begin{equation*}
\FL(\pi_1)\cong (\FL(\ovl\sigma_1)\otimes\chi_{(-1)^{r+1}\mfk d})\boxplus\uno=\Pi_1.
\end{equation*}
Thus it follows from the global seesaw identity again that there exist automorphic forms $f_0\in\pi_0$ and $f_1\in\pi_1$ \sut
\begin{equation*}
\mcl P_{\bx{GP}}(f_0, f_1)\ne 0.
\end{equation*}

Let $E\subset \bb C$ be a strong coefficient field of $\Pi$. The isomorphism $\iota_\ell: \bb C\xr\sim\ovl{\bb Q_\ell}$ induces a place $\lbd$ of $E$. We check that $\lbd$ is a preadmissible place of $E$ with respect to $(E,\Pi_1)$ (see Definition~\ref{aosppwoienuvneu}).

\begin{itemize}
\item
For (pL1), note that the restriction of $\rho_{\Pi, \lbd}$ to $\Gal_{F_{\mfk q}}$ is a direct sum of a residually absolutely irreducible self-dual representation $\sigma$ with a self-dual character $\chi$ by Proposition~\ref{ddddddeieopws} and the definition of $\Pi_{1, \mfk p}^{\flat, \prime}$. If the semi-simplified residual representation $\ovl\rho_{\Pi, \lbd}$ is not irreducible, then it is a sum of a self-dual absolutely irreducible representation with a self-dual character. On the other hand, if it is irreducible, then it is absolutely irreducible, because otherwise $\ovl\rho_{\Pi, \lbd}\otimes_{\kappa_\lbd}\ovl\kappa_\lbd$ is a sum of several irreducible representations of the same dimension, contradicting the fact that $\ovl\rho_{\Pi, \lbd}|_{\Gal_{F_{\mfk q}}}=\ovl\sigma\oplus\ovl\chi$.
\item
(pL2) holds by our choice of $\mfk p$ and the definition of $\Pi_{1, \mfk p}^{\flat, \prime}$; see Definition~\ref{ossewwttoeiiinfiemisw}.
\item
(pL3) holds by the definition of $\Pi_{1, \mfk p}^{\flat\flat, \prime}$ and the Chebotarev density theorem applied to the representation $\ovl\rho_{\Pi, \lbd}\oplus\ovl\ve_\ell$ of $\Gal_F$.
\end{itemize}

We now apply Conjecture~\ref{ocneienipwowie} to the preadmissible place $\lbd$ to get
\begin{equation*}
\bx H^1_f\paren{F, \Sym^{2r-1}\etH^1(A; \bb Q_\ell)(r)}\otimes_{\bb Q_\ell}E_\lbd=\bx H^1_f\paren{F, \rho_{\Pi_0, \ell}(r)}\otimes_{\bb Q_\ell}E_\lbd=\bx H^1_f\paren{F, \rho_{\Pi_0, \lbd}(r)}=0.
\end{equation*}
Thus $\bx H^1_f\paren{F, \Sym^{2r-1}\etH^1(A; \bb Q_\ell)(r)}$ vanishes.

The theorem is proved.
\end{proof}

\appendix

\section{Polarized local Galois representations}\label{oggieeotiieimfiiwis}

In this appendix, we construct certain (conjugate) self-dual local Galois representations of special kind. These representations will be used in the Burger--Sarnak type principles.

\subsection{Special conjugate self-dual local Galois representations}

In this subsection, we construct certain conjugate self-dual local Galois representations of special kind.

Let $p$ be an odd rational prime, and let $K$ be a finite extension of $\bb Q_p$. Denote by $\kappa$ the residue field of $K$, of cardinality $q$. Let $K_1$ be the unramified quadratic extension of $K$. Let $\mcl O_K$ (resp. $\mcl O_{K_1}$) denote the ring of integers of $K$ (resp. $K_1$) with maximal ideal $\mfk m_K$ (resp. $\mfk m_{K_1}$). Denote by $\kappa_1$ the residue field of $K_1$. Fix a uniformizer $\varpi_K$ of $K$.

We care about representations of $W_{K_1}$ that are conjugate-orthogonal, that is, if we write $\Pi^\theta\defining (\Pi^s)^\vee$, where $\Pi^s$ is the conjugate of $\Pi$ by an element $s\in W_K$ which maps to $\cc\in \Gal(K_1/K)$, then there is an isomorphism $f: \Pi^\theta\xr\sim \Pi$ satisfying $(f^\vee)^s=f$. Constructing irreducible conjugate self-dual representations of $W_{K_1}$ is more complicated than expected. We will only provide the following construction of residually absolutely irreducible conjugate-orthogonal representations when there is a tamely ramified cyclic extension of degree $2r$.

\begin{lm}\label{tteoeoeires}
Let $\ell$ be a rational prime distinct from $p$, with a fixed isomorphism $\iota_\ell: \bb C\xr\sim \ovl{\bb Q_\ell}$. Suppose $\ell$ is coprime to $2pr$ and $2r|(q^2-1)$. Then there exists a conjugate-orthogonal supercuspidal representation $\Pi$ of $\GL_{2r}(K_1)$ \sut  the Galois representation
\begin{equation*}
\iota_\ell\rec_{2r}(\Pi): W_{K_1}\to \GL_{2r}(\ovl{\bb Q_\ell})
\end{equation*}
attached to $\Pi$ via local Langlands correspondence is residually absolutely irreducible.
\end{lm}
\begin{proof}
By local Langlands correspondence for $\GL_{2r}(K_1)$, it suffices to construct a residually absolutely irreducible $2r$-dimensional representation $(\rho, V)$ of $W_{K_1}$ with $\ovl{\bb Q_\ell}$-coefficients, a lift $s\in W_K$ of $\cc\in\Gal(K_1/K)$, and a nondegenerate pairing $\bra{-, -}: V\times V\to \ovl{\bb Q_\ell}$ satisfying
\begin{equation}\label{oeimeigieimsiws}
\begin{cases}
\bra{\rho(\tau)f, \rho(s\tau s^{-1})g}=\bra{f, g}\\
\bra{g, f}=\bra{f, \rho(s^2)g}
\end{cases}
\end{equation}
for all $\tau\in W_{K_1}$ and $f, g\in V$.

Let $\ovl\gamma\in\kappa_1^\times$ be \sut $\{\ovl\gamma^q, \ovl\gamma\}$ is a $\kappa$-basis of $\kappa_1$. Let $\gamma\in K_1^\times$ denote the \Teichmuller lift of $\ovl\gamma$, and set
\begin{equation*}
E=K_1((\gamma\varpi_K)^{1/2r}),
\end{equation*}
which is a totally (tamely) ramified cyclic Galois extension of $K_1$ of degree $2r$ since $2r|(q^2-1)$. Let $\mcl O_E$ denote the ring of integers of $E$ with maximal ideal $\mfk m_E$. Let $W_E\subset W_{K_1}$ denote the corresponding Weil groups, and write $\ab_E: W_E\to W_E^\ab$ for the Abelianization map. Let
\begin{equation*}
\Art_E: E^\times\xr\sim W_E^\ab
\end{equation*}
be the local Artin map, normalized so that uniformizers are mapped to geometric Frobenius classes.

Let $\tau$ be a generator of $\Gal(E/K_1)\cong\bb Z/2r\bb Z$ and let $\varpi_E$ be a uniformizer of $E$ \sut $\tau(\varpi_E)=\zeta\varpi_E$ for some $2r$-th root of unity $\zeta\in K_1^\times$. Let $\phi\in \Gal(E/K)$ be lift of $\cc\in\Gal(K_1/K)$. By considering the action of $\phi$ we may change the lift so that $\phi(\gamma)=\gamma^q$ and $\phi(\varpi_E)=-\varpi_E$. In particular, $\phi^2=1$.

Recall the group decomposition
\begin{equation*}
E^\times=\bra{\varpi_E}\times \kappa_1^\times\times U^1_E, \quad U^1_E=1+\mfk m_E,
\end{equation*}
where $\kappa_1^\times$ embeds into $K_1^\times$ via the \Teichmuller lift $[-]: \kappa_1^\times\to K_1^\times$. Since $p>2$, the $p$-adic logarithm
\begin{equation*}
\log: U^1_E\to \mfk m_E: 1+x\mapsto \sum_{k\in\bb Z_+}\frac{(-x)^{k+1}}{k}
\end{equation*}
is a continuous group homomorphism and is $\Gal(E/K)$-equivariant. We extend $\log$ to a map $E^\times\to \mfk m_E$ by setting $\log(\varpi_E)=0$ and $\log|_{\kappa_1^\times}=0$.

Let $e_K$ denote the ramification index of $K$, and set $k_0=\floor{\frac{2re_K}{p-1}}+1$. Fix a positive integer $m>k_0$ to be determined later. Define, for $x\in E$,
\begin{equation*}
\Psi_E\defining \iota_\ell e^{2\pi\ii\cdot \tr_{E/\bb Q_p}(x)/p^{m+[K: \bb Q_p]}},
\end{equation*}
which is an additive character of $E$ of conductor at most $-(2re_K(m-1)+1)$. Let $\chi: E^\times\to\ovl{\bb Q_\ell}^\times$ denote the character
\begin{equation*}
\chi(x)\defining\Psi_E(\varpi_E\log x), \quad x\in E^\times.
\end{equation*}
When $m$ is sufficiently large, $\chi^\phi=\chi^{-1}$ and $\chi^\sigma\ne \chi$ for every nontrivial element $\sigma\in\Gal(E/K_1)$. Here we use that $\mfk m_E^{k_0}\subset \log(E^\times)$. Set
\begin{equation*}
\xi\defining \chi\circ\Art_E^{-1}\circ\ab_E: W_E\to \ovl{\bb Q_\ell}^\times.
\end{equation*}

Let
\begin{equation*}
(\rho, V)\defining\Ind_{W_E}^{W_{K_1}}\xi
\end{equation*}
denote the induced representation of $W_{K_1}$ of dimension $2r$. It follows from Mackey's theory that $\rho$ is absolutely irreducible. Fix an element $s\in W_K$ lifting $\phi\in\Gal(E/K)$, and define a pairing on $V$ given by
\begin{equation*}
\bra{f, g}\defining \sum_{[x]\in W_E\bsh W_{K_1}}f(x)g(sxs^{-1}).
\end{equation*}
Here for each $[x]\in W_E\bsh W_K$, $x\in W_K$ is a lift of $[x]$. Note that this is well-defined because replacing $x$ by $hx$ gives
\begin{equation*}
f(hx)g(shxs^{-1})=\xi(h)f(x)g(shs^{-1}(sxs^{-1}))=\xi(h)\xi^\phi(h)f(x)g(sxs^{-1})=f(x)g(sxs^{-1}).
\end{equation*}
This pairing is clearly nondegenerate. We check \eqref{oeimeigieimsiws}:
\begin{equation*}
\bra{\rho(\tau)f, \rho(s\tau s^{-1})g}=\sum_{[x]\in W_E\bsh W_{K_1}}f(x\tau)g(sx\tau s^{-1})=\sum_{[x]\in W_E\bsh W_{K_1}}f(x)g(sxs^{-1}).
\end{equation*}
\begin{equation*}
\bra{g, f}=\sum_{[x]\in W_E\bsh W_{K_1}}g(x)f(sxs^{-1})=\sum_{[x]\in W_E\bsh W_{K_1}}f(x)g(s^{-1}xs)=\xi(s^2)^{-1}\bra{f, \rho(s^2)g}.
\end{equation*}
Here we use that conjugation by $s$ permutes left $W_E$-cosets. We claim that $\xi(s^2)=1$. In fact, this claim is independent of the lift $s$ chosen, because for any other lift $s'=hs$ with $h\in W_E$,
\begin{equation*}
\xi((s')^2)=\xi(h)\xi(shs^{-1})\xi(s^2)=\xi(h)\xi^\phi(h)\xi(s^2)=\xi(s^2).
\end{equation*}
To prove the claim, we let $H$ denote the subgroup of elements of $W_K$ whose images in $\Gal(E/K)$ lie in $\bra{\phi}$. Then there is an exact sequence
\begin{equation*}
1\to \xi(W_E)\to H/\ker(\xi)\to \bra{\phi}\to 1.
\end{equation*}
Note that $\xi(W_E)$ is a finite $p$-group. So it follows from the Schur--Zassenhaus theorem that there exists a lift $s\in W_K$ of $\phi$ satisfying $s^2\in \ker(\xi)$. In particular, $\xi(s^2)=1$, and \eqref{oeimeigieimsiws} is proved.

We check that $\rho$ is residually absolutely irreducible: As $\rho$ factors through $W_K/\ker(\xi)$, which is a finite group with order dividing $2rq^M$ for some positive integer $M$. Thus $\ovl{\bb F_\ell}[\Im(\ovl\rho)]$ is a semisimple algebra, because $\ell$ is coprime to $2pr$. Thus the same Mackey theory argument implies that $\rho$ is residually absolutely irreducible.

The lemma is proved.
\end{proof}

\begin{defi}\label{osoeiiinfiemisw}
Let $\ell$ be a rational prime distinct from $p$, with a fixed isomorphism $\iota_\ell: \bb C\xr\sim \ovl{\bb Q_\ell}$. Let $B$ be a finite subset of $\ovl{\bb F_\ell}$. A $B$-avoiding good representation (with respect to $\iota_\ell$) is a representation $\Pi$ of $\GL_{2r}(K_1)$ \sut there exists some lift $F\in \Gal_K$ of the arithmetic Frobenius element satisfying
\begin{itemize}
\item
there is a partition $n=\sum_{i=1}^kn_i$ \sut $\Pi$ is an isobaric sum of distinct representations $\Pi_i$ where each $\Pi_i$ is a supercuspidal representation of $\GL_{n_i}(K_1)$;
\item
for each $1\le i\le k$, if we write $\Pi_i^\theta\defining (\Pi_i^s)^\vee$, where $\Pi_i^s$ is the conjugate of $\Pi_i$ by an element $s\in W_K$ which maps to $\cc\in \Gal(K_1/K)$, then there is an isomorphism $f_i: \Pi_i^\theta\xr\sim \Pi_i$ satisfying $(f_i^\vee)^s=f_i$;
\item
the Galois representation $\iota_\ell\rec_{2r}(\Pi): W_{K_1}\to \GL_{2r}(\ovl{\bb Q_\ell})$ attached to $\Pi$ via local Langlands maps $F^2$ to an element with generalized eigenvalues $\{\alpha_1, \ldots, \alpha_{2r}\}$ in which $\alpha_i$ is an $\ell$-adic unit with residue not in $B$ for every $1\le i\le 2r$.
\end{itemize}

Suppose $\Pi_0$ is a constituent of an unramified principal series of $\GL_{2r}(K_1)$ with Satake parameter $\alpha(\Pi_0)=\{\beta_1, \ldots, \beta_{2r}\}$. If $\iota_\ell(\beta_i)$ is an $\ell$-adic unit for every $1\le i\le r$, then we say a representation $\Pi$ of $\GL_{2r}(K_1)$ is $\Pi_0$-avoiding (\wrt $\iota_\ell$) if it is $B\defining\{-q, q\iota_\ell(\beta_1), \ldots, q\iota_\ell(\beta_{2r})\}\modu\ell$-avoiding.
\end{defi}

For a given finite subset $B\subset\ovl{\bb F_\ell}$, constructing $B$-avoiding good representations is more complicated than we expected. In fact, we do not know how to construct supercuspidal $B$-avoiding good representations. Nonetheless, we have the following result which is enough for our purpose.

\begin{lm}\label{soosseteeiifeimiefisws}
Let $\ell$ be a rational prime distinct from $p$, with a fixed isomorphism $\iota_\ell: \bb C\xr\sim \ovl{\bb Q_\ell}$. Suppose $\ell$ is coprime to $2pr$. For any finite subset $B\subset\ovl{\bb F_\ell}$, a $B$-avoiding good representation exists. We can further ensure that the generalized eigenvalues $\{\alpha_1, \ldots, \alpha_{2r}\}$ as defined in Definition~\ref{osoeiiinfiemisw} satisfy
\begin{equation*}
q^2\notin\{\alpha_i\alpha_j^{-1}|1\le i\ne j\le 2r\}\cup\{\alpha_i|1\le i\le 2r\}\subset\ovl{\bb F_\ell}.
\end{equation*}
\end{lm}
\begin{proof}
By local Langlands correspondence for $\GL_{2r}(K_1)$, it suffices to construct 
\begin{itemize}
\item
a $2r$-dimensional representation $(\rho, V)$ of $W_{K_1}$ with $\ovl{\bb Q_\ell}$-coefficients that is a sum of $r$ distinct $2$-dimensional irreducible representations,
\item
a lift $F\in W_K$ of the arithmetic Frobenius element,
\item
a lift $s\in W_K$ of $\cc\in\Gal(K_1/K)$, and
\item
a nondegenerate pairing $\bra{-, -}: V\times V\to \ovl{\bb Q_\ell}$
\end{itemize}
satisfying
\begin{enumerate}
\item
\begin{equation}\label{oeieeemeigieimsiws}
\begin{cases}
\bra{\rho(\tau)f, \rho(s\tau s^{-1})g}=\bra{f, g}\\
\bra{g, f}=\bra{f, \rho(s^2)g}
\end{cases}
\end{equation}
for all $\tau\in W_{K_1}$ and $f, g\in V$; and
\item
The eigenvalues $\{\alpha_1, \ldots, \alpha_{2r}\}$ of $\rho(F^2)$ are $\ell$-adic units with residues not in $B$, and
\begin{equation*}
q^2\notin\{\alpha_i\alpha_j^{-1}|1\le i\ne j\le 2r\}\cup\{\alpha_i|1\le i\le 2r\}\subset\ovl{\bb F_\ell}.
\end{equation*}
\end{enumerate}

Let $R/K$ be a quadratic ramified extension and let $E=RK_1$. Let $\mcl O_E$ denote the ring of integers of $E$ with maximal ideal $\mfk m_E$. Let $\tau$ denote the nontrivial element of $\Gal(R/K)\cong\bb Z/2\bb Z$. Then there is a natural identification
\begin{equation*}
\Gal(E/K)=\bra{\tau}\times\bra{\cc}\cong\bb Z/2\bb Z\times\bb Z/2\bb Z.
\end{equation*}
Let $W_E\subset W_{K_1}\subset W_K$ denote the corresponding Weil groups, and let $\ab_E: W_E\to W_E^\ab$ denote the Abelianization map. Let the Artin map
\begin{equation*}
\Art_E: E^\times\xr\sim W_E^\ab
\end{equation*}
be normalized so that uniformizers are mapped to geometric Frobenius classes. Choose a uniformizer $\varpi_R$ of $R$ with $\tau(\varpi_R)=-\varpi_R$.

Recall the group decomposition
\begin{equation*}
E^\times=\bra{\varpi_R}\times \kappa_1^\times\times U^1_E, \quad U^1_E=1+\mfk m_E,
\end{equation*}
where $\kappa_1^\times$ embeds into $K_1^\times$ via the \Teichmuller lift $[-]: \kappa_1^\times\to K_1^\times$. Since $p>2$, the $p$-adic logarithm
\begin{equation*}
\log: U^1_E\to \mfk m_E: 1+x\mapsto \sum_{k\in\bb Z_+}\frac{(-x)^{k+1}}{k}
\end{equation*}
is a continuous group homomorphism and is $\Gal(E/K)$-equivariant. We extend $\log$ to a map $E^\times\to \mfk m_E$ by setting $\log(\varpi_R)=0$ and $\log|_{\kappa_1^\times}=0$.

Let $d_K$ denote the different exponent of $K$, so the different ideal $\mfk d_K$ of $K$ over $\bb Q_p$ satisfies $\mfk d_K=\mfk m_K^{d_K}$. Let $e_K$ denote the ramification index of $K$, and set $k_0=\floor{\frac{2e_K}{p-1}}+1$. Fix a positive integer $m>k_0$ to be determined later. Let
\begin{equation*}
\Psi_{\bb Q_p}: \bb Q_p\to \ovl{\bb Q_\ell}^\times: x\mapsto\iota_\ell e^{2\pi\ii x}
\end{equation*}
denote the standard additive character of $\bb Q_p$ of conductor $0$, and set
\begin{equation*}
\Psi_K\defining \Psi_{\bb Q_p}\paren{\tr_{K/\bb Q_p}(\varpi_K^{-d_K-m}x)}, \quad \Psi_E\defining \Psi_K\circ\tr_{E/K}
\end{equation*}
Then $\Psi_E$ is an additive character of $E$ of conductor $1-2m$.

For each $1\le i\le r$, let $\chi_i: E^\times\to\ovl{\bb Q_\ell}^\times$ denote the character given by
\begin{equation*}
\chi_i(x)\defining\Psi_E(p^{1-i}\varpi_R\log x), \quad x\in E^\times.
\end{equation*}
Set $\phi=\tau\cc\in\Gal(E/K)$. Then $\chi_i^\phi=\chi_i^\tau=\chi_i^{-1}\ne\chi_i$ for every $1\le i\le r$. It is clear that $\chi_i\ne \chi_j^\pm$ for $1\le i<j\le r$. Here we use that $\mfk m_E^{k_0}\subset\log(E^\times)$. Set
\begin{equation*}
\xi_i\defining \chi_i\circ\Art_E^{-1}\circ\ab_E: W_E\to \ovl{\bb Q_\ell}^\times, \quad 1\le i\le r.
\end{equation*}

For each $1\le i\le r$, let
\begin{equation*}
(\rho_i, V_i)\defining\Ind_{W_E}^{W_{K_1}}\xi_i
\end{equation*}
denote the induced representation of $W_K$ of dimension $2r$. And we define
\begin{equation*}
(\rho, V)\defining\bplus_{i=1}^r(\rho_i, V_i).
\end{equation*}
It follows from Mackey's theory that $\rho$ is a direct sum of $r$ distinct $2$-dimensional absolutely irreducible representations. Fix any element $s\in W_K$ lifting $\phi\in\Gal(E/K)$, and define a pairing on $V_i$ given by
\begin{equation*}
\bra{(f_1, \ldots, f_r), (g_1, \ldots, g_r)}\defining\sum_{i=1}^r\sum_{[x]\in W_E\bsh W_{K_1}}f_i(x)g_i(sxs^{-1}).
\end{equation*}
Here for each $[x]\in W_E\bsh W_K$, $x\in W_K$ is a lift of $[x]$. The same argument as in the proof of Lemma~\ref{tteoeoeires} shows that \eqref{oeieeemeigieimsiws} holds.

We compute the Frobenius eigenvalues of the residual representation of $\rho$. Fix a lift $F_0\in W_K$ of $\phi\in\Gal(E/K)$. For each $t\in I_E$, $F=tF_0$ is also a lift of $\cc\in\Gal(E/K)$. The characteristic polynomial of $\ovl\rho(F^2)$ is
\begin{equation*}
\chi_{\ovl\rho(F^2)}(X)=\prod_{i=1}^r(X-\ovl\xi_i(F^2))(X-\ovl\xi_i(F^2)^{-1}).
\end{equation*}

For each $1\le i\le r$, note that
\begin{align*}
\xi_i(F^2)&=\xi_i(F_0^2)\Psi_E\paren{p^{1-i}\varpi_R\tr_{E/R}(\log\Art_E^{-1}\ab_E(i))}\\
&=\xi(F_0^r)\Psi_{\bb Q_p}\paren{\tr_{K/\bb Q_p}\varpi_K^{-d_K-m}\tr_{R/K}\paren{2p^{1-i}\varpi_R\tr_{E/R}(\log\Art_E^{-1}\ab_E(i))}}.
\end{align*}
When $t$ varies, $\log\Art_E^{-1}\ab_E(t)$ ranges over all elements in $\mfk m_E^{k_0}$. Since $E/R$ is unramified, $\tr_{E/R}(\log\Art_E^{-1}\ab_E(t))$ ranges over all elements in $\mfk m_R^{k_0}$, and
\begin{equation*}
\tr_{R/K}\paren{2p^{1-i}\varpi_R\tr_{E/R}(\log\Art_E^{-1}\ab_E(i))}
\end{equation*}
ranges over all elements in $\mfk m_K^{k_0}$. Thus when $t$ varies,
\begin{equation*}
\Psi_{\bb Q_p}\paren{\tr_{K/\bb Q_p}\varpi_K^{-d_K-m}\tr_{R/K}\paren{2p^{1-i}\varpi_R\tr_{E/R}(\log\Art_E^{-1}\ab_E(i))}}
\end{equation*}
can be every $p^{\floor{\frac{m-k_0}{e_K}}}$-th roots of unity in $\ovl{\bb Q_\ell}$. Thus, when $s$ is large, it is clear that we can take some $t\in I_E$ \sut $\ovl\xi_i(F^2)$ is not contained in the set
\begin{equation*}
\{b^{\pm1}|b\in B\}\cup\{q^{\pm1}, q^{\pm2}\}\subset\ovl{\bb F_\ell}
\end{equation*}
for every $1\le i\le r$. As a result, $\ovl\rho(F^2)$ has no eigenvalues in $B$. 

Similarly, when $s$ is large, we can further assume that $t\in I_E$ is chosen so that each of the elements
\begin{equation*}
\ovl\xi_i(F^2)\ovl\xi_j(F^2)=\xi_i(F_0^2)\Psi_E\paren{(p^{1-i}+p^{1-j})\varpi_R\tr_{E/R}(\log\Art_E^{-1}\ab_E(i))}, \quad 1\le i<j\le r
\end{equation*}
and
\begin{equation*}
\ovl\xi_i(F^2)\ovl\xi_j^{-1}(F^2)=\xi_i(F_0^2)\Psi_E\paren{(p^{1-i}-p^{1-j})\varpi_R\tr_{E/R}(\log\Art_E^{-1}\ab_E(i))}, \quad 1\le i\le r
\end{equation*}
is not contained in $\{q^{\pm2}\}\subset\ovl{\bb F_\ell}$.

The desired properties of $\rho$ are all proved.
\end{proof}

\subsection{Special self-dual local Galois representations}

Let $p$ be an odd rational prime, and let $K$ be a finite extension of $\bb Q_p$. Let $\kappa$ denote the residue field of $K$, of cardinality $q$. let $\mcl O_K$ denote the ring of integers of $K$ with maximal ideal $\mfk m_K$. Fix a uniformizer $\varpi_K$ of $K$.

\begin{defi}\label{ossewwttoeiiinfiemisw}
Let $\ell$ be a rational prime distinct from $p$, with a fixed isomorphism $\iota_\ell: \bb C\xr\sim \ovl{\bb Q_\ell}$. For a finite subset $B\subset\ovl{\bb F_\ell}$, a supercuspidal $B$-avoiding good representation (with respect to $\iota_\ell$) is a supercuspidal representation $\Pi$ of $\GL_{2r}(K)$ \sut there exists some lift $F\in \Gal_K$ of the arithmetic Frobenius element satisfying
\begin{itemize}
\item
there is an isomorphism $f: \Pi^\vee\xr\sim \Pi$ satisfying $f^\vee=f$.
\item
the Galois representation $\iota_\ell\rec_{2r}(\Pi): W_K\to \GL_{2r}(\ovl{\bb Q_\ell})$ attached to $\Pi$ via local Langlands maps $F$ to an element with generalized eigenvalues $\{\alpha_1, \ldots, \alpha_{2r}\}$ in which $\alpha_i$ is an $\ell$-adic unit with residue not in $B$ for every $1\le i\le 2r$.
\end{itemize}

If $\Pi_0$ is a constituent of an unramified principal series of $\GL_{2r}(K)$ with Satake parameter $\alpha(\Pi_0)=\{\beta_1, \ldots, \beta_{2r}\}$, then we say a representation $\Pi$ of $\GL_{2r}(K)$ is $\Pi_0$-avoiding if it is $B$-avoiding for 
\begin{equation*}
B\defining\{\pm1, \pm q^{\pm1}, \ldots, \pm q^{\pm4r}\}\cup\{q^{1/2}\beta_1, \ldots, q^{1/2}\beta_{2r}\}\subset\ovl{\bb F_\ell}.
\end{equation*}
\end{defi}

\begin{lm}\label{sooeiifeimiefisws}
Let $\ell$ be a rational prime with an isomorphism $\iota_\ell: \bb C\xr\sim \ovl{\bb Q_\ell}$ satisfying $\ell\nmid 2pr$. For any finite subset $B\subset\ovl{\bb F_\ell}$, a supercuspidal $B$-avoiding good representation $\Pi$ exists. Moreover, we can make sure that $\iota_\ell\rec_{2r}(\Pi)$ is absolutely residually irreducible. If $q^{2r}-1$ is not divisible by $\ell$, we can further make sure that the generalized eigenvalues $\{\alpha_1, \ldots, \alpha_{2r}\}$ as defined in Definition~\ref{ossewwttoeiiinfiemisw} satisfy
\begin{equation*}
q\notin\{\pm\alpha_i\alpha_j^{-1}|1\le i\ne j\le 2r\}\cup\{\pm\alpha_i|1\le i\le 2r\}\subset\ovl{\bb F_\ell}.
\end{equation*}
\end{lm}
\begin{proof}
It suffices to show that we can find a residually absolutely irreducible $2r$-dimensional representation $(\rho, V)$ of $W_K$ with $\ovl{\bb Q_\ell}$-coefficients and a lift $F\in\Gal_K$ of the arithmetic Frobenius element satisfying
\begin{itemize}
\item
there exists a $W_K$-invariant nondegenerate $\ovl{\bb Q_\ell}$-valued symmetric pairing on $V$;
\item
the eigenvalues $\{\alpha_1, \ldots, \alpha_{2r}\}$ of $\rho(F)$ are $\ell$-adic units with residues not in $B$; and
\item
\begin{equation*}
q\notin\{\pm\alpha_i\alpha_j^{-1}|1\le i\ne j\le 2r\}\cup\{\pm\alpha_i|1\le i\le 2r\}\subset\ovl{\bb F_\ell}.
\end{equation*}
holds if $q^{2r}-1$ is not divisible by $\ell$.
\end{itemize}

Choose $U/K$ unramified of degree $r$ with Frobenius class $\sigma$, and $R/K$ ramified quadratic with Galois group $\Gal(R/K)=\{1, \tau\}$. Set $E=UR$. Let $\mcl O_E$ denote the ring of integers of $E$ with maximal ideal $\mfk m_E$. Then there is a natural identification
\begin{equation*}
\Gal(E/K)=\bra{\phi}\times\bra{\tau}\cong\bb Z/r\bb Z\times\bb Z/2\bb Z.
\end{equation*}
Let $W_E\subset W_K$ be the corresponding Weil groups, and let $\ab_E: W_E\to W_E^\ab$ denote the Abelianization map. Let the Artin map
\begin{equation*}
\Art_E: E^\times\xr\sim W_E^\ab
\end{equation*}
be normalized so that uniformizers are mapped to geometric Frobenius classes.

Let $\kappa_U$ denote the residue field of $U$. Choose a uniformizer $\varpi_R$ of $R$ satisfying $\tau(\varpi_R)=-\varpi_R$. Recall the decomposition
\begin{equation*}
E^\times=\bra{\varpi_R}\times \kappa_U^\times\times U^1_E, \quad U^1_E=1+\mfk m_E,
\end{equation*}
where $\kappa_U^\times$ embeds into $K^\times$ via the \Teichmuller lift $[-]: \kappa_U^\times\to K^\times$. Since $p>2$, the $p$-adic logarithm
\begin{equation*}
\log: U^1_E\to \mfk m_E: 1+x\mapsto \sum_{k\in\bb Z_+}\frac{(-x)^{k+1}}{k}
\end{equation*}
is a continuous group homomorphism and is $\Gal(E/K)$-equivariant. We extend $\log$ to a map $E^\times\to \mfk m_E$ by setting $\log(\varpi_R)=0$ and $\log|_{\kappa_U^\times}=0$.

Let $d_K$ denote the different exponent of $K$, so the different ideal $\mfk d_K$ of $K$ over $\bb Q_p$ satisfies $\mfk d_K=\mfk m_K^{d_K}$. Let $e_K$ denote the ramification index of $K$, and set $k_0=\floor{\frac{2e_K}{p-1}}+1$. Fix a positive integer $s>k_0$ to be determined later. Let
\begin{equation*}
\Psi_{\bb Q_p}: \bb Q_p\to \ovl{\bb Q_\ell}^\times: x\mapsto\iota_\ell e^{2\pi\ii x}
\end{equation*}
denote the standard additive character of $\bb Q_p$ of conductor $0$, and set
\begin{equation*}
\Psi_K\defining \Psi_{\bb Q_p}\paren{\tr_{K/\bb Q_p}(\varpi_K^{-d_K-s}x)}, \quad \Psi_E\defining \Psi_K\circ\tr_{E/K}
\end{equation*}
Then $\Psi_E$ is an additive character of $E$ of conductor $1-2s$.

Take an element $\ovl\gamma\in \kappa_U^\times$ satisfying
\begin{enumerate}
\item
$\ovl\gamma^{q^i}\ne \pm\ovl\gamma$ for every $1\le i\le r-1$, and
\item
$\tr_{\kappa_U/\kappa}(\ovl\gamma)\ne 0$.
\end{enumerate}
Such an element $\ovl\gamma$ exists by normal basis theorem. Indeed, we can take $\ovl\gamma$ \sut
\begin{equation*}
\{\sigma^{q^i}(\ovl\gamma)|0\le i\le r-1\}
\end{equation*}
is a $\kappa$-basis of $\kappa_U$. Set $\alpha=\varpi_R[\ovl\gamma]$. Then
\begin{enumerate}
\item
$\tau(\alpha)=-\alpha$, and
\item
$\sigma(\alpha)-\alpha\in\mfk m_E\setm\mfk m_E^2$ for every nontrivial element $\sigma\in \Gal(E/K)$.
\end{enumerate}

Let $\chi: E^\times\to\ovl{\bb Q_\ell}^\times$ denote the character given by
\begin{equation*}
\chi(x)\defining\Psi_E(\alpha\log x), \quad x\in E^\times.
\end{equation*}
Then $\chi^\tau=\chi^{-1}$ and $\chi^\sigma\ne \chi$ for every nontrivial element $\sigma\in\Gal(E/K)$. Here we use that $\mfk m_E^{k_0}\subset\log(E^\times)$. Set
\begin{equation*}
\xi\defining \chi\circ\Art_E^{-1}\circ\ab_E: W_E\to \ovl{\bb Q_\ell}^\times.
\end{equation*}

Let
\begin{equation*}
(\rho, V)\defining\Ind_{W_E}^{W_K}\xi
\end{equation*}
be the induced representation of $W_K$ of dimension $2r$. It follows from Mackey's theory that $\rho$ is absolutely irreducible. Fix any element $y\in W_K$ lifting $\tau\in\Gal(E/K)$, and define a pairing on $V$ given by
\begin{equation*}
\bra{f, g}\defining \sum_{[x]\in W_E\bsh W_K}f(x)g(y^{-1}x).
\end{equation*}
Here for each $[x]\in W_E\bsh W_K$, $x\in W_K$ is a lift of $[x]$. Note that this is well-defined because replacing $x$ by $hx$ gives
\begin{equation*}
f(hx)g(y^{-1}hx)=\xi(h)f(x)g(y^{-1}hy(y^{-1}x))=\xi(h)\xi^\tau(h)f(x)g(y^{-1}x)=f(x)g(y^{-1}x).
\end{equation*}
This pairing is clearly $W_K$-invariant and nondegenerate. We check that it is symmetric:
\begin{align*}
\bra{f, g}&=\sum_{[x]\in W_E\bsh W_K}f(x)g(y^{-1}x)\\
&=\sum_{[x]\in W_E\bsh W_K}f(yxy^{-1})g(xy^{-1})\\
&=\sum_{[x]\in W_E\bsh W_K}f(yx)g(x)\\
&=\xi(y^2)\sum_{[x]\in W_E\bsh W_K}g(x)f(y^{-1}x)\\
&=\xi(y^2)\bra{g, f}
\end{align*}
Here we use that conjugation by $y$ permutes left $W_E$-cosets. We claim that $\xi(y^2)=1$. In fact, this claim is independent of the lift $y$ chosen, because for any other lift $y'=hy$ with $h\in W_E$,
\begin{equation*}
\xi((y')^2)=\xi(h)\xi(yhy^{-1})\xi(y^2)=\xi(h)\xi^\tau(h)\xi(y^2)=\xi(y^2).
\end{equation*}
To prove the claim, we let $H$ denote the subgroup of elements of $W_K$ whose images in $\Gal(E/K)$ lie in $\bra{\tau}$. Then there is an exact sequence
\begin{equation*}
1\to \xi(W_E)\to H/\ker(\xi)\to \bra{\tau}\to 1.
\end{equation*}
Note that $\xi(W_E)$ is a finite $p$-group. So it follows from the Schur--Zassenhaus theorem  that there exists a lift $y\in W_K$ of $\tau$ satisfying $y^2\in \ker(\xi)$. In particular, $\xi(y^2)=1$, and the form $\bra{-, -}$ is symmetric.

We check that $\rho$ is residually absolutely irreducible: As $\rho$ factors through $W_K/\ker(\xi)$, which is a finite group with order dividing $2rq^{r(2s-3)}$. Thus $\ovl{\bb F_\ell}[\Im(\ovl\rho)]$ is a semisimple algebra, because $\ell$ is coprime to $2pr$. Thus the same Mackey theory argument implies that $\rho$ is residually absolutely irreducible.

We compute the Frobenius eigenvalues of the residual representation of $\rho$. Fix a lift $F_0\in W_K$ of $\phi\in\Gal(E/K)$. For each $t\in I_E$, $F=tF_0$ is also lift of $\bra{\phi}\in\Gal(E/K)$. The characteristic polynomial of $\ovl\rho(F)$ is
\begin{equation*}
\chi_{\ovl\rho(F)}(X)=(X^r-\ovl\xi(F^r))(X^r-\ovl\xi(F^r)^{-1}).
\end{equation*}
Note that
\begin{align*}
\xi(F^r)&=\xi(F_0^r)\Psi_E\paren{\alpha\tr_{E/R}(\log\Art_E^{-1}\ab_E(i))}\\
&=\xi(F_0^r)\Psi_{\bb Q_p}\paren{\tr_{K/\bb Q_p}\varpi_K^{-d_K-s}\tr_{R/K}\paren{\tr_{E/R}(\alpha)\tr_{E/R}(\log\Art_E^{-1}\ab_E(i))}}.
\end{align*}
It follows from the choice of $\gamma$ that 
\begin{equation*}
\tr_{E/R}(\alpha)=\varpi_R\tr_{E/R}\gamma\in\mfk m_R\setm\mfk m_R^2.
\end{equation*}
When $t$ varies, $\log\Art_E^{-1}\ab_E(t)$ ranges over all elements in $\mfk m_E^{k_0}$. Since $E/R$ is unramified and $\tr_{E/R}(\alpha)\in \mfk m_R\setm\mfk m_R^2$, $\tr_{E/R}(\log\Art_E^{-1}\ab_E(t))$ ranges over all elements in $\mfk m_R^{k_0}$, and
\begin{equation*}
\tr_{R/K}\paren{\tr_{E/R}(\alpha)\tr_{E/R}(\log\Art_E^{-1}\ab_E(i))}
\end{equation*}
ranges over all elements in $\mfk m_K^{k_0}$. Thus when $t$ varies,
\begin{equation*}
\Psi_{\bb Q_p}\paren{\tr_{K/\bb Q_p}\varpi_K^{-d_K-s}\tr_{R/K}\paren{\tr_{E/R}(\alpha)\tr_{E/R}(\log\Art_E^{-1}\ab_E(i))}}
\end{equation*}
can be every $p^{\floor{\frac{s-k_0}{e_K}}}$-th roots of unity in $\ovl{\bb Q_\ell}$. When $s$ is large, it is clear that we can take some $t\in I_E$ \sut $\ovl\xi(F^r)$ is not contained in the set
\begin{equation*}
\{b^{\pm r}|b\in B\}\cup\{(\pm q)^{\pm r}\}\subset\ovl{\bb F_\ell}.
\end{equation*}
As a result, $\ovl\rho(F)$ has no eigenvalues in $B$. Moreover, if $q^{2r}-1$ is not divisible by $\ell$, it is clear that 
\begin{equation*}
q\notin\{\pm\alpha_i\alpha_j^{-1}|1\le i\ne j\le 2r\}\cup\{\pm\alpha_i|1\le i\le 2r\}\subset\ovl{\bb F_\ell}
\end{equation*}
is satisfied.

The desired properties of $\rho$ are all proved.
\end{proof}

\bibliography{bibliography}

\begin{bibdiv}
\begin{biblist}

\bib{A-B98}{article}{
      author={Adams, Jeffrey},
      author={Barbasch, Dan},
       title={Genuine representations of the metaplectic group},
        date={1998},
        ISSN={0010-437X,1570-5846},
     journal={Compositio Math.},
      volume={113},
      number={1},
       pages={23\ndash 66},
         url={https://doi.org/10.1023/A:1000450504919},
      review={\MR{1638210}},
}

\bib{A-C89}{book}{
      author={Arthur, James},
      author={Clozel, Laurent},
       title={Simple algebras, base change, and the advanced theory of the
  trace formula},
   publisher={Princeton University Press},
        date={1989},
      number={120},
}

\bib{A-G17}{article}{
      author={Atobe, Hiraku},
      author={Gan, Wee~Teck},
       title={On the local {L}anglands correspondence and {A}rthur conjecture
  for even orthogonal groups},
        date={2017},
        ISSN={1088-4165},
     journal={Represent. Theory},
      volume={21},
       pages={354\ndash 415},
         url={https://doi.org/10.1090/ert/504},
      review={\MR{3708200}},
}

\bib{AMRT}{book}{
      author={Ash, A.},
      author={Mumford, D.},
      author={Rapoport, M.},
      author={Tai, Y.},
       title={Smooth compactification of locally symmetric varieties},
      series={Lie Groups: History, Frontiers and Applications},
   publisher={Math Sci Press, Brookline, MA},
        date={1975},
      volume={Vol. IV},
      review={\MR{457437}},
}

\bib{And04}{book}{
      author={Andr{\'e}, Yves},
       title={Une introduction aux motifs (motifs purs, motifs mixtes,
  p\'eriodes)},
      series={Panoramas et Synth\`eses [Panoramas and Syntheses]},
   publisher={Soci\'et\'e{} Math\'ematique de France, Paris},
        date={2004},
      volume={17},
        ISBN={2-85629-164-3},
      review={\MR{2115000}},
}

\bib{Art13}{book}{
      author={Arthur, James},
       title={The endoscopic classification of representations},
      series={American Mathematical Society Colloquium Publications},
   publisher={American Mathematical Society, Providence, RI},
        date={2013},
      volume={61},
        ISBN={978-0-8218-4990-3},
         url={https://doi.org/10.1090/coll/061},
        note={Orthogonal and symplectic groups},
      review={\MR{3135650}},
}

\bib{B-C83}{article}{
      author={Borel, A.},
      author={Casselman, W.},
       title={{$L^2$}-cohomology of locally symmetric manifolds of finite
  volume},
        date={1983},
        ISSN={0012-7094,1547-7398},
     journal={Duke Math. J.},
      volume={50},
      number={3},
       pages={625\ndash 647},
         url={https://doi.org/10.1215/S0012-7094-83-05029-9},
      review={\MR{714821}},
}

\bib{B-D05}{article}{
      author={Bertolini, M.},
      author={Darmon, H.},
       title={Iwasawa's main conjecture for elliptic curves over anticyclotomic
  {$\mathbb Z_p$}-extensions},
        date={2005},
        ISSN={0003-486X,1939-8980},
     journal={Ann. of Math. (2)},
      volume={162},
      number={1},
       pages={1\ndash 64},
         url={https://doi.org/10.4007/annals.2005.162.1},
      review={\MR{2178960}},
}

\bib{BFKMMS}{article}{
      author={Blomer, Valentin},
      author={Fouvry, \'Etienne},
      author={Kowalski, Emmanuel},
      author={Michel, Philippe},
      author={Mili\'cevi\'c, Djordje},
      author={Sawin, Will},
       title={The second moment theory of families of {$L$}-functions---the
  case of twisted {H}ecke {$L$}-functions},
        date={2023},
        ISSN={0065-9266,1947-6221},
     journal={Mem. Amer. Math. Soc.},
      volume={282},
      number={1394},
       pages={v+148},
         url={https://doi.org/10.1090/memo/1394},
      review={\MR{4539366}},
}

\bib{B-G14}{incollection}{
      author={Buzzard, Kevin},
      author={Gee, Toby},
       title={The conjectural connections between automorphic representations
  and {G}alois representations},
        date={2014},
   booktitle={Automorphic forms and {G}alois representations. {V}ol. 1},
      series={London Math. Soc. Lecture Note Ser.},
      volume={414},
   publisher={Cambridge Univ. Press, Cambridge},
       pages={135\ndash 187},
         url={https://doi.org/10.1017/CBO9781107446335.006},
      review={\MR{3444225}},
}

\bib{B-K90}{incollection}{
      author={Bloch, Spencer},
      author={Kato, Kazuya},
       title={{$L$}-functions and {T}amagawa numbers of motives},
        date={1990},
   booktitle={The {G}rothendieck {F}estschrift, {V}ol.\ {I}},
      series={Progr. Math.},
      volume={86},
   publisher={Birkh\"auser Boston, Boston, MA},
       pages={333\ndash 400},
      review={\MR{1086888}},
}

\bib{Beu14}{article}{
      author={Beuzart-Plessis, Rapha\"el},
       title={Expression d'un facteur epsilon de paire par une formule
  int\'egrale},
        date={2014},
        ISSN={0008-414X,1496-4279},
     journal={Canad. J. Math.},
      volume={66},
      number={5},
       pages={993\ndash 1049},
         url={https://doi.org/10.4153/CJM-2013-042-4},
      review={\MR{3251763}},
}

\bib{Beu15}{article}{
      author={Beuzart-Plessis, R.},
       title={Endoscopie et conjecture locale raffin\'ee de
  {G}an-{G}ross-{P}rasad pour les groupes unitaires},
        date={2015},
        ISSN={0010-437X,1570-5846},
     journal={Compos. Math.},
      volume={151},
      number={7},
       pages={1309\ndash 1371},
         url={https://doi.org/10.1112/S0010437X14007891},
      review={\MR{3371496}},
}

\bib{Beu16}{article}{
      author={Beuzart-Plessis, Rapha\"el},
       title={La conjecture locale de {G}ross-{P}rasad pour les
  repr\'esentations temp\'er\'ees des groupes unitaires},
        date={2016},
        ISSN={0249-633X,2275-3230},
     journal={M\'em. Soc. Math. Fr. (N.S.)},
      number={149},
       pages={vii+191},
         url={https://doi.org/10.24033/msmf.457},
      review={\MR{3676153}},
}

\bib{BCZ22}{article}{
      author={Beuzart-Plessis, Rapha{\"e}l},
      author={Chaudouard, Pierre-Henri},
      author={Zydor, Micha{\l}},
       title={The global {G}an-{G}ross-{P}rasad conjecture for unitary groups:
  the endoscopic case},
        date={2022},
        ISSN={0073-8301,1618-1913},
     journal={Publ. Math. Inst. Hautes \'Etudes Sci.},
      volume={135},
       pages={183\ndash 336},
         url={https://doi.org/10.1007/s10240-021-00129-1},
      review={\MR{4426741}},
}

\bib{BPLZZ}{article}{
      author={Beuzart-Plessis, Rapha{\"e}l},
      author={Liu, Yifeng},
      author={Zhang, Wei},
      author={Zhu, Xinwen},
       title={Isolation of cuspidal spectrum, with application to the
  {G}an-{G}ross-{P}rasad conjecture},
        date={2021},
        ISSN={0003-486X,1939-8980},
     journal={Ann. of Math. (2)},
      volume={194},
      number={2},
       pages={519\ndash 584},
         url={https://doi.org/10.4007/annals.2021.194.2.5},
      review={\MR{4298750}},
}

\bib{B-S91}{article}{
      author={Burger, M.},
      author={Sarnak, P.},
       title={Ramanujan duals. {II}},
        date={1991},
        ISSN={0020-9910,1432-1297},
     journal={Invent. Math.},
      volume={106},
      number={1},
       pages={1\ndash 11},
         url={https://doi.org/10.1007/BF01243900},
      review={\MR{1123369}},
}

\bib{Car12}{article}{
      author={Caraiani, Ana},
       title={Local-global compatibility and the action of monodromy on nearby
  cycles},
        date={2012},
        ISSN={0012-7094,1547-7398},
     journal={Duke Math. J.},
      volume={161},
      number={12},
       pages={2311\ndash 2413},
         url={https://doi.org/10.1215/00127094-1723706},
      review={\MR{2972460}},
}

\bib{Car14}{article}{
      author={Caraiani, Ana},
       title={Monodromy and local-global compatibility for {$l=p$}},
        date={2014},
        ISSN={1937-0652,1944-7833},
     journal={Algebra Number Theory},
      volume={8},
      number={7},
       pages={1597\ndash 1646},
         url={https://doi.org/10.2140/ant.2014.8.1597},
      review={\MR{3272276}},
}

\bib{C-H13}{article}{
      author={Chenevier, Ga\"etan},
      author={Harris, Michael},
       title={Construction of automorphic {G}alois representations, {II}},
        date={2013},
        ISSN={2168-0930,2168-0949},
     journal={Camb. J. Math.},
      volume={1},
      number={1},
       pages={53\ndash 73},
         url={https://doi.org/10.4310/CJM.2013.v1.n1.a2},
      review={\MR{3272052}},
}

\bib{Clo90}{incollection}{
      author={Clozel, Laurent},
       title={Motifs et formes automorphes: applications du principe de
  fonctorialit\'e},
        date={1990},
   booktitle={Automorphic forms, {S}himura varieties, and {$L$}-functions,
  {V}ol.\ {I} ({A}nn {A}rbor, {MI}, 1988)},
      series={Perspect. Math.},
      volume={10},
   publisher={Academic Press, Boston, MA},
       pages={77\ndash 159},
      review={\MR{1044819}},
}

\bib{C-Z21}{article}{
      author={Chen, Rui},
      author={Zou, Jialiang},
       title={Local {L}anglands correspondence for even orthogonal groups via
  theta lifts},
        date={2021},
        ISSN={1022-1824,1420-9020},
     journal={Selecta Math. (N.S.)},
      volume={27},
      number={5},
       pages={Paper No. 88, 71},
         url={https://doi.org/10.1007/s00029-021-00704-8},
      review={\MR{4308934}},
}

\bib{C-Z24}{article}{
      author={Chen, Rui},
      author={Zou, Jialiang},
       title={Arthur's multiplicity formula for even orthogonal and unitary
  groups},
        date={2024},
     journal={Preprint},
      eprint={2103.07956},
         url={https://arxiv.org/abs/2103.07956},
}

\bib{Del79}{incollection}{
      author={Deligne, P.},
       title={Valeurs de fonctions {$L$}\ et p\'eriodes d'int\'egrales},
        date={1979},
   booktitle={Automorphic forms, representations and {$L$}-functions ({P}roc.
  {S}ympos. {P}ure {M}ath., {O}regon {S}tate {U}niv., {C}orvallis, {O}re.,
  1977), {P}art 2},
      series={Proc. Sympos. Pure Math.},
      volume={XXXIII},
   publisher={Amer. Math. Soc., Providence, RI},
       pages={313\ndash 346},
        note={With an appendix by N. Koblitz and A. Ogus},
      review={\MR{546622}},
}

\bib{Fal83}{incollection}{
      author={Faltings, Gerd},
       title={On the cohomology of locally symmetric {H}ermitian spaces},
        date={1983},
   booktitle={Paul {D}ubreil and {M}arie-{P}aule {M}alliavin algebra seminar,
  35th year ({P}aris, 1982)},
      series={Lecture Notes in Math.},
      volume={1029},
   publisher={Springer, Berlin},
       pages={55\ndash 98},
         url={https://doi.org/10.1007/BFb0098927},
      review={\MR{732471}},
}

\bib{F-H95}{article}{
      author={Friedberg, Solomon},
      author={Hoffstein, Jeffrey},
       title={Nonvanishing theorems for automorphic {$L$}-functions on {${\rm
  GL}(2)$}},
        date={1995},
        ISSN={0003-486X,1939-8980},
     journal={Ann. of Math. (2)},
      volume={142},
      number={2},
       pages={385\ndash 423},
         url={https://doi.org/10.2307/2118638},
      review={\MR{1343325}},
}

\bib{Fin21}{article}{
      author={Fintzen, Jessica},
       title={Types for tame {$p$}-adic groups},
        date={2021},
        ISSN={0003-486X,1939-8980},
     journal={Ann. of Math. (2)},
      volume={193},
      number={1},
       pages={303\ndash 346},
         url={https://doi.org/10.4007/annals.2021.193.1.4},
      review={\MR{4199732}},
}

\bib{F-P23}{article}{
      author={Fakhruddin, Najmuddin},
      author={Pilloni, Vincent},
       title={Hecke operators and the coherent cohomology of {S}himura
  varieties},
        date={2023},
        ISSN={1474-7480,1475-3030},
     journal={J. Inst. Math. Jussieu},
      volume={22},
      number={1},
       pages={1\ndash 69},
         url={https://doi.org/10.1017/S1474748021000050},
      review={\MR{4556929}},
}

\bib{Gar05}{article}{
      author={Garrett, Paul},
       title={Universality of holomorphic discrete series},
        date={2005},
     journal={notes available at homepage of author},
        note={\url{https://www-users.cse.umn.edu/~garrett/m/v/holo_disc.pdf}},
}

\bib{GGP12}{incollection}{
      author={Gan, Wee~Teck},
      author={Gross, Benedict~H.},
      author={Prasad, Dipendra},
       title={Symplectic local root numbers, central critical $l$-values, and
  restriction problems in the representation theory of classical groups},
    language={en},
        date={2012},
   booktitle={Sur les conjectures de gross et prasad. i},
      editor={Gan, W.~T.},
      editor={Gross, B.~H.},
      editor={Prasad, D.},
      editor={Waldspurger, J.-L.},
      series={Ast\'erisque},
   publisher={Soci\'et\'e math\'ematique de France},
       pages={1\ndash 109},
         url={https://www.numdam.org/item/AST_2012__346__1_0/},
      review={\MR{3202556}},
}

\bib{G-I16}{article}{
      author={Gan, Wee~Teck},
      author={Ichino, Atsushi},
       title={The {G}ross-{P}rasad conjecture and local theta correspondence},
        date={2016},
        ISSN={0020-9910,1432-1297},
     journal={Invent. Math.},
      volume={206},
      number={3},
       pages={705\ndash 799},
         url={https://doi.org/10.1007/s00222-016-0662-8},
      review={\MR{3573972}},
}

\bib{G-I18}{article}{
      author={Gan, Wee~Teck},
      author={Ichino, Atsushi},
       title={The {S}himura-{W}aldspurger correspondence for {${\rm
  Mp}_{2n}$}},
        date={2018},
        ISSN={0003-486X,1939-8980},
     journal={Ann. of Math. (2)},
      volume={188},
      number={3},
       pages={965\ndash 1016},
         url={https://doi.org/10.4007/annals.2018.188.3.5},
      review={\MR{3866889}},
}

\bib{G-P92}{article}{
      author={Gross, Benedict~H.},
      author={Prasad, Dipendra},
       title={On the decomposition of a representation of {${\rm SO}_n$} when
  restricted to {${\rm SO}_{n-1}$}},
        date={1992},
        ISSN={0008-414X,1496-4279},
     journal={Canad. J. Math.},
      volume={44},
      number={5},
       pages={974\ndash 1002},
         url={https://doi.org/10.4153/CJM-1992-060-8},
      review={\MR{1186476}},
}

\bib{G-P94}{article}{
      author={Gross, Benedict~H.},
      author={Prasad, Dipendra},
       title={On irreducible representations of {${\rm SO}_{2n+1}\times{\rm
  SO}_{2m}$}},
        date={1994},
        ISSN={0008-414X,1496-4279},
     journal={Canad. J. Math.},
      volume={46},
      number={5},
       pages={930\ndash 950},
         url={https://doi.org/10.4153/CJM-1994-053-4},
      review={\MR{1295124}},
}

\bib{GRS93}{article}{
      author={Gelbart, S.},
      author={Rogawski, J.},
      author={Soudry, D.},
       title={On periods of cusp forms and algebraic cycles for {${\rm
  U}(3)$}},
        date={1993},
        ISSN={0021-2172,1565-8511},
     journal={Israel J. Math.},
      volume={83},
      number={1-2},
       pages={213\ndash 252},
         url={https://doi.org/10.1007/BF02764643},
      review={\MR{1239723}},
}

\bib{G-S12}{article}{
      author={Gan, Wee~Teck},
      author={Savin, Gordan},
       title={Representations of metaplectic groups {I}: epsilon dichotomy and
  local {L}anglands correspondence},
        date={2012},
        ISSN={0010-437X,1570-5846},
     journal={Compos. Math.},
      volume={148},
      number={6},
       pages={1655\ndash 1694},
         url={https://doi.org/10.1112/S0010437X12000486},
      review={\MR{2999299}},
}

\bib{G-T16}{article}{
      author={Gan, Wee~Teck},
      author={Takeda, Shuichiro},
       title={A proof of the {H}owe duality conjecture},
        date={2016},
        ISSN={0894-0347,1088-6834},
     journal={J. Amer. Math. Soc.},
      volume={29},
      number={2},
       pages={473\ndash 493},
         url={https://doi.org/10.1090/jams/839},
      review={\MR{3454380}},
}

\bib{G-Z86}{article}{
      author={Gross, Benedict~H.},
      author={Zagier, Don~B.},
       title={Heegner points and derivatives of {$L$}-series},
        date={1986},
        ISSN={0020-9910,1432-1297},
     journal={Invent. Math.},
      volume={84},
      number={2},
       pages={225\ndash 320},
         url={https://doi.org/10.1007/BF01388809},
      review={\MR{833192}},
}

\bib{Har07}{incollection}{
      author={Harris, Michael},
       title={Cohomological automorphic forms on unitary groups. {II}. {P}eriod
  relations and values of {$L$}-functions},
        date={2007},
   booktitle={Harmonic analysis, group representations, automorphic forms and
  invariant theory},
      series={Lect. Notes Ser. Inst. Math. Sci. Natl. Univ. Singap.},
      volume={12},
   publisher={World Sci. Publ., Hackensack, NJ},
       pages={89\ndash 149},
         url={https://doi.org/10.1142/9789812770790_0005},
      review={\MR{2401812}},
}

\bib{HJL23}{article}{
      author={Hoffstein, Jeff},
      author={Jung, Junehyuk},
      author={Lee, Min},
       title={Non-vanishing of symmetric cube {$L$}-functions},
        date={2023},
        ISSN={0024-6107,1469-7750},
     journal={J. Lond. Math. Soc. (2)},
      volume={107},
      number={1},
       pages={153\ndash 188},
         url={https://doi.org/10.1112/jlms.12683},
      review={\MR{4535011}},
}

\bib{HKS96}{article}{
      author={Harris, Michael},
      author={Kudla, Stephen~S.},
      author={Sweet, William~J.},
       title={Theta dichotomy for unitary groups},
        date={1996},
        ISSN={0894-0347,1088-6834},
     journal={J. Amer. Math. Soc.},
      volume={9},
      number={4},
       pages={941\ndash 1004},
         url={https://doi.org/10.1090/S0894-0347-96-00198-1},
      review={\MR{1327161}},
}

\bib{H-L98}{article}{
      author={Harris, Michael},
      author={Li, Jian-Shu},
       title={A {L}efschetz property for subvarieties of {S}himura varieties},
        date={1998},
        ISSN={1056-3911,1534-7486},
     journal={J. Algebraic Geom.},
      volume={7},
      number={1},
       pages={77\ndash 122},
      review={\MR{1620690}},
}

\bib{How06}{article}{
      author={Howard, Benjamin},
       title={Bipartite {E}uler systems},
        date={2006},
        ISSN={0075-4102,1435-5345},
     journal={J. Reine Angew. Math.},
      volume={597},
       pages={1\ndash 25},
         url={https://doi.org/10.1515/CRELLE.2006.062},
      review={\MR{2264314}},
}

\bib{How89a}{article}{
      author={Howe, Roger},
       title={Remarks on classical invariant theory},
        date={1989},
        ISSN={0002-9947,1088-6850},
     journal={Trans. Amer. Math. Soc.},
      volume={313},
      number={2},
       pages={539\ndash 570},
         url={https://doi.org/10.2307/2001418},
      review={\MR{986027}},
}

\bib{How89}{article}{
      author={Howe, Roger},
       title={Transcending classical invariant theory},
        date={1989},
        ISSN={0894-0347,1088-6834},
     journal={J. Amer. Math. Soc.},
      volume={2},
      number={3},
       pages={535\ndash 552},
         url={https://doi.org/10.2307/1990942},
      review={\MR{985172}},
}

\bib{I-I10}{article}{
      author={Ichino, Atsushi},
      author={Ikeda, Tamutsu},
       title={On the periods of automorphic forms on special orthogonal groups
  and the {G}ross-{P}rasad conjecture},
        date={2010},
        ISSN={1016-443X,1420-8970},
     journal={Geom. Funct. Anal.},
      volume={19},
      number={5},
       pages={1378\ndash 1425},
         url={https://doi.org/10.1007/s00039-009-0040-4},
      review={\MR{2585578}},
}

\bib{Ish24}{article}{
      author={Ishimoto, Hiroshi},
       title={The endoscopic classification of representations of
  non-quasi-split odd special orthogonal groups},
        date={2024},
        ISSN={1073-7928,1687-0247},
     journal={Int. Math. Res. Not. IMRN},
      volume={2024},
      number={14},
       pages={10939\ndash 11012},
         url={https://doi.org/10.1093/imrn/rnae113},
      review={\MR{4776199}},
}

\bib{J-R11}{incollection}{
      author={Jacquet, Herv{\'e}},
      author={Rallis, Stephen},
       title={On the {G}ross-{P}rasad conjecture for unitary groups},
        date={2011},
   booktitle={On certain {$L$}-functions},
      series={Clay Math. Proc.},
      volume={13},
   publisher={Amer. Math. Soc., Providence, RI},
       pages={205\ndash 264},
      review={\MR{2767518}},
}

\bib{J-S81}{article}{
      author={Jacquet, H.},
      author={Shalika, J.~A.},
       title={On {E}uler products and the classification of automorphic forms.
  {II}},
        date={1981},
        ISSN={0002-9327,1080-6377},
     journal={Amer. J. Math.},
      volume={103},
      number={4},
       pages={777\ndash 815},
         url={https://doi.org/10.2307/2374050},
      review={\MR{623137}},
}

\bib{Kat04}{incollection}{
      author={Kato, Kazuya},
       title={{$p$}-adic {H}odge theory and values of zeta functions of modular
  forms},
        date={2004},
       pages={ix, 117\ndash 290},
        note={Cohomologies $p$-adiques et applications arithm\'etiques. III},
      review={\MR{2104361}},
}

\bib{KMSW}{article}{
      author={Kaletha, Tasho},
      author={Minguez, Alberto},
      author={Shin, Sug~Woo},
      author={White, Paul-James},
       title={Endoscopic classification of representations: Inner forms of
  unitary groups},
        date={2014},
     journal={Preprint},
      eprint={1409.3731},
         url={https://arxiv.org/abs/1409.3731},
}

\bib{Kol90}{incollection}{
      author={Kolyvagin, V.~A.},
       title={Euler systems},
        date={1990},
   booktitle={The {G}rothendieck {F}estschrift, {V}ol.\ {II}},
      series={Progr. Math.},
      volume={87},
   publisher={Birkh\"auser Boston, Boston, MA},
       pages={435\ndash 483},
      review={\MR{1106906}},
}

\bib{K-R05}{incollection}{
      author={Kudla, Stephen~S.},
      author={Rallis, Stephen},
       title={On first occurrence in the local theta correspondence},
        date={2005},
   booktitle={Automorphic representations, {$L$}-functions and applications:
  progress and prospects},
      series={Ohio State Univ. Math. Res. Inst. Publ.},
      volume={11},
   publisher={de Gruyter, Berlin},
       pages={273\ndash 308},
         url={https://doi.org/10.1515/9783110892703.273},
      review={\MR{2192827}},
}

\bib{K-R90}{article}{
      author={Kudla, Stephen~S.},
      author={Rallis, Stephen},
       title={Degenerate principal series and invariant distributions},
        date={1990},
        ISSN={0021-2172},
     journal={Israel J. Math.},
      volume={69},
      number={1},
       pages={25\ndash 45},
         url={https://doi.org/10.1007/BF02764727},
      review={\MR{1046171}},
}

\bib{K-R94}{article}{
      author={Kudla, Stephen~S.},
      author={Rallis, Stephen},
       title={A regularized {S}iegel-{W}eil formula: the first term identity},
        date={1994},
        ISSN={0003-486X,1939-8980},
     journal={Ann. of Math. (2)},
      volume={140},
      number={1},
       pages={1\ndash 80},
         url={https://doi.org/10.2307/2118540},
      review={\MR{1289491}},
}

\bib{KSZ21}{article}{
      author={Kisin, Mark},
      author={Shin, Sug~Woo},
      author={Zhu, Yihang},
       title={The stable trace formula for {S}himura varieties of abelian
  type},
        date={2021},
     journal={Preprint},
      eprint={2110.05381},
         url={https://arxiv.org/abs/2110.05381},
}

\bib{Kud86}{article}{
      author={Kudla, Stephen~S.},
       title={On the local theta-correspondence},
        date={1986},
        ISSN={0020-9910,1432-1297},
     journal={Invent. Math.},
      volume={83},
      number={2},
       pages={229\ndash 255},
         url={https://doi.org/10.1007/BF01388961},
      review={\MR{818351}},
}

\bib{Kud94}{article}{
      author={Kudla, Stephen~S.},
       title={Splitting metaplectic covers of dual reductive pairs},
        date={1994},
        ISSN={0021-2172,1565-8511},
     journal={Israel J. Math.},
      volume={87},
      number={1-3},
       pages={361\ndash 401},
         url={https://doi.org/10.1007/BF02773003},
      review={\MR{1286835}},
}

\bib{Li90}{article}{
      author={Li, Jian-Shu},
       title={Theta lifting for unitary representations with nonzero
  cohomology},
        date={1990},
        ISSN={0012-7094,1547-7398},
     journal={Duke Math. J.},
      volume={61},
      number={3},
       pages={913\ndash 937},
         url={https://doi.org/10.1215/S0012-7094-90-06135-6},
      review={\MR{1084465}},
}

\bib{Liu16}{article}{
      author={Liu, Yifeng},
       title={Hirzebruch-{Z}agier cycles and twisted triple product {S}elmer
  groups},
        date={2016},
        ISSN={0020-9910,1432-1297},
     journal={Invent. Math.},
      volume={205},
      number={3},
       pages={693\ndash 780},
         url={https://doi.org/10.1007/s00222-016-0645-9},
      review={\MR{3539925}},
}

\bib{Liu19}{article}{
      author={Liu, Yifeng},
       title={Bounding cubic-triple product {S}elmer groups of elliptic
  curves},
        date={2019},
        ISSN={1435-9855,1435-9863},
     journal={J. Eur. Math. Soc. (JEMS)},
      volume={21},
      number={5},
       pages={1411\ndash 1508},
         url={https://doi.org/10.4171/JEMS/865},
      review={\MR{3941496}},
}

\bib{Liu21}{article}{
      author={Liu, Yifeng},
       title={Fourier-{J}acobi cycles and arithmetic relative trace formula
  (with an appendix by {C}hao {L}i and {Y}ihang {Z}hu)},
        date={2021},
        ISSN={2168-0930,2168-0949},
     journal={Camb. J. Math.},
      volume={9},
      number={1},
       pages={1\ndash 147},
         url={https://doi.org/10.4310/CJM.2021.v9.n1.a1},
        note={Appendix by Chao Li and Yihang Zhu},
      review={\MR{4325259}},
}

\bib{L-L21}{article}{
      author={Li, Chao},
      author={Liu, Yifeng},
       title={Chow groups and {$L$}-derivatives of automorphic motives for
  unitary groups},
        date={2021},
        ISSN={0003-486X,1939-8980},
     journal={Ann. of Math. (2)},
      volume={194},
      number={3},
       pages={817\ndash 901},
         url={https://doi.org/10.4007/annals.2021.194.3.6},
      review={\MR{4334978}},
}

\bib{Lom15}{article}{
      author={Lombardo, Davide},
       title={Bounds for {S}erre's open image theorem for elliptic curves over
  number fields},
        date={2015},
        ISSN={1937-0652,1944-7833},
     journal={Algebra Number Theory},
      volume={9},
      number={10},
       pages={2347\ndash 2395},
         url={https://doi.org/10.2140/ant.2015.9.2347},
      review={\MR{3437765}},
}

\bib{Lon06}{article}{
      author={Longo, Matteo},
       title={On the {Birch} and {Swinnerton-Dyer} conjecture for modular
  elliptic curves over totally real fields},
    language={en},
        date={2006},
     journal={Annales de l'Institut Fourier},
      volume={56},
      number={3},
       pages={689\ndash 733},
         url={https://aif.centre-mersenne.org/articles/10.5802/aif.2197/},
      review={\MR{2244227}},
}

\bib{Lon07}{article}{
      author={Longo, Matteo},
       title={Euler systems obtained from congruences between {H}ilbert modular
  forms},
        date={2007},
     journal={Rendiconti del Seminario Matematico della Università di Padova},
      volume={118},
}

\bib{LTX24}{article}{
      author={Liu, Yifeng},
      author={Tian, Yichao},
      author={Xiao, Liang},
       title={Iwasawa's main conjecture for {R}ankin--{S}elberg motives in the
  anticyclotomic case},
        date={2024},
      eprint={2406.00624},
         url={https://arxiv.org/abs/2406.00624},
}

\bib{LTXZZ}{article}{
      author={Liu, Yifeng},
      author={Tian, Yichao},
      author={Xiao, Liang},
      author={Zhang, Wei},
      author={Zhu, Xinwen},
       title={On the {B}eilinson--{B}loch--{K}ato conjecture for
  {R}ankin--{S}elberg motives},
        date={2022},
        ISSN={0020-9910,1432-1297},
     journal={Invent. Math.},
      volume={228},
      number={1},
       pages={107\ndash 375},
         url={https://doi.org/10.1007/s00222-021-01088-4},
      review={\MR{4392458}},
}

\bib{LTXZZa}{article}{
      author={Liu, Yi~Feng},
      author={Tian, Yi~Chao},
      author={Xiao, Liang},
      author={Zhang, Wei},
      author={Zhu, Xin~Wen},
       title={Deformation of rigid conjugate self-dual {G}alois
  representations},
        date={2024},
        ISSN={1439-8516,1439-7617},
     journal={Acta Math. Sin. (Engl. Ser.)},
      volume={40},
      number={7},
       pages={1599\ndash 1644},
         url={https://doi.org/10.1007/s10114-024-1409-x},
      review={\MR{4777059}},
}

\bib{LTXZZb}{article}{
      author={Liu, Yifeng},
      author={Tian, Yichao},
      author={Xiao, Liang},
      author={Zhang, Wei},
      author={Zhu, Xinwen},
       title={Survey on bounding {S}elmer groups for {R}ankin--{S}elberg
  motives},
        date={2025},
      eprint={2509.16881},
         url={https://arxiv.org/abs/2509.16881},
}

\bib{LXZ25}{article}{
      author={Leslie, Spencer},
      author={Xiao, Jingwei},
      author={Zhang, Wei},
       title={Unitary {F}riedberg-{J}acquet periods and their twists:
  Fundamental lemmas},
        date={2025},
      eprint={2503.09500},
         url={https://arxiv.org/abs/2503.09500},
}

\bib{LXZ25a}{article}{
      author={Leslie, Spencer},
      author={Xiao, Jingwei},
      author={Zhang, Wei},
       title={Unitary {F}riedberg-{J}acquet periods and their twists: Relative
  trace formulas},
        date={2025},
      eprint={2503.09664},
         url={https://arxiv.org/abs/2503.09664},
}

\bib{Min12}{article}{
      author={M\'inguez, Alberto},
       title={The conservation relation for cuspidal representations},
        date={2012},
        ISSN={0025-5831,1432-1807},
     journal={Math. Ann.},
      volume={352},
      number={1},
       pages={179\ndash 188},
         url={https://doi.org/10.1007/s00208-011-0636-5},
      review={\MR{2885581}},
}

\bib{Mok15}{article}{
      author={Mok, Chung~Pang},
       title={Endoscopic classification of representations of quasi-split
  unitary groups},
        date={2015},
        ISSN={0065-9266,1947-6221},
     journal={Mem. Amer. Math. Soc.},
      volume={235},
      number={1108},
       pages={vi+248},
         url={https://doi.org/10.1090/memo/1108},
      review={\MR{3338302}},
}

\bib{M-R18}{incollection}{
      author={Moeglin, Colette},
      author={Renard, David},
       title={Sur les paquets d'{A}rthur des groupes classiques et unitaires
  non quasi-d\'eploy\'es},
        date={2018},
   booktitle={Relative aspects in representation theory, {L}anglands
  functoriality and automorphic forms},
      series={Lecture Notes in Math.},
      volume={2221},
   publisher={Springer, Cham},
       pages={341\ndash 361},
      review={\MR{3839702}},
}

\bib{MVW87}{book}{
      author={Moeglin, Colette},
      author={Vign\'eras, Marie-France},
      author={Waldspurger, Jean-Loup},
       title={Correspondances de {H}owe sur un corps {$p$}-adique},
      series={Lecture Notes in Mathematics},
   publisher={Springer-Verlag, Berlin},
        date={1987},
      volume={1291},
        ISBN={3-540-18699-9},
         url={https://doi.org/10.1007/BFb0082712},
      review={\MR{1041060}},
}

\bib{Nek00}{incollection}{
      author={Nekov\'a{\v r}, Jan},
       title={{$p$}-adic {A}bel-{J}acobi maps and {$p$}-adic heights},
        date={2000},
   booktitle={The arithmetic and geometry of algebraic cycles ({B}anff, {AB},
  1998)},
      series={CRM Proc. Lecture Notes},
      volume={24},
   publisher={Amer. Math. Soc., Providence, RI},
       pages={367\ndash 379},
         url={https://doi.org/10.1090/crmp/024/18},
      review={\MR{1738867}},
}

\bib{Nek12}{article}{
      author={Nekov\'a{\v r}, Jan},
       title={Level raising and anticyclotomic {S}elmer groups for {H}ilbert
  modular forms of weight two},
        date={2012},
        ISSN={0008-414X,1496-4279},
     journal={Canad. J. Math.},
      volume={64},
      number={3},
       pages={588\ndash 668},
         url={https://doi.org/10.4153/CJM-2011-077-6},
      review={\MR{2962318}},
}

\bib{N-N16}{article}{
      author={Nekov{\'a}{\v r}, Jan},
      author={Nizio{\l}, Wies{\l}awa},
       title={Syntomic cohomology and {$p$}-adic regulators for varieties over
  {$p$}-adic fields},
        date={2016},
        ISSN={1937-0652,1944-7833},
     journal={Algebra Number Theory},
      volume={10},
      number={8},
       pages={1695\ndash 1790},
         url={https://doi.org/10.2140/ant.2016.10.1695},
        note={With appendices by Laurent Berger and Fr\'ed\'eric D\'eglise},
      review={\MR{3556797}},
}

\bib{N-T21}{article}{
      author={Newton, James},
      author={Thorne, Jack~A.},
       title={Symmetric power functoriality for holomorphic modular forms,
  {II}},
        date={2021},
        ISSN={0073-8301,1618-1913},
     journal={Publ. Math. Inst. Hautes \'Etudes Sci.},
      volume={134},
       pages={117\ndash 152},
         url={https://doi.org/10.1007/s10240-021-00126-4},
      review={\MR{4349241}},
}

\bib{N-T22}{article}{
      author={Newton, James},
      author={Thorne, Jack~A.},
       title={Symmetric power functoriality for {H}ilbert modular forms},
        date={2022},
     journal={Preprint},
      eprint={2212.03595},
         url={https://arxiv.org/abs/2212.03595},
}

\bib{Pen25a}{article}{
      author={Peng, Hao},
       title={Fargues-{S}cholze parameters and torsion vanishing for special
  orthogonal and unitary groups},
        date={2025},
     journal={Preprint},
      eprint={2503.04623},
         url={https://arxiv.org/abs/2503.04623},
}

\bib{Pen26}{article}{
      author={Peng, Hao},
       title={Higher {K}olyvagin theorems in the arithmetic {G}ross--{P}rasad
  setting},
        date={2026},
     journal={Ph.D. thesis to appear},
}

\bib{Pra07}{article}{
      author={Prasad, Dipendra},
       title={Relating invariant linear form and local epsilon factors via
  global methods},
        date={2007},
        ISSN={0012-7094,1547-7398},
     journal={Duke Math. J.},
      volume={138},
      number={2},
       pages={233\ndash 261},
         url={https://doi.org/10.1215/S0012-7094-07-13823-7},
        note={With an appendix by Hiroshi Saito},
      review={\MR{2318284}},
}

\bib{Pra99}{article}{
      author={Prasad, Dipendra},
       title={Some remarks on representations of a division algebra and of the
  {G}alois group of a local field},
        date={1999},
        ISSN={0022-314X,1096-1658},
     journal={J. Number Theory},
      volume={74},
      number={1},
       pages={73\ndash 97},
         url={https://doi.org/10.1006/jnth.1998.2289},
      review={\MR{1670568}},
}

\bib{Sha79}{incollection}{
      author={Piatetski-Shapiro, I.~I.},
       title={Multiplicity one theorems},
        date={1979},
   booktitle={Automorphic forms, representations and {$L$}-functions ({P}roc.
  {S}ympos. {P}ure {M}ath., {O}regon {S}tate {U}niv., {C}orvallis, {O}re.,
  1977), {P}art 1},
      series={Proc. Sympos. Pure Math.},
      volume={XXXIII},
   publisher={Amer. Math. Soc., Providence, RI},
       pages={209\ndash 212},
      review={\MR{546599}},
}

\bib{Rog92}{incollection}{
      author={Rogawski, Jonathan~D.},
       title={Analytic expression for the number of points mod {$p$}},
        date={1992},
   booktitle={The zeta functions of {P}icard modular surfaces},
   publisher={Univ. Montr\'eal, Montreal, QC},
       pages={65\ndash 109},
      review={\MR{1155227}},
}

\bib{Roh94}{incollection}{
      author={Rohrlich, David~E.},
       title={Elliptic curves and the {W}eil-{D}eligne group},
        date={1994},
   booktitle={Elliptic curves and related topics},
      series={CRM Proc. Lecture Notes},
      volume={4},
   publisher={Amer. Math. Soc., Providence, RI},
       pages={125\ndash 157},
         url={https://doi.org/10.1090/crmp/004/10},
      review={\MR{1260960}},
}

\bib{Rao93}{article}{
      author={Ranga~Rao, R.},
       title={On some explicit formulas in the theory of {W}eil
  representation},
        date={1993},
        ISSN={0030-8730,1945-5844},
     journal={Pacific J. Math.},
      volume={157},
      number={2},
       pages={335\ndash 371},
         url={http://projecteuclid.org/euclid.pjm/1102634748},
      review={\MR{1197062}},
}

\bib{Rub00}{book}{
      author={Rubin, Karl},
       title={Euler systems},
      series={Annals of Mathematics Studies},
   publisher={Princeton University Press, Princeton, NJ},
        date={2000},
      volume={147},
        ISBN={0-691-05075-9; 0-691-05076-7},
         url={https://doi.org/10.1515/9781400865208},
        note={Hermann Weyl Lectures. The Institute for Advanced Study},
      review={\MR{1749177}},
}

\bib{R-Y23}{article}{
      author={Radziwi{\l}{\l}, Maksym},
      author={Yang, Liyang},
       title={Non-vanishing of twists of $\text{GL}_4(\mathbb{A}_{\mathbb{q}})$
  $l$-functions},
        date={2023},
      eprint={2304.09171},
         url={https://arxiv.org/abs/2304.09171},
}

\bib{Ser72}{article}{
      author={Serre, Jean-Pierre},
       title={Propri\'et\'es galoisiennes des points d'ordre fini des courbes
  elliptiques},
        date={1972},
        ISSN={0020-9910,1432-1297},
     journal={Invent. Math.},
      volume={15},
      number={4},
       pages={259\ndash 331},
         url={https://doi.org/10.1007/BF01405086},
      review={\MR{387283}},
}

\bib{S-T14}{article}{
      author={Shin, Sug~Woo},
      author={Templier, Nicolas},
       title={On fields of rationality for automorphic representations},
        date={2014},
        ISSN={0010-437X,1570-5846},
     journal={Compos. Math.},
      volume={150},
      number={12},
       pages={2003\ndash 2053},
         url={https://doi.org/10.1112/S0010437X14007428},
      review={\MR{3292292}},
}

\bib{Swe25}{article}{
      author={Sweeting, Naomi},
       title={On the {B}loch-{K}ato conjecture for some four-dimensional
  symplectic {G}alois representations},
        date={2025},
     journal={Preprint},
      eprint={2503.19226},
         url={https://arxiv.org/abs/2503.19226},
}

\bib{S-Z15}{article}{
      author={Sun, Binyong},
      author={Zhu, Chen-Bo},
       title={Conservation relations for local theta correspondence},
        date={2015},
        ISSN={0894-0347,1088-6834},
     journal={J. Amer. Math. Soc.},
      volume={28},
      number={4},
       pages={939\ndash 983},
         url={https://doi.org/10.1090/S0894-0347-2014-00817-1},
      review={\MR{3369906}},
}

\bib{T-Y07}{article}{
      author={Taylor, Richard},
      author={Yoshida, Teruyoshi},
       title={Compatibility of local and global {L}anglands correspondences},
        date={2007},
        ISSN={0894-0347,1088-6834},
     journal={J. Amer. Math. Soc.},
      volume={20},
      number={2},
       pages={467\ndash 493},
         url={https://doi.org/10.1090/S0894-0347-06-00542-X},
      review={\MR{2276777}},
}

\bib{Wal10a}{article}{
      author={Waldspurger, J.-L.},
       title={Une formule int\'egrale reli\'ee \`a{} la conjecture locale de
  {G}ross-{P}rasad},
        date={2010},
        ISSN={0010-437X,1570-5846},
     journal={Compos. Math.},
      volume={146},
      number={5},
       pages={1180\ndash 1290},
         url={https://doi.org/10.1112/S0010437X10004744},
      review={\MR{2684300}},
}

\bib{Wal12}{incollection}{
      author={Waldspurger, Jean-Loup},
       title={Calcul d'une valeur d'un facteur {$\epsilon$} par une formule
  int\'egrale},
        date={2012},
       pages={1\ndash 102},
        note={Sur les conjectures de Gross et Prasad. II},
      review={\MR{3155344}},
}

\bib{Wal12a}{incollection}{
      author={Waldspurger, Jean-Loup},
       title={La conjecture locale de {G}ross-{P}rasad pour les
  repr\'esentations temp\'er\'ees des groupes sp\'eciaux orthogonaux},
        date={2012},
       pages={103\ndash 165},
        note={Sur les conjectures de Gross et Prasad. II},
      review={\MR{3155345}},
}

\bib{Wal12b}{incollection}{
      author={Waldspurger, Jean-Loup},
       title={Une formule int\'egrale reli\'ee \`a{} la conjecture locale de
  {G}ross-{P}rasad, 2e partie: extension aux repr\'esentations temp\'er\'ees},
        date={2012},
       pages={171\ndash 312},
        note={Sur les conjectures de Gross et Prasad. I},
      review={\MR{3202558}},
}

\bib{Wal90}{incollection}{
      author={Waldspurger, J.-L.},
       title={D\'emonstration d'une conjecture de dualit\'e{} de {H}owe dans le
  cas {$p$}-adique, {$p\neq 2$}},
        date={1990},
   booktitle={Festschrift in honor of {I}. {I}. {P}iatetski-{S}hapiro on the
  occasion of his sixtieth birthday, {P}art {I} ({R}amat {A}viv, 1989)},
      series={Israel Math. Conf. Proc.},
      volume={2},
   publisher={Weizmann, Jerusalem},
       pages={267\ndash 324},
      review={\MR{1159105}},
}

\bib{Wan22a}{article}{
      author={Wang, Haining},
       title={Arithmetic level raising on triple product of {S}himura curves
  and {G}ross-{S}choen diagonal cycles, {I}: {R}amified case},
        date={2022},
        ISSN={1937-0652,1944-7833},
     journal={Algebra Number Theory},
      volume={16},
      number={10},
       pages={2289\ndash 2338},
         url={https://doi.org/10.2140/ant.2022.16.2289},
      review={\MR{4546490}},
}

\bib{Xue14}{article}{
      author={Xue, Hang},
       title={The {G}an-{G}ross-{P}rasad conjecture for {${\rm U}(n)\times{\rm
  U}(n)$}},
        date={2014},
        ISSN={0001-8708,1090-2082},
     journal={Adv. Math.},
      volume={262},
       pages={1130\ndash 1191},
         url={https://doi.org/10.1016/j.aim.2014.06.010},
      review={\MR{3228451}},
}

\bib{Yam14}{article}{
      author={Yamana, Shunsuke},
       title={L-functions and theta correspondence for classical groups},
        date={2014},
        ISSN={0020-9910,1432-1297},
     journal={Invent. Math.},
      volume={196},
      number={3},
       pages={651\ndash 732},
         url={https://doi.org/10.1007/s00222-013-0476-x},
      review={\MR{3211043}},
}

\bib{Y-Z25}{article}{
      author={Yang, Xiangqian},
      author={Zhu, Xinwen},
       title={On the generic part of the cohomology of shimura varieties of
  abelian type},
        date={2025},
      eprint={2505.04329},
         url={https://arxiv.org/abs/2505.04329},
}

\bib{Zan24}{book}{
      author={Zanarella, Murilo~Corato},
       title={First {E}xplicit {R}eciprocity {L}aw for {U}nitary
  {F}riedberg---{J}acquet {P}eriods},
   publisher={ProQuest LLC, Ann Arbor, MI},
        date={2024},
        ISBN={979-8304-95316-0},
  url={https://gateway.proquest.com/openurl?url_ver=Z39.88-2004&rft_val_fmt=info:ofi/fmt:kev:mtx:dissertation&res_dat=xri:pqm&rft_dat=xri:pqdiss:31850998},
        note={Thesis (Ph.D.)--Massachusetts Institute of Technology},
      review={\MR{4890378}},
}

\bib{Zha01}{article}{
      author={Zhang, Shouwu},
       title={Heights of {H}eegner points on {S}himura curves},
        date={2001},
        ISSN={0003-486X,1939-8980},
     journal={Ann. of Math. (2)},
      volume={153},
      number={1},
       pages={27\ndash 147},
         url={https://doi.org/10.2307/2661372},
      review={\MR{1826411}},
}

\bib{Zha14}{article}{
      author={Zhang, Wei},
       title={Fourier transform and the global {G}an-{G}ross-{P}rasad
  conjecture for unitary groups},
        date={2014},
        ISSN={0003-486X,1939-8980},
     journal={Ann. of Math. (2)},
      volume={180},
      number={3},
       pages={971\ndash 1049},
         url={https://doi.org/10.4007/annals.2014.180.3.4},
      review={\MR{3245011}},
}

\end{biblist}
\end{bibdiv}

\vspace{2em}
\noindent\textsc{Hao Peng}\\
Department of Mathematics, Massachusetts Institute of Technology, Cambridge, MA 02139, USA\\
\textit{Email}: \texttt{hao\_peng@mit.edu}

\end{document}